\documentclass[11pt]{amsart}

\usepackage{amsfonts}
\usepackage{amssymb}
\usepackage{graphicx}
\usepackage{amsmath}
\usepackage{latexsym}
\usepackage{amscd}
\usepackage{xypic}
\usepackage[all,cmtip]{xy}
\usepackage{mathrsfs}
\usepackage{enumitem} 
\usepackage{braket}
\usepackage[margin=1.3in]{geometry}
\usepackage{comment}
\usepackage{bm}
\usepackage{cancel}

\usepackage{tikz} \usetikzlibrary{arrows,snakes,shapes,decorations.markings,patterns,cd}
\usepackage{imakeidx}
\makeindex[title=Index of Notation]

\usepackage{hyperref}

\allowdisplaybreaks 

\numberwithin{equation}{section}
\newtheorem{theorem}{Theorem}[section]
\newtheorem{lemma}[theorem]{Lemma}
\newtheorem{corollary}[theorem]{Corollary}

\newtheorem{proposition}[theorem]{Proposition}

\newtheorem*{theorem*}{Theorem}

\theoremstyle{definition}
\newtheorem{definition}[theorem]{Definition}

\theoremstyle{remark}
\newtheorem*{remark}{Remark}

\theoremstyle{remark}

\theoremstyle{definition}

\newcommand{\NN}{\mathbb{N}}

\newcommand{\QQ}{\mathbb{Q}}
\newcommand{\RR}{\mathbb{R}}
\renewcommand{\SS}{\mathbb{S}}

\newcommand{\ZZ}{\mathbb{Z}}

\newcommand{\cB}{\mathcal B}
\newcommand{\cC}{\mathcal C}

\newcommand{\cF}{\mathcal F}
\newcommand{\cG}{\mathcal G}
\renewcommand{\cH}{\mathcal H}

\newcommand{\cK}{\mathcal K}

\newcommand{\cM}{\mathcal M}
\newcommand{\cN}{\mathcal N}

\newcommand{\cS}{\mathcal S}
\newcommand{\cT}{\mathcal T}
\newcommand{\cU}{\mathcal U}

\newcommand{\be}{\mathbf{e}}
\newcommand{\bE}{\mathbf{E}}
\newcommand{\bp}{\mathbf{p}}

\newcommand{\bH}{\mathbf{H}}

\newcommand{\bF}{\mathbf{F}}

\newcommand{\bW}{\mathbf{W}}
\newcommand{\bx}{\mathbf{x}}
\newcommand{\bX}{\mathbf{X}}
\newcommand{\by}{\mathbf{y}}
\newcommand{\bY}{\mathbf{Y}}

\newcommand{\bZ}{\mathbf{Z}}
\newcommand{\bOh}{\mathbf{0}}

\newcommand{\ft}{\mathfrak{t}}
\newcommand{\fS}{\mathfrak{S}}

\DeclareMathOperator{\tr}{tr}

\DeclareMathOperator{\Span}{span}

\DeclareMathOperator{\supp}{supp}

\newcommand{\bangle}[1]{\left\langle #1 \right\rangle}

\DeclareMathOperator{\Ric}{Ric}

\DeclareMathOperator{\Id}{Id} 

\newcommand{\eps}{\varepsilon}

\DeclareMathOperator{\sing}{sing}
\DeclareMathOperator{\reg}{reg}
\DeclareMathOperator{\genus}{genus}
\DeclareMathOperator{\rank}{rank}

\DeclareMathOperator{\coker}{coker}
\DeclareMathOperator{\image}{image}

\DeclareMathOperator{\Graph}{graph}

\DeclareMathOperator{\Div}{div}

\newcommand{\bomega}{{\boldsymbol\omega}}

\title{Mean curvature flow with generic initial data II}

\author{Otis Chodosh} 
\address{OC: Department of Mathematics, Bldg.\ 380, Stanford University, Stanford, CA 94305, USA}
\email{ochodosh@stanford.edu}

\author{Kyeongsu Choi}
\address{KC: School of Mathematics, Korea Institute for Advanced Study, 85 Hoegiro, Dongdaemun-gu, Seoul 02455, Republic of Korea}
\email{choiks@kias.re.kr}
\author{Felix Schulze}
\address{FS: Department of Mathematics, Zeeman Building, University of Warwick, Gibbet Hill Road, Coventry CV7 4AL,
UK}
\email{felix.schulze@warwick.ac.uk} 

\begin{document}

\begin{abstract}
We show that the mean curvature flow of a generic closed surface in $\RR^3$ avoids multiplicity one tangent flows that are not round spheres/cylinders. In particular, we show that any non-cylindrical self-shrinker with a cylindrical end cannot  arise generically. 
\end{abstract}

\maketitle

\tableofcontents

\newpage


\section{Introduction}

Mean curvature flow is the natural heat equation for submanifolds. A family of surfaces $M(t) \subset \RR^{3}$ flows by mean curvature flow if 
\begin{equation} \label{eq:mean-curvature-flow}
\left(\tfrac{\partial}{\partial t} \bx \right)^{\perp} = \bH_{M(t)}(\bx),
\end{equation}
where $\bH_{M(t)}(\bx)$ denotes the mean curvature vector of $M(t)$ at $\bx$. When the initial surface $M(0)$ is compact, mean curvature flow is guaranteed to become singular in finite time. The simplest way to analyze such a singularity is to parabolically dilate around a singular point in space-time. Huisken's monotonicity formula \cite{Huisken:sing} (cf.~\cite{Ilmanen:singularities}) guarantees that a subsequential limit of such dilations will weakly limit to a \emph{tangent flow} which will be a weak solution to \eqref{eq:mean-curvature-flow}, evolving only by homothety. Ilmanen has showed that such a tangent flow will be supported on a smooth surface $\Sigma\subset \RR^3$ so that $(-\infty,0) \mapsto \sqrt{-t}\Sigma$ is a solution to \eqref{eq:mean-curvature-flow}. Such a surface is called a \emph{self-shrinker} (see Section \ref{subsec:convent}). 

A deeper understanding of such tangent flows is necessary in order to continue the flow past the onset of singularities (either by constructing a flow with surgery, or by showing that weak solutions to the flow have good partial regularity and well-posedness properties). The two main obstructions to doing so are the potential presence of \emph{multiplicity} and of tangent flows supported on self-shrinkers whose scalar mean curvature \emph{changes sign}. By a result of Huisken \cite{Huisken:sing} (cf.\ \cite{ColdingMinicozzi:generic}) a shrinker whose mean curvature does not change sign is either a plane or a round sphere or cylinder. Such singularities (when they occur with multiplicity one) are known to be well-behaved, thanks to regularity results of Colding--Minicozzi \cite{ColdingMinicozzi:sing-generic} (cf.\ \cite{White:stratification,ColdingIlmanenMinicozzi,ColdingMinicozzi:uniqueness-tangent-flow}) and the resolution of the mean convex neighborhood conjecture by Choi--Haslhofer--Hershkovits \cite{ChoiHaslhoferHershkovits}. 

In this paper, we show that the second obstruction does not occur for \emph{generic} initial data $M(0)\subset \RR^3$. Our main result is roughly as follows (see Section \ref{subsec:sketch-genus-drop} for a more detailed overview and Theorem \ref{theo:generic.flow.R3} for the precise statement):
\begin{quote}
\emph{The mean curvature flow of a generic initial surface $M(0)\subset \RR^3$ encounters only spherical and cylindrical singularities at least until multiplicity occurs for some tangent flow. In particular, up to the first time multiplicity occurs, the flow is well-posed and is completely smooth for almost every time. }
\end{quote}
This confirms---up to the potential occurrence of multiplicity---a long-standing conjecture of Huisken \cite[\# 8]{Ilmanen:problems}. 

The well-known \emph{multiplicity-one conjecture} posits that multiplicity never occurs, but this is widely open even for generic initial data (we note that considerable progress has been made under additional geometric assumptions on $M(0)$ or the flow; see \cite{White:size,ShengWang,Andrews:noncollapse,Lin:star,Brendle:inscribed-sharp,LiWang1,LiWang2,BrendleNaff:noncollapse}). We review previous work related to generic mean curvature flow in Section \ref{subsec:related-work} below. 

There is strong evidence that the \emph{generic} hypothesis in the above result is necessary. Indeed, Ilmanen--White have indicated \cite{white:ICM} a construction of a closed surface $M(0)\subset \RR^3$ whose mean curvature flow is not well-posed after the onset of singularities. (On the other hand, if $M(0)\subset \RR^3$ has genus zero, an important work of Brendle \cite{Brendle:genus0} shows that the statement above holds without ``generic.'')

\begin{remark} 
By combining our main result with a surgery construction by Daniels-Holgate \cite{Daniels-Holgate}, one can construct a mean curvature flow with surgery for a generic initial $M(0)\subset \RR^3$ until the first time multiplicity occurs for the weak flow above. 
\end{remark}

\subsection{The no-cylinder conjecture} 
At the first time singularities appear (or more generally, the first time that some singularity other than multiplicity-one spherical/cylindrical singularities occur) we have seen that any tangent flow is supported on a smooth self-shrinker $\Sigma$ (possibly with multiplicity). A fundamental result of Wang \cite{Wang:ends-conical} shows that if $\Sigma\subset\RR^3$ is a non-compact self-shrinker then any end of $\Sigma$ is smoothly asymptotic---with multiplicity one---to a smooth cone or a cylinder. 

Ilmanen has asked if it is possible that a non-cylindrical shrinker has a cylindrical end (the \emph{no-cylinder conjecture} \cite[\# 12]{Ilmanen:problems}). See \cite{Wang:cylindrical-ends-unique} for some partial progress towards non-existence. Moreover, we note that are by now many constructions of self-shrinkers (both numerical and rigorous) \cite{Angenent:doughnuts,Chopp,KKM:AC,Nguyen:AC,Ketover:self-shrinkers,SunWangZhou,BuzanoNguyenSchulz} but no example with a cylindrical end has been found.

The main new ingredient this paper can be stated as follows:
\begin{quote}
 \emph{Even if a non-cylindrical shrinker with a cylindrical end exists, it does not arise generically (with multiplicity one at the first singular time).}
 \end{quote}
 This builds on our  previous work with Mantoulidis \cite{CCMS:generic1} where we proved that asymptotically conical and compact shrinkers do not arise generically, but left open the possibility of a shrinker with a cylindrical end. As we discuss below, the analysis in the cylindrical end case is significantly complicated by the \emph{a priori} lack of graphicality of the corresponding ancient flow along the cylindrical end.

\subsection{Generic initial data and one-sided ancient flows}
In our previous paper with Mantoulidis \cite{CCMS:generic1}, we developed a new approach to the study of the mean curvature flow of generic initial data by relating this to the classification of certain ancient solutions. (One may view this as the \emph{parabolic} analogue of the Hardt--Simon foliation and corresponding generic regularity results for area minimizing hypersurfaces in $\RR^8$ \cite{HardtSimon:foliation,Smale,Wang:HSfol}, cf.\ \cite{CMS:910}.) 

\subsubsection{One-sided flows as local models for generic mean curvature flow} The basic strategy is as follows (see also \cite[\S 1.5]{CCMS:generic1} for a more detailed exposition of the method). Embed the initial surface $M_0(0)\subset \RR^3$ in a local foliation $M_s(0)$ and flow all of the surface simultaneously by mean curvature flow $M_s(t)$. Suppose that $M_0(t)$ encounters a singularity at time $T$ at point $\bx \in \RR^3$. (Let us assume that this is the first singular time, but in reality we can work with the first time with a singularity other a multiplicity one sphere/cylinder.) 

As explained above, the parabolically dilated flows 
\[
\lambda(M_0(T + \lambda^{-2}t  ) - \bx) : = M_0^\lambda(t)
\]
weakly limit subsequentially (as $\lambda\to\infty$) to the shrinking flow $t\mapsto \sqrt{-t}\Sigma$, $t<0$, associated to a smooth self-shrinker $\Sigma$. (Of course, this limit could also occur with some integer multiplicity $k\geq 2$; we will assume that $k=1$ below.)

The key observation in \cite{CCMS:generic1} is that if we \emph{simultaneously} dilate the entire foliation 
\[
\lambda(M_s(T + \lambda^{-2}t  ) - \bx) : = M_s^\lambda(t)
\]
then by choosing a subsequence of $s\to 0$ appropriately with $\lambda\to\infty$ we can pass $M_s^\lambda(t)$ to the limit to find $\check M(t)$ a (nontrivial) ancient mean curvature flow that's disjoint from $t\mapsto\sqrt{-t}\Sigma$ (we call such a flow \emph{one-sided}). The main strategy is then to show that $\check M(t)$ has \emph{improved} behavior as compared to $M_0(t)$ and $t\mapsto\sqrt{-t}\Sigma$. This is achieved by showing (see Section \ref{subsec:sketch}  for a more detailed discussion) that up to a parabolic dilation, there exists a \emph{unique} ancient one-sided flow $\overline M(t)$ so that $\tfrac{1}{\sqrt{-t}}\overline M(t)$ is close (in a $C^\infty_\textrm{loc}$ sense, with multiplicity one) to $\Sigma$ as $t\to-\infty$ (existence is proven in Proposition \ref{prop:exist-ancient-rescaled-MCF-C2-rescaling-argument} and uniqueness is proven in Theorem \ref{thm:uniqueness-global}) and moreover $\overline M(t)$ has the following improved properties (see Propositions \ref{coro:summary-tleq0-non-rescaled} and \ref{prop:genus-drop-one-sided-flow} )
\begin{enumerate}
\item $\overline M(t)$ is a shrinker mean-convex flow; geometrically this means that the space-time track is parabolically star-shaped\footnote{Compare with the ``elliptic'' star-shapedness of the Hardt--Simon foliation. } around $(\bOh,0)$,
\item as a consequence of the shrinker mean-convexity, all singularities of $\overline M(t)$ are spherical/cylindrical with multiplicity-one, 
\item for $t\ll0$, $\frac{1}{\sqrt{-t}}\overline M(t)$ is a small $C^\infty$ graph over $\Sigma$, and 
\item for $t<0$ sufficiently close to $0$, $\overline M(t)$ has genus zero. 
\end{enumerate}
Properties (1) and (2) can be used to show that for $s\sim 0$ appropriately chosen, the flow $M_s(t)$ will only have spherical/cylindrical singularities near $(\bx,T)$. Thus, if we assumed that the tangent flow to $M_0(t)$ at $(\bx,T)$ was supported on a more complicated shrinker, the flow $M_s(t)$ has improved \emph{locally}. Of course, this does not rule out the possibility of $M_s(t)$ having other such singularities elsewhere. 

To ensure that this improvement is actually reflected \emph{globally} it is crucial that we only have to iterate this perturbation process finitely many times (if we perturb infinitely many times, smooth regions that we already controlled might become singular in the limit). This is guaranteed by combining (3) and (4) with the classification of genus zero self-shrinkers due to Brendle \cite{Brendle:genus0}. Indeed, at any non-spherical/cylindrical singularity, \cite{Brendle:genus0} and (3) implies that $\overline M(t)$ has non-trivial genus for $t \ll 0$. In particular, some handle of $M_s(t)$ ``concentrates'' near $(\bx,T)$. However, handle pinches off before time $T$, since by (4) it does so for the local model $\overline M(t)$. Along with a result of White \cite{White:topology-weak} stating that genus is monotone non-increasing along mean curvature flow, this implies that we need only perturb finitely many times. 

\subsubsection{Existence/uniqueness of ancient flows} Fix a self-shrinker $\Sigma\subset \RR^3$. In the above sketch, we crucially used the existence of a one-sided ancient flow with properties (1)-(4) as well as the uniqueness of this flow. The basic mechanism involves the spectral properties of the $L$-operator
\begin{equation}\label{eq:intro-L-operator}
L u = \Delta_\Sigma u - \frac 12 \bx^T \cdot \nabla_\Sigma u + |A_\Sigma|^2 u + \frac 12 u
\end{equation}
obtained by linearizing $\bH + \tfrac 12 \bx^\perp$ at $\Sigma$ (with respect to normal variations). In their foundational work \cite{ColdingMinicozzi:generic}, Colding--Minicozzi introduced the $L$-operator and proved (among many other things) that the first eigenvalue $\mu$ of $L$ (with respect to the Gaussian weighted space) is negative. 

We can thus give an indication at the linear level why the existence and uniqueness result should hold. After an appropriate rescaling, an ancient solution to (rescaled\footnote{Given $N(t)$ a solution to mean curvature flow, the rescaled solution around $(\bx,T)$ is the one-parameter family of surfaces $\tau\mapsto e^{\tau/2}(N(T-e^{-\tau}) -\bx)$. It is easily checked that this procedure (around $(\bOh,0)$) converts the shrinking self-shrinker $t\mapsto \sqrt{-t}\Sigma$ into a static $\tau\mapsto\Sigma$. In general, the rescaled family will solve $\left(\tfrac{\partial}{\partial t} \bx \right)^{\perp} = \bH(\bx) + \tfrac 12 \bx^\perp$ (compare with \eqref{eq:mean-curvature-flow}). }) mean curvature flow that's close to $\Sigma$ can be modeled on an ancient solution to the parabolic equation
\begin{equation}\label{eq:par-L-eqn-intro}
\partial_\tau u = L u
\end{equation}
with $u\to 0$ as $\tau \to-\infty$. One-sidedness is reflected at the linear level by $u>0$ and shrinker mean convexity is reflected at the linear level by $\partial_\tau u \geq 0$ (shrinker mean convexity is equivalent to the rescaled mean curvature flow moving in one direction). 

For ``existence,'' we have the solution $e^{-\mu\tau}\varphi$, for $\varphi$ the first eigenfunction. Since the first eigenfunction does not change sign, we see that this is a ``one-sided'' solution and since $\mu<0$ we see that this is a ``shrinker mean convex'' solution. Conversely, for ``uniqueness'' we note that by a straightforward linear analysis, any ancient solution to \eqref{eq:par-L-eqn-intro} will be of the form
\begin{equation}\label{eq:intro-lin-analysis-stupid}
u(\bx,t) = c e^{-\mu\tau}\varphi +  \sum_{\mu < \lambda < 0} e^{-\lambda\tau} \varphi_\lambda(\bx)
\end{equation}
where the sum is taken over eigenvalues of $L$ and $\varphi_\lambda$ is a some corresponding eigenfunction (each of which must change sign, by the usual characterization of the first eigenfunction). Since $e^{-\lambda\tau} \gg e^{-\mu\tau}$ for $\tau\to-\infty$, we see that if $u$ is ``one-sided'' (i.e., $u>0$) then it holds that $u(\bx,t) = c e^{-\mu\tau}\varphi$, giving uniqueness in the linear setting. 

Of course, in the actual problem one must account for the nonlinearity of the equation. To prove existence we rely on a new upper/lower barrier construction for the ancient flow (cf.\ Lemma \ref{lemm:sub-super-one-sided}). (This allows for a somewhat more direct proof that the one-sided flow is shrinker mean convex, cf.\ \cite[\S 7]{CCMS:generic1}.) 

For uniqueness, if we were to assume \emph{a priori} that the unknown one-sided rescaled flow is an \emph{entire} graph (with good decay properties) over $\Sigma$ for very negative times, then a Merle--Zaag type analysis readily yields uniqueness, cf.\ Proposition \ref{prop:cond.unique.2}. However, establishing this graphicality represents a major issue in this work. Indeed, in contrast to the conical ends (where graphicality follows automatically from pseudolocality, see \cite[Lemma 7.18]{CCMS:generic1}), an ancient one-sided rescaled flow is not \emph{a priori} an entire graph over the cylindrical ends of $\Sigma$ for very negative times. Establishing this graphicality is the main technical ingredient needed to carry out the above linear analysis. We prove graphicality and the corresponding uniqueness statement in Section \ref{sec:estimates-ancient-general}, see Section \ref{subsec:sketch} below for an outline of the proof.  

\subsection{Related work} \label{subsec:related-work} 

The first major advance towards a theory of generic mean curvature flow was taken by Colding--Minicozzi who introduced the notion of entropy (cf.\ Section \ref{subsec:dens-mon-entropy}) and classified the sphere/cylinder as the unique \emph{linearly stable} self-shrinkers in a precise sense \cite{ColdingMinicozzi:generic}. In particular, their work shows that if $M(0) = \Sigma\subset \RR^3$ is a non-spherical/cylindrical closed self-shrinker, then there is $M'(0)$ a small $C^\infty$-graph over $M(0)$ so that the mean curvature flow $M'(t)$ does not have the shrinking flow associated to $\Sigma$ at any tangent flow. 

\begin{remark} It follows from Colding--Minicozzi's linear analysis that any shrinker becomes shrinker mean-convex when perturbed to one side. This crucial observation was first used by Colding--Ilmanen--Minicozzi--White \cite{ColdingMinicozziIlmanenWhite} to classify the sphere as the closed self-shrinker of minimal entropy (see Section \ref{subsec:dens-mon-entropy}). This was later generalized in various ways by Bernstein--Wang \cite{BernsteinWang:1,BernsteinWang:TopologicalProperty,BernsteinWang:topology-small-ent} (cf.\ \cite{Zhu:entropy,HershkovitsWhite:sharp-entropy}). Of particular relevance to this article, we note that Bernstein--Wang's proof \cite{BernsteinWang:TopologicalProperty} that the cylinder has the second-least entropy among non-flat shrinkers in $\RR^3$ proceeds by constructing a ``pre-ancient'' one-sided mean curvature flow (on one side of a hypothetical lower entropy asymptotically conical shrinker) and observing that it must lose genus at (non-rescaled) time $t=0$. 
\end{remark}

More recently, in a joint work with Mantoulidis \cite{CCMS:generic1} we introduced a new approach to the study of mean curvature flow with generic initial data. In particular, as explained above, we showed how the problem was related to the existence/uniqueness of ancient one-sided flows for a given shrinker $\Sigma$ and studied the case where $\Sigma$ is smoothly asymptotic to a cone. In this case, using pseudolocality (cf.\ \cite{ChodoshSchulze}), such a flow is automatically an entire graph over $\Sigma$ in the ancient past and the techniques from work of the second-named author and Mantoulidis \cite{ChoiMantoulidis} could be used to prove existence/uniqueness. In particular, \cite{CCMS:generic1} proved that asymptotically conical self-shrinkers do not arise as multiplicity one tangent flows at the first (non-generic) singular time for a generic initial surface in $\RR^3$. (The current paper generalizes the techniques from \cite{CCMS:generic1} to account for cylindrical ends.)

We have also developed a related, but different, strategy to study generic mean curvature flows in a separate joint work with Mantoulidis \cite{CCMS:low-ent-gen}. This approach avoids the need for a classification of ancient flows and instead relies on a softer \emph{density drop} argument. This argument is most effective in a \emph{low entropy} setting and in particular proved that in $\RR^3$ a surface $M\subset \RR^3$ with entropy $\lambda(M) \leq 2$ can be perturbed so as to only encounter multiplicity one spherical/cylindrical singularities. However, the argument does not seem sensitive to the \emph{genus drop} proven in this paper (and in \cite{CCMS:generic1}) which limits its applicability---at present---to the low entropy case. Similarly, the work \cite{CCMS:low-ent-gen} applied in $\RR^4$ to yield a similar result for hypersurfaces $M\subset \RR^4$ with entropy $\lambda(M) \leq \lambda(\SS^1\times \RR^2)$; combined with the surgery result of Daniels--Holgate \cite{Daniels-Holgate} this proves that the Schoenflies conjecture holds under such a hypothesis (this improves on an influential work of Bernstein--Wang who established the Schoenflies conjecture under the assumption $\lambda(M) \leq \lambda(\SS^2\times \RR)$, cf.\  \cite{BernsteinWang:SpaceOfExpanders,BernsteinWang:expander-compactness,BernsteinWang:degree-expander,BernsteinWang:relative-entropy,bernsteinWang:top-uniqueness-expanders,Bernstein-Wang:mtn-pass,Bernstein-Wang:schoenflies}). 

\begin{remark} The existence and uniqueness result for ancient flows proven here in fact holds for self-shrinking hypersurfaces $\Sigma\subset \RR^{n+1}$ ($n+1 = 3,\dots,7$) with smooth conical and cylindrical ends. In fact, we can even allow for $\Sigma$ to have ends modeled on cylinders of the form $\hat\Sigma\times\RR$, for $\hat\Sigma\subset \RR^n$ an arbitrary closed shrinker. However, in the case of ends modeled on $\hat\Sigma\times \RR$, we have to impose an extra condition on $\Sigma$, namely we require that the first eigenfunction of the $L$ operator is $o(1)$ at spatial infinity (we call this $\Sigma$ having \emph{strictly stable ends}). (This holds automatically for $\hat\Sigma$ a round sphere, by work of Colding--Minicozzi \cite{ColdingMinicozzi:generic}; cf.\ Proposition \ref{prop:ends.shrinkers.R3} and Lemma \ref{lem:decay-lowest-eigenfuncton}). This generalization has some consequences for perturbing away such singularities in higher dimensions (in particular, yielding precise local information in contrast with the local perturbation result described in  \cite[Appendix C]{CCMS:low-ent-gen}).
\end{remark}

\begin{remark}
Some time after our work with Mantoulidis \cite{CCMS:generic1,CCMS:low-ent-gen} appeared, a more analytic approach to the mean curvature flow of generic initial data has been studied by Sun--Xue \cite{SunXue:cpt,SunXue:conical,SunXue:cylinder} (modeled on an earlier program introduced by Colding--Minicozzi \cite{ColdingMinicozzi:dynamics}). They have proven that tangent flows modeled on (multiplicity one) compact or asymptotically conical tangent flows (other than spheres/cylinders) at the first singular time can be \emph{locally} perturbed away by adjusting the initial data. It is yet unclear if their methods can be applied to shrinkers with cylindrical ends (as we do in this paper) or if their methods can recover the crucial \emph{genus drop} property used here to prove \emph{global} results. 
\end{remark}

\subsection{Outline of existence/uniqueness for the one-sided ancient flow}\label{subsec:sketch}

In this section we give a more detailed description of the existence and uniqueness of the one-sided ancient flow. Consider $\Sigma\subset \RR^{n+1}$ (for\footnote{This is a technical assumption related to certain difficulties in White's theory of mean convex flows in dimensions admitting non-flat stable minimal cones; we expect that it could be removed, but have not pursued this further. Of course, this  restriction does not prevent applications in the most relevant dimensions $n=2,3$.} $n\leq 6$) a smoothly embedded self-shrinker. We assume that any end of $\Sigma$ is either smoothly asymptotic to a smooth cone (with multiplicity one) or smoothly asymptotic to a cylinder $\hat\Sigma\times \RR$ for $\hat\Sigma\subset \RR^n$ a closed smooth self-shrinker. We call this the \emph{nice ends} hypothesis (cf.\ Definition \ref{defi:nice-ends}). 

It follows from standard linear theory for the Gaussian weighted $L^2$-space that the $L$ operator \eqref{eq:intro-L-operator} admits a positive first eigenfunction $\varphi$. We will assume that $\Sigma$ has \emph{strictly stable ends}  (cf.\ Definition \ref{defi:nice-stable}) meaning that $\varphi(\bx) = o(1)$ as $\bx\to \infty$. By work of Colding--Minicozzi and the barrier argument given in Lemma \ref{lem:decay-lowest-eigenfuncton} this automatically holds if all cylindrical ends of $\Sigma$ are asymptotic to \emph{round cylinders} but we have been unable to determine if it holds in general. (Note that for $\Sigma\subset \RR^3$, work of Wang and Colding--Minicozzi guarantee that $\Sigma$ has nice, stable ends, see Proposition \ref{prop:ends.shrinkers.R3}.)

\begin{remark}
The stable ends hypothesis does not seem to have been considered elsewhere. If there exists $\Sigma\subset \RR^{n+1}$ with nice but unstable ends, we conjecture that there exists an ancient one-sided flow coming out of $\Sigma$, but the flow is not graphical in the ancient past. Indeed, one may expect that the ancient flow along a cylindrical end $\hat\Sigma\times \RR$ is modeled on the elliptic regularization translator solution that limits to the one-sided flow coming out of $\hat\Sigma$ (in one dimension lower). 
To prove existence and uniqueness in this setting one would need to analyze the ``tips" of the cylindrical ends by generalizing the arguments from \cite{ADS2} to apply in this setting. Since this is not needed for the main result in $\RR^3$, we leave such investigations for future work.
\end{remark}

We now state (somewhat less informally) our existence and uniqueness results for the ancient one-sided flow. 

\begin{theorem*}[Existence of ancient one-sided flow, informal]
Consider $\Sigma\subset \RR^{n+1}$, $n\leq 6$, a self-shrinker with nice, strictly stable ends. Fix one of the two open sets $\Omega$ with $\Sigma = \partial\Omega$ and choose the inwards pointing unit normal for $\Sigma$. There exists a weak\footnote{precisely, $\bar M(t)$ is an integral cyclic unit regular Brakke flow with corresponding weak set flow $\cK$ so that $\bar M(t) = \cH^n\lfloor \partial\cK(t)$} rescaled mean curvature flow $(\bar M(\tau))_{\tau \in \RR}$, with the following properties:
\begin{enumerate}
\item the flow is one-sided: $\bar M(\tau) \subset \Omega$
\item the flow is shrinker mean convex: the speed of $\tau \mapsto \bar M(\tau)$ is everywhere positive 
\item the flow has good regularity properties: $\bar M(\tau)$ has only spherical/cylindrical multiplicity one singularities
\item the flow is graphical in the ancient past: for $\tau \ll0$, there is $v(\cdot,\tau) > 0 \in C^\infty(\Sigma)$ so that $\bar M(\tau)$ is the normal graph of $v(\cdot,\tau)$,
\item convergence in the ancient past: the function $v(\cdot,\tau)\to0$ in $C^k(\Sigma)$ as $\tau \to -\infty$, and
\item the graph is first eigenfunction dominated: the function $v(\cdot,\tau)$ satisfies $\tfrac 12 e^{-\mu\tau} \varphi(\cdot) \leq v(\cdot,\tau) \leq 2 e^{-\mu\tau} \varphi(\cdot)$, for $\mu$ and $\varphi$ the first eigenvalue and eigenfunction of the $L$-operator. 
\end{enumerate}
\end{theorem*}
Precise references are as follows: for (1)-(3) see Propositions \ref{theo:basic-prop-cK-cM}  and \ref{coro:summary-tleq0-non-rescaled}, for (4)-(6) see Proposition \ref{prop:exist-ancient-rescaled-MCF-C2-rescaling-argument}.

\begin{theorem*}[Uniqueness of ancient one-sided flow, informal]
For $\Sigma,\Omega,\bar M(\tau)$ as in the existence result, consider $\check M(\tau)$ an ancient weak\footnote{integral unit regular Brakke flow} rescaled mean curvature flow with the following properties: 
\begin{enumerate}
\item one-sided: $\check M(\tau) \subset \Omega$ 
\item smoothly limiting in the ancient past: $\check M(\tau)$ limits to $\Sigma$ in $C^\infty_\textnormal{loc}(\Sigma)$ with multiplicity one as $\tau\to-\infty$. 
\end{enumerate}
Then there is $\tau_0 \in \RR$ so that $\check M(\tau)  = \bar M(\tau + \tau_0)$. 
\end{theorem*}
The precise reference is Theorem \ref{thm:uniqueness-global}. 

\subsubsection{Outline of the proof of existence}
For $\mu,\varphi$ the first eigenvalue/eigenfunction for the $L$-operator, we prove in Lemma \ref{lemm:sub-super-one-sided} that the graphs of
\[
v_\pm(\bx,\tau) : =  (e^{-\mu \tau}\pm M e^{-2\mu \tau}) \varphi(\bx),
\]
define a supersolution (for $v_+$) and subsolution (for $v_-$) for rescaled mean curvature flow, for $\tau \ll0$ and $M\gg0$ (see Lemma \ref{lemm:sub-super-one-sided}).\footnote{This is a key innovation in the existence part of the present paper, allowing us to simplify several arguments from \cite{CCMS:generic1} even though the cylindrical ends seem like they should complicate certain steps.} The key feature is that to highest order $v_\pm \sim e^{-\mu\tau}\varphi$ decays at exactly the expected rate. Using these solutions as barriers, we see that the rescaled mean curvature flow with initial condition the graph of $\eps\varphi$ starting at time $\tau = |\mu|^{-1} \log \eps$---where $\eps>0$ is sufficiently small---exists until some fixed time $\tau_0$ independent of $\eps>0$ (this follows from the bootstrap arguments in Section \ref{sec:bootstrap}, see Proposition \ref{prop:short-time-exist-Sigma-eps}). 

Because the flows exist for a uniform amount of time as $\eps\to 0$, we can pass to the $\eps\to 0$ limit to obtain a graphical ancient one-sided rescaled mean curvature flow $\bar M(\tau)$ defined for $\tau \in (-\infty,\tau_0]$. In fact, using standard elliptic regularization methods we can continue $\bar M(\tau)$ to all $\tau \in\RR$ as a weak flow. This yields (1). Properties (4)-(6) follow from the barriers. 

Shrinker mean convexity for all $\tau\in\RR$ follows from shrinker mean convexity in the ancient past and spatial infinity. This proves (2) and allows us (as in \cite{CCMS:generic1}) to apply White's regularity theory to deduce (3). 

\subsubsection{Outline of the proof of uniqueness}

As mentioned above, establishing uniqueness of the ancient one-sided flow represents the central difficulty in this paper. 

For a shrinker $\Sigma = \partial\Omega$ with nice, stable ends consider an ancient weak rescaled mean curvature flow $\check M(\tau)$ that converges to $\Sigma$ in $C^\infty_\textrm{loc}$ with multiplicity-one as $\tau\to-\infty$. We claim that $\check M (\tau) = \bar M(\tau + \tau_0)$ for some $\tau_0\in\RR$, where $\bar M(\cdot)$ is the ancient one-sided flow constructed above. 

For simplicity, we assume in this outline that $\Sigma$ has exactly one end and is asymptotically cylindrical, modeled on $\hat\Sigma\times \RR$. Multiple cylindrical ends do not interact in a significant way and conical ends can be handled by the arguments from \cite{CCMS:generic1}, so this is not a significant loss of generality. We will choose coordinates $(\bomega,z) \in \RR^n\times \RR$ so that the $\RR$-factor of the asymptotic cylinder  points in the $\be = \partial_z$ direction. 

The steps to prove uniqueness are as follows:
\begin{enumerate}
\item We prove a nearly sharp upper bound for the graphical function $\check u$ in a fixed compact region $\Sigma_0\Subset \Sigma$ (with smooth boundary) as $\tau\to-\infty$. Namely, for any $\mu_0$ the first eigenvalue of $L$ on $\Sigma_0$ (with Dirichlet boundary conditions) we show (see Theorem \ref{thm:interior decay}) that 
\begin{equation}\label{eq:nearly-sharp-decay-intro-sketch}
\sup_{\Sigma_0} \check u \leq C e^{-\mu_0\tau}
\end{equation}
as $\tau\to-\infty$. (By taking $\Sigma_0$ larger we can take $\mu_0$ as close to $\mu$ as we want.)

The basic idea is to consider $\varphi_0$ the first eigenfunction for $L$ on $\Sigma_0$ with Dirichlet boundary and to form a subsolution\footnote{In reality, this expression fails to be a subsolution near $\partial\Sigma_0$. We use a different argument to prevent contact there by contrasting the Hopf boundary point lemma with a differential Harnack inequality proven for $\check u$. } of the form 
\[
a e^{-\mu_0\tau}(1-aM e^{-\mu_0\tau})\varphi_0
\]
(this is similar to the ansatz used in the proof of existence). If the desired decay was false then by using the (parabolic) Harnack inequality we would be able to fit the subsolution below $\check M(\tau)$ for larger and larger values of $a$ (as $\tau\to-\infty$). 

This part of the argument is very general and we expect it will be useful in other contexts. 
\item Fixing $\delta(\tau) : = \Vert \check u(\cdot,\tau)\Vert_{C^4(\Sigma\cap B_{4R}(\bOh))}$ for $R\gg0$ fixed, step (1) (along with Schauder theory) imply that $\delta(\tau)$ has nearly sharp exponential decay as $\tau \to-\infty$. We combine this with a barrier argument to extend this to graphicality/decay in the $z$-direction (i.e., along the cylindrical end). In particular, we prove that there is some $\alpha \gg0$ so that $\check M(\tau)$ is a small graph over the exponentially growing region $\Sigma \cap B_{\tfrac 12 \delta(\tau)^{-\frac{1}{1+\alpha}}}(\bOh)$ (see Proposition \ref{prop:global-decay}). The barrier argument proceeds in the following manner:
\begin{enumerate}
\item Construction of the barriers: In the case that the cylindrical end was modeled on a round cylinder, we could try to use the foliation of the complement of the end of the cylinder by self-shrinkers as introduced by Angenent--Daskalopoulos--\v{S}e\v{s}um \cite{ADS} (cf.\ \cite{KleeneMoller}). 

However, since the actual end of $\Sigma$ is likely to cross in and outside of the cylinder, these shrinkers will not fit perfectly inside/outside of $\Sigma$. To solve this problem, we observe that we can construct a ``long'' barrier by first writing the end of $\Sigma$ as a graph over $\hat\Sigma\times \RR$ and then perturbing this graph by $\eta z^\alpha \hat\varphi(\bomega)$ for $\hat \varphi$ the first eigenfunction of $L_{\hat\Sigma}$ (see Proposition \ref{prop:bendy-barrier-good-ends}). When $\eta>0$ is sufficiently small, this property holds for $\alpha \leq z \leq \eta^{-\frac 1\alpha}$. When $z = \eta^{-\frac 1\alpha}$, the long barrier has peeled away from $\Sigma$ a sufficient amount so that it can be ``welded''  (cf.\ Proposition \ref{prop:welding}) to the $z\geq \eta^{-\frac 1\alpha}$ portion of a ADS shrinker that fits within the end of $\Sigma$. 

(When $\hat\Sigma$ is some arbitrary closed shrinker rather than the round sphere, we construct a replacement for the ADS shrinkers formed from the translator approximations to the rescaled (ancient) mean curvature flow coming out of $\hat\Sigma$; see Section \ref{subsec:trans-barr}. We note also that Sun--Wang have very recently \cite{SunWang:translator} constructed complete translators asymptotic to $\hat\Sigma\times \RR$ which could alternatively be used here.)
\item The flow fits under the barrier with $\eta \sim \delta(\tau)$: We then need to show that the flow $\check M(\tau)$ can be controlled along the end of $\Sigma$ by using the barrier described in (a). To do so, we prove that the barriers from (a) either have bounded length or are asymptotically conical (see Lemma \ref{lemm:lin-grow-trans}) and that $\check M(\tau)$ decays to $\Sigma$ in a sublinear manner (see Proposition \ref{prop:sublinear-growth}). Combining these facts (and the decay from (1) above, we see that the flow $\check \cM(\tau)$ is pinched between $\Sigma$ and the barrier. 
\item Upgrading the barrier argument to graphicality:  \emph{A priori} it could hold that the flow bends into multiple sheets or develops a singularity in the region where it is pinched between the barrier and $\Sigma$. We know that this does not occur on fixed compact sets. Using a geometric continuity/bootstrap argument we can then extend this to the full barrier region (see Proposition \ref{prop:barrier.to.graph}). 
\end{enumerate}
\item At this point the barrier argument from (b) has given that $\check M(\tau)$ is graphical on a exponentially growing region (with nearly sharp pointwise estimates). We can use this, along with a Merle--Zaag type spectral dynamic analysis to prove that the graphical function $\check u$ is equal to $\check a e^{-\mu \tau} \varphi + O(e^{(-\mu + \tfrac {\delta_0} {4}) \tau})$ on compact sets, i.e., $\check u$ behaves like one would expect based on a naive linear analysis as in \eqref{eq:intro-lin-analysis-stupid} above (see Theorem \ref{thm:decaying-mode}). Below we will assume that $\check a = 1$ (this can be arranged by a translation in time unless $\check a = 0$; we ignore this possibility in this sketch, as it is easily handled by a modification of the $\check a = 1$ argument). We will also assume that $\bar M(\tau)$ (the graphical flow constructed above) has the same top order spectral behavior, i.e., $\bar a=1$. 
\item We now prove that $\bar M(\tau)$ lies (weakly) \emph{below} $\check M(\tau)$. (Later we will show they agree.) The main idea is for $\tau_m\to-\infty$ appropriately chosen, we can cut off $(1-o(1))e^{-\mu\tau_m}\varphi(\cdot)$ near radius $r_m\to\infty$ to form an initial condition for the rescaled mean curvature flow that lies below both solutions $\check M(\tau_m),\bar M(\tau_m)$. Since $\bar M(\tau_m)$ remains graphical (with good exponential growth estimates) until a fixed time (independent of $m$), a bootstrap argument implies that the rescaled mean curvature flow of this initial condition also exists until a fixed time (independent of $m$) and remains below both solutions $\check M(\tau),\bar M(\tau)$. 

As $m\to\infty$ these new solutions limit to ancient one-sided solutions that satisfy similar graphicality/decay estimates to the constructed graphical solution $\bar M(\tau)$ (see Proposition \ref{prop:est-vm}). Because the initial condition is formed by cutting off $(1-o(1))e^{-\mu\tau_m}\varphi(\cdot)$, a simple spectral dynamics and uniqueness argument implies that in fact this new solution agrees with $\bar M(\tau)$. Thus, $\bar M(\tau)$ lies below $\check M(\tau)$ (see Proposition \ref{prop:entire-sol-below}). 
\item It thus remains to estimate $\check M(\tau)$ ``from above'' (in reality we do not yet know that $\check M(\tau)$ is an entire graph over $\Sigma$). As a first step, we observe that we could repeat the barrier argument used in (2) to pinch $\check M(\tau)$ between barriers and $\bar M(\tau)$ (instead of pinching it between barriers and $\Sigma$). This involves observing that the ``long barrier'' can be also considered as a graph over $\bar M(\tau)$ instead of $\Sigma$ and one can construct suitable barriers in this setting as well. This allows us to bootstrap estimates for the quantity 
\[
W(\tau)^2 : = \sup_{s\leq \tau} \int_{\Sigma\cap B_{5R}(\bOh)} (\check u(\cdot,s)-\bar u(\cdot,s))^2 e^{-\frac 14 |\bx|^2} .
\]
Indeed, when $W(\tau)$ is very small, we can fit very close barriers (upgrading $L^2$-smallness to $L^\infty$-smallness using Schauder estimates) to obtain graphicality on a very large scale $\sim W(\tau)^{-\frac{1}{1+\alpha}}$. Then, this implies that the cutoff error estimates in a spectral dynamics argument for $\check u(\cdot,s)-\bar u(\cdot,s)$ are very small, leading to further improved estimates for $W(\tau)$. Eventually (see Theorem \ref{thm:uniqueness-global}), this implies that $W(\tau) \equiv 0$, finishing the proof.
\end{enumerate}

\subsection{Outline of application to mean curvature flow in $\RR^3$}\label{subsec:sketch-genus-drop}

Consider $M(0)\subset \RR^3$ a closed embedded surface. Assume that some weak\footnote{integral unit regular Brakke flow} mean curvature flow $M(t)$ starting at $M(0)$ has only multiplicity one spherical and cylindrical tangent flows until time $t=T$\footnote{By the resolution of the mean convex neighborhood conjecture \cite{ChoiHaslhoferHershkovits} any two integral unit regular Brakke flows with the given initial condition agree at least until $t=T$. }. Assume that  all singularities at time $t=T$ have multiplicity one. For simplicity, we assume that there is exactly one singular point at $t=T$, at $\bx \in \RR^3$. 

We now consider a local foliation $M_s(0)$ around $M(0)$ and flow each leaf simultaneously into $M_s(t)$. Fixing $s_j\to 0$ so that the level set flow of $M_{s_j}(0)$ does not fatten, we parabolically dilate $M_{s_j}(t)$ and $M_0(t)$ around $(\bx,T)$ by the parabolic distance $d_j$ of the (space-time) support of the flow $t\mapsto M_{s_j}(t)$ to $(\bx,T)$. Passing both flows to the limit, we obtain an ancient flow $\check M(t)$ on one side of some\footnote{Uniqueness of tangent flows has not been proven in the case of shrinkers with cylindrical ends, cf.\ \cite{Schulze:Loj,ColdingMinicozzi:uniqueness-tangent-flow,ChodoshSchulze}. } tangent flow to $t\mapsto M_0(t)$ at $(\bx,T)$. By a standard argument based on Huisken's monotonicity formula (cf.\ the proof of Lemma \ref{lemm:Tgen-est-one-sidedR3}) we have $\lambda(\check M(t)) \leq F(\Sigma)$, so the tangent flow to $t\mapsto \check M(t)$ is $t\mapsto \cH^2\lfloor \sqrt{-t}\Sigma$ (i.e., the multiplicity is one).  By the choice of rescaling, $\check M(t)$ is not $t\mapsto \cH^2\lfloor \sqrt{-t}\Sigma$, so by the uniqueness result discussed above, it must be a parabolic dilation of the ancient one-sided flow $\bar M(t)$. We have that $t\mapsto \bar M(t)$ has only multiplicity one spherical/cylindrical singularities, at least for $t<0$\footnote{For the arguments developed here, we only need to analyze the flow for $t<0$. This is a strong contrast to \cite{CCMS:generic1} which considers the properties of the one-sided flow for all $t \in \RR$ (see, however \cite[\S 8.4]{CCMS:generic1}). }. 

In particular, returning to the original scale, we find that $t\mapsto M_{s_j}(t)$ has only multiplicity one spherical/cylindrical singularities, at least until time $ t= T - o(d_j^2)$ as $j\to\infty$. We claim that the genus of $M_{s_j}(T - o(d_j^2))$ has strictly dropped in comparison to the genus of $M_0(t)$ right before $t=T$. Recalling that the genus is non-increasing along mean curvature flow by work of White \cite{White:topology-weak}, we thus conclude the following result after $\leq \genus(M(0))$ such perturbations:
\begin{theorem*}[Mean curvature flow of generic initial data in $\RR^3$, informal]
For $M(0)\subset \RR^3$ closed embedded surface, there is $M'(0)\subset\RR^3$ an arbitrarily small $C^\infty$ normal graph over $M(0)$ so that the mean curvature flow $t\mapsto M'(t)$ satisfies one of the following conditions:
\begin{itemize}
\item all tangent flows at singular points are multiplicity one shrinking spheres or cylinders, or
\item at the first time this fails, some tangent flow has multiplicity $\geq 2$. 
\end{itemize} 
\end{theorem*}
See Theorem \ref{theo:generic.flow.R3} for the precise statement. 

As such, the remaining ingredient we need to discuss is the genus drop. It suffices to show that the genus of the one-sided ancient rescaled flow $\tau \mapsto \bar M(\tau)$ loses genus at some (rescaled) time $\tau<\infty$. Indeed, by an important result of Brendle \cite{Brendle:genus0}, $\Sigma$ has non-zero genus (if it is not a sphere/cylinder) and thus $\bar M(\tau)$ does as well, for $\tau\ll0$, since it is graphical at such times. In \cite{CCMS:generic1} we established the genus drop (when $\Sigma$ is asymptotically conical or compact) by generalizing an observation of Bernstein--Wang \cite{BernsteinWang:TopologicalProperty} who noted that shrinker mean-convexity implies that the flow is \emph{star-shaped} (and thus genus zero) at time $t=0$. 

Here, we use a softer argument (cf.\ Proposition \ref{prop:genus-drop-one-sided-flow}), still using shrinker mean convexity, but only in the sense that the rescaled flow $\tau\mapsto \bar M(\tau)$ flows in one direction. If $M(\tau)$ has positive genus for $\tau \gg0$ we can find a non-trivial loop in the component of $\RR^3\setminus M(\tau)$ that $M(\tau)$ is flowing \emph{into}. (Some care must be taken to avoid a loop $\gamma$ that just goes around a cylindrical end.) The avoidance principle for mean curvature flows rescales to imply that $M(\tau)$ diverges as $\tau \to\infty$. Thus the loop $\gamma$ can be chosen outside of any fixed compact set. In particular, with respect to $M(\tau)$, $\tau \ll0$ (so $M(\tau)$ is $C^\infty$ close to $\Sigma$), the loop $\gamma$ lies near infinity and thus bounds a disk (near infinity $\Sigma$ only consists of cylindrical/conical ends and we can choose $\gamma$ not going around any of these ends). By work of White \cite{White:topology-weak}, a loop cannot bound a disk in the complement of a mean curvature flow at an earlier time but not a later time. This is a contradiction, completing the proof.

\subsection{Organization of the paper} We collect several fundamental definitions and conventions in Section \ref{sec:prelim}. In Section \ref{sec:self-shrinkers} we discuss self-shrinkers including their ends, and linear theory for the $L$ operator. Section \ref{sec:bootstrap} contains some techniques for constructing graphical rescaled mean curvature flows over shrinkers. We prove \emph{existence} of the ancient one-sided flow $\tau \mapsto \bar M(\tau)$ in Section \ref{sec:existence-one-sided-ancient}. In Section \ref{sec:barriers} we construct barriers along the cylindrical end and then use these (along with other arguments) in Section to prove \emph{uniqueness} of ancient one-sided flows in Section \ref{sec:estimates-ancient-general}. We prove the genus drop for ancient one-sided flows in $\RR^3$ in Section \ref{sec:genus-drop-ancient} and combine everything to prove the main generic flow result in Section \ref{sec:main-generic}. 

There are also several appendices. Appendix \ref{app:graph-shrinker} contains some important computations concerning the specific form of the quadratic error term in the linearization of the rescaled mean curvature flow equation over a shrinker. Appendix \ref{app:graph-shrinker-pos} continues this investigation with the additional assumption that the graph of a positive function, proving a Li--Yau estimate used in the proof of Theorem \ref{thm:interior decay}. Appendix \ref{app:integral-brakke-X-flows} recalls a definition of Brakke flow with additional forces while Appendix \ref{app:ilmanen.avoidance} recalls Ilmanen's avoidance principle. Appendix \ref{app:bf-avoid-bd-curv} contains a result (possibly of independent interest) concerning the avoidance of a Brakke flow and a non-compact flow with uniformly bounded curvature.  Appendix \ref{sec:top-regions-R3} contains some results about the topology of regions in $\RR^3$ and Appendix \ref{app:loc-top-mon} recalls some concepts from \cite{CCMS:generic1} concerning localized topological monotonicity for  mean curvature flows. Finally, Appendix \ref{app:unique-weak-tf} contains some results about weak versions of uniqueness of tangent flows in $\RR^3$ (also possibly of independent interest). 

\subsection{Acknowledgements}
We are grateful to Christos Mantoulidis for countless discussions about generic mean curvature flow and related topics. We would also like to thank Steve Kerckhoff and Lu Wang for answering some questions.  O.C. was supported by an NSF grant (DMS-2016403), a Terman Fellowship, and a Sloan Fellowship.  K.C. was supported by KIAS Individual Grant MG078901 and POSCO Science Fellowship. F.S. was partially supported by a Leverhulme Trust Research
Project Grant RPG-2016-174.


\section{Preliminaries}\label{sec:prelim}

In this section we collect some preliminary definitions, conventions, and results. 

\subsection{Spacetime and the level set flow}
We define the \emph{time} map $\ft : \RR^{n+1}\times \RR\to \RR$ to be the projection $\ft(\bx,t) = t$. For $E\subset \RR^{n+1}\times \RR$ we will write $E(t) : = E\cap\ft^{-1}(t)$. The knowledge of $E(t)$ for all $t$ is the same thing as knowing $E$, so we will often ignore the distinction. 

For a compact $n$-manifold $M$ (possibly with boundary), we consider $f: M\times [a,b] \to \RR^{n+1}$ so that (i) $f$ is continuous (ii) $f$ is smooth on $(M\setminus\partial M)\times (a,b]$ (iii) $f|_{M\times \{t\}}$ is injective for each $t \in [a,b]$ and (iii) $t\mapsto f(M\setminus\partial M,t)$ is flowing by mean curvature flow. In this case we call
\[
\cM : =\cup_{t\in [a,b]} f(M,t)\times \{t\} \subset \RR^{n+1}\times \RR
\]
a \emph{classical mean curvature flow} and define the \emph{heat boundary} of $\cM$ by
\[
\partial\cM : = f(M,a) \cup f(\partial M,[a,b]). 
\]
Classical flows that intersect must intersect in a point that belongs to at least one of their heat boundaries (cf.\ \cite[Lemma 3.1]{White:topology-weak}). 

For $\Gamma \subset \RR^{n+1}\times [0,\infty)$, $\cM \subset \RR^{n+1}\times \RR$ is a \emph{weak set flow} (generated by $\Gamma$) if $\cM(0) =\Gamma(0)$ and if $\cM'$ is a classical flow with $\partial \cM'$ disjoint from $\cM$ and $\cM'$ disjoint from $\Gamma$ then $\cM'$ is disjoint from $\cM$. There may be more than one weak set flow generated by $\Gamma$. 

The biggest such flow is called the \emph{level set flow}, which can be constructed as follows: For $\Gamma \subset \RR^{n+1}\times [0,\infty)$ as above, we set
\[
W_0 : = \{(\bx,0) : (\bx,0) \not \in \Gamma\}
\]
and then let $W_{k+1}$ denote the union of all classical flows $\cM'$ with $\cM'$ disjoint from $\Gamma$ and $\partial\cM\subset W_k$. The \emph{level set flow} generated by $\Gamma$ is then defined by
\[
\cM : = (\RR^{n+1}\times [0,\infty))\setminus \cup_k W_k \subset \RR^{n+1}\times [0,\infty). 
\]
See \cite{EvansSpruck1,Ilmanen:elliptic,White:topology-weak}. If $\Gamma\subset \RR^{n+1}\times \{0\}$, we will write $F_t(\Gamma) : = \cM(t)$ for the time $t$ slice of the corresponding level set flow. 

Fix $\Gamma \subset \RR^{n+1}$ compact. We say that the level set flow of $\Gamma$ is \emph{non-fattening} if $F_t(\Gamma)$ has no interior for each $t \geq 0$. This condition holds generically, namely if $u_0$ is a continuous function with compact level sets $u_0^{-1}(s)$ then the level set flow of $u_0^{-1}(s)$ fattens for at most countably many values of $s$ \cite[\S 11.3-4]{Ilmanen:elliptic}.

\subsection{Integral Brakke flows} An ($n$-dimensional\footnote{Of course one can consider $k$-dimensional flows in $\RR^{n+1}$ but we will never do so in this paper, so we will often omit the ``$n$-dimensionality'' and implicitly assume that all Brakke flows are flows of ``hypersurfaces.''}) integral Brakke flow in $\RR^{n+1}$ is a $1$-parameter family of Radon measures $(\mu(t))_{t\in I}$ so that
\begin{enumerate}
\item For almost every $t\in I$ there is an integral $n$-dimensional varifold $V(t)$ with $\mu(t) = \mu_{V(t)}$ and so that $V(t)$ has locally bounded first variation and mean curvature $\bH$ orthogonal to $\textrm{Tan}(V(t),\cdot)$ almost everywhere. 
\item For a bounded interval $[t_1,t_2]\in I$ and $K\subset \RR^{n+1}$ compact, we have
\[
\int_{t_1}^{t_2} \int_K (1+|\bH|^2) d\mu(t) dt < \infty.
\]
\item If $[t_1,t_2]\subset I$ and $f\in C^1_c(\RR^{n+1}\times [t_1,t_2])$ has $f\geq 0$ then
\[
\int f(\cdot,t_2) d\mu(t_2) - \int f(\cdot,t_1) d\mu(t_1) \leq \int_{t_1}^{t_2} \int (-|\bH|^2 f + \bH \cdot\nabla f + \tfrac{\partial f}{\partial t}) d\mu(t) dt.
\]
\end{enumerate}
We will sometimes write $\cM$ to represent a Brakke flow. 

We define the support of $\cM = (\mu(t))_t$ to be $\overline{\cup_t \supp \mu(t) \times \{t\}} \subset \RR^{n+1}\times \RR$.  It is useful to recall that the support a Brakke flow (with $t \in [0,\infty)$) is a weak set flow (generated by $\supp \mu(0)$) \cite[10.5]{Ilmanen:elliptic}. 

We say that a sequence of integral Brakke flows $\cM_i$ converges to another integral Brakke flow $\cM$ (written $\cM_i \rightharpoonup \cM$) if $\mu_i(t)$ weakly converges to $\mu(t)$ for all $t$ and for almost every $t$, after passing to a further subsequence depending on $t$, the associated integral varifolds converge $V_i(t) \to V(t)$. (Recall that if $\cM_i$ is a sequence of integral Brakke flows with uniform local mass bounds then a subsequence converges to an integral Brakke flow \cite[\S 7]{Ilmanen:elliptic}.)

We denote the set of points so that an integral Brakke flow $\cM$ is locally (in a backwards-forwards parabolic neighborhood) a smooth multiplicity one flow as $\reg \cM$ and if $\cM$ is defined on $[T_0,\infty)$, we set $\sing \cM : = \supp\cM \setminus (\reg\cM \cup \RR^{n+1}\times \{T_0\})$.

\subsection{Density, Huisken's monotonicity, and entropy}\label{subsec:dens-mon-entropy} For $X_0 = (\bx_0,t_0)\in\RR^{n+1}\times \RR$ we consider the ($n$-dimensional) backwards heat kernel based at $X_0$:
\begin{equation}\label{eq:gauss.dens.defn}
\rho_{X_0} (\bx,t) : = (4\pi(t_0-t))^{-\frac n2} \exp\left( -\frac{|\bx-\bx_0|^2}{4(t_0-t)}\right)
\end{equation}
for $\bx\in\RR^{n+1},t<t_0$. For $\cM$ a Brakke flow defined on $[T_0,\infty)$, $t_0>T_0$ and $0<r\leq\sqrt{T_0-t_0}$, we set
\[
\Theta_\cM(X_0,r) : = \int \rho_{X_0}(\bx,t_0-r^2) d\mu(t_0-r^2). 
\]
Huisken's monotonicity formula \cite{Huisken:sing,Ilmanen:singularities} implies that $r\mapsto \Theta_{\cM}(X_0,r)$ is non-decreasing (and constant only for a shrinking self-shrinker centered at $X_0$, cf.\ Section \ref{subsec:convent}). In particular we can define the density of $\cM$ at such $X_0$ by
\[
\Theta_\cM(X_0) : = \lim_{r\searrow 0} \Theta_\cM(X_0,r). 
\]
We call an integral Brakke flow $\cM$ \emph{unit-regular} if $\cM$ is smooth in a forwards-backwards space-time neighborhood of any space-time point $X$ with $\Theta_\cM(X) = 1$. Note that we can then write $\sing\cM = \{X \in\RR^{n+1}\times \RR : \Theta_\cM(X) > 1\}$.  Note that by \cite[Theorem 4.2]{SchulzeWhite} the class of unit-regular integral Brakke flows is closed under the convergence of Brakke flows. Furthermore, combining  \cite[Lemma 4.1]{SchulzeWhite} and \cite{White:Brakke} it follows that there is $\varepsilon_0>0$, depending only on dimension, such that every point $X\in \sing\cM$ has $\Theta_\cM(X) \geq 1 + \varepsilon_0$. Upper semi-continuity of density then implies that $\sing\cM$ is closed.

Following \cite{ColdingMinicozzi:generic} for a Radon measure $\mu$ on $\RR^{n+1}$ we define the entropy of $\mu$ by 
\[
\lambda(\mu) : = \sup_{\bx_0\in\RR^{n+1},t_0>0} \int \rho_{(\bx_0,t_0)}(\bx,0) d\mu
\]
where $\rho$ is the Gaussian as in \eqref{eq:gauss.dens.defn}. When $\mu = \cH^n \lfloor M$ for $M\subset \RR^{n+1}$ a properly embedded submanifold we will write $\lambda(M)$ instead. We define $\lambda_j : = \lambda(\SS^j \subset \RR^{j+1})$ and note that $\lambda_j = \lambda(\SS^j \times \RR^{n-j})$.

\subsection{Cyclic Brakke flows}
We call an integral Brakke flow $\cM = (\mu(t))_{t\in I}$ \emph{cyclic} if for a.e.\ $t\in I$ it holds that $\mu(t) = \mu_{V(t)}$ for an integral varifold $V(t)$ whose uniquely associated rectifiable mod-$2$ flat chain $[V(t)]$ has $\partial[V(t)]=0$. Loosely speaking, the cyclic condition can be used to rule out triple junctions (and generalizations thereof). By \cite[Theorem 4.2]{White:cyclic} the cyclic property is preserved under weak limits of integral Brakke flows.

\part{Existence and uniqueness of one sided ancient flows}
\section{Self-shrinkers}\label{sec:self-shrinkers}
\subsection{Self-shrinkers and curvature conventions} \label{subsec:convent}We recall that a \emph{self-shrinker} is a hypersurface $\Sigma^n \subset \RR^{n+1}$ satisfying 
\begin{equation}\label{eq:shrinker-def}
\bH +\frac{\bx^\perp}{2} = 0\, 
\end{equation}
where $\bH$ is its \emph{mean curvature vector}, $\bx$ the position vector and ${}^\perp$ the projection to the normal space.  

Recall that \eqref{eq:shrinker-def} is equivalent to any of the following conditions:
\begin{itemize}
\item $t\mapsto \sqrt{-t}\Sigma$ is a mean curvature flow for $t<0$,
\item $\Sigma$ is a minimal hypersurface with respect to $e^{-\frac{1}{2n}|\bx|^2} g_{\RR^{n+1}}$, or
\item $\Sigma$ is a critical point of the $F$-functional
\[
F(\Sigma) : = (4\pi)^{-\frac n2} \int_\Sigma e^{-\frac 14 |\bx|^2}
\]
among compactly supported deformations and translations/dilations. 
\end{itemize}
See \cite[Section 3]{ColdingMinicozzi:generic}. 

We denote with $\cS_n$ the set of smooth properly embedded self shrinkers $\Sigma^n \subset \RR^{n+1}$ with finitely many ends and let $\cS_n^*$ denote the subset of non-flat ones. Moreover (cf.\ \cite{BernsteinWang:topology-small-ent}) set 
\[
\cS_n(\Lambda) = \{\Sigma \in \cS_n : \lambda(\Sigma) < \Lambda\}.
\]
Furthermore, we consider tuples $(\Sigma, \Omega)$ where $\Sigma \in \cS_n$ and $\Omega$ is one of the two unique choices of open subsets $\Omega \subset \RR^{n+1}$ such that $\partial \Omega = \Sigma$. By an abuse of notation we will write $(\Sigma,\Omega) \in \cS_n$. 

We also define
\[
\cS^\textrm{gen}_n : = \left\{O(\SS^{j}(\sqrt{2j}) \times \RR^{n-j}) \in \cS_n : j = 1,\dots,k, \, O \in O(n+1) \right\}
\]
to be the set of (round) self-shrinking spheres and cylinders in $\RR^{n+1}$. 

We denote with $\nu_\Sigma$ the choice of unit normal vector field to $\Sigma$ pointing into $\Omega$. We define its \emph{shape operator} (or Weingarten map)
\begin{equation}\label{eq:convent.shape.oper}
S_p: T_p\Sigma \to T_p\Sigma, \qquad \xi \mapsto - D_\xi\nu_\Sigma 
\end{equation}
and its second fundamental form
\begin{equation}\label{eq:convent.sff}
 A: T_p\Sigma \times T_p\Sigma \to \RR,\qquad (\xi,\zeta) \mapsto A(\xi,\zeta) = S_p(\xi) \cdot \zeta\, .
 \end{equation}
We fix the sign of the \emph{scalar mean curvature} $H$ as follows
$$ \bH = H\, \nu_\Sigma\, ,$$
and thus $H = \tr S  = \tr A $, with the principal curvatures of $\Sigma$ being the eigenvalues of $S$. Observe that the shrinker mean curvature can be written as
\[
\bH + \frac{\bx^\perp}{2} = \left(H + \frac 12  \bx \cdot \nu_\Sigma \right) \nu_\Sigma
\]
Note that with these conventions, a sphere bounding a ball has normal vector pointing to the inside, positive mean curvature and positive principal curvatures. Similarly, a sphere bounding the complement of a closed ball has normal vector pointing to the outside, negative mean curvature and negative principal curvatures.

For a Riemannian manifold $(M,g)$ with Levi--Civita connection $\nabla$, we define the \emph{Riemannian curvature tensor} by
\begin{align}
R(\bX,\bY)\bZ & = \nabla^2_{\bX,\bY} \bZ - \nabla^2_{\bY,\bX} \bZ \label{eq:curv.con1}\\
R(\bX,\bY,\bZ,\bW) & = g(R(\bX,\bY)\bW,\bZ)\label{eq:curv.con2}
\end{align}
With this convention, if $\be_1,\be_2\in T_p M$ are orthonormal, the sectional curvature of $\Pi = \Span\{\be_1,\be_2\}$ is given by $K(\Pi) = R(\be_1,\be_2,\be_1,\be_2)$. Furthermore, if $M \subset \RR^{n+1}$ is a hypersurface then the Gauss equations become
\begin{equation}\label{eq:Gauss.convention}
R(\bX,\bY,\bZ,\bW) = A(\bX,\bZ)A(\bY,\bW) - A(\bX,\bW)A(\bY,\bZ). 
\end{equation}
In particular, we have
\begin{equation}\label{eq:Gauss.convention.traced}
\Ric(\bX,\bY) = H A(\bX,\bY) - A^2(\bX,\bY),
\end{equation}
where $A^2(\bX,\bY) = \sum_{i=1}^n A(\bX,\be_i)A(\bY,\be_i)$ for an orthonormal basis $\be_i$.

\subsection{Ends of shrinkers}For any $\Sigma \in \cS_n$, define 
\[
\sqrt{0}\Sigma : = \lim_{t\nearrow 0} \sqrt{-t}\Sigma
\]
where the limit is taken in the Hausdorff sense. For $\bx \in \sqrt{0}\Sigma$, define $\cT_\Sigma(\bx)$ to be the set of tangent flows to the flow $\{\sqrt{-t}\Sigma\}_{\{t<0\}}$ at $(\bx,0)$. Note that if $\bx \neq \bOh$, any tangent flow in $\cT_\Sigma(\bx)$ (weakly) splits a line in the $\bx$ direction. 

We define the class of shrinkers that we consider in this paper as follows.

\begin{definition}[Shrinkers with nice ends]\label{defi:nice-ends}
We say that $\Sigma \in \cS_n$ has \emph{nice ends} if for $\bx \in \sqrt{0}\Sigma\setminus\{\bOh\}$, the set $\cT_\Sigma(\bx)$ has precisely one element $(-\infty,0) \ni t\mapsto \cM(t)$ and moreover, one of the following holds:
\begin{enumerate}
\item $\cM(t)$ is a multiplicity one plane, or
\item $\cM(t)$ is the the multiplicity one flow associated to the smooth flow $\sqrt{-t} \hat\Sigma \times \RR_\bx$, for $\hat\Sigma$ a \emph{compact} smooth shrinker in $\RR^{n}_{\bx^\perp} : = \{\by \in \RR^{n+1} : \bangle{\by,\bx} = 0\}$. 
\end{enumerate} 
If $\Sigma\in\cS_n$ has nice ends, we write $\Sigma\in \cS'_n$. We define $\cS_n'(\Lambda) = \cS_n'\cap\cS_n(\Lambda)$. 
\end{definition}

The following lemma follows immediately from the definition of $\cS_n'$ combined with Brakke regularity. 

\begin{lemma}[Asymptotic decomposition of shrinkers with nice ends] \label{lem:ends-decomposition}
For $\Sigma \in \cS_n'$, there is $R=R(\Sigma)$ sufficiently large so that 
\[
\Sigma \setminus B_R(\bOh) = \Sigma_\textnormal{cyl} \cup \Sigma_\textnormal{con}
\]
so that $\lambda \Sigma_\textnormal{con}$ converges in $C^\infty_\textnormal{loc}(\RR^{n+1}\setminus\{0\})$ with multiplicity one to $\sqrt{0}\Sigma_\textnormal{con}$ as $\lambda \to 0$, and for any connected component $\Sigma ' \subset \Sigma_\textnormal{cyl}$ there exists $\bx \in \SS^n \subset \RR^{n+1}$ and a compact smooth shrinker $\hat \Sigma \subset \RR^n_{\bx^\perp}$ so that $\Sigma' - s\bx$ converges in $C^\infty_\textnormal{loc}(\RR^{n+1})$ with multiplicity one to $\hat \Sigma\times \RR_\bx$ as $s\to \infty$.
\end{lemma}

\subsection{Function spaces} For ease of reference, we collect below various norms to be used in later sections. We will use the notation $L^{2}_{W}$ for the Gaussian weighted $L^{2}$-space on a shrinker $\Sigma$, i.e.,
\[
\bangle{u,v}_{W} = (4\pi)^{-\frac n2} \int_{\Sigma} vw e^{-\tfrac 14 |\bx|^{2}}
\]
The following weighted H\"older spaces will be used along the conical ends. 
\begin{align*}
\Vert f \Vert^{(d)}_{k;\Sigma_{\textrm{con}}} & : = \sum_{i=0}^{k} \sup_{\bx \in\Sigma_{\textrm{con}}} |\bx|^{-d+i} |\nabla^{i}_{\Sigma} f(\bx)| \\
[f]^{(d)}_{\alpha;\Sigma_{\textrm{con}}} & : = \sup_{\bx\neq \by \in \Sigma_{\textrm{con}}} \frac{1}{|\bx|^{d-\alpha}+|\by|^{d-\alpha}} \frac{|f(\bx)-f(\by)|}{d_{\Sigma}(\bx,\by)^{\alpha}}\\
\Vert f \Vert^{(d)}_{k,\alpha;\Sigma_{\textrm{con}}} & : = \Vert f \Vert^{(d)}_{k;\Sigma_{\textrm{con}}} + [\nabla^{k}_{\Sigma}f]^{(d-k)}_{\alpha;\Sigma_{\textrm{con}}}. 
\end{align*}
When working with flows that are not known to be graphical over the entire cylindrical end, we will use the following semi-global norms
\[
\Vert \cdot \Vert_{k,\alpha; \text{cyl}}(\rho)=   \Vert  \cdot \Vert_{k,\alpha;\Sigma \cap B_{2R_0}(\bOh)} + \Vert \cdot \Vert_{k,\alpha;\Sigma_\textnormal{cyl} \cap B_{\rho}(\bOh)}
\]
and
\[ \Vert \cdot \Vert_{k,\alpha}(\rho)=  \Vert \cdot \Vert_{k,\alpha; \text{cyl}}(\rho) +  \Vert  \cdot \Vert_{k,\alpha;\Sigma_{\textnormal{con}}}^{(1)} 
\, ,\]
where the first norm only considers the compact part along with a finite part of the  cylindrical piece, while the second normal also controls the conical ends (with an appropriate weight). These norms are used in the proof of uniqueness in the ancient past in Section \ref{sec:estimates-ancient-general}. 

\subsection{Linear theory on shrinkers with stable ends}  Consider a shrinker with nice ends $\Sigma \in \cS'_n$. By Lemma \ref{lem:ends-decomposition} the ends of $\Sigma$  decompose as the disjoint union $\Sigma_\textnormal{cyl} \cup \Sigma_\textnormal{con}$ of cylindrical and smoothly asymptotically conical ends. We denote the asymptotically conical ends by $\Sigma'_{\textnormal{con}, i}$ for $i =1, \dots, N_\textnormal{con}$ with asymptotic cones $\cC_i$ and the asymptotically cylindrical ends $\Sigma'_{\textnormal{cyl}, j}$ for $j =1, \dots, N_\textnormal{cyl}$ with asymptotic directions $\bx_j \in \SS^n$ and asymptotic compact smooth shrinkers $\hat{\Sigma}_j\subset \RR^n_{\bx_j^\perp}$.

We denote by $\mu$ the smallest eigenvalue of the following operator\footnote{Recall that we say that $\mu$ is an eigenvalue of $L$ with eigenfunction $\varphi$ if $L\varphi = - \mu \varphi$. We will write $L_\Sigma$, $\Delta_\Sigma$, etc., when the dependence on $\Sigma$ is not clear.} acting on functions $u \in C^2_\textrm{loc}(\Sigma)$
$$Lu=\Delta u -\frac{1}{2} \bx^T \cdot \nabla u +\tfrac{1}{2}u +|A|^2 u$$
(where $A$ is the second fundamental form of $\Sigma$) and will write $0<\varphi \in H^1_W(\Sigma)$ for the corresponding first eigenfunction (unique up to scaling). We normalize $\Vert \varphi \Vert_{W} = 1$. The existence and uniqueness of $\varphi$ follows as in the proof of \cite[Proposition 4.1]{BernsteinWang:TopologicalProperty}, using the results in \cite{ColdingMinicozzi:generic}. We denote the corresponding eigenvalues and first eigenfunctions of the asymptotic shrinkers $\hat{\Sigma}_j$ by $\hat{\mu}_{j}$ and $\hat{\varphi}_{j}$ respectively. Recall that by Colding--Minicozzi's classification of entropy stable shrinkers, \cite[Theorems 0.17 and 9.36]{ColdingMinicozzi:generic}, all non flat shrinkers which are not a sphere or a (generalised) round cylinder satisfy $\mu < -1$.

\begin{definition}[Shrinkers with nice stable ends] \label{defi:nice-stable0} We say a shrinker $\Sigma \in \cS_n$ has \emph{stable ends} if the first eigenfunction $\varphi$ satisfies $\varphi \in L^{\infty}(\Sigma)$. 
If $\Sigma\in\cS'_n$ has nice stable ends, we write $\Sigma\in \cS''_n$. We define $\cS_n''(\Lambda) = \cS_n''\cap\cS_n(\Lambda)$. 
\end{definition}

\begin{definition}[Shrinkers with nice strictly stable ends] \label{defi:nice-stable} We say a shrinker $\Sigma \in \cS_n$ has \emph{strictly stable ends} if the first eigenfunction $\varphi$ satisfies $\varphi(\bx) = o(1)$ for $\Sigma\ni \bx \to\infty$. If $\Sigma\in\cS'_n$ has nice strictly stable ends, we write $\Sigma\in \cS'''_n$. We define $\cS_n'''(\Lambda) = \cS_n'''\cap\cS_n(\Lambda)$. 
\end{definition}

Clearly if $\Sigma = \hat \Sigma\times \RR$ for $\hat\Sigma \in \cS_{n-1}$ a smooth compact shrinker, it holds that $\Sigma$ has nice stable ends $\Sigma \in \cS_{n}''$ (but $\Sigma \not \in \cS_n'''$). For a non-cylindrical shrinker with nice ends, we show in part (2) of Proposition \ref{lem:decay-lowest-eigenfuncton} below that if $\hat{\mu}_{j}$ for $j = 1, \dots, N_\textnormal{cyl}$ are the first eigenvalues of the cylindrical ends of $\Sigma \in \cS_{n}'$, then a sufficient condition for $\Sigma$ to have strictly stable ends is $\mu< \hat{\mu}_{j}$ for $j = 1, \dots, N_\textnormal{cyl}$. 

We note that by the result of Wang \cite{Wang:ends-conical}, for $\Sigma \in \cS_2$ (i.e., an embedded shrinker in $\RR^3$), all ends are either cylindrical, with $\hat \Sigma$ being the shrinking circle, or asymptotically conical. Thus, $\Sigma \in \cS_{2}'$. Combining this  with Colding--Minicozzi's estimate $\mu<-1$  and  properties (1) and (2) Proposition \ref{lem:decay-lowest-eigenfuncton} below we see that any non-cylindrical shrinker in $\cS_2$ has nice, strictly stable ends.

\begin{proposition}[\cite{Wang:ends-conical,ColdingMinicozzi:generic}]\label{prop:ends.shrinkers.R3}
$\cS_2 = \cS_2' = \cS_2''$ and $\cS_2 \setminus \cS_2^\textnormal{gen} = \cS_2'''$. 
\end{proposition}

It is an interesting question to what extent the inclusions $\cS_3'' \subset \cS_3'\subset \cS_3$, $\cS_3'''\subset \cS_3' \setminus \cS_3^\textrm{gen}$ are strict.\footnote{For example, if we fix a smooth closed embedded shrinker $\hat\Sigma\subset\RR^3$ (e.g., the round sphere, or the Angenent torus), we do not know if there exists $\Sigma \in \cS_3'$ with an end asymptotic to $\hat\Sigma\times \RR$ (other than $\Sigma = \hat \Sigma\times \RR$) and if such $\Sigma$ exists, we do not know if it must have (strictly) stable ends when $\Sigma$ is not a round sphere. These questions are likely to be relevant to the understanding of (generic) mean curvature flow of hypersurfaces in $\RR^4$.}

We extend here the analysis in \cite[Proposition 4.1]{BernsteinWang:TopologicalProperty} of the first eigenfunction of $L_\Sigma$ from asymptotically conical shrinkers to shrinkers with nice ends.

\begin{lemma}[Decay for first eigenfunction and sufficient condition for stable ends] \label{lem:decay-lowest-eigenfuncton}
Assume $(\Sigma, \Omega) \in \cS'_n$ is a smooth shrinker with nice ends. Write $\mu$ for the first eigenvalue of $L_{\Sigma}$ on $\Sigma$. 
\begin{enumerate}
\item Fix $\beta>0$. Along $\Sigma_\textnormal{con}$ we have
\begin{align*}
(1+|\mathbf{x}|^2)^{\frac 12 + \mu - \beta} \lesssim \varphi(\mathbf{x}) & \lesssim (1+|\mathbf{x}|^2)^{\frac 12 + \mu + \beta},\\
|\nabla^m \varphi(\mathbf{x})| & \lesssim (1+|\mathbf{x}|^2)^{\frac 12 + \mu + \beta - \frac m2}.
\end{align*}
\item Write $\mu_{j}$ for the first eigenvalue of $L_{\hat\Sigma_{j}\times\RR}$ for the exact cylinder associated to each cylindrical end $\Sigma_{\textnormal{cyl},j}'$, $j=1,\dots,N_{\textnormal{cyl}}$. Assume that 
\begin{equation}\label{eq:eig-cond-impl-stable-ends}
\mu<\mu_{j} \textrm{ for } j=1,\dots,N_{\textnormal{cyl}}.
\end{equation}
Then, along $\Sigma'_{\textnormal{cyl}, j} \subset \Sigma_\textnormal{cyl}$, we have
\begin{align}
(1+|\mathbf{x}|^2)^{\mu - \hat{\mu}_{j} - \beta} \lesssim \varphi (\mathbf{x}) & \lesssim (1+|\mathbf{x}|^2)^{\mu - \hat{\mu}_{j}+ \beta} \label{eq:C0_estimate_cylindrical_end},\\
|\nabla^m \varphi (\mathbf{x})| & \lesssim (1+|\mathbf{x}|^2)^{\mu - \hat{\mu}_{j}+ \beta}. \label{eq:Ck_estimate_cylindrical_end}
\end{align}
In particular, \eqref{eq:eig-cond-impl-stable-ends} implies that $\Sigma \in \cS_{n}'''$. 
\item Assume that $\Sigma\in\cS_{n}''$, i.e., $\Sigma$ has stable ends (but do not necessarily assume the $\mu<\mu_{j}$ condition in \eqref{eq:eig-cond-impl-stable-ends}). Then, $|\nabla^{m}\varphi|\in L^{\infty}(\Sigma)$ and $\varphi^{-1}|\nabla \varphi|^{2} \in L^{\infty}(\Sigma)$. 

Moreover, there is $\eps_0 = \eps_0(\Sigma)>0$ with the following properties. For $\eps \in (0,\eps_0)$, the normal graph of $\eps\varphi$ is a smooth surface $\Sigma_\eps \subset \Omega$. Denote $\Omega_\eps \subset \Omega$ by the open set with $\partial\Omega_\eps = \Sigma_\eps$. The surface $\Sigma_\eps$ is strictly shrinker mean convex to the interior of $\Omega_\eps$ in the sense that
\begin{equation}\label{eq:shrinker_mean_convexity}
2H_{\Sigma_\eps} + \mathbf{x}\cdot \nu_{\Sigma_\eps} \geq  \eps (2 |\mu| - C\eps) \varphi(\bx)
\end{equation}
for $C=C(\Sigma)$. Furthermore,  the entropy of $\Sigma_\eps$ satisfies the bound $\lambda(\Sigma_\eps) \leq F(\Sigma) + o(1)$ as $\eps \to 0$.
\end{enumerate}
\end{lemma}
\begin{proof} Recall from the proof of \cite[Proposition 4.1]{BernsteinWang:TopologicalProperty} that for every $\beta>0$ there exists $R = R(\Sigma, \beta)$ such along each asymptotically conical end $\Sigma'_{\textnormal{con}, i}$ we have barriers
\begin{equation}\label{eq: barrier asym.1}
 \overline{g}(\bx) = |x|^{1+2\mu+2\beta} \text{ and } \underline{g}(\bx) = |x|^{1+2\mu-2\beta}\, ,
\end{equation}
which satisfy
\begin{equation}\label{eq: barrier asym.2}
 L\overline{g} + \mu \overline{g} < 0 \text{ and } L\underline{g} + \mu \underline{g} > 0 
 \end{equation}
on $\Sigma'_{\textnormal{con}, i} \setminus \overline{B}_R$.

Similarly, for an asymptotically cylindrical end $\Sigma'_{\textnormal{cyl}, j}$, so that $\mu_{j}<\mu$, 
let $\be:=\bx_j$ and denote the coordinate in direction $\be$ by $z$. Let $\bar{\varphi} \in C^\infty_c(\RR^n_{\be^\perp})$ be an extension of the first eigenfunction of $L_{\hat \Sigma_i}$, $\hat{\varphi}_{j}$, to a neighborhood of $\hat{\Sigma}\subset \RR^n_{\be^\perp}$. We make the ansatz for a barrier as the ambient function 
$$g := z^{-\alpha} \bar{\varphi}$$ for some $\alpha>0$, where we extend $\bar{\varphi}$ constant in $\be$-direction. Note that
\begin{align}
\nabla_{\RR^{n+1}} g &= - \alpha z^{-\alpha-1} \bar{\varphi} \be + z^{-\alpha} \nabla_{\RR^{n+1}} \bar{\varphi} \notag \\
\nabla^2_{\RR^{n+1}} g &= \alpha (\alpha + 1) z^{-\alpha-2} \bar{\varphi}\, dz \otimes dz - \alpha z^{-\alpha -1} \left( dz \otimes d\bar{\varphi} + d\bar{\varphi} \otimes dz \right) + z^{-\alpha} \nabla^2_{\RR^{n+1}} \bar{\varphi} \label{eq: hess psi}
\end{align}
Since $g$ is an ambient function we compute (compare the proof of \cite[Lemma 3.15]{ChodoshSchulze})
$$L g = \Delta_{\RR^{n+1}}g - \nabla^2_{\RR^{n+1}} g(\nu, \nu) - \tfrac{1}{2} \bx \cdot \nabla_{\RR^{n+1}} g + \tfrac{1}{2} g + |A|^2 g\, .$$
Note that 
$$ - \tfrac{1}{2} \bx \cdot \nabla_{\RR^{n+1}}g = \tfrac{\alpha}{2} g - \tfrac{1}{2} z^{-\alpha} \bx \cdot \nabla_{\RR^{n}_{\be^{\perp}}} \bar{\varphi} \, .$$
Furthermore, since $\hat{\varphi}_{j} > 0$, using \eqref{eq: hess psi}, we can estimate as $z \to \infty$ that
\begin{align*}
\nabla^2_{\RR^{n+1}} g(\nu_\Sigma,\nu_\Sigma) &= z^{-\alpha} \nabla^2_{\RR^{n}_{\be^{\perp}}} \bar{\varphi}(\nu_{\hat{\Sigma}_j}, \nu_{\hat{\Sigma}_j}) + o(1) g \\
\Delta_{\RR^{n+1}}g &=  z^{-\alpha} \Delta_{\RR^n_{\be^\perp}} \bar{\varphi} + o(1) g\\
|A_\Sigma|^2 g &=  z^{-\alpha} |A_{\hat{\Sigma}_j}|^2\bar{\varphi} + o(1) g\, ,
\end{align*}
where have extended $\nu_{\hat{\Sigma}_j}$ and $|A_{\hat{\Sigma}_j}|^2$ to a neighborhood of $\hat{\Sigma}\subset \RR^n_{\be^\perp}$. We thus can estimate
\begin{equation*}\begin{split}
L g &= z^{-\alpha} \left( \Delta_{\RR^n_{\be^\perp}} \bar{\varphi} - \nabla^2_{\RR^n_{\be^\perp}} \bar{\varphi}(\nu_{\hat{\Sigma}_j}, \nu_{\hat{\Sigma}_j}) - \tfrac{1}{2} \bx \cdot \nabla_{\RR^{n}_{\be^{\perp}}} \bar{\varphi} + \tfrac{1}{2} \bar{\varphi} + |A_{\hat{\Sigma}_j}|^2\bar{\varphi} \right) + \tfrac{\alpha}{2} g + o(1) g \\
&= \left( \tfrac{\alpha}{2} - \hat{\mu}_{j} + o(1)\right) g\, .
\end{split}\end{equation*}
Thus for every $\beta > 0$ there is $R = R(\Sigma'_j, \beta)$ such that choosing $\alpha =  2(\hat{\mu}_{j} - \mu - \beta)$ we have an upper barrier $\overline{g}$ and choosing $\alpha =  2(\hat{\mu}_{j} - \mu + \beta)$ we have a lower barrier $\underline{g}$ satisfying
\begin{equation}\label{eq: barrier asym.3}
 L\overline{g} + \mu \overline{g} < 0 \text{ and } L\underline{g} + \mu \underline{g} > 0 
 \end{equation}
on $\Sigma'_j \setminus B_R$.
In both the asymptotically conical and the asymptotically cylindrical case we can choose $C>1$ such that
$$ C^{-1} \underline{g} < \varphi < C \overline{g} $$
on $\partial(\Sigma' \setminus B_R)$.  As such, $v = \varphi - C \overline{g}$ satisfies $L v + \mu v > 0$ and $v < 0$ in a neighborhood of $\partial(\Sigma' \setminus B_R)$. We now argue that $v \leq 0 $ on $\Sigma' \setminus B_R$. Since $v^+ \in H^1_W$, we find that
\begin{equation*}\begin{split}
-\mu \int_{\Sigma \setminus B_R} (v^+)^2 \rho\, d\cH^n &\leq \int_{\Sigma \setminus B_R} v^+ L v \rho\, d\cH^n \\
&= \int_{\Sigma \setminus B_R} \left(-|\nabla v^+|^2 + \left(\frac{1}{2} + |A_\Sigma|^2 \right)(v^+)^2\right) \rho\, d\cH^n\, .
\end{split}
\end{equation*}
Thus, using Ecker's Sobolev inequality \cite{Ecker:Sobolev} (cf.\ \cite[Proposition 3.9]{ChodoshSchulze}) and that $|A_\Sigma|$ is bounded, we find that
$$ R^2 \int_{\Sigma \setminus B_R}  (v^+)^2 \rho\, d\cH^n \leq C \int_{\Sigma \setminus B_R}  (v^+)^2 \rho\, d\cH^n\, .$$
Thus choosing $R$ sufficiently large, if necessary, we thus see that $v^+ \equiv 0$ on $\Sigma'\setminus B_R(0)$. 

Thus $\varphi \leq C \overline{g}$ on $\Sigma'\setminus B_R(0)$. The same reasoning gives $\varphi \geq C^{-1} \underline{g}$ on $\Sigma'\setminus B_R(0)$. This proves the $C^{0}$-estimate in properties (1) and (2). In particular, this proves that \eqref{eq:eig-cond-impl-stable-ends} implies $\Sigma \in \cS_n'''$. 

We now explain the derivative estimates in (1) and (2). In the asymptotically conical case, the decay estimates for the higher derivatives of $\varphi$ follow as in the proof of \cite[Proposition 4.1]{BernsteinWang:TopologicalProperty}.
For the asymptotically cylindrical ends, consider for $t<0$ and $\mathbf{q} \in \sqrt{-t} \Sigma : = \Sigma_t$ the function
$$ \tilde{\varphi}(\mathbf{q},t) = \varphi\left(\frac{\mathbf{q}}{\sqrt{-t}}\right)\, .$$
Then since $L\varphi = -\mu \varphi$ we have that
$$ \left(\frac{d}{dt} - \Delta_{\Sigma_t}\right) \tilde{\varphi} = \left(|A_{\Sigma_t}|^2 + \frac{1}{2} + \mu\right) \tilde{\varphi}\, .$$
Note that \eqref{eq:C0_estimate_cylindrical_end} implies that for $t \in (-3,-1]$ along $(\Sigma'_j)_t$
$$ \tilde{\varphi}(\bx) \leq C (1+ |\bx|^2)^{\mu-\hat{\mu}_j + \beta}$$
which is decaying for $\beta < \hat{\mu}_j - \mu$ and thus bounded. Standard interior estimates imply that all higher derivatives of $\tilde\varphi$ are uniformly bounded along $(\Sigma'_j)_t$ for $t \in [-2,-1]$ . Then \eqref{eq:Ck_estimate_cylindrical_end} follows from \eqref{eq:C0_estimate_cylindrical_end} via interpolation.

Finally, we consider (3). Assume that $\Sigma\in \cS_{n}''$. We first note that the bounds for higher derivatives of the first eigenfunction (with no decay) follows from interior estimates (on $t\mapsto\sqrt{-t}\Sigma$) exactly as in the above argument. We now estimate $\varphi^{-1}|\nabla \varphi|^{2}$. Using the fact that $\Sigma$ has  bounded geometry and the fact that $\varphi$ has bounded $C^{2}$ norm, there is $r>0$ so that normal coordinates exist on a ball of radius $r$ centered at each $\bx \in \Sigma$ and for $\by \in B_{r}(\bx)$ we have the following Taylor expansion
\[
0 \leq \varphi(\by) \leq \varphi(\bx) + \partial_{i}\varphi|_{\bx} (\by^{i}-\bx^{i}) + C|\by-\bx|^{2}. 
\]
in the normal coordinates (with constant $C$ uniformly bounded for all such $\bx,\by$). Since $\nabla \varphi \in L^{\infty}(\Sigma)$, we can choose $\by^{i} := - s \partial_{i}\varphi|_{\bx} + \bx^{i} \in B_{r}(\bx)$, for $s>0$ sufficiently small (independent of $\bx$). Thus, we find
\[
0 \leq \varphi(\bx) - s |\nabla \varphi(\bx)|^{2} + C s^{2} |\nabla \varphi(\bx)|^{2}. 
\]
Taking $s>0$ even smaller, we conclude that $\varphi^{-1}|\nabla \varphi|^{2}\in L^{\infty}(\Sigma)$.\footnote{Note that this estimate could alternatively have been proven by appropriately modifying the argument in \cite[Lemma 6.6]{Ilmanen:elliptic}.}

Because $\Sigma$ has a uniform tubular neighborhood, for $\eps>0$, the normal graph of $\eps\varphi$, $\Sigma_{\eps}$, is a smooth surface. Moreover, by Corollary \ref{coro:expand-rescaled-mcf-app} (and recalling that $v = (\nu_{\Gamma}\cdot\nu_{\Sigma})^{-1} = 1+o(1)$ as $\eps\to 0$)  we have 
\[
v(\bx) (H + \tfrac 12 \bx_{\Gamma}\cdot\nu_{\Gamma}) = \eps L \varphi + E = - \eps \mu \varphi + E
\]
where
\[
E= \eps \varphi E_{1} + \eps^{2} E_{2}(\nabla \varphi,\nabla \varphi)
\]
with $|E_{1}|= O(\eps),|E_{2}| = O(1)$ as $\eps\to 0$, uniformly on $\Sigma$. Using the $\varphi^{-1}|\nabla \varphi|^{2} \in L^{\infty}(\Sigma)$ estimate from above, we find
\[
v(\bx) (H + \tfrac 12 \bx_{\Gamma}\cdot\nu_{\Gamma}) = \eps L \varphi + E = - \eps \mu \varphi + O(\eps^{2} \varphi) 
\]
This proves the strictly shrinker mean convex assertion. 

The final statement on the entropy follows directly from the smooth convergence of $\Sigma_\eps \to \Sigma$ as $\eps \to 0$. 
\end{proof}


\section{Geometric estimates for graphical flows}\label{sec:bootstrap}
This section contains some well-known considerations relating graphs over shrinkers to rescaled mean curvature flows, including a long-time existence for rescaled flows with small $L^\infty$-norm. 

We begin by using pseudolocality and interior estimates to prove higher derivative estimates for a graphical rescaled mean curvature flow over a shrinker. 
\begin{proposition}[Bootstrap for graphical flows]\label{prop:bootstrap.graphical}
Fix a shrinker $\Sigma \in \cS_n'$. There are constants $\psi_0= \psi_0(\Sigma), \psi_1= \psi_1(\Sigma)$ with the following properties. For $T > 1$, suppose that $v \in C^2_\textnormal{loc}(\Sigma \times [0,T))$ has 
\begin{equation}\label{eq:restart.graph.flow}
\sup_\Sigma |v(\cdot,\tau)| \leq \psi_0, \qquad \sup_\Sigma |\nabla_\Sigma v(\cdot,\tau)|\leq \psi_1
\end{equation}
for all $\tau \in [0,T)$ and $\tau \mapsto \Gamma(\tau) : = \Graph_\Sigma v(\cdot,\tau)$ is a rescaled mean curvature flow. 

Then, for $m\in \NN$, it holds that
$$
\sup_{\Sigma\times [1,T)} |\nabla^m_\Sigma v| \leq C= C(m,\Sigma)
$$
Furthermore, if
$$
\limsup_{\tau \to T} \sup_\Sigma |v(\cdot,T)| < \psi_0
$$
then we can extend the flow to $v' \in (\Sigma\times [0,T'])$ for some $T'>T$ with $v'$ still satisfying \eqref{eq:restart.graph.flow} for all $\tau \in [0,T']$. 
\end{proposition}
\begin{proof}
Choose $\psi_1=\psi_1(\Sigma)$ sufficiently small so that if $w \in C^2_\textnormal{loc}(\Sigma)$ has 
\[
\sup_\Sigma |w(\cdot)| , \sup_\Sigma |\nabla_\Sigma w(\cdot)|\leq \psi_1
\]
 then $\Graph_\Sigma w(\cdot)$ is a uniformly small Lipschitz graph in balls of uniform size (over tangent planes to $\Sigma$) and $F(\Graph_\Sigma w(\cdot)) \leq (2-\delta) F(\Sigma)$ for some $\delta\in(0,1)$ fixed. 

Now, consider $T > 1$ and $v \in C^2_\textnormal{loc}(\Sigma \times [0,T))$ with 
\begin{equation}\label{eq:restart.graph.flow.1}
\sup_\Sigma |v(\cdot,\tau)| \leq \psi_1, \qquad \sup_\Sigma |\nabla_\Sigma v(\cdot,\tau)|\leq \psi_1
\end{equation}
for all $\tau \in [0,T)$. Assume that $\tau\mapsto\Gamma(\tau)$ is a rescaled mean curvature flow. Due to the choice of $\psi_1$ above, we can locally write the flow, 
\[
[-e,-1] \ni t\mapsto \sqrt{-t} \left( \Graph_\Sigma v(\cdot,\bar \tau - \log(-t))\right)
\]
as a uniformly bounded Lipschitz graphs over balls of uniform size. Interior estimates \cite[Corollary 3.5]{EckerHuisken:interior} for mean curvature flow applied at $t=-1$ thus imply that 
\[
\sup_{\Gamma(\bar\tau)}|\nabla^m A_{\Gamma(\bar\tau)}| \leq C'(m,\Sigma)
\]
This yields the asserted $C^m$-estimates as long as $\psi_0 \leq \psi_1$. 

It remains to choose $\psi_0=\psi_0(\Sigma) \leq \psi_1$ sufficiently small so that the second property holds. Before doing so, we observe that the curvature estimates just derived show that $\Gamma(\tau)$ smoothly limits to some $\Gamma(T)$ and combined with \cite[Theorem 4.2]{EckerHuisken:interior}, we find that there is $T''>T$ (independent of $\psi_0$) so that the rescaled mean curvature flow $\tau \mapsto \Gamma(\tau)$ exists for $\tau \in [0,T'']$. The same argument used above implies that 
\begin{equation}\label{eq:curv.est.boostrap}
\sup_{\Gamma(\bar\tau)}|\nabla^m A_{\Gamma(\bar\tau)}| \leq C'(m,\Sigma)
\end{equation}
holds for $\tau \in [T,T'']$ (after shrinking $T''$ if necessary). 

We now fix $\psi_0=\psi_0(\Sigma,\psi_1) \leq \psi_1$ so that if $\Gamma \subset \RR^{n+1}$ is a properly embedded hypersurface with $\Gamma \subset U_{\psi_0}(\Sigma)$, $F(\Gamma) \leq (2-\delta) F(\Sigma)$ and $|A_\Gamma| \leq C'(1,\Sigma)$ (the constant in \eqref{eq:curv.est.boostrap} with $m=1$), then 
\[
\Gamma = \Graph_\Sigma w,
\]
where $w\in C^2_\textrm{loc}(\Sigma)$ satisfies
\[
\sup_\Sigma |w(\cdot)|\leq \psi_0,\qquad \sup_\Sigma |\nabla_\Sigma w(\cdot)|\leq \psi_1.
\]
(To see that $\psi_{0}$ exists, note that as $\psi_{0}\to 0$, $\Gamma$ will converge to $\Sigma$ in $C^{1,\alpha}$ due to the entropy and second fundamental form bounds. In this limit, we have the uniform decay $\nabla_\Sigma w = o(1)$, while $|w|\leq \psi_0$ follows from $\Gamma \subset U_{\psi_0}(\Sigma)$.) 

We now assume that 
\[
\limsup_{\tau \to T} \sup_\Sigma |v(\cdot,T)| < \psi_0
\]
where we have chosen $\psi_0$ above. Using small spheres as barriers we can conclude that there is $T'\in[T,T'']$ so that $\Gamma(\tau) \subset U_{\psi_0}(\Sigma)$ for $\tau \in [T,T']$. Moreover, $F(\Gamma(\tau)) \leq (2-\delta)F(\Sigma)$ by monotonicity of $F$ under the rescaled flow. Finally, \eqref{eq:curv.est.boostrap} implies that $|A_{\Gamma(\tau)}| \leq C'(1,\Sigma)$ for $\tau \in [T,T']$. Thus, by the choice of $\psi_0$, we see that $\Gamma(\tau) = \Graph_\Sigma v(\cdot,\tau)$ with 
\[
\sup_\Sigma |v(\cdot,\tau)|\leq \psi_0,\qquad \sup_\Sigma |\nabla_\Sigma v(\cdot,\tau)|\leq \psi_1.
\]
for all $\tau \in [T,T']$. This completes the proof. 
\end{proof}

\begin{corollary}[Existence of graphical flows]\label{coro:exist.graph.bootstrap}
Fix a shrinker $\Sigma \in \cS_n'$. There is a constant $\psi^* = \psi^*(\Sigma)$ so that if $v_0 \in C^2(\Sigma)$ has $\Vert v_0\Vert_{C^2(\Sigma)} \leq \psi^*$ then there is $T\in [1,\infty]$ and $v \in C^2_\textnormal{loc}(\Sigma\times [0,T])$ so that:
\begin{enumerate}
\item $v(\cdot,0) = v_0(\cdot)$, 
\item $\Gamma(\tau) : = \Graph_\Sigma v(\cdot,\tau)$ is a rescaled mean curvature flow, 
\item we have 
\[
\sup_\Sigma |\nabla^m v(\cdot,\tau)| \leq C(m,\Sigma)
\]
for all $\tau \in [1,T]$ and $m \in\NN$, and
\item if $T<\infty$, then it must hold that
\[
\sup_\Sigma |v(\cdot,T)| = \psi_0,
\]
for $\psi_0=\psi_0(\Sigma)$ fixed in Proposition \ref{prop:bootstrap.graphical}. 
\end{enumerate}
\end{corollary}
\begin{proof}
Taking $\psi^*=\psi^*(\Sigma)$ sufficiently small it is easy to see (e.g.\ using \cite[Theorem 4.2]{EckerHuisken:interior}) that the graphical rescaled mean curvature flow starting at $v_0$ exists and satisfies 
\begin{equation}\label{eq:bootstrap.graphical.conditions}
\sup_\Sigma |v(\cdot,\tau)|\leq \psi_0, \qquad \sup_\Sigma |\nabla_\Sigma v(\cdot,\tau)| \leq \psi_1,
\end{equation}
at least for $\tau \in [0,2]$. Choose $T \in (1,\infty]$ maximal so that the flow exists and satisfies \eqref{eq:bootstrap.graphical.conditions} for $\tau \in [0,T)$. Property (3) for $\tau < T$ follows immediately from Proposition \ref{prop:bootstrap.graphical}. This verifies (1) and (2) (since if $T<\infty$ then (3) guarantees the flow extends to $\tau =T$.)

In the case that $T=\infty$, condition (4) is vacuous so the proof is completed. Thus, we must verify (4) when $T<\infty$. If (4) fails, then 
\[
\sup_\Sigma |v(\cdot,T)| < \psi_0.
\]
However, due to Proposition \ref{prop:bootstrap.graphical}, this would would contradict the maximality of $T$, so this case cannot occur. This completes the proof.  
\end{proof}

\begin{proposition}[Barrier principle for graphical flows]\label{prop:barrier-graphical}
Fix a shrinker $(\Sigma,\Omega) \in \cS_{n}'$ and the corresponding constant $\psi^*$ from Corollary \ref{coro:exist.graph.bootstrap}. For $v_0 \in C^2(\Sigma) \setminus\{0\}$ with $v_0 \geq 0$ and $\Vert v_0\Vert_{C^2(\Sigma)} \leq \psi^{*}$, let $[0,T]\ni \tau \mapsto v(\cdot,\tau)$ denote the function whose graph is a rescaled mean curvature flow as in Corollary \ref{coro:exist.graph.bootstrap}. Then $1\leq T< \infty$. 

In particular,
\[
\sup_\Sigma |v(\cdot,T)| = \psi_0.
\]
Moreover, if $[0,T]\ni \tau \mapsto \cM(\tau)$ is a rescaled Brakke flow with $\supp \cM(\tau) \subset\Omega$ for all $\tau \in [0,T]$ and
\[
\supp\cM(0) \subset \Omega\setminus \{\bx + s v_0(\bx)\nu_\Sigma(\bx) : \bx \in\Sigma,s\in [0,1]\}
\]
then
\[
\supp\cM(\tau) \subset \Omega\setminus  \{\bx + s v(\bx,\tau)\nu_\Sigma(\bx) : \bx \in\Sigma,s\in [0,1)\}. 
\]
\end{proposition}
\begin{proof}
Assume that $T = \infty$. First note that $v(\cdot,\tau) >0 $ for all $\tau \in [0,\infty)$ by the Ecker--Huisken maximum principle \cite[Theorem 4.3]{EckerHuisken:interior}. Then, Ilmanen's localized avoidance principle (Theorem \ref{theo:ilmanen-avoidance}) applied to $t\mapsto \sqrt{-t}\Sigma$ and $t\mapsto \Graph_\Sigma v(\cdot,-\log(-t))$ implies that for any $R>0$, there is $t_0=t_0(R) < 0$ so that $\Graph_\Sigma v(\cdot,-\log(-t_0)) \cap B_R(\bOh) = \emptyset$. This is a contradiction, when $R$ is large. Thus $T<\infty$, so the ``blowup'' condition
\[
\sup_\Sigma |v(\cdot,T)| = \psi_0.
\]
follows from (4) in Corollary \ref{coro:exist.graph.bootstrap}. 

The remaining statement (``Moreover...'') follows from Proposition \ref{prop:avoidance-noncompact}. 
\end{proof}

\section{Existence of a one-sided ancient flow}\label{sec:existence-one-sided-ancient}

Fix $n\leq 6$ and a shrinker $(\Sigma, \Omega) \in  \cS''_n$ with nice stable ends. (Recall that this means that the ends of $\Sigma$ are smoothly asymptotic to cones or cylinders and the first eigenfunction of $L$ is bounded.) We construct here an ancient shrinker mean convex rescaled Brakke flow inside of $\Omega$ converging to $\Sigma$ (with multiplicity one) in the ancient past.

Recall that the unit normal to $\Sigma$ points into $\Omega$. By Proposition \ref{lem:decay-lowest-eigenfuncton}, the first eigenfunction $\varphi$ (with first eigenvalue $\mu$) of $L$ along $\Sigma$ has the property that  there is $\eps_0 >0$ so that for $\eps\in (0,\eps_0)$, the normal graph of $\eps \varphi$ (denoted here by $\Sigma_\eps$) is a smooth strictly shrinker mean convex hypersurface contained in $\Omega$. 

We first construct barriers for the one-sided flow.\footnote{These barriers allows us to simplify certain arguments used in \cite{CCMS:generic1}.}
\begin{lemma}[Global barriers for one-sided flow]\label{lemm:sub-super-one-sided}  
There is $M =M(\Sigma)>0$ sufficiently large and $\tau_0 = \tau_0(\Sigma)$ sufficiently negative so that $Me^{-\mu\tau_0} \leq 1/2$ and so that defining functions
\[
v_\pm(\bx,\tau) : =  (e^{-\mu \tau}\pm M e^{-2\mu \tau}) \varphi(\bx),
\]
 the hypersurfaces $\Sigma_\pm(\tau) : = \Graph_\Sigma v_\pm(\cdot,\tau)$ satisfy:
\[
(-\infty,\tau_0] \ni \tau \mapsto \Sigma_+(\tau)
\]
is a supersolution\footnote{Namely, a smooth solution to rescaled mean curvarure flow cannot make contact with $\Sigma_+$ from \emph{below} (with respect to the unit normal induced by $\Sigma$). } to rescaled mean curvature flow and
\[
(-\infty,\tau_0] \ni \tau \mapsto \Sigma_-(\tau)
\]
is a subsolution to rescaled mean curvature flow.
\end{lemma}
\begin{proof}
We consider $\Sigma_+$ for simplicity (the argument for $\Sigma_-$ is essentially identical). We first compute
\[
(\partial_\tau - L) v_+ =   - M  \mu e^{-2\mu \tau} \varphi = M  |\mu| e^{-2\mu \tau} \varphi
\]
By Corollary \ref{coro:expand-rescaled-mcf-app}, 
\[
(\nu_{\Sigma_+}\cdot\nu_\Sigma)^{-1} (\partial_{\tau}\bx_{\Gamma} - \bH -\tfrac 12 \bx_{\Gamma}) \cdot \nu_{\Gamma} = (\partial_{\tau} - L) u + E = M |\mu| e^{-2\mu\tau}\varphi + E,
\]
at $\bx_{\Gamma} = \bx + v_{+}(\bx,\tau)\nu_{\Gamma}(\bx)$ for  $E=v_{+}E_1 + E_2(\nabla v_{+},\nabla v_{+})$ with $E_1=O(e^{-\mu\tau}), E_2(\cdot, \cdot)=O(1)$ uniformly on $\Sigma$ as $\tau\to-\infty$. Note that 
\begin{align*}
|E(\bx)| & \leq v_{+}(\bx) |E_{1}(\bx)| + |E_{2}(\nabla v_{+}(\bx),\nabla v_{+}(\bx))|\\
& \leq C e^{-\mu\tau} (e^{-\mu\tau} + M e^{-2\mu\tau}) \varphi(\bx) + C (e^{-\mu\tau} + M e^{-2\mu\tau})^{2} |\nabla \varphi(\bx)|^{2}\\
&  \leq C (1 + M^2 e^{-2\mu \tau})  e^{-2\mu\tau} \varphi(\bx). 
\end{align*}
The final estimate uses $\varphi^{-1}|\nabla \varphi|^{2} \in L^{\infty}(\Sigma)$ as proven in part (3) of Lemma \ref{lem:decay-lowest-eigenfuncton}. Hence, 
\begin{align*}
(\nu_{\Sigma_+}\cdot\nu_\Sigma)^{-1} (\partial_{\tau}\bx_{\Gamma} - \bH -\tfrac 12 \bx_{\Gamma}) \cdot \nu_{\Gamma}  & \geq (M|\mu| - C - C M^2 e^{-2\mu \tau})  e^{-2\mu \tau}\varphi
\end{align*}
Taking $M$ large and then $\tau_{0}$ sufficiently negative ensures that $M|\mu| - C - C M^2 e^{-2\mu \tau} >0$ for all $\tau \in (-\infty,\tau_0)$. This proves the assertion. 
\end{proof}

The following result is an extension of \cite[Proposition 4.4]{BernsteinWang:TopologicalProperty} to allow for asymptotically cylindrical ends. 
We will denote $\eps_0'=\eps_0'(\Sigma)$ the constant from Lemma \ref{lem:decay-lowest-eigenfuncton} (denoted by $\eps_0$ there).  
\begin{proposition}[Existence of pre-ancient graphical one-sided flow]\label{prop:short-time-exist-Sigma-eps}
There is $\eps_0 = \eps_0(\Sigma) \in (0,\eps_0']$ so that for $\eps \in (0,\eps_0)$ there is a smooth mean curvature flow $\Sigma_\eps(t)$ for $t \in [-1,- (\eps/\eps_0)^{1/|\mu|}]$ with $\Sigma_\eps(-1) = \Sigma_\eps$ and $C=C(\Sigma)$ with the following properties:
\begin{enumerate}
\item the flow $\Sigma_\eps(t)$ remains strictly shrinker mean convex to the inside with the bound
\[
2t H_{\Sigma_\eps(t)} + \mathbf{x}\cdot\nu_{\Sigma_\eps(t)} \geq C^{-1}\eps (1+|\mathbf{x}|^2 + 2n(t+1))^{\mu} 
\]
for all $t \in [-1,- (\eps/\eps_0)^{1/|\mu|} ]$,
\item the rescaled surface $\tfrac{1}{\sqrt{-t}}\Sigma_\eps(t)$ is the normal exponential graph over $\Sigma$ of a function $v_\eps(\cdot,t)$ with $v_\eps \in C^\infty(\Sigma \times [-1, - (\eps/\eps_0)^{1/|\mu|}])$ and $\|v_\eps(\cdot,t)\|_{C^0(\Sigma)} \leq 2\eps_0$
for $t \in  [-1,- (\eps/\eps_0)^{1/|\mu|} ]$,
\item $\Vert v_\eps(\cdot,t)\Vert_{C^3(\Sigma)} \leq C $ for all $t \in [-1, - (\eps/\eps_0)^{1/|\mu|}]$, and
\item the following estimate
\[
\frac{\eps}{2} e^{-\mu \tau} \varphi(\bx) \leq v_{\eps}(\bx,-e^{-\tau}) \leq 2 \eps e^{-\mu \tau} \varphi(\bx)
\]
holds for all $(\bx,\tau) \in \Sigma\times [0,\tfrac{1}{\mu} (\log(\eps) - \log(\eps_0))]$. 
\end{enumerate} 
\end{proposition}
\begin{proof}
Fix $\psi^{*}=\psi^{*}(\Sigma)$ and $\psi_{0}=\psi_{0}(\Sigma)$ as in Corollary \ref{coro:exist.graph.bootstrap}. Choose $\eps_{0} \in (0,\eps_{0}']$ (for $\eps_{0}'$ the constant from Lemma \ref{lem:decay-lowest-eigenfuncton} so that (i) $\eps_{0} \leq \Vert \varphi\Vert_{C^{2}(\Sigma)}^{-1}\psi^{*}$, (ii)  $\eps_{0} \Vert \varphi\Vert_{C^{0}(\Sigma)} (1+M \eps_{0}) < \psi_{0}$ for $M$ fixed in Lemma \ref{lemm:sub-super-one-sided} and (iii) $\eps_{0}\leq e^{-\mu\tau_{0}}$, for $\tau_{0}$ fixed in Lemma \ref{lemm:sub-super-one-sided}.

Condition (i) above and (the first part of) Proposition \ref{prop:barrier-graphical} yield $T \in [1,\infty)$ and a graphical rescaled mean curvature flow $[0,T]\ni \tau \mapsto v_{\eps}(\cdot,\tau)$ with $v_{\eps}(\cdot,0) = \eps \varphi(\cdot)$ and so that
\begin{equation}\label{eq:pre-ancient-blowup-barrier-crt}
\sup_{\Sigma}|v_{\eps}(\cdot,T)|  = \psi_{0}. 
\end{equation}

Choose $\bar \tau \leq \tau_0$ so that $e^{-\mu \bar \tau} = \eps$ (possible by (iii) above). Then we have that
\[
 v_-(\bx,\bar \tau) =  \eps (1- M\eps) \varphi(\bx) \leq \eps \varphi(\bx) \leq  \eps (1+ M\eps) \varphi(\bx)  = v_+(\bx,\bar \tau) 
\]
where $v_{\pm}$ are defined in Lemma \ref{lemm:sub-super-one-sided}. The Ecker--Huisken maximum principle \cite[Theorem 4.3]{EckerHuisken:interior} and Lemma \ref{lemm:sub-super-one-sided} yields
\[
 v_-(\bx,\bar \tau + \tau)  \leq v_{\eps}(\bx,\tau) \leq  v_+(\bx,\bar \tau + \tau) 
\]
for $\tau \leq T_{0} : = \min\{T,\tau_{0}-\bar \tau\}$. 

We claim that $T_{0}\geq \mu^{-1} \log (\eps/\eps_{0})$. Indeed, by (iii) (and $\mu<0$) we note that
\[
\mu^{-1} \log (\eps/\eps_{0}) \leq \tau_0 - \bar \tau.
\]
Thus, if $T_{0}< \mu^{-1} \log (\eps/\eps_{0})$ then $T_{0} = T$. Using \eqref{eq:pre-ancient-blowup-barrier-crt} and (ii) above, we thus find
\[
\psi_{0} = \sup_{\Sigma} |v_{\eps}(\cdot,T)| \leq \sup |v_{+}(\cdot,\bar \tau + \mu^{-1}\log(\eps/\eps_{0}))| \leq \eps_{0} \Vert \varphi\Vert_{C^{0}(\Sigma)} (1+M \eps_{0}) < \psi_{0} .
\]
This is a contradiction. The graph of $v_{\eps}(\cdot,\tau)$ yields the desired mean curvature flow (after rescaling), with the asserted existence time.

The shrinker mean convexity follows from the Ecker--Huisken maximum principle and Lemma \ref{lem:decay-lowest-eigenfuncton} exactly as in \cite[Proposition 3.3]{BernsteinWang:TopologicalProperty}. The remaining properties follow by construction. This finishes the proof.
\end{proof}

\begin{proposition}[Existence of ancient graphical one-sided flow]\label{prop:exist-graphical-one-sided}
Fix $(\Sigma,\Omega)\in \cS_n''$. Set $\tau_0 = |\mu|^{-1}\log \eps_0$. There is an ancient graphical shrinker mean convex rescaled mean curvature flow $  (-\infty,\tau_0] \ni \tau \mapsto \Graph_\Sigma v(\cdot,\tau)$  with  $\lambda(\Graph_\Sigma v(\cdot,\tau)) \leq F(\Sigma)$, $v(\bx,\tau) > 0$ for all $(\bx,\tau) \in \Sigma \times (-\infty,\tau_0]$, as well as the following properties for $\tau \in (-\infty,\tau_0]$:
\begin{enumerate}
\item $\Vert v(\cdot,\tau)\Vert_{C^3(\Sigma)} \leq C$,
\item $\frac{1}{2} e^{-\mu \tau} \varphi \leq v(\cdot, \tau) \leq 2 e^{-\mu \tau} \varphi$,
\item  $\Vert v(\cdot, \tau)\Vert_{C^k(\Sigma)} \leq C_k e^{-\mu \tau}$ for all $k\in \NN$, and
\item $ |e^{\mu\tau} \langle v(\cdot, \tau), \varphi\rangle_W -a| \leq Ce^{-\mu \tau}$ for some fixed $a >0$. 
\end{enumerate}
\end{proposition}
\begin{proof}
For $\Sigma_\eps(t)$ defined in Proposition \ref{prop:short-time-exist-Sigma-eps}, we set
\[
S_\eps(\tau) : = e^{(\tau-\tau(\eps))/2} \Sigma_\eps(-e^{-\tau+\tau(\eps)})
\]
for $\tau \in [\tau(\eps),\tau_0]$, $\tau_0=|\mu|^{-1}\log\eps_0$, and $\tau(\eps) = |\mu|^{-1}\log\eps$. Observe that $S_\eps(\tau)$ is a rescaled mean curvature flow and $S_\eps(\tau)$ is the graph of $\hat v_\eps(\cdot,\tau) : = v_\eps(\cdot,-e^{-\tau+\tau(\eps)})$, defined in Proposition \ref{prop:short-time-exist-Sigma-eps}. Property (3) in Proposition \ref{prop:short-time-exist-Sigma-eps} yields $\Vert \hat v_\eps(\cdot,\tau)\Vert_{C^3(\Sigma)} \leq C$ for $\tau \in [\tau(\eps), \tau_0]$. Similarly, (4) in Proposition \ref{prop:short-time-exist-Sigma-eps} yields 
\begin{equation}\label{eq:barrier_est}
\frac{1}{2} e^{-\mu \tau} \varphi(\bx) \leq \hat v_{\eps}(\bx, \tau) \leq 2 e^{-\mu \tau} \varphi(\bx)
\end{equation}
for $(\bx,\tau) \in\Sigma\times [\tau(\eps), \tau_0]$. 

We can thus find $\eps_i\to 0$ so that $\hat v_\eps$ converges in $C^\infty_\textrm{loc}(\Sigma\times (-\infty,\tau_0])$ to $v(\cdot,\tau)$ with $\Vert v(\cdot,\tau) \Vert_{C^3(\Sigma)} \leq C$ for $\tau \in (-\infty,\tau_0]$ and so that $\Graph_\Sigma v(\cdot,\tau)$ is  rescaled mean curvature flow. This proves property (1). Property (2) follows by passing \eqref{eq:barrier_est} to the limit. Property (3) follows from this, combined with interior Schauder estimates applied to the (non-rescaled) mean curvature flow equation. 

Finally, to prove (4), we recall that by \ref{coro:expand-rescaled-mcf-app}, $v$ satisfies
$$ \partial_\tau u = Lu + E\, ,$$
where (3) implies that $|E| \leq C e^{-2\mu\tau}$. Since 
$$ \partial_\tau \langle v,\varphi\rangle_W = -\mu \langle v,\varphi\rangle_W + \langle E,\varphi\rangle_W\, ,$$
the function $f(\tau):= e^{\mu \tau} \langle v,\varphi\rangle_W$ satisfies
$$ |\partial_\tau f| \leq C e^{-\mu \tau}\, .$$
Integrating this proves (4) (note that $a>0$ by (2)). 

The shrinker mean convexity follows from the Ecker--Huisken maximum principle and Lemma \ref{lem:decay-lowest-eigenfuncton} exactly as in \cite[Proposition 3.3]{BernsteinWang:TopologicalProperty}. The remaining properties follow by construction. This finishes the proof.
\end{proof}

\subsection{The flow for all times $t<0$} In this section we extend the flow constructed previously to all (non-rescaled) times $t<0$. To do so, we must assume that $\Sigma \in \cS_n'''$ has nice, \emph{strictly} stable ends. Recall that this means that the first eigenfunction $\varphi$ decays to $0$ at spatial infinity. 

\begin{remark}
We note that the strictly stable end condition excludes cylinders $\hat\Sigma \times \RR$, $\hat\Sigma \in \cS_{n-1}$ compact. In the exact cylindrical case one can construct the ancient flow coming out of $\hat\Sigma$ (using the arguments discussed below) and then cross with $\RR$. As such, we will not discuss this case further in this section. 
\end{remark}

For later reference we also note the following basic stability result.
\begin{lemma}[Stability of cylinders] \label{lem:smooth-stability-cylinder}
For every $\delta > 0$, $\Lambda>0$ there exists a $D \geq \delta^{-1}$ such that the following holds. Let $\cM$ be unit regular integral Brakke flow, such that $\lambda(\cM) \leq \Lambda$ and $\cM(-1)\lfloor B_{D}(\bOh)$ is a smooth graph over $\hat{\Sigma}_j \times \RR_{\bx_j}$ with the $C^3$-norm of the graph function bounded by $D^{-1}$. Then $\cM(t)$ remains a smooth graph over $(\hat{\Sigma}_j)_t \times \RR_{\bx_j}$ on $B_{\delta^{-1}}$ for all $t \in [-1, -\delta^2]$ with (rescaled) $C^3$-norm bounded by $\delta$.
\end{lemma}
\begin{proof} By a direct contradiction argument, this follows from the uniqueness of the evolution of $\hat{\Sigma}_j \times \RR_{\bx_j}$.
\end{proof}

We recall the definition of the parabolic dilation map
\[ \cF_\lambda : \RR^{n+1} \times \RR\to  \RR^{n+1} \times \RR, \qquad \cF_\lambda : (\bx,t) \mapsto (\lambda \bx,\lambda^2 t) \]
and also that $\ft(\bx,t) = t$ is the projection onto the time coordinate.

In the following we will consider $\eps_0=\eps_0(\Sigma)$ as in Proposition \ref{prop:short-time-exist-Sigma-eps} 
and $t_0=-\eps_0^{-1/|\mu|}$. 

\begin{proposition}[Existence of a pre-ancient  weak one-sided flow]\label{prop:long-time-exist-Sigma-eps}
For $(\Sigma,\Omega)\in \cS_n''$ and $\eps \in (0,\eps_0)$ there exists unit-regular integral Brakke flows (resp.\ weak set flows $\cK_\eps$) with initial condition $\cM_\eps(-1) = \cH^n\lfloor \Sigma_\eps$ (resp.~$\cK_\eps(-1) = K_\eps$, the unique closed set with $K_\eps \subset \Omega$ and $\partial K_\eps = \Sigma_\eps)$
with the following properties:
\begin{enumerate}
 \item $\partial \cK_\eps \setminus \ft^{-1}(-1) \subset \supp \cM_\eps \subset \cK_\eps$
 \item for $t\in [-1, - (\eps/\eps_0)^{1/|\mu|}]$ we have $ \cM_\eps(t) = \cH^n\lfloor \Sigma_\eps(t)$, for $\Sigma_\eps(t)$ the flow constructed in Proposition \ref{prop:short-time-exist-Sigma-eps}. 
\end{enumerate}

Now, assume that $\Sigma \in \cS_n'''$ has strictly stable ends and define the dilated flows
\[
\hat\cM_\eps = \cF_{\eps^{-1/(2|\mu|)}}(\cM_\eps) \text{ resp. } \hat\cK_\eps = \cF_{\eps^{-1/(2|\mu|)}}(\cK_\eps).
\]
There is a monotonically non-decreasing function $R(t)$ on $(-\infty,0)$ with $R(t) \equiv 0$ for $t \in (-\infty,t_0]$ independent of $\eps \in (0,\eps_0)$ so that
\begin{equation}\label{eq:rescaled-pre-ancient-weak-reg-infty}
\hat\cM_\eps(t) \lfloor (\RR^{n+1}\setminus B_{R(t)}) \text{ and } \partial\hat\cK_\eps(t) \cap (\ft^{-1}(t)\setminus B_{R(t)})
\end{equation}
are the same, smooth, strictly shrinker mean convex mean curvature flow, which we will denote by $\hat\Sigma_\eps(t)$; moreover $\tfrac{1}{\sqrt{-t}} \hat\Sigma_\eps(t)$  can be written as a small $C^2$-graph over a subset of $\Sigma$. 
\end{proposition}
\begin{proof}
Approximating $\Sigma_\eps$ by smooth compact hypersurfaces whose level set flow does not fatten, we may use \cite{Ilmanen:elliptic} to find a sequence of compact unit regular integral Brakke flows $\cM_{\eps, \rho}$ and weak set flows $\cK_{\eps,\rho}$ satisfying
 \begin{equation}\label{eq:approx_flows}
  \supp \cM_{\eps,\rho} \cap \ft^{-1}((-1,\infty)) = \partial\cK_{\eps,\rho}\cap \ft^{-1}((-1,\infty))
 \end{equation}
 with $\cM_{\eps,\rho}$ converging to a Brakke flow $\cM_\eps$ as Brakke flows and $\cK_{\eps,\rho}$ converging to a weak set flow $\cK_\eps$ in the local Hausdorff sense (the flows $\cM_\eps,\cK_\eps$ are easily seen to have the asserted initial conditions). See  \cite[Proposition 7.3]{CCMS:generic1} for the (standard) details of this construction. 
 
 We recall here the argument from \cite[Lemma 7.5]{CCMS:generic1} proving (1). For $X \in \partial\cK_\eps\setminus \ft^{-1}(-1)$, there is $X_i \in \partial\cK_{\eps,\rho_i}\setminus \ft^{-1}(-1) = \supp\cM_{\eps,\rho_i}\setminus \ft^{-1}(-1)$ converging to $X$. The monotonicity formula guarantees that $X \in \supp\cM_\eps$. This proves the first inclusion in (1). For the second inclusion, take $X \in \supp\cM_\eps$. The same argument yields $X_i \in \cK_{\eps,\rho_i}$ with $X_i \to X$. Thus $X \in \cK_\eps$ since $\cK_\eps$ is closed. 
 
 Using pseudolocality, unit-regularity and uniqueness of smooth solutions to mean curvature flow with bounded curvature \cite[Theorem 1.1]{ChenYin}, it follows that $\cM_\eps(t) = \cH^n\lfloor \Sigma_\eps(t)$ as long as $\Sigma_\eps(t)$ has bounded curvature (which holds at least until $t = -(\eps/\eps_0)^{1/|\mu|}$). Thus the flows satisfy (2) as well. 
 
It thus remains to prove \eqref{eq:rescaled-pre-ancient-weak-reg-infty} under the assumption that $\Sigma \in \cS_n'''$ has strictly stable ends. We note that
\[
\cM_\eps(t_0) = \cH^n\lfloor \hat\Sigma_\eps(t_0)
\]
where $(-t_0)^{-1/2} \hat\Sigma_\eps(t_0) = (\eps/\eps_0)^{-1/2|\mu|} \Sigma_\eps((\eps/\eps_0)^{1/|\mu|}) = \Graph_\Sigma v_\eps(\cdot, (\eps/\eps_0)^{1/|\mu|})$ where 
\[
v_\eps(\bx,(\eps/\eps_0)^{1/2|\mu|}) \leq 2 \eps (\eps/\eps_0) \varphi(\bx) = 2 \eps_0 \varphi(\bx) 
\]
(this follows from (4) in Proposition \ref{prop:short-time-exist-Sigma-eps}). In particular, the strongly stable end hypothesis (and (3) in Lemma \ref{lem:decay-lowest-eigenfuncton}) implies that for any $\bx_i \in \hat\Sigma_\eps(t_0)$, $(-t_0)^{-1/2} \hat\Sigma_\eps(t_0) - \bx_i$ converges in $C^3_\textrm{loc}$ (up to passing to a subsequence and applying a fixed rotation) to either $\RR^n$ (on the conical part) or $\hat\Sigma_j\times \RR$ (one of the finitely many cylindrical ends of $\Sigma$). In particular, by applying pseudolocality (in the conical case) or stability of cylinders Lemma \ref{lem:smooth-stability-cylinder} (in the cylindrical case) we find $R(t)$ monotonically non-decreasing so that 
\[
\hat\cM_\eps(t) \lfloor (\RR^{n+1}\setminus B_{R(t)})
\]
is a smooth mean curvature flow, which we will denote by $\hat\Sigma_\eps(t)$; moreover $\tfrac{1}{\sqrt{-t}} \hat\Sigma_\eps(t)$  can be written as a small $C^2$-graph over a subset of $\Sigma$. Unit regularity at the level of the approximating flows $\cM_{\eps,\rho}$ guarantees that the same thing holds for $\cK_\eps$. 

That $[-\eps^{-1/|\mu|},t_0]\mapsto \hat\Sigma_\eps(t)$ is strictly shrinker mean convex follows from Proposition \ref{prop:short-time-exist-Sigma-eps}. Finally, we can use Ilmanen's localized avoidance principle (cf.\ Theorem \ref{theo:ilmanen-avoidance}) and the Ecker--Huisken maximum principle as in \cite[Corollary 7.11]{CCMS:generic1} to see that $\hat\Sigma_\eps(t)$ is weakly shrinker mean convex for all $t \in [-\eps^{-1/|\mu|},0)$.  Strict mean convexity then follows from the strong maximum principle and the barriers constructed in property (4) of Proposition \ref{prop:short-time-exist-Sigma-eps}. This completes the proof.
\end{proof}

\begin{remark}
In \cite[\S7]{CCMS:generic1} (i.e., for $\Sigma$ with only asymptotically conical ends), the estimate $|\bx| |A| + |\bx|^2 |\nabla A| + |\bx|^3 |\nabla^2 A|\leq C$ for the flow $\hat\Sigma_\eps$ is proven and used in several places. Such an estimate will not hold for asymptotically cylindrical ends, so in any reference below to  \cite[\S7]{CCMS:generic1}, such an estimate should be replaced by $|A| + |\nabla A| + |\nabla^2A|\leq C$. Fortunately, this does not affect the logic of any of the arguments referred to below.
\end{remark}

\begin{proposition}[Existence of ancient weak flow]\label{prop:exist-ancient-rescaled-MCF-C2-rescaling-argument}
Fix $\Sigma \in \cS'''_n$. 
There exists an ancient unit regular integral Brakke flow $\cM$ and a corresponding weak set flow $\cK$ with the following properties:
\begin{enumerate}
\item $\partial \cK \subset \supp \cM \subset \cK$,
\item $\lambda(\cM) \leq F(\Sigma)$,
\item there is a continuous, monotonically increasing function $R(t)$ on $(-\infty,0)$,  with $R(t)=0$ for $t \in (-\infty , t_0]$, $t_0<0$, such that
$$ \cM(t) \lfloor (\RR^{n+1}\setminus B_{R(t)}(\bOh)) \text{ and } \partial\cK \cap (\ft^{-1}(t)\setminus B_{R(t)}(\bOh)) $$
are the same, smooth, strictly shrinker mean convex mean curvature flow, which we will denote by $\Sigma(t)$; moreover $(-t)^{-1/2} \Sigma(t)$  can be written as a small $C^3$-graph over a subset of $\Sigma$, 
\item for $t \in (-\infty ,t_0]$,  $ (-t)^{-1/2}\Sigma(t) =\Graph_\Sigma v(\cdot, \tau)$, $t = -e^{-\tau}$, for $v$ as in Proposition \ref{prop:exist-graphical-one-sided}. 
\end{enumerate}
\end{proposition}
\begin{proof}
The proof of Proposition \ref{prop:exist-graphical-one-sided} yields $\eps_i\to0$ so that $[-\eps^{1/|\mu|},t_0]\ni t\mapsto \hat\Sigma_\eps(t)$ converges in $C^\infty_\textrm{loc}$ to $(-\infty,t_0] \ni t\mapsto (-t)^{-1/2} \Graph_\Sigma v(\cdot, -\log(-t)) : = \Sigma(t)$  for $v$ as in Proposition \ref{prop:exist-graphical-one-sided}. Proposition \ref{prop:long-time-exist-Sigma-eps} allows us to find $\Sigma(t)$ a flow in $\RR^{n+1}\setminus B_{R(t)}(\bOh)$ so that (3) is satisfied.\footnote{A priori, $\Sigma(t)$ must only be weakly shrinker mean convex, but by  Proposition \ref{prop:exist-graphical-one-sided} it is strictly shrinker mean convex for $t \in (-\infty,t_0]$ so the strong maximum principle guarantees strict shrinker mean convexity of $\Sigma(t)$. } Letting the weak flows $\hat\cM_{\eps_i},\hat\cK_{\eps_i}$ pass to a subsequential limit we find an ancient unit regular integral Brakke flow $\cM$ and weak set flow $\cK$ which (thanks to Proposition \ref{prop:long-time-exist-Sigma-eps}) satisfy the various assertions. 
\end{proof}

We can now combine the previous results with Ilmanen's localised avoidance principle to extend the strict shrinker mean convexity to the whole flow for $t<0$ (cf.\ \cite[Theorem 7.17]{CCMS:generic1}). Recall that $\cK^\circ$ denotes the interior of $\cK$. 

\begin{proposition}[Basic properties of ancient weak flow I] \label{theo:basic-prop-cK-cM} The ancient unit regular integral Brakke flow and the corresponding weak set flow $(\cM, \cK)$ from the previous proposition have the following properties: let $\check\cK = \cK \cap \ft^{-1}((-\infty, 0))$, then
\begin{enumerate}
\item we have $d((\bOh,0), \check\cK) > 0$,
\item for $\lambda > 1$ we have $\cF_\lambda(\check\cK)\subset \check\cK^{\circ}$ and $\supp\cM \cap \cF_\lambda(\check\cK) = \emptyset$.
\end{enumerate}
\end{proposition}

\begin{proof}
Since the flow $t\mapsto \Sigma(t)$ is strictly shrinker mean convex we have that $\Sigma(t) \cap \Sigma = \emptyset$ for all $t<0$. Thus, using Ilmanen's localised avoidance principle, Theorem \ref{theo:ilmanen-avoidance},  we see that $\cM(t) \cap \Sigma = \emptyset$ for all $t<0$. Applying Ilmanen's localised avoidance principle another time, we thus see that $d((\bOh,0), \supp \cM) > 0$, which by Proposition \ref{prop:exist-ancient-rescaled-MCF-C2-rescaling-argument} implies $(1)$.

To prove (2), we note that strict shrinker mean convexity of the exterior flow $\Sigma(t)$ guarantees that for $\lambda>1$, $\supp\cM$ and $\cF_{\lambda}(\check \cK)$ are disjoint outside of a set $D$ in space-time which has $D \cap \mathfrak{t}^{-1}([a,b])$ compact for any $a<b<0$. Thus, we may apply Ilmanen's localized avoidance principle, Theorem \ref{theo:ilmanen-avoidance}, to show that $\supp \cM$ and $\cF_{\lambda}(\check\cK)$ are indeed disjoint. Using again Proposition \ref{prop:exist-ancient-rescaled-MCF-C2-rescaling-argument} this completes the proof of (2). 
\end{proof}

We can replace \cite[Theorem 7.17]{CCMS:generic1} with Proposition \ref{prop:exist-ancient-rescaled-MCF-C2-rescaling-argument} and Proposition \ref{theo:basic-prop-cK-cM} for all the results in \cite[Section 8.1]{CCMS:generic1} to carry over. This yields the following result, replacing \cite[Corollary 8.12]{CCMS:generic1}.\footnote{Property (6) follows from \cite[Corollary 1.4]{ColdingMinicozzi:sing-generic}.}

\begin{proposition}[Basic properties of ancient weak flow II]\label{coro:summary-tleq0-non-rescaled}
For $n\leq 6$, fix $\Sigma \in \cS'''_n$. 
The flows $(\cM,\cK)$ constructed in Proposition \ref{prop:exist-ancient-rescaled-MCF-C2-rescaling-argument} have the following properties for $t<0$:
\begin{enumerate}
\item $\cM(t) = \cH^{n}\lfloor \partial\cK(t)$,
\item $\partial\cK(t) \subset \sqrt{-t}\Omega$ for all $t<0$,
\item $\sing\cM \cap\{t<0\}$ has parabolic Hausdorff dimension $\leq n-1$ and for $t<0$, $\sing\cM(t)$ has spatial Hausdorff dimension $\leq n-1$,
\item any limit flow at $X=(\bx,t)$ with $t<0$ is weakly convex on the regular part and all tangent flows are multiplicity one generalized cylinders, 
\item any singular point has a (strict) mean-convex neighborhood, and
\item if $n=3,4$, $\ft(\sing\cM)\cap\{t<0\}$ has measure zero. 
\end{enumerate}
\end{proposition}

\subsection{Conditional uniqueness results} Proposition \ref{coro:summary-tleq0-non-rescaled} implies the following conditional uniqueness result.

\begin{proposition}[Conditional uniqueness I]\label{prop:cond.unique.1}
 \label{theo:one.sided.uniqueness} For $n\leq 6$, fix $\Sigma \in \cS'''_n$. 
Let $(\cM,\cK)$ be as constructed in Proposition \ref{prop:exist-ancient-rescaled-MCF-C2-rescaling-argument}. Let $(\check \cM(t))_{-\infty < t < \infty}$ be an ancient unit-regular integral Brakke flow such that there is $T<0$ so that
	\begin{equation} \label{eq:one.sided.assumption}
		 \check \cM(t) = \cM(t) \text{ for all } t \leq  T.
	\end{equation}
Then $\check \cM(t) = \cM(t)$ for $t<0$.
\end{proposition}
\begin{proof} We can assume that
 $T<0$ is sufficiently negative so that for $t \leq T$, $\cM(t) = \cH^{n}\lfloor \Sigma(t)$, where $\Sigma(t)$ is the smooth flow from Proposition \ref{prop:exist-ancient-rescaled-MCF-C2-rescaling-argument} (3).

Below, we consider only times $t<0$. As in \cite[Proposition 7.10]{CCMS:generic1}, Ilmanen's localized avoidance principle (Theorem \ref{theo:ilmanen-avoidance}) combined with Ecker--Huisken's maximum principle at infinity \cite[Theorem D.1]{CCMS:generic1}, we see that $(\supp \check\cM)_{t}$ is disjoint from $\cF_{\lambda}(\supp \cM)$ for $\lambda \not = 1$. This implies that, $(\supp \check\cM)_{t} \subset \supp\cM$. 

Finally, since $\reg\cM$ is connected for by (5) in Proposition \ref{coro:summary-tleq0-non-rescaled} and \cite[Corollary F.5]{CCMS:generic1}, we see that $\check\cM(t) = \cM(t)$ in $(\sing \cM)^{c}$ (using the unit-regularity of $\cM$ and $\check\cM$). This completes the proof. 
\end{proof}

Under a global graphicality assumption, we can also ensure uniqueness. Note that this result holds with just the (non-strict) stable end hypothesis. 
\begin{proposition}[Conditional uniqueness II]\label{prop:cond.unique.2} For $\Sigma \in \cS_n''$, consider two ancient rescaled mean curvature flows $\Sigma_i(\tau)$ for $\tau \in (-\infty, \tau_0), i=1,2$, such that $\Sigma_i(\tau) = \Graph_\Sigma u_i(\cdot, \tau)$. There exists $\eta = \eta(\Sigma)>0$ such that if
\begin{equation}
 \label{eq:estimate-decay}
  \|u_i(\cdot, \tau)\|_{C^3(\Sigma)} \leq C e^{-\mu\tau}
\end{equation}
for $\tau \in (-\infty, \tau_0)$ and $i =1,2$ as well as 
$$ \lim_{\tau \to -\infty} e^{\mu \tau} \| w \|_W =0\, ,$$
where $w =u_1-u_2$, then $w\equiv 0$.
\end{proposition}
\begin{proof} The proof follows similarly as the proof of \cite[Corollary 5.2]{CCMS:generic1}. Denoting 
$$w: =u_1-u_2,\  E^w:= E(u_1)-E(u_2)$$ 
we have
$$ \big(\tfrac{\partial}{\partial \tau} - L \big) w = E^w\, .$$
From  Lemma \ref{lemm:relative-shrinker-mean-curvature} we see that we can write
$$ E^w= w F + \nabla_\Sigma w \cdot \mathbf{F} + \nabla^2_\Sigma w \cdot \mathcal{F}\, .$$
	We take the $L^2_W$ dot product of \eqref{eq:uniqueness.w.error} with $w$ to find
\begin{align*}
\bangle{E^w,w}_W & =   \bangle{wF,w}_W + \bangle{\nabla_\Sigma w\cdot \bF,w}_W + \bangle{\nabla^2_\Sigma w\cdot \cF,w}_W\\
& \leq C e^{-\mu\tau} \Vert w\Vert_{W,1}^2 + \bangle{\nabla^2_\Sigma w\cdot \cF,w}_W,
\end{align*}
where we recall that $\Vert f\Vert_{W,1}^2 = \Vert f\Vert_W^2 + \Vert \nabla_\Sigma f\Vert_W^2$ is the weighted Sobolev norm. To handle the remaining term, we integrate by parts. We obtain 
\[
 \bangle{\nabla^2_\Sigma w\cdot \cF,w}_W = O(\Vert w\Vert_{W,1}^2 \Vert \cF\Vert_{C^1}) +  \int_\Sigma w\, (\bx^T \otimes \nabla_\Sigma w) \cdot \mathcal{F} \, e^{-\tfrac{|\bx|^2}{4}}\, d\mathcal{H}^n.
 \]
 Thanks to the estimate $\Vert \cF\Vert_{C^1}\leq Ce^{-\mu\tau}$, the first term is as desired. Again by Lemma \ref{lemm:relative-shrinker-mean-curvature} we see that
$$ \left| \int_\Sigma w\, (\bx^T \otimes \nabla_\Sigma w) \cdot \mathcal{F} \,e^{-\tfrac14|\bx|^{2}} d\mathcal{H}^n \right| \leq C e^{-\mu\tau} \| w\|_{W,1}^2 \, ,$$
which combining with the previous estimates implies
	\begin{equation*} \label{eq:one.sided.decay.uniqueness.w.dot.error}
		|\langle w(\cdot, \tau), E^w(\cdot, \tau) \rangle_W| \leq C_2 e^{-\mu \tau} \Vert w(\cdot, \tau) \Vert^2_{W,1}, \; \tau \leq \tau_0,
	\end{equation*}
This replaces \cite[(5.12)]{CCMS:generic1}. The rest of the argument is identical to the remainder of the proof of \cite[Corollary 5.2]{CCMS:generic1}. 
\end{proof}


\section{Barriers along the cylindrical ends} \label{sec:barriers}

In this section we consider $(\Sigma,\Omega)\in \cS_n'$. (Recall that this means that $\Sigma = \partial\Omega$ and the ends of $\Sigma$ are  either asymptotic to smooth cones or cylinders over smooth compact shrinkers.) Angenent--Daskalopoulos--\v{S}e\v{s}um have recently \cite{ADS}, constructed a self-shrinker foliation inside and outside of the end of a round cylinder $\Sigma = \SS^{n-1}(\sqrt{2(n-1)})\times \RR$ using ODE techniques (see also \cite{KleeneMoller}). This foliation plays an important role in the study of ancient solutions with cylindrical tangent flow at $t=-\infty$ in \cite{ADS,ADS2,BrendleChoi:3d,BrendleChoi:nD,CHH:4D,ChoiHaslhoferHershkovits,ChoiHaslhoferHershkovitsWhite}. The goal of this section is to (partially) generalize this to the case of shrinkers with nice ends. 

\subsection{Setup and statement of barrier construction} 
 After a rotation, we will assume that $\Sigma$ has a cylindrical end in the direction $\be = (0,\dots,0,1)$ modeled on $\hat\Sigma \times \RR$. More precisely, we consider $\hat\Sigma^{n-1} \subset \RR^n$ a smoothly embedded closed self-shrinker, so that the surfaces $\Sigma - s\be$ converge smoothly on compact sets to $\hat\Sigma\times \RR$ as $s\to\infty$. We will use $(\bomega,z) \in \RR^n\times \RR$ as coordinates on $\RR^{n+1}$. We assume that $\rho>0$ is chosen such that $\hat\Sigma \subset B_{\rho}(\bOh) \subset \RR^n$.
 
Fix $\hat \Omega \subset \RR^n$ so that $\partial\hat\Omega = \hat\Sigma$ and $\Omega - s\be$ converges to $\hat\Omega\times \RR$ as $s\to\infty$. Set $\Omega_Z : = \Omega \cap \{z\geq Z\}$, and $\Sigma_Z = \Sigma \cap \{z\geq Z\}$. For $\delta>0$, we define
\begin{equation}\label{eq:move-hat-sigma-inside-by-delta-barriers}
\hat\Sigma_\delta : = \{\bomega + \delta \hat\varphi(\bomega) \nu_{\hat\Sigma}(\bomega) : \bomega \in \hat\Sigma \}\subset \RR^n. 
\end{equation}
Let $\hat\Omega_\delta \subset  \hat\Omega$ denote the region (possibly unbounded) to the ``inside'' of $\hat\Sigma_\delta$ (as determined by the inwards pointing unit normal). Note that $\hat\Sigma_\delta$ is smooth and strictly shrinker mean convex (to the inside) and $F(\hat\Sigma_\delta) < F(\hat\Sigma)$ for $\delta \in (0,\delta_0)$, where $\delta_0 = \delta_0(\hat\Sigma)$. 

By Lemma \ref{lem:ends-decomposition}, there is $\bar Z = \bar Z(\Sigma)$ large and a smooth function $\tilde v : \hat\Sigma \times \{z\geq \bar Z\} \to \RR$ so that 
\begin{equation}\label{eq:cyl.end.graph.cyl}
\Sigma \cap (B_{2\rho}(\bOh) \times [\bar Z,\infty)) = \{ (\bomega,z) + \tilde v(\bomega,z) \nu_{\hat\Sigma}(\bomega): (\bomega,z) \in\hat\Sigma \times[ \bar Z,\infty)\}\, .
\end{equation}
We have $|\nabla^k_{\hat\Sigma\times \RR} \tilde v| = o(1)$ as $z\to\infty$ for all $k\geq 0$. Fix $\hat \mu < 0$ the first eigenvalue of $L_{\hat\Sigma\times \RR}$. 

For $Z_{0}\geq \bar Z$ we assume there is $[-\infty,T_{0}] \ni \tau \mapsto \tilde v_{\tau} \in C_\textrm{loc}^{\infty}(\hat\Sigma \times [Z_{0},\infty))$ so that 
\begin{equation}\label{eq:barrier-w-tau-decay-to-end-at-infty}
\Vert \tilde v_\tau\Vert_{C^3;\hat\Sigma\times [Z_0,\infty)} = o(1) \textrm{ as $\tau\to-\infty$},
\end{equation}
and so that 
\[
\Graph_{\hat\Sigma \times [Z_{0},\infty)} \tilde v_{\tau}(\cdot)\textrm{ is a rescaled mean curvature flow}.
\]
(An important special case is $\tilde v_{\tau}\equiv \tilde v$, so $\Graph_{\hat\Sigma \times [Z_{0},\infty)} \tilde v_{\tau}(\cdot)$ is the chosen end of $\Sigma$, but we will also consider the case that $\tilde v_{\tau}$ is the end of the ancient flow constructed in Proposition \ref{prop:exist-graphical-one-sided}.)

Taking $Z_{0}$ larger if necessary we can assume that there is $\vartheta_{0} > 0$ so that 
\begin{equation}\label{eq:sep.cyl.ends.from.other.ends}
\Sigma \cap \{z>Z_{0}\} \cap \{t\bx : \bx \in \partial B_{1}(\bOh), d(\bx,\be) < \vartheta_{0}, t>0\} = \Sigma \cap (B_{2\rho}(\bOh)\times (Z_{0},\infty)).
\end{equation}
We will use this later to separate the chosen cylindrical ends from all other ends.

\begin{theorem}[Barriers along cylindrical ends] \label{theo:barriers-inner-outer}
Consider the setup described above. There is $Z_1 \geq  Z_{0}$, $T_{1}\leq T_{0}$, $\alpha > 2|\hat \mu|$, $\eta_0 > 0 $, and $h_{0}>0$ so that for $\eta \in (0,\eta_0)$, there is $\vartheta = \vartheta(\eta)\in (0,\vartheta_{0})$ with the following properties. For $T_{2}\leq T_{1}$, suppose that that $(-\infty,T_{2}]\ni \tau\mapsto \cM(\tau)$ is a rescaled Brakke flow in $\Omega \cap \{z > Z_{1}\} \subset \RR^{n+1}$ so that
\begin{enumerate}
\item for all $\tau \in (-\infty,T_{2}]$, it holds that
\begin{multline*}
\supp \cM(\tau) \cap (B_{2\rho}(\bOh) \times  [Z_{1},2Z_{1}]) \subset\\
\{ (\bomega,z) + (\tilde v_{\tau}(\bomega,z) + s) \nu_{\hat\Sigma}(\bomega) : z \in [Z_{1},2Z_{1}], s \in (0,\eta z^{\alpha}) \} ,
\end{multline*}
\item there is some $r\geq 0$ so that 
\[
\supp \cM(\tau) \subset B_{r}(\bOh) \cup \{ tx : \bx \in \partial B_{1}(\bOh), d(\be,\bx) < \vartheta, t\geq 0\}
\]
for all $\tau \in (-\infty, T_{2}]$, 
and
\item $\supp \cM(\tau)$ converges in the local Hausdorff sense to $\Sigma \cap (B_{2\rho}(\bOh) \times [Z_{1},\infty))$ as $\tau \to -\infty$. 
\end{enumerate}
Then, for all $\tau \in (-\infty,T_{2}]$ it holds that 
\begin{multline*}
\supp \cM(\tau) \cap  \{Z_{1}<z<\eta^{-\kappa}\} \subset\\
\{ (\bomega,z) + (\tilde v_{\tau}(\bomega,z) + s) \nu_{\hat\Sigma}(\bomega) : z \in [Z_{1},\eta^{-\kappa}], s \in [0,\eta z^{\alpha}) \} ,
\end{multline*}
for $\kappa = (1+\alpha)^{-1} \in (0,1)$. 
\end{theorem}

The construction of these barriers will follow by welding a barrier (constructed using the linearized shrinker equation) that covers a long portion of the end of $\Sigma$ to a ``translator'' that approximates a one-sided rescaled mean curvature flow coming out of $\hat\Sigma$. We will prove Theorem \ref{theo:barriers-inner-outer} in Section \ref{subsec:proof-of-theo-barriers} after constructing these pieces. 

\subsection{Translator barriers}\label{subsec:trans-barr}

As we recall in Appendix \ref{app:integral-brakke-X-flows}, Hershkovits--White have recently defined a notion of Brakke flow with an ambient vector field  (cf.\ \cite[Section 13]{HershkovitsWhite:set-theoretic}). In particular, a rescaled mean curvature flow is a $\frac{\bx}{2}$-flow in their framework. Below, we will consider a related object, a $\frac{\bomega}{2}$-flow. These arise in the elliptic regularization construction of rescaled mean curvature flow.

\begin{lemma}[Weighted area estimates for $\bX$-translators] \label{lemm:area.bds.translators}
Consider a compact Riemannian manifold $(M^{n},g)$, a gradient vector field $\bX = \nabla f$ on $M$, and a hypersurface $\Gamma^{n} \subset M\times [0,\infty)$ so that $\bH + \bX^{\perp} + \lambda \be^{\perp} = 0$, where $\be=\partial_{z}$ is the unit vector field on $M\times \RR$ in the $\RR$ direction. Assume that $\Gamma \cap M\times \{0\} = \partial\Gamma$. For any $0\leq a<b$ so that $\Gamma$ intersects $\{z\in\{a,b\}\}$ transversely, we have
\[
\int_{\Gamma\cap\{a\leq z\leq b\}} e^{-f} \leq (b-a+\lambda^{-1}) \int_{\partial\Gamma} e^{-f}. 
\]
and
\[
\int_{\Gamma(a,b)}|\be^{\perp}|_{g}^{2} e^{-f} \leq \lambda^{-1}  \int_{\partial\Gamma} e^{-f}. 
\]
\end{lemma}
\begin{proof}
Set $\Gamma(a,b) : = \Gamma \cap \{a\leq z\leq b\}$ and $\Gamma(s) = \Gamma \cap\{z=s\}$. By Sard's theorem, $\Gamma(a,b)$ is a smooth hypersurface with boundary for a.e. $0<a<b$. We now follow the proof of local mass bounds for the translators obtained in elliptic regularization \cite[\S 5]{Ilmanen:elliptic} (see also \cite[\S 9.3]{White:MCF-lectures}).

We will abuse notation slightly and write $g$ for the product metric on $(M,g)\times \RR$. Then, for $a<b$ with $\Gamma$ transversal to $\{z=a\}\cup\{z=b\}$, we find 
 \begin{align*}
 \int_{\Gamma(a,b)} g(-\bX^{T},\be) e^{-f}
& = \int_{\Gamma(a,b)} \Div_{\Gamma,g}(e^{-f}\be)\\
& = \int_{\Gamma(a,b)} g(-\bH,\be) e^{-f} + \int_{\Gamma(b)} |\be^{T}|_{g} e^{-f} - \int_{\Gamma(a)} |\be^{T}|_{g} e^{-f}\\
& = \int_{\Gamma(a,b)} g(\bX^{\perp} + \lambda \be^{\perp},\be) e^{-f} + \int_{\Gamma(b)} |\be^{T}|_{g} e^{-f} - \int_{\Gamma(a)} |\be^{T}|_{g} e^{-f}.
\end{align*}
Observe that
\[
g(\bX^\perp + \bX^T,e) = g(\bX,e) = 0,
\]
so, we find
\begin{equation}\label{eq:translator-barriers-loc-mass-bounds}
\int_{\Gamma(a)} |\be^T|_g e^{-f} = \int_{\Gamma(a,b)} \lambda |\be^\perp|_g^2 e^{-f}  + \int_{\Gamma(b)} |\be^T|_g e^{-f} . 
\end{equation}
In particular $s\mapsto \int_{\Gamma(s)} |\be^T|_g e^{-f}$ is non-increasing so 
\[
\int_{\Gamma(s)} |\be^T| e^{-f} \leq \int_{\partial \Gamma} e^{-f} 
\]
for $s>0$ so that $\Gamma$ intersects $\{z=s\}$ transversely. Combined with \eqref{eq:translator-barriers-loc-mass-bounds} we find 
\begin{align*}
\int_{\Gamma(a,b)}|\be^{\perp}|_{g}^{2} e^{-f} & = \lambda^{-1} \int_{\Gamma(a,b)}\lambda |\be^{\perp}|^{2}_{g} e^{-f}\\
& \leq  \lambda^{-1}  \int_{\Gamma(a)} |\be^{T}|_{g} e^{-f}\\
& \leq  \lambda^{-1}  \int_{\partial\Gamma} e^{-f}
\end{align*}
proving the second equation. Similarly,
\begin{align*}
\int_{\Gamma(a,b)} e^{-f} & = \int_{\Gamma(a,b)} |\be^\perp|_g^2 e^{-f} + \int_{\Gamma(a,b)} |\be^T|_g^2 e^{-f} \\
& \leq \lambda^{-1} \int_{\partial \Gamma }  e^{-f} + \int_{\Gamma(a,b)} |\be^T|_g^2 e^{-f} \\
& = \lambda^{-1} \int_{\partial \Gamma }   e^{-f} + \int_a^b \int_{\Gamma(s)} |\be^T|_g e^{-f} \\
& \leq (b-a+ \lambda^{-1}) \int_{\partial \Gamma } e^{-f},
\end{align*}
proving the first equation. This completes the proof. 
\end{proof}

\begin{lemma}[Existence of translators]\label{lemm:translator-barrier}
For $\lambda>0$ and $\delta\in (0,\delta_0)$, there exists a smooth hypersurface $\Gamma_\delta^\lambda \subset \RR^{n+1}$ with $\partial\Gamma^{\lambda}_{\delta} = \hat\Sigma_\delta \times\{0\}\subset \RR^{n+1}$ so that $\mathring \Gamma_\delta^\lambda \subset \hat \Omega_\delta \times [0,\infty)$ and the mean curvature of $\Gamma_\delta^\lambda$ satisfies
\[
\bH + \tfrac{\bomega^\perp}{2} + \lambda \be^\perp = 0.
\]
The hypersurfaces  $\mathring\Gamma^\lambda_\delta$ are graphical over $\hat\Omega_\delta\times \{0\}$. Furthermore, the hypersurfaces
\[
\Gamma^\lambda_\delta(t) : = \Gamma_\delta^\lambda - \lambda t \be.
\]
have normal velocity $\bH + \tfrac{\bomega^{\perp}}{2}$ and are thus $\tfrac{\bomega}{2}$-flows. As $\lambda\to\infty$, the flows converge in the sense of Brakke $\tfrac{\bomega}{2}$-flows to $\cH^n \lfloor (\partial^*\Omega_\delta(t)\times \RR)$ where $\Omega_\delta(t)$ is the interior of the rescaled mean curvature flow of $\hat\Sigma_\delta$. 
\end{lemma}
\begin{proof}
We consider the case that $\hat\Omega$ is unbounded (the bounded case is similar but easier). For $R_{j}\to\infty$, consider $(M,g_{j})$ where $M=\SS^{n}$ and $g_{j}$ contains a region isometric to $B_{2R_{j}}(\bOh) \subset \RR^{n}$. We assume that $\hat\Sigma \subset B_{R_{j}}(\bOh)$. Choose a gradient vector field $\bX_{j}$ on $(M,g)$ so that $\bX_{j} = \tfrac 12 \bomega$ on the region isometric to $B_{R_{j}}(\bOh)$ (recalling that $\bomega$ is the coordinate on $\RR^{n}$). Note that $(M,g_{j},\bX_{j},\bOh)$ converges in a pointed sense to $(\RR^{n},g_{\RR^{n}},\tfrac 12 \bomega,\bOh)$. 

Consider the region $\hat\Omega_{\delta,j} \subset M$ that agrees with $\hat\Omega_{\delta}$ in the $B_{R_{j}}(\bOh)$ region and $\partial \hat\Omega_{\delta,j} = \hat\Sigma_{\delta}$  and currents $\Gamma$ with support in the closure of $\hat\Omega_{\delta,j} \times \RR$ and $\partial \Gamma = \hat\Sigma_{\delta}\times \{0\}$. Note that by construction, $\hat\Omega_{\delta,j}$ is  strictly $\bX_{j}$-mean convex to the inside and thus $\partial \hat\Omega_{\delta,j}\times \RR$ acts as barrier for constrained minimization. Hence, by modifying the arguments used in elliptic regularization to $\bX$-flows (cf.\ \cite{EvansSpruck1,ChenGigaGoto,Ilmanen:levelset,Ilmanen:elliptic,HaslhoferHershkovits,HershkovitsWhite:set-theoretic,HershkovitsWhite:sharp-entropy}) for each $\lambda >0$,  there exists a smooth hypersurface $\Gamma_{j} \subset M\times [0,\infty)$ with (i) $\partial\Gamma_{j} = \hat\Sigma_{\delta}\times \{0\}$, (ii) $\Gamma_{j}$ is graphical over $\hat\Omega_{\delta,j}\times\{0\}$ and (iii) $\Gamma_{j}$ satisfies the $\bX_{j}$-translator equation $\bH + \bX^{\perp}_{j} + \lambda \be^{\perp} = 0$. 

Choose $f_{j} \in C^{\infty}(M)$ so that $\nabla f_{j} =\bX_{j}$ and $f_{j} = \tfrac 1 4 |\bomega|^{2}$ in $B_{R_{j}}(\bOh)$. Lemma \ref{lemm:area.bds.translators} implies that for almost every $a<b$, it holds that 
\[
\int_{\Gamma_{j} \cap \{a\leq z \leq b\}} e^{-f_{j}} d\mu_{g_{j}}\leq (b-a+\lambda^{-1}) \int_{\partial\Gamma} e^{-\frac{|\bomega|^{2}}{4}} d\mu_{\RR^{n}}. 
\]
In particular, for all $R>0$, it holds that
\[
\limsup_{j\to\infty} |\Gamma_{j} \cap (B_{R}(\bOh)\times \{0\leq z \leq R\} )|_{g_{\RR^{n+1}}} < \infty. 
\]
We can now pass to a subsequential limit as $j\to\infty$ by following the arguments in \cite[Appendix]{HIMW:trans} (one should replace convexity in  \cite[Appendix]{HIMW:trans} by the fact that the region $\hat \Omega_{j}$ is mean convex with respect to the metric $e^{-2f_{j}/n}g^{(j)}$ and that $\Gamma_j$ is a constrained minimizer in $\hat\Omega_{\delta,j} \times \RR$) but otherwise the argument is unchanged; the Frankel property for self shrinkers (see for example \cite[Corollary C.4]{CCMS:generic1}) prevents cylinders from forming as in \cite[Theorem 12.2(2)]{HIMW:trans}). This yields the smooth hypersurface $\Gamma^{\lambda}_{\delta}$ graphical over $\hat\Omega_{\delta}\times \{0\}$ with $\bH + \frac{\bomega^{\perp}}{2} + \lambda\be^{\perp} = 0$. 

Using Lemma \ref{lemm:area.bds.translators}, the remaining assertions follow exactly as in the construction of (mean convex) mean curvature flow via elliptic regularization (cf.\ \cite[\S 8]{Ilmanen:elliptic} or \cite[Theorem 9.11]{White:MCF-lectures}). Indeed, $\Gamma^{\lambda}_{\delta}-\lambda t\be$ defines an integral unit regular Brakke $\frac{\bomega}{2}$-flow and using the local area estimates that follow from Lemma \ref{lemm:area.bds.translators}, we can take a subsequential limit as $\lambda\to\infty$. Since $\bomega$ is $z$-translation invariant, it is easy to show that the limit is of the form $\mu(t)\times \RR$ where $\mu(t)$ is an integral unit regular Brakke $\frac{\bomega}{2}$-flow on $\RR^{n}$. 

Finally, writing  $\pi : \SS^{n}\times  \RR \to\SS^n$ for the projection to $\SS^{n}$, we compute 
\begin{align*}
& \int_{\pi(\Gamma^{\lambda}_{\delta,j}\cap \{0<z<b\})}e^{-f_{j}}d\mu_{g^{(j)}}   \leq \int_{\Gamma^{\lambda}_{\delta,j}\cap \{0<z<b\}} |\be^{\perp}|_{g^{(j)}} e^{-f_{j}}d\mu_{g^{(j)}} \\
& \leq \left( \int_{\Gamma^{\lambda}_{\delta,j}\cap \{0<z<b\}} |\be^{\perp}|_{g^{(j)}}^{2} e^{-f_{j}}d\mu_{g^{(j)}}  \right)^{1/2} (b+\lambda^{-1})^{1/2} \left( \int_{\partial\hat\Omega_{\delta,j} } e^{-f_j} d\mu_{g^{(j)}} \right)^{1/2}\\
& \leq \lambda^{-1/2} (b+\lambda^{-1})^{1/2} \left( \int_{\partial\hat\Omega_{\delta,j} } e^{-f_j} d\mu_{g^{(j)}} \right).
\end{align*}
Sending $j\to\infty$ and then $\lambda\to\infty$, we can follow \cite[Theorem 9.11]{White:MCF-lectures} or \cite[8.11 Slice Convergence]{Ilmanen:elliptic} to find that $\mu(0) = \cH^{n-1}\lfloor \partial^{*}\hat\Omega_{\delta}$. Since $\partial^{*}\Omega_{\delta}$ is $\tfrac{\bomega}{2}$-mean convex to the inside, any integral unit regular Brakke $\tfrac{\bomega}{2}$-flow starting at $\partial^{*}\hat\Omega_{\delta}$ agrees with the rescaled mean curvature flow of $\hat\Sigma_{\delta}$. 
\end{proof}

\begin{lemma}[Linear growth of translators]\label{lemm:lin-grow-trans}
Fix $\delta \in (0,\delta_{0})$ and $\lambda > \tfrac 12 Z_{0}$ sufficiently large so that $\hat\Omega_{\delta}\times[2\lambda,\infty) \subset \Omega$. 
Then, there is $\vartheta = \vartheta(\lambda,\delta)>0$, $R = R(\lambda,\delta)>0$ so that 
\[
(\Gamma_{\delta}^{\lambda} + 2 \lambda \be) \cap \{ t\bx : \bx \in \partial B_{1}(\bOh), d(\bx,\be) < \vartheta, t>R\} = \emptyset 
\]
\end{lemma}
\begin{proof}
We begin by assuming that $\hat\Omega$ is unbounded. (We will consider the easier case when $\hat\Omega$ is bounded below.)

Note that the shrinker mean curvature of $\Gamma_{\delta}^{\lambda} + 2 \lambda \be$ satisfies
\begin{equation}\label{eq:mean.curv.trans.shift}
\bH + \tfrac 12 \bx^\perp  = (\tfrac 12 z - \lambda ) \be^\perp 
\end{equation}
at $(\bomega,z) \in \Gamma_{\delta}^{\lambda} + 2 \lambda \be$ and such a point must have $z\geq 2\lambda$. Consider the open set $Q \subset \{z>2\lambda\}$ with
\[
\partial Q =  (\RR^{n} \setminus \hat\Omega_{\delta}) \times \{2\lambda\}  \cup (\Gamma_{\delta}^{\lambda} + 2 \lambda \be). 
\]
Note that $Q$ is weakly shrinker mean convex to the inside in the sense of Meeks--Yau \cite[\S 1]{MeeksYau}. This is obvious along the planar boundary, while along the translator portion of the boundary we can use \eqref{eq:mean.curv.trans.shift} and the fact that $\Gamma^\lambda_\delta$ is graphical over $\hat\Omega_{\delta}\times\{0\}$ as proven in Lemma \ref{lemm:translator-barrier}. 

Following \cite[Proposition 12]{Brendle:genus0} we can choose a sequence of Riemannian metrics $g_{k}$ on $\RR^{n+1}$ converging in $C^{\infty}_{\textrm{loc}}(\RR^{n+1})$ to $e^{-\frac{|\bx|^{2}}{2n}}g_{\RR^{n+1}}$ so that $Q \cap B_{2k}(\bOh)$ is mean convex with respect to $g_{k}$ for $k$ large and $\Sigma \cap Q \cap B_{2k}(\bOh)$ is a minimal surface with respect to $g_{k}$. Write $A_{k} : = (\RR^{n+1}\setminus\overline{ \Omega}) \cap Q \cap B_{2k}(\bOh)$ for the ``inside'' of $\Sigma \cap Q \cap B_{2k}(\bOh)$. Choose a Caccioppoli set $B_{k}$ in $Q\cap B_{2k}(\bOh)$ with (i) $A_{k}\subset B_{k}$, (ii) $\overline B_{k} \cap \partial(Q \cap B_{2k}(\bOh)) = \overline A_{k} \cap \partial(Q \cap B_{2k}(\bOh))$ and (iii) $B_{k}$ minimizes $g_{k}$-perimeter among Caccioppoli sets satisfying (i) and (ii). Set $\Upsilon_{k} : = \partial^{*} B_{k}$. 

Taking a subsequence, we can thus find $\Upsilon = \partial^{*}B$ for $B \subset Q$ locally finite perimeter with $\overline B \cap \{z=\lambda\} = (\RR^{n+1}\setminus \Omega) \cap\{z=\lambda\}$, so that $\Upsilon$ has at worst a co-dimension $7$ singular set. It is thus straightforward to prove that $t\mapsto \sqrt{-t} \Upsilon$ is a (integral) Brakke flow. Recall that $\sqrt{0}\Upsilon$ is a conical closed set (not necessarily smooth) and satisfies 
\[
\sqrt{0}\Upsilon \cap \be [1,\infty) = \emptyset
\]
by Ilmanen's avoidance principle (Theorem \ref{theo:ilmanen-avoidance}) applied to the disjoint flows $\sqrt{-t}\Upsilon$ and $\sqrt{-t}\Sigma$. In particular, there is $\vartheta_{1}$ so that 
\[
\sqrt{0} \Upsilon \cap \{t \bx : \bx \in \partial B_{1}(\bOh), d(\bx,\be) < \vartheta_{1}, t>0 \} = \emptyset. 
\]
This is easily seen to imply the assertion (when $\hat\Omega$ is unbounded). In the bounded case, a similar proof works with the conclusion that $\Upsilon$ (and thus $\Gamma$) is bounded. This completes the proof. 
\end{proof}

\subsection{Long barriers}  
We continue with the setup of Theorem \ref{theo:barriers-inner-outer}.

For constants $h_0>0$ to be fixed small below and $\alpha > \max\{ 2|\hat \mu|,Z_{0}\}$ to be fixed large below (depending on $\Sigma$ and the function $\tau \mapsto \tilde v_{\tau}(\cdot)$), for any $\eta \in (0,\alpha^{-\alpha} h_0)$ fixed we consider the function 
\[
U(\bomega,z) =  \eta z^\alpha \hat\varphi(\bomega).
\]
Define $R=R(h_0,\eta,\alpha)>\alpha$ so that $\eta R^\alpha = h_0$. Then define
\[
(-\infty,T_{1}]\ni \tau \mapsto \Xi_{\tau} : = \{(\bomega,z) + (\tilde v_{\tau}(\bomega,z) + U(\bomega,z)) \nu_{\hat \Sigma}(\bomega) : (\bomega,z) \in \hat \Sigma \times [\alpha,R]\}.
\]
As long as $T_{1}\leq T_{0}$ is sufficiently negative, $h_0$ is small, and $\alpha$ is large (depending only on $\Sigma$ and $\tilde v_{\tau}$), we see that $\Xi_{\tau}$ is a smooth hypersurface (with boundary) contained in $\Omega$. We will write $\nu_{\Xi_{\tau}}$ for the ``upwards pointing'' unit normal, i.e., $\nu_{\Xi_{\tau}}\cdot\nu_{\hat\Sigma} > 0$. Define
\[
v_{\Xi_{\tau}} = (\nu_{\Xi_{\tau}} \cdot \nu_{\hat\Sigma})^{-1}
\]
along $\hat\Sigma\times [\alpha,R]$, where $\nu_{\Xi_{\tau}}$ is evaluated at $(\bomega,z) + (\tilde v_{\tau}(\bomega,z) + U(\bomega,z)) \nu_{\hat \Sigma}(\bomega)$, for $(\bomega,z) \in \hat \Sigma \times [\alpha,R]$, as usual. Note that the normal speed of $\Xi_{\tau}$ is given by
\[
v_{\Xi_{\tau}}^{-1}(\partial_{\tau} \tilde v_{\tau})
\]
and the scalar (rescaled) mean curvature satisfies
\[
\bH_{\Xi_{\tau}} + \tfrac 12 \bx^\perp = (H_{\Xi_{\tau}} + \tfrac 12 \bx\cdot\nu_{\Xi_{\tau}}) \nu_{\Xi_{\tau}}.
\]
As such, the next proposition says that $\tau \mapsto \Xi_{\tau}$ is a barrier for rescaled mean curvature flows pinched between $\Xi_{\tau}$ and $\Sigma$. 
\begin{proposition}[Verifying long barrier property]\label{prop:bendy-barrier-good-ends}
There is $T_{1}'\leq T_{0}$, $h_0$ small, and $\alpha_0 >\max\{2|\hat \mu|, Z_{0}\}$ large so that for $\alpha \geq \alpha_0$ and $\eta \in (0,\alpha^{-\alpha}h_0)$, it holds that
\[
v_{\Xi_{\tau}}^{-1}(\partial_{\tau} \tilde v_{\tau}) -  (H_{\Xi_{\tau}} + \tfrac 12 \bx\cdot\nu_{\Xi_{\tau}}) > 0. 
\]
\end{proposition}
\begin{proof}
We begin by observing that, 
\begin{equation}\label{eq:bendy-barrer-function-higher-derivatives}
|\nabla_{\hat\Sigma\times \RR}^k U|(\bomega,z) \leq C_k \eta z^\alpha \hat\varphi(\bomega)
\end{equation}
for $z \geq \alpha$ and $k\geq 0$, (recall that $|\hat \mu|\geq 1$). This will be used in several places below to control various error terms. 

Combining \eqref{eq:barrier-w-tau-decay-to-end-at-infty} and \eqref{eq:bendy-barrer-function-higher-derivatives} we see that Lemma \ref{lemm:relative-shrinker-mean-curvature} applies,\footnote{Here we just use the part of Lemma \ref{lemm:relative-shrinker-mean-curvature} estimating $F,\bF,\mathcal{F}$, not the second half of the lemma.} yielding
\begin{equation}\label{eq:application-relative-shrinker-mean-curvature-to-long-barrier}
v_{\Xi_{\tau}} (H_{\Xi_{\tau}} + \tfrac 12 \bx\cdot\nu_{\Xi_{\tau}}) - v_{\Graph \tilde v_{\tau}} (H_{\Graph \tilde v_{\tau}} + \tfrac 12 \bx\cdot\nu_{\Graph \tilde v_{\tau}}) = L_{\hat\Sigma\times \RR} U + E
\end{equation}
where
\begin{equation}\label{eq:application-relative-shrinker-mean-curvature-to-long-barrier-ERROR}
|E(\bx)| \leq C (|U(\bx)| + |\nabla_{\Sigma\times\RR} U(\bx)| + |\nabla^{2}_{\Sigma\times\RR} U(\bx)|)
\end{equation}
and $v_{\Graph \tilde v_{\tau}} = (\nu_{\Graph \tilde v_{\tau}}\cdot\nu_{\hat\Sigma})^{-1}$. Because $\tau \mapsto \Graph \tilde v_{\tau}$ is a rescaled mean curvature flow, we find
\begin{equation}\label{eq:long-barrier-w-tau-is-rescaled-MCF}
\partial_{\tau} \tilde v_{\tau} = v_{\Graph \tilde v_{\tau}} (H_{\Graph \tilde v_{\tau}} + \tfrac 12 \bx\cdot\nu_{\Graph \tilde v_{\tau}}), 
\end{equation}
so combining \eqref{eq:application-relative-shrinker-mean-curvature-to-long-barrier} and \eqref{eq:long-barrier-w-tau-is-rescaled-MCF} we find
\[
\partial_{\tau} \tilde v_{\tau} - v_{\Xi_{\tau}}  (H_{\Xi_{\tau}} + \tfrac 12 \bx\cdot\nu_{\Xi_{\tau}}) = - L_{\hat\Sigma\times \RR} U - E. 
\]
We now compute
\[
- L_{\hat\Sigma\times \RR}U =  \left( \tfrac \alpha 2 -  \alpha(\alpha - 1) z^{-2} - |\hat\mu| \right)\eta z^\alpha \hat\varphi(\bomega),
\]
so when combined with \eqref{eq:application-relative-shrinker-mean-curvature-to-long-barrier-ERROR} and \eqref{eq:bendy-barrer-function-higher-derivatives} we find
\[
- L_{\hat\Sigma\times \RR} U - E \geq \left( \tfrac \alpha 2 -  \alpha(\alpha - 1) z^{-2} - |\hat\mu| - C\right)\eta z^\alpha \hat\varphi(\bomega) 
\]
for $C$ independent of $z,\alpha,\tau$. Recalling that $z \geq \alpha$ on the region defining $\Xi_{\tau}$, we can thus take $\alpha$ sufficiently large to complete the proof. 
\end{proof}

\subsection{The proof of Theorem \ref{theo:barriers-inner-outer}}\label{subsec:proof-of-theo-barriers}

We continue with the setup of Theorem \ref{theo:barriers-inner-outer}.

We fix $\alpha_0,h_0>0$, $T_{1}'\leq T_{0}$ using Proposition \ref{prop:bendy-barrier-good-ends}. It follows that for any $\alpha \geq\alpha_0$ and $\eta \in (0,\alpha^{-\alpha}h_0)$, defining $R$ by $\eta R^\alpha = h_0$, the family of hypersurfaces 
\[
(-\infty,T_{1}']\ni\tau \mapsto \Xi_{\tau} : = \{(\bomega,z) + (\tilde v_{\tau}(\bomega,z) + U(\bomega,z)) \nu_{\hat \Sigma}(\bomega) : (\bomega,z) \in \hat \Sigma \times [\alpha,R]\} \subset \Omega. 
\]
is a supersolution to mean curvature flow. 

Define $\delta>0$ so that if $\hat \Sigma_\delta(t)$ is the rescaled mean curvature flow in $\hat\Omega\subset\RR^n$ with $\hat \Sigma_\delta(0) = \hat\Sigma_\delta$ (defined in \eqref{eq:move-hat-sigma-inside-by-delta-barriers}) then $[0,3] \ni t \mapsto \Sigma_\delta(t)$ is smooth and
\begin{equation}\label{eq:barr.defn.delta.small.welding}
\hat\Sigma_\delta(t) \cap \hat \Omega_{h_0} = \emptyset
\end{equation}
The existence of such a $\delta>0$ follows from the fact that as $\delta\to 0$, the flows $t\mapsto \hat \Sigma_\delta(t)$ converge smoothly (on compact subsets of space-time) to the static flow given by $\hat\Sigma$ (since $\hat \Sigma$ is a shrinker). We now fix $\alpha\geq \alpha_0$ so that 
\begin{equation}\label{eq:barr.defn.alpha.big.welding}
h_0 2^{-\alpha} < \delta/2
\end{equation}
and $\hat \Omega_\delta \times \{\alpha\} \subset \Omega$. 

Now recall the surfaces $\Gamma^{\lambda}_\delta$ as defined in Lemma \ref{lemm:translator-barrier}. We set
\begin{equation}
\Gamma : = \Gamma^{R/4}_\delta + \tfrac R 2 \be. 
\end{equation}
We choose the unit normal to $\Gamma$ so that $\nu_\Gamma \cdot \be > 0$. Note that the shrinker mean curvature of $\Gamma$ satisfies
\begin{align*}
\bH_\Gamma + \tfrac 12 \bx^\perp 
& = (\tfrac 12 z - \tfrac R 4 ) \be^\perp 
\end{align*}
so
\[
H_\Gamma + \tfrac 12 \bx \cdot \nu_\Gamma = (\tfrac 12 z - \tfrac R 4 ) \be \cdot \nu_\Gamma \geq 0.
\]
\begin{remark}
Below, we will choose a family of regions $Q_{\tau}$ so that parts of $\Xi_{\tau}$ and $\Gamma$ are contained in $\partial Q_{\tau}$. The unit normal that $\Xi_{\tau}$ inherits from its graphical nature will be \emph{outwards pointing} relative to $Q_{\tau}$. 
\end{remark}

There is a unique choice of open set $Q_\Gamma$ so that 
\begin{equation*}
\partial Q_\Gamma = ( (\Omega \cap \{z=\tfrac R 2\}) \setminus (\hat\Omega_\delta \times \{\tfrac R 2\}))
\cup (\Sigma \cap (B_{2\rho}(\bOh)\times [\tfrac R2 ,\infty))) \cup \Gamma 
\end{equation*}
and so that $Q_\Gamma$ is shrinker mean convex (to the inside) in the generalized sense allowing for the corners of $\partial Q_\Gamma$ (cf.\  \cite{MeeksYau}). Similarly, there is a unique choice of families of open sets $\tau \mapsto Q_{\Xi_{\tau}}$ so that 
\begin{multline*}
\partial Q_{\Xi_{\tau}} = \{(\bomega,z) + (\tilde v_{\tau}(\bomega,z)+sU(\bomega,z)) \nu_{\hat \Sigma}(\bomega) : \bomega \in \hat \Sigma, z \in \{\alpha,R\}, s \in [0,1]\} \\
\cup  (\Sigma \cap (B_{2\rho}(\bOh)\times [\alpha ,R]))  \cup \Xi_{\tau}.
\end{multline*}

We now consider the boundaries of $\Xi_{\tau},\Gamma$ at $z=R,z=R/2$ respectively and show that they do not affect the barrier argument if $\Xi,\Gamma$ are welded together appropriately. Define 
\[
\tau \mapsto Q_{\tau} = (Q_{\Xi_{\tau}} \cap (Q_\Gamma \cup \{z \leq \tfrac R2\})) \cup (Q_\Gamma \cap (Q_{\Xi_{\tau}} \cup \{z \geq R\})). 
\] 
This will be the fundamental barrier region. The following proposition will be used to verify that as long as the parameter $\eta$ is chosen small in the definition of $\Xi_{\tau}$ then $\partial Q_{\tau}$ will serve as a barrier for a rescaled mean curvature flow in $Q_{\tau}$. In the next lemma it is important to recall that $\nu_{\Xi_{\tau}}$ was chosen to be the graphical unit normal, so it points \emph{out} of $Q_{\tau}$.

\begin{proposition}[Geometric properties used to weld the translator to the long barrier]\label{prop:welding}
There is $\eta_{0} \in (0,\alpha^{-\alpha}h_0)$ sufficiently small, $T_1 \leq T_1'$ sufficiently negative, and $\gamma_{0} > 0$ so that if $\eta \in (0,\eta_0)$, $\tau \in (-\infty,T_{1}]$ and $\bx \in B_{\gamma_{0}}(\Xi_{\tau}) \cap Q_{\tau}$, writing $\by \in \Xi_{\tau}$ a closest point to $\bx$, then 
\begin{enumerate}
\item[(1a)] $\by \not \in \partial\Xi_{\tau} \cap \{z>\tfrac R 2 \}$, and
\item [(1b)] $(\by - \bx) \cdot  \nu_{\Xi_{\tau}}(\by) >  0$. 
\end{enumerate}
Similarly, if $\bx \in B_{\gamma_{0}}(\Gamma) \cap Q_{\tau}$ and $\by \in \Gamma$ is a closest point, then 
\begin{enumerate}
\item[(2a)] $\by \not \in \partial \Gamma$, and 
\item[(2b)] $(\by - \bx) \cdot \be < 0$. 
\end{enumerate}
\end{proposition}
\begin{proof}
Consider $\eta\to 0$ and $T_1 \to -\infty$ along some sequence (not labeled). The corresponding parameter $R\to \infty$. Hence $\partial \Xi_{\tau} -  R\be$ converges in $C^{\infty}_\textrm{loc}$ to $\hat\Sigma_{h_0} \times \{0\}$ uniformly with respect to $\tau \in (-\infty,T_1]$ (thanks to  \eqref{eq:barrier-w-tau-decay-to-end-at-infty}). On the other hand, recall that $\Gamma^\lambda_\delta - \lambda t \be$ converges weakly to the flow defined by $\partial^*\hat \Omega_\delta(t)\times \RR$ and $[0,3] \to \hat\Sigma_\delta(t)$ is smooth we find that 
\[
\Gamma - R\be = \Gamma^{R/2}_\delta - \tfrac R 2 2\be \to \partial^*\hat \Omega_\delta(2) \times \RR
\]
again in $C^{\infty}_{\textrm{loc}}$ (and $\Gamma$ is independent of $\tau$) so
\[
Q_{\tau} -  R\be \to ((\hat\Omega \setminus {\overline{ \hat \Omega_{\delta}(2)}}) \times \RR)
\]
in the local Hausdorff sense (along with $C^{\infty}_{\textrm{loc}}$ convergence of the boundary). By \eqref{eq:barr.defn.delta.small.welding} we can thus choose $T_1 \ll 0$, $\gamma_{0}>0$, and $\eta_{0}>0$ small so that
\[
B_{\gamma_{0}}(\partial\Xi_{\tau} \cap \{z>\tfrac R2 \}) \cap Q_{\tau} = \emptyset
\]
for all $\eta \in (0,\eta_0)$, $\tau \in (-\infty,T_{1}]$. This ensures (1a) holds. Similarly, we note that as $T_1\to-\infty$, $\eta \to 0$, $\Xi_{\tau} - \tfrac R 2 \be$ converges to $\hat\Sigma_{2^{-\alpha}h_0} \times \RR$ uniformly with respect to $\tau$. Thus, after taking $T_1\ll0$ more negative and $\gamma_{0}>0,\eta_{0}>0$ smaller, condition (2a) holds (thanks to \eqref{eq:barr.defn.alpha.big.welding}). 

We now consider condition (1b). Observe that if $\bx \in Q_{\Xi_{\tau}}$ then (1b) holds by the graphicality of $\Xi_{\tau}$ over (a portion of) $\hat\Sigma\times \RR$. Observe that
\[
Q_{\tau} \setminus  Q_{\Xi_{\tau}} = Q_{\Gamma} \cap \{z \geq R\}. 
\]
On the other hand, as $\eta\to 0$, $T_1\to-\infty$, we find that $B_{\gamma_{0}}(\Xi_{\tau}) - R\be$ converges to $B_{\gamma_{0}}(\hat\Sigma_{h_{0}} \times (-\infty,0])$ (uniformly in $\tau$) and $Q_{\Gamma} \cap \{z\geq R\} - R\be$ converges to $((\hat\Omega \setminus {\overline{\hat\Omega_{\delta}(2)}}) \times [0,\infty))$ (where both convergence statements hold in the local Hausdorff sense). These sets are disjoint by choice of $\gamma_{0}$ above, so we have verified (1b). 

We finally consider (2b). As above, $\be$-graphicality of $\Gamma$ over $\hat\Omega_{\delta} \times \{\tfrac R2\}$ allows us to consider only the case of $\bx \in B_{\gamma_{0}}(\Gamma) \cap (Q_{\tau} \setminus Q_{\Gamma})$. As before, we have already guaranteed this intersection is empty. This completes the proof.
\end{proof}

We can now complete the proof of the main barrier construction. 

\begin{proof}[Proof of Theorem \ref{theo:barriers-inner-outer}]
Fix $\eta_{0},\gamma_{0},T_1$ as in Proposition \ref{prop:welding}.

Suppose that for $T_{2}\leq T_{1}$, $(-\infty,T_{2}]\ni \tau\mapsto \cM(\tau)$ is a rescaled Brakke flow in $\Omega \cap \{z > Z_{1}\} \subset \RR^{n+1}$ so that for $\Xi_{\tau},Q_{\tau}$ defined based on some $\eta\in (0,\eta_{0})$, it holds that
\begin{enumerate}
\item for all $\tau \in (-\infty,T_{2}]$, it holds that
\begin{multline*}
\supp \cM(\tau) \cap (B_{2\rho}(\bOh) \times  (Z_{1},2Z_{1})) \subset\\
\{ (\bomega,z) + (\tilde v_{\tau}(\bomega,z) + s) \nu_{\hat\Sigma}(\bomega) : z \in [Z_{1},2Z_{1}], s \in (0,\eta z^{\alpha}) \} ,
\end{multline*}
\item there is some $r\geq 0$ so that 
\[
\supp \cM(\tau) \subset B_{r}(\bOh) \cup \{ tx : \bx \in \partial B_{1}(\bOh), d(\be,\bx) < \theta, t\geq 0\}
\]
for all $\tau \in (-\infty, T_{2}]$, 
and
\item $\supp \cM(\tau)$ converges in the local Hausdorff sense to $\Sigma \cap (B_{2\rho}(\bOh) \times [Z_{1},\infty))$ as $\tau \to -\infty$. 
\end{enumerate}
We claim that $\supp \cM(\tau) \subset Q_{\tau}$ for all $\tau \in (-\infty,T_{2}]$. Note that since $\eta^{-\kappa} \ll R$, it holds that
\begin{equation*}
Q_{\tau} \cap \{Z_{1}<z<\eta^{-\kappa}\} = 
\{ (\bomega,z) + (\tilde v_{\tau}(\bomega,z) + s) \nu_{\hat\Sigma}(\bomega) : z \in [Z_{1},\eta^{-\kappa}], s \in (0,\eta z^{\alpha}) \} ,
\end{equation*}
so this will suffice to prove Theorem \ref{theo:barriers-inner-outer}. 

To prove that $\supp \cM(\tau) \subset Q_{\tau}$ for all $\tau \in (-\infty,T_{2}]$ we can modify the argument in \cite[10.5]{Ilmanen:elliptic} as follows. First, by combining properties (2) and (3) with Lemma \ref{lemm:lin-grow-trans} (linear growth of translators) we find that there is some {maximal }$\tau_{2} \in (-\infty,T_{2}]$ so that $\supp \cM(\tau) \subset Q_{\tau}$ for all {$\tau \in (-\infty,\tau_{2})$} and moreover, if $\tau_{2} < T_{2}$ {there exists}
\[
\bx_{2} \in \supp\cM(\tau_{2}) \cap (\partial Q_{\tau_{2}} \cap \Omega \cap \{z > Z_{1}\}). 
\]
(Note that small spherical barriers rule out points of $\supp\cM(\tau_{2})$ lying in the interior of the complement of $Q_{\tau_{1}}$). Note that $\be \cdot \bx_{2} > 2Z_{1}$ by (1). 

It must hold that $\bx_{2} \in \Xi_{\tau_{2}}$ or $\Gamma$. In each case we reach a contradiction using a similar argument. We first consider $\bx_{2} \in \Xi_{\tau_{2}}$. Consider
\[
\phi(\bx,\tau) = \begin{cases}
(\gamma - d(\bx,\Xi_{\tau}))^{\beta} & d(\bx,\Xi_{\tau}) < \gamma \\
0 & d(\bx,\Xi_{\tau}) \geq \gamma 
\end{cases}
\]
for $\gamma \in (0,\gamma_{0})$ and $\beta>2$. Taking $\gamma \in (0,\gamma_{0})$ sufficiently small, we can arrange that for $\tau \in [\tau_{2}-1,\tau_{2}]$, $\supp \phi(\cdot,\tau) \cap \supp \cM(\tau) \cap \{z \in [Z_{1},2Z_{1}]\} = \emptyset$ by property (1). Thus, even though $\phi$ is not compactly supported in $\Omega$, it is an allowable test function in the Brakke flow condition for $\cM(\tau)$. Furthermore, (after passing to the non-rescaled flow) the calculations in \cite[10.5]{Ilmanen:elliptic} carry over essentially without change. The only potential problem is that $\Xi_{\tau}$ has nontrivial boundary and is a supersolution to (rescaled) mean curvature flow (see Proposition \ref{prop:bendy-barrier-good-ends}) as opposed to being a solution. To handle this, we can rely on Proposition \ref{prop:welding}. Indeed, for $\tau \in [\tau_{2}-1,\tau_{2}]$, at any point of 
\[
\bx \in \supp \phi(\cdot,\tau) \cap \supp\cM(\tau),
\]
the closest point $\by \in \Xi_{\tau}$ is not in $\partial\Xi_{\tau}$ and $\bx$ lies ``below'' $y$ with respect to $\nu_{\Xi_{\tau}}$ (in the sense of Proposition \ref{prop:welding}). Using this, we readily check that the remainder of the argument in \cite[10.5]{Ilmanen:elliptic} yields a contradiction. Similarly, when $\bx_{2} \in \Gamma$, we can use (2) and Lemma \ref{lemm:lin-grow-trans} to show $\Gamma$ and $\supp\cM(\tau)$ are well separated at infinity. The argument is then essentially the same as before (replacing $\Xi_{\tau}$ by $\Gamma$ in the definition of $\phi$). This completes the proof. 
\end{proof}


\section{Uniqueness in the ancient past}\label{sec:estimates-ancient-general}
The goal of this section is to show that if we are given $(\Sigma,\Omega)\in\cS_{n}''$ and an arbitrary one-sided ancient Brakke flow converging to $\Sigma$ with multiplicity one as $\tau \to -\infty$, then the Brakke flow satisfies the hypothesis of Theorem \ref{prop:cond.unique.1}, thus proving an unconditional uniqueness result for such flows. 

Because certain parts of the argument below hold without the nice, stable end hypothesis on $\Sigma$, we will specify the specific situation considered at the beginning of each subsection. However, we will never weaken a hypothesis in later sections after imposing it earlier, so any result derived earlier in the section will always be applicable.

\subsection{Nearly sharp upper bound in the compact region} In this section consider a smooth shrinker $(\Sigma, \Omega) \in \cS_{n}$ (with no condition at infinity). Consider $(\check \cM(\tau))_{\tau <T}$ an ancient (integral unit regular) rescaled Brakke flow with $(\supp \check \cM)_{\tau} \subset \Omega$ for all $\tau <T$. Assume that $\check \cM(\tau)$ limits smoothly to $\Sigma$ with multiplicity one as $\tau\to -\infty$. 

Choose $\Sigma_{0} \subset \Sigma$ compact domain with smooth boundary. Let $\mu_{0}$ denote the first Dirichlet eigenvalue of $L_{\Sigma}$ on $\Sigma_{0}$. Note that $\inf \mu_{0} = \mu \leq -1$, where $\inf \mu_{0}$ denotes the the infimum of $\mu_0$ over all such $\Sigma_{0}$ and $\mu$ is the first eigenvalue of $L_{\Sigma}$ on the entire $\Sigma$.

By time translation we can assume that $\check \cM(\tau)$ is the normal graph of a positive function $\check u : \Sigma' \times (-\infty,0] \to\RR$ for some $\Sigma'\Supset \Sigma_{0}$ with $\Sigma ' \supset \Sigma \cap B_{2R}(\bOh)$, where $R$ is sufficiently large depending on $\Sigma$. To be precise, $R$ should satisfy \eqref{eq:app_Li-Yau.2}, \eqref{eq:app_Li-Yau.2.1}, and \eqref{eq:app_Li-Yau.3}.  

Our goal is to prove the following nearly sharp upper bound for the graphical function. We expect this result and the method of proof to be useful in other contexts. 
\begin{theorem}[Nearly sharp upper bound] \label{thm:interior decay}
There are $C>0, \bar{\tau}<0$ depending on $\Sigma,\Sigma_{0},R,\check \cM$ satisfying
\[
\check u(\bx,\tau) \leq C e^{-\mu_{0} \tau }
\]
for $(\bx,\tau) \in \Sigma_{0} \times (-\infty, \bar{\tau})$. 
\end{theorem}
Denote by $\varphi_{0}>0$ the first Dirichlet eigenfunction of $L_{\Sigma}$ on $\Sigma_{0}$ with $\Vert \varphi_{0}\Vert_{\infty} = 1$. Note that the Hopf boundary lemma implies that 
\begin{equation}\label{eq:Hopf_lemma}
\inf_{\partial\Sigma_{0}} |\nabla \varphi_{0}| : = \delta_{0} > 0.
\end{equation}
Consider $\Sigma^a_{-}(\tau)$ defined to be the normal graph of
\[
v_{-}^a(\bx,\tau)=   ae^{-\mu_{0}\tau} (1 - aM  e^{-\mu_{0}\tau}) \varphi_{0}(\bx)
\]
over $\Sigma_0$, where  $M=M(\Sigma_{0})$ is the  constant given in the next proposition.
\begin{proposition}[Barrier verification]\label{prop:lower.barrier.compact.bc}
There are constants $\eta = \eta(\Sigma_{0}) \in (0,\tfrac 12)$ and $M=M(\Sigma_{0})\geq 1$ with the following significance. Suppose $a>0$ and $ \tau_0\leq 0$ are chosen to satisfy 
\begin{enumerate}
\item $a M e^{-\mu_{0}\tau_{0}} \leq \eta$, and
\item $v_{-}^a(\bx,\tau_{0}) < \check u(\bx,\tau_{0})$ for $\bx \in \Sigma_{0}$.
\end{enumerate}
Then, for $\tau_0 \leq \tau \leq 0$ with $a M e^{-\mu_{0}\tau} \leq \eta$, it holds that $v_{-}(\bx,\tau) \leq \check u(\bx,\tau)$ for $\bx \in \Sigma_{0}$.
\end{proposition}
\begin{proof}
Let $C_{2} = 20nR^2$ denote the constant in Proposition \ref{prop:Li-Yau}. We claim there is $\delta = \delta(\Sigma,\Sigma_{0},\Sigma')$ so that
\begin{equation}\label{eq:interior-decay.1}
|\nabla \log v_{-}^a|^{2} - 2 \partial_{\tau} \log v_{-}^a > C_{2}
\end{equation}
on $(\partial\Sigma_{0})_{2\delta} \cap \Sigma_{0}$, where $(\partial\Sigma_{0})_{2\delta}$ denotes the $2\delta$-neighborhood of $\partial \Sigma_0$, for all times $\tau$ satisfying $a M e^{-\mu_{0}\tau} \leq \eta$. Indeed, remembering  $a M e^{-\mu_{0}\tau} \leq \eta < \tfrac 12$, we can compute
\[
|\nabla \log v_{-}^a|^{2} - 2 \partial_{\tau} \log v_{-}^a = \varphi_{0}^{-2} |\nabla \varphi_{0}|^{2} + 2\mu_{0} \left( 1 - \frac{a M e^{-\mu_{0}\tau}}{1- a M e^{-\mu_{0}\tau}} \right)>\varphi_{0}^{-2} |\nabla \varphi_{0}|^{2}.
\]
In addition, $|\nabla \varphi_{0}|$ is bounded away from zero at $\partial\Sigma_{0}$ by the Hopf boundary point lemma \eqref{eq:Hopf_lemma}, but $\varphi_{0}\to 0$ as $\bx \to \partial\Sigma_0$. Thus, we can take $\delta$ small enough (independent of $M$) so that \eqref{eq:interior-decay.1} holds. This fixes $\delta$. 

Next, we can observe
\[
(\partial_{\tau} - L) v_{-}^a =  a^2 M \mu_{0} e^{-2\mu_{0}\tau} \varphi_{0}.
\]
Thus, by Corollary \ref{coro:expand-rescaled-mcf-app}, the graph $\Sigma^a_-(\tau)$ satisfies
\[
(\nu_{-}^a\cdot\nu_{\Sigma})^{-1}(\partial_{\tau}\bx  - \bH - \tfrac 12 \bx )\cdot\nu_{-}^a  = (\partial_{\tau} - L) v_{-}^a  + E =  a^2 M e^{-2\mu_{0}\tau}\mu_{0} \varphi_{0} + E,
\]
where $\nu_{-}^a$ is a unit normal vector to $\Sigma^a_-(\tau)$ and
\[
|E(\bx)| \leq C a^2 e^{-2\mu_{0}\tau}(1 + a^2 M^{2}e^{-2\mu_{0}\tau}) (\varphi_{0}(\bx)^{2} + |\nabla \varphi_{0}(\bx)|^{2} + |\nabla^{2} \varphi_{0}(\bx)|^{2})
\]
as long as $v_{-}^a \leq \eta' $ for small $\eta' = \eta'(\Sigma_{0})$ and large $C=C(\Sigma_0)$. Hence, there is $N=N(\Sigma_0)$ so that 
\[
|E(\bx)| \leq N  a^2 e^{-2\mu_{0}\tau}(1 + a^2 M^{2}e^{-2\mu_{0}\tau}).
\]
Therefore,
\[
(\nu_{-}^a\cdot\nu_{\Sigma})^{-1}(\partial_{\tau}\bx  - \bH - \tfrac 12 \bx )\cdot\nu_{-}^a\leq ( \mu_{0}   M  \varphi_{0} + N   + N a^2 M^{2} e^{-2\mu_{0}\tau})a^2 e^{-2\mu_{0}\tau}.
\]
We now fix $M=M(\Sigma_{0})\geq 1$ by the requirement that\footnote{Note that we make essential use of the fact that $\delta>0$ here, since $\varphi_{0}\to 0$ as  $\bx\to \partial\Sigma_{0}$. } 
\[
 \mu_{0} M \varphi_{0} + N \leq -1
\]
on $\Sigma_{0}\setminus (\partial\Sigma_{0})_{\delta}$. Finally we fix $\eta = \eta(\Sigma_{0})$ by 
\[
\eta = \min\{ N^{-\frac12}, \tfrac{1}{2}, \eta'\}.
\]
Note that $ a M e^{-\mu_{0}\tau} \leq \eta$  and $M\geq 1$ imply that $v_{-}^a\leq \eta \leq \eta'$, so the above estimates are valid. Combining the above inequalities yields
\begin{equation}\label{eq:interior-decay.2}
(\nu_{-}^a\cdot\nu_{\Sigma})^{-1}(\partial_{\tau}\bx  - \bH - \tfrac 12 \bx )\cdot\nu_{-}^a=\partial_{\tau} - L) v_{-}^a  + E\leq (-1  + N\eta^2)a^2 e^{-2\mu_{0}\tau}\leq 0
\end{equation}
over $\Sigma_{0}\setminus (\partial\Sigma_{0})_\delta$. 

Now consider $(\bx,\tau)$ the first time of contact with $a M e^{-\lambda_{0}\tau} \leq \eta$ towards a contradiction. We have
\[
\check u = v_{-}^a, \qquad \nabla \check u = \nabla v_{-}^a, \qquad \partial_{\tau} \check u \leq \partial_{\tau} v_{-}^a
\]
at $(\bx,\tau)$. This implies that  $\bx \in \Sigma_{0}\setminus (\partial\Sigma_{0})_{\delta}$. Indeed, if $\bx \in \Sigma_{0} \cap (\partial\Sigma_{0})_{2\delta}$, then thanks to \eqref{eq:interior-decay.1} we find that 
\[
|\nabla \log \check u|^{2} - 2 \partial_{\tau} \log \check u \geq |\nabla \log v_{-}^a|^{2} - 2 \partial_{\tau} \log v_{-}^a > C_{2}
\]
at $(\bx,\tau)$. This contradicts Proposition \ref{prop:Li-Yau}. 

On the other hand, by \eqref{eq:interior-decay.2}, we see that at the point of contact $\bx \in \Sigma_{0}\setminus (\partial\Sigma_{0})_{\delta}$, the graph of $v_{-}^a$ is a subsolution. This is a contradiction. 
\end{proof}

\begin{proof}[Proof of Theorem \ref{thm:interior decay}]
We first note that by the parabolic Harnack inequality it suffices to prove that there is $C>0,\tau_{0}<0$ so that
\[
\inf_{\bx \in \Sigma_{0}} \check u(\bx,\tau) \leq C e^{-\mu_{0} \tau }
\]
for $\tau \leq \tau_{0}$. 

Assume this fails. Then, there are $\tau_{k}\to-\infty,C_{k}\to\infty$ so that
\begin{equation}\label{eq:interior-decay.3}
\inf_{\bx \in\Sigma_{0}} \check u(\bx,\tau_{k}) \geq C_{k} e^{-\mu_{0} \tau_{k}}.
\end{equation}
Set
\[
a_{k} :=\min\{ \tfrac 12  C_{k} ,  \eta M^{-1} e^{\mu_0\tau_k}\} \to \infty,
\]
where $C_{k}$ is defined as in \eqref{eq:interior-decay.3} and $M,\eta$ are fixed in Proposition \ref{prop:lower.barrier.compact.bc}.  Consider $v_{k}(\bx,\tau): = v_{-}^{a_k}(\bx,\tau)$. Note that $v_{k}(\bx,\tau_k) \leq \tfrac 12 C_{k} e^{-\mu_{0}\tau_{k}} < \inf_{\bx\in\Sigma_{0}} \check u(\bx,\tau_{k})$ by \eqref{eq:interior-decay.3}. Using Proposition \ref{prop:lower.barrier.compact.bc}, we thus find that $v_{k}(\bx,\tau) < u(\bx,\tau)$ as long as $\tau_k \leq \tau \leq 0$ and $a_{k} M e^{-\mu_{0}\tau} \leq \eta$. Note that
\[
a_{k} M e^{-\mu_{0}\tau} \leq \eta \qquad  \Leftrightarrow \qquad \tau \leq \tfrac{1}{\mu_{0}} (\log M + \log a_{k} - \log \eta). 
\]
Set
\[
\tilde \tau_{k} = \tfrac{1}{\mu_{0}} (\log M + \log a _{k} - \log \eta) \to -\infty,
\]
namely $a_k M e^{-\mu_0\tilde \tau_k} = \eta$. Note that $\tilde\tau_k \geq \tau_k$ by definition of $a_k$.

As explained above, Proposition \ref{prop:lower.barrier.compact.bc} thus yields
\[
M^{-1} \eta (1-  \eta) \varphi_0(\bx) = v_{k}(\bx,\tilde\tau_{k}) \leq \check u(\bx,\tilde\tau_k) 
\]
On the other hand, since $\tau_k\to-\infty$, we have
\[
\sup_{\bx\in\Sigma_0} \check u(\bx,\tilde\tau_k)  \to 0
\]
as $k\to\infty$. This is a contradiction, completing the proof. 
\end{proof}

\subsection{Spatial sub-linear growth of ancient solutions} 
In this section  we consider $(\Sigma,\Omega)\in \cS_n'$. We will prove that an ancient one-sided Brakke flow has a weak sublinear growth property along the ends of $\Sigma$. (Compare with the stronger conclusion \cite[Lemma 7.18]{CCMS:generic1} in the conical case.)

By Lemma \ref{lem:ends-decomposition}, we can fix $R>0$ sufficiently large so that $\Sigma \setminus B_R(\bOh) = \Sigma_\textrm{cyl}\cup\Sigma_\textrm{con}$. For $\vartheta$ small, set
\[
U_\textrm{con}^{\vartheta} : = \{ t \bx : \bx \in \partial B_1(\bOh), d(\bx, \sqrt{0} \Sigma_\textrm{con})<\vartheta, t\geq 0 \}
\]
and
\[
U_\textrm{cyl}^{\vartheta} : = \cup_{j=1}^{N_\textrm{cyl}} \{ t\bx : \bx \in \partial B_1(\bOh), d(\bx, \bx_j)<\vartheta, t\geq 0\}
\]
where $\bx_j$ are the asymptotic directions of the cylindrical ends.  

\begin{proposition}[Sublinear growth estimates]\label{prop:sublinear-growth}
Fix an ancient unit regular integral Brakke flow $\check \cM(t)$ so that $(\supp \check \cM)_t \subset \sqrt{-t}\Omega$ for all $t<0$ and $\Theta(\check\cM,-\infty)< \infty$. Then, for any $\vartheta>0$, there is $R = R(\vartheta) \geq R_0$ so that for $t \leq -1$,
\[
\frac{1}{\sqrt{-t}} (\supp\check\cM)_t \subset B_R(\bOh) \cup U_\textnormal{con}^{\vartheta} \cup U_\textnormal{cyl}^{\vartheta}. 
\]
\end{proposition}
\begin{proof}
By the Frankel property for self-shrinkers, all tangent flows $\tilde \cM$ at $t=-\infty$ to $\check \cM$ satisfy
\[
\tilde \cM(t) =  k \cH^n\lfloor \sqrt{-t} \Sigma
\]
for all $t<0$, where\footnote{Density considerations imply that $k$ is independent of the tangent flow, implying uniqueness of the tangent flow, although this will not be relevant in the sequel.} $k\in \NN$. In particular, if $\bx \in \RR^{n+1}$ has $\Theta((\bx,0),\tilde \cM) > 0$ then $\bx \in \sqrt{0}\Sigma$. 

Now, we assume that the assertion fails. That is, there is $(\bx_k,t_k) \in \supp\check \cM$ with $t_{k}\leq -1$, so that $\bx_k \not \in U_\textrm{con}^\vartheta\cup U_\textrm{cyl}^\vartheta$ and $|\bx_k|/\sqrt{-t_k}\to\infty$. Note that $|\bx_{k}|\to\infty$, since $t_{k}\leq -1$. Define $\tilde \cM_k(t)$ and $(\tilde \bx_k,\tilde t_k)$ by dilating around $(\bOh,0)$ so that $|\tilde\bx_k|=1$. Passing to a subsequence, $\tilde \cM_k(t) \rightharpoonup \tilde\cM(t)$ a tangent flow to $\check\cM$ at $t=-\infty$ and $\tilde \bx_{k}\to\tilde \bx$. Note that $\tilde t_{k}\to 0$ by construction. In particular, we find that
\[
\Theta((\bx,0),\tilde \cM) > 0,
\]
so $\bx \in \sqrt{0}\Sigma$. However, by construction, 
\[
\sqrt{0}\Sigma \setminus\{\bOh\} \subset U_\textnormal{con}^{\vartheta} \cup U_\textnormal{cyl}.
\]
This is a contradiction. 
\end{proof}

\subsection{Graphicality and estimates along the conical ends}
In this section we consider $(\Sigma,\Omega)\in \cS_n'$ and  $(\check \cM(\tau))_{\tau}$ an ancient (integral unit regular) rescaled Brakke flow with $(\supp \check \cM)_{\tau} \in \Omega$ for all $t<0$. Assume that $ \check \cM(\tau)$ limits in $C^{\infty}_{\textrm{loc}}$ to $\Sigma$ with multiplicity one as $\tau \to -\infty$. 

By Lemma \ref{lem:ends-decomposition}, we can fix $R_0>0$ sufficiently large so that for $R\geq R_0$, $\Sigma \setminus B_R(\bOh) = \Sigma_\textrm{cyl}\cup\Sigma_\textrm{con}$. For $\vartheta$ small, set
\[
U_\textrm{con}^{\vartheta} : = \{ s \bx : \bx \in \partial B_1(\bOh), d(\bx, \sqrt{0} \Sigma_\textrm{con})<\vartheta, s\geq 0 \}
\]
and
\[
U_\textrm{cyl}^{\vartheta} : = \cup_{j=1}^{N_\textrm{cyl}} \{ s\bx : \bx \in \partial B_1(\bOh), d(\bx, \bx_j)<\vartheta, s\geq 0\}
\]
where $\bx_j$ are the asymptotic directions of the cylindrical ends. Assume that $\vartheta$ is chosen sufficiently small and $R$ sufficiently large so that (i) there is exactly one component of $\Sigma_{\textrm{con}}$ in each component of $U^{\vartheta}_{\textrm{con}}\setminus B_{R}(\bOh)$ and (ii) for $\bx \in \Sigma_{\textrm{con}}$ it holds that
\[
\bx \pm \tfrac 12 |A_{\Sigma}| \nu_{\Sigma}(\bx) \not \in U_{\textrm{con}}. 
\]
(This second condition can be verified by considering the blow-down of $\Sigma_{\textrm{con}}$, a smooth cone by assumption, and it will guarantee that graphs over $\Sigma_{\textrm{con}}$ that remain in $U_{\textrm{con}}^{\theta}$ lie within a tubular neighborhood of $\Sigma_{\textrm{con}}$). 

We first recall the weighted norms from \cite{CCMS:generic1} along the conical ends are defined as follows:
\begin{align*}
\Vert f \Vert^{(d)}_{k;\Sigma_{\textrm{con}}} & : = \sum_{i=0}^{k} \sup_{\bx \in\Sigma_{\textrm{con}}} |\bx|^{-d+i} |\nabla^{i}_{\Sigma} f(\bx)| \\
[f]^{(d)}_{\alpha;\Sigma_{\textrm{con}}} & : = \sup_{\bx\neq \by \in \Sigma_{\textrm{cyl}}} \frac{1}{|\bx|^{d-\alpha}+|\by|^{d-\alpha}} \frac{|f(\bx)-f(\by)|}{d_{\Sigma}(\bx,\by)^{\alpha}}\\
\Vert f \Vert^{(d)}_{k,\alpha;\Sigma_{\textrm{con}}} & : = \Vert f \Vert^{(d)}_{k;\Sigma_{\textrm{cyl}}} + [\nabla^{k}_{\Sigma}f]^{(d-k)}_{\alpha;\Sigma_{\textrm{cyl}}}. 
\end{align*}

\begin{lemma}[Decay along conical ends]\label{lem:decay-conical-ends}
Fix $\vartheta$ satisfying (i) and (ii) above. For $R \geq R_0$, there is $T<0$ so that for $\tau<T$, 
\[
\check\cM(\tau) \lfloor (U_{\textnormal{con}}^{\vartheta} \cup B_{R})
\]
agrees with the multiplicity-one graph of $\check u(\cdot,\tau)$ defined on some subset of $\Sigma$. Furthermore, for any integer $k\geq 1$,
\[
\lim_{\tau\to-\infty} (\Vert \check u(\cdot,\tau)\Vert_{k;\Sigma_{\textnormal{con}}}^{(1)} + \Vert \check u (\cdot,\tau)\Vert_{C^{k}(B_{2R}(\bOh))}) = 0 
\]
\end{lemma}
\begin{proof}
As in \cite[Lemma 7.18]{CCMS:generic1}, this follows by rescaling the fact that the tangent cone to the un-rescaled flow at $-\infty$ is $t\mapsto\sqrt{-t}\Sigma$ (with multiplicity-one). 
\end{proof}

\begin{proposition}[Controlling decay along conical ends by the interior]\label{prop:estimates-conical-ends}
There exists $R_1\geq R_0$ and $C>0$ such that for all $R\geq R_1$ and $\tau_0\leq T$,
\[ \Vert \check u(\cdot,\tau_0)\Vert_{2,\alpha;\Sigma_{\textnormal{con}}}^{(1)} \leq C \sup_{\tau\leq \tau_0} \Vert \check u(\cdot, t)\Vert_{C^0(\Sigma_\textnormal{con}\cap B_{2R}(\bOh))}\, .\]

\end{proposition}

\begin{proof}
We construct barriers for the rescaled flow. To do this, we consider a smooth function $\varphi$ defined on $\Sigma_\textrm{con} \cap \{|\bx| \geq R\}$ to be chosen below. Set
\[\Gamma_\varphi = \{ \bx + \varphi(\bx)  \nu_{\Sigma'}(\bx) : \bx \in \Sigma_\textrm{con}, |\bx| \geq R\}\, .\]
Note that if $|\varphi(\bx)| |A_\Sigma|(\bx) < 1$ for all $\bx \in \Sigma_\textrm{con} \cap \{|\bx| \geq R\}$, then $\Gamma_\varphi$ will be a smooth hypersurface. 

We would like to choose $\varphi$ so that $\bH + \bx^\perp/2$ points towards $\Sigma$, or equivalently, we would like that $H + \tfrac{1}{2} \bx \cdot \nu \leq 0$. Expanding this equation as in Appendix \ref{app:graph-shrinker}, we can write this as
\[
 L\varphi + E(\varphi) \leq 0, \,
 \]
where $E$ is as in Lemma \ref{lemm:conical-end-decomp-error-lin} (cf.\ Corollary \ref{coro:expand-rescaled-mcf-app} and \cite[Corollary 3.7]{CCMS:generic1}). Choosing $\varphi:= \alpha |\bx| -\beta$ as in \cite[Lemma 3.15]{ChodoshSchulze} we have
\begin{equation}\label{eq:est.barr.est.lin.term.con.ends}
L\varphi \leq -\tfrac{1}{2} (1 + O(r^{-2}))\beta + O(r^{-1}) \alpha\, .
\end{equation}
Using Lemma \ref{lemm:conical-end-decomp-error-lin}, we can write
\begin{align*}
|E(\varphi)(\bx)| & \leq C(|\bx|^{-2}|\varphi(\bx)| + |\bx|^{-1}|\nabla \varphi(\bx)|)(|\bx|^{-1}|\varphi(\bx)| + |\nabla \varphi(\bx)| + |\bx| |\nabla^2\varphi(\bx)|)\\
& \leq C |\bx|^{-1} (\alpha + \beta |\bx|^{-1})^2
\end{align*}
For $R >0$ large we choose 
$\beta = \alpha (R-1)$. Thus $\varphi(\bx) = \alpha$ on $\Sigma_\textnormal{con} \cap \{|\bx| =R \}$, and $|\varphi(\bx)| \leq \alpha (|\bx| +1 - R)$ on $\Sigma_\textnormal{con} \cap \{|\bx|  \geq R \}$. In particular, since $|A_{\Sigma_\textrm{con}}| = O(|\bx|^{-1})$ (by scaling), we see that $\Gamma_\varphi$ is indeed a smooth hypersurface, for $\alpha \leq \alpha_0$ small. 

In particular, we find that
\begin{equation}\label{eq:est.barr.est.err.term.con.ends}
|E(\varphi)(\bx)| = O(\alpha^2|\bx|^{-1})
\end{equation}
on $\Sigma_\textnormal{con} \cap \{|\bx|  \geq R \}$. Combining \eqref{eq:est.barr.est.lin.term.con.ends} and \eqref{eq:est.barr.est.err.term.con.ends} we find that for $R$ sufficiently large (independent of $\alpha \leq \alpha_0$), 
\[
L\varphi + E(\varphi) \leq 0
\]
on $\Sigma_\textnormal{con} \cap \{|\bx|  \geq R \}$, as desired. 
Combining the fact that $  \check \cM(\tau)$ limits smoothly to $\Sigma$ together with Proposition \ref{prop:sublinear-growth} this implies 
\[ \Vert \check u(\cdot,\tau_0)\Vert_{0;\Sigma_{\textnormal{con}}}^{(1)} \leq C \sup_{\tau\leq \tau_0} \Vert \check u(\cdot, \tau)\Vert_{C^0(\Sigma_\textnormal{con}\cap B_{2R}(\bOh))}\, .\]
Finally, the assertion follows from the Schauder estimates stated in Proposition \ref{prop:schauder.graphical.mcf.conical} below. 
\end{proof}

The following estimates are proven in \cite[\S 3]{CCMS:generic1}.
\begin{proposition}[Schauder estimates along conical ends] \label{prop:schauder.graphical.mcf.conical}
For $\check u(\cdot,\tau)$ as defined in Lemma \ref{lem:decay-conical-ends}, it holds that 
\[
\Vert \check u(\cdot, \tau) \Vert_{2,\alpha;\Sigma_\textnormal{con}}^{(1)} \leq  \sup_{\tau\leq \tau_0}  C\left(  \Vert \check u(\cdot, \tau)\Vert_{C^0(\Sigma_\textnormal{con}\cap B_{2R}(\bOh))} + \Vert \check u(\cdot,\tau_0)\Vert_{0;\Sigma_{\textnormal{con}}}^{(1)} \right)
\]
\end{proposition}
\begin{proof}
We give the references to the relevant results from \cite[\S 3]{CCMS:generic1}. By \cite[Lemma 3.5]{CCMS:generic1} (Schauder estimates for the linearized equation $Lu=h$), we find that 
\begin{equation}\label{eq:schauder.nonlinearity.from.gen.1}
\Vert \check u(\cdot, \tau) \Vert_{2,\alpha;\Sigma_\textnormal{con}}^{(1)} \leq C\sup_{\tau\leq \tau_0} \left(    \Vert \check u(\cdot, \tau)\Vert_{C^0(\Sigma_\textnormal{con}\cap B_{2R}(\bOh))} + \Vert E(\check u)(\cdot,\tau_0) \Vert_{0,\alpha;\Sigma_{\textnormal{con}}}^{(-1)} \right)
\end{equation}
where $E(\check u)$ is the nonlinear error term satisfying the following weighted H\"older estimates proven in \cite[Corollary 3.7]{CCMS:generic1}:
\[
\Vert E(\check u)(\cdot,\tau_0) \Vert_{0,\alpha;\Sigma_{\textnormal{con}}}^{(-1)} \leq C \Vert \check u(\cdot,\tau_0) \Vert_{1,\alpha;\Sigma_{\textnormal{con}}}^{(1)} \Vert \check u(\cdot,\tau_0) \Vert_{2,\alpha;\Sigma_{\textnormal{con}}}^{(1)}.
\]
Using Lemma \ref{lem:decay-conical-ends}, we can thus absorb this term into the left-hand side of \eqref{eq:schauder.nonlinearity.from.gen.1}, completing the proof. 
\end{proof}

\subsection{Upgrading closeness to graphicality along the cylindrical ends} 

In this section we consider $(\Sigma,\Omega)\in \cS_n'$ and  $(\check \cM(t))_{t}$ an ancient (integral unit regular) non-rescaled Brakke flow with $(\supp \check \cM)_{t} \in \sqrt{-t} \Omega$ for all $t<0$. Assume that $\frac{1}{\sqrt{-t}} \check \cM(t)$ limits in $C^{\infty}_{\textrm{loc}}$ to $\Sigma$ with multiplicity one as $t\to -\infty$. 

\begin{definition}[Barrier estimate radius along cylindrical ends]\label{defi:barrier.est} Fix $\vartheta$ as in Lemma \ref{lem:decay-conical-ends}.
For $T<0$ fixed, consider smooth positive strictly monotone functions $b,\eps:(-\infty,T) \to (0,\infty)$ with $\lim_{t\to-\infty} b(t) = \infty$ and $\lim_{t\to-\infty} \eps(t) = 0$. We say that $(b(t),\eps(t))$ forms a \emph{barrier estimate} for $\check \cM$ along the cylindrical ends if 
\[
\tfrac{1}{\sqrt{-t}} (\supp\check \cM)_t\cap \big(U^\vartheta_\text{cyl}\setminus B_{2R_0}(\bOh)\big) \cap B_{b(t)}(\bOh) \subset U_{\eps(t)}(\Sigma_\text{cyl}),
\]
for $t \in (-\infty,T)$. Here, $U_r(\Sigma_\text{cyl})$ is the $r$-neighborhood of $\Sigma_\text{cyl}$. 
\end{definition}

\begin{definition}[Graphical radius along cylindrical ends]
For $\eps_0>0$ to be fixed in Proposition \ref{prop:barrier.to.graph}, define the \emph{graphical radius} $\rho(t)$ to be supremum of $\rho>0$ so that $(\frac{1}{\sqrt{-t}} \check \cM(t)) \lfloor \big((U^\vartheta_\text{cyl}\setminus B_{2R_0}(\bOh)) \cap B_{\rho}(\bOh)\big)$ is the graph of some function defined on a subset of $\Sigma_\text{cyl}$, with $C^5$-norm $\leq \eps_0$. 
\end{definition}

Note that we do not claim that $\rho(t)$ is continuous (it would be possible to prove that $\rho(t)$ has certain partial continuity properties, but we will not need that below). However, it is clear that in the setting considered here $\lim_{t\to-\infty}\rho(t) = \infty$.

\begin{proposition}[Barrier estimates imply graphical estimates]\label{prop:barrier.to.graph}
Suppose that $(b(t),\eps(t))$ forms a barrier estimate for $\check \cM$. Then, we can take $\eps_0=\eps_0(\Sigma)>0$ small and $T=T(\Sigma,\check \cM,b(t),\eps(t))<0$ sufficiently negative so that $\rho(t) \geq b(t)$ for $t<T$. 
\end{proposition}
\begin{proof}
Assume (for contradiction) that there is $t_i \to -\infty$ so that $\rho(t_i) \leq b(t_i)$. Because $\rho(t) \to \infty$ as $t\to-\infty$, 
\[
 \inf\{ s \leq t_i : \rho(s) \leq \rho(t_i) \}  > -\infty. 
\]
Without some continuity property of $\rho(t)$ we cannot ensure that the infimum is attained, but we can find
\[
t_i' \in \{ s \leq t_i : \rho(s) \leq \rho(t_i) \}
\] 
with
\[
t_i' \leq \inf \{ s \leq t_i : \rho(s) \leq \rho(t_i) \} + 1. 
\]
Note that for $s \leq t_i' -1$, we have $\rho(s) \geq \rho(t_i) \geq \rho(t_i')$. Furthermore, we have $\rho(t_i') \leq \rho(t_i) \leq b(t_i) \leq b(t_i')$ (since $b(t)$ is strictly decreasing) and $t_i'\to-\infty$ (since $\rho(t) \to \infty$ as $t\to-\infty$). 

By assumption, there are points $\bx_i \in \supp\check\cM \cap \overline{B_{\rho(t_i')}(\bOh)}$ so that $\check \cM(t_i')$ is not locally written as a graph over a subset of $\Sigma$ with $C^5$-norm $\leq \eps_0$ (if this held at each point in $\supp\check\cM \cap \overline{B_{\rho(t_i')}(\bOh)}$ we could patch these graphs together to write $\supp\check\cM \cap B_{\rho(t_i')+\delta_0}(\bOh)$---for some $\delta_0>0$ small---as a graph over a subset of $\Sigma$ with $C^5$-norm $\leq \eps_0$; note that there can only be a single sheet by the multiplicity one covergence). It is clear that $\bx_i\to\infty$ as $i\to\infty$. 

We first suppose that (after passing to a subsequence), $\bx_i \in U_\textrm{con}^{\vartheta}$, where $\vartheta$ is fixed as in Lemma \ref{lem:decay-conical-ends}. Applying Lemma \ref{lem:decay-conical-ends} with $k=5$, we find a contradiction. 

Thus, we may pass to a subsequence and assume that $\bx_i \in U_\textrm{cyl}^{\vartheta}$. Up to passing to a subsequence and rotating we assume that $\bx_i/|\bx_i |\to \be$ and $\Sigma$ has a cylindrical end (modeled on $\hat \Sigma \times \RR$) in the $\be$ direction. 

Write $\bx_i = (\bomega_i,z_i)$. Consider the flows
\[
\check\cM_i(t):= \tfrac{1}{\sqrt{|t_i'|}}\check \cM(|t_i'| t) - z_i \be. 
\]
Passing to a subsequence, these flows converge to a Brakke flow $\tilde\cM(t)$ with 
\[
\tfrac{1}{\sqrt{-t}} (\supp\tilde\cM)_t\subset \overline{\hat \Omega\times \RR}. 
\]
Now, fix $t_0 < -1$ and consider $t \leq t_0$. Observe that by definition of the graphical radius, 
\[
\Big(\tfrac{1}{\sqrt{-|t_i'|t}} \check\cM(|t_i'|t)\Big) \lfloor B_{\rho(|t_i'|t)} (\bOh)
\]
is graphical over some subset of $\Sigma$ (with $C^5$ norm $\leq \eps_0$). Note that (since $t\leq t_0<-1$)
\[
|t_i'| t < t_i' - 1
\]
for $i$ sufficiently large (depending on $t_0$), so for such $i$, we find that 
\[
\rho(|t_i'|t) \geq \rho(t_i'). 
\]
Hence, 
\[
\Big(\tfrac{1}{\sqrt{-|t_i'|t}} \check\cM(|t_i'|t)\Big) \lfloor B_{\rho(t_i')} (\bOh)
\]
is graphical over some subset of $\Sigma$ (with $C^5$ norm $\leq \eps_0$). Translating in the $\be$ direction, we find 
\[
\tfrac{1}{\sqrt{-t}} \check \cM_{i}(t) \lfloor B_{\rho(t_i')} (-\tfrac{1}{\sqrt{-t}}z_i\be)
\]
is graphical over some subset of $\Sigma-\tfrac{1}{\sqrt{-t}}z_i\be$ (with $C^5$ norm $\leq \eps_0$). Because $t<-1$, we note that $B_{\rho(t_i')} (-\tfrac{1}{\sqrt{-t}}z_i\be)$ limits to $\RR^{n+1}$ as $i\to\infty$. Passing this information to the limit (taking $\eps_0$ sufficiently small so that we can use interior estimates) and then sending $t_0\nearrow -1$, we thus find that $\frac{1}{\sqrt{-t}} \tilde\cM(t)$ is graphical over $\hat \Sigma \times \RR$ with $C^5$ norm $\leq \eps_0$ for all $t < -1$. 

On the other hand, we have that 
\[
\Big(\tfrac{1}{\sqrt{-t_i'}} (\supp \cM)_{t_i'}\Big) \cap B_{b(t_i')}(\bOh) \subset U_{\eps(t_i')}(\Sigma),
\]
because $(b(t),\eps(t))$ forms a barrier estimate. Translating this as above, we find $t<-1$,
\[
\tfrac{1}{\sqrt{-t}}(\supp \check\cM_i)_t \cap B_{b(t_i')}(-\tfrac{1}{\sqrt{-t}}z_{i}\be) \subset U_{\eps(t_i')} (\Sigma -\tfrac{1}{\sqrt{-t}}z_{i}\be).
\]
Since $b(t_i') \geq \rho(t_i') \geq z_{i} $ and $\eps(t_i')\to 0$, we can conclude as above that $\frac{1}{\sqrt{-t}}(\supp \tilde \cM)_t \subset \hat\Sigma\times \RR$ for all $t<-1$. Using the integral unit regular property of $\tilde \cM$, we thus find that 
\[
\tilde \cM(t) = \cH^n \lfloor \sqrt{-t} (\hat\Sigma\times \RR)
\]
for all $t<0$. Note that we have unit multiplicity here, due to the fact that $\frac{1}{\sqrt{-t}} \tilde\cM(t)$ is graphical over $\hat \Sigma \times \RR$ with $C^5$ norm $\leq \eps_0$ for all $t < -1$. 

By unit regularity, we thus find that for $i$ sufficiently large, $\check\cM_i(-1) = \frac{1}{\sqrt{|t_i'|}} \check \cM(t_i') - z_i\be$ is graphical near $(\bomega_i,0)$ (with $C^5$ norm $\leq o(1)$). This contradicts our assumption about $\bx_i$, completing the proof. 
\end{proof}

\subsection{Global decay estimates}
We now impose the stable end hypothesis on $\Sigma$. More precisely, in this section we consider $(\Sigma,\Omega)\in \cS_n''$ and  $(\check \cM(\tau))_{\tau}$ an ancient (integral unit regular) rescaled Brakke flow with $(\supp \check \cM)_{\tau} \in \Omega$ for all $\tau < T$. Assume that $\check \cM(\tau)$ limits in $C^{\infty}_{\textrm{loc}}$ to $\Sigma$ with multiplicity one as $\tau\to -\infty$. 

We define the following norm, adapted to the graphical radius along the cylindrical ends
\begin{equation}\label{eq:defi.partial.seminorm.1} 
\Vert \cdot \Vert_{k,\alpha; \text{cyl}}(\rho)=   \Vert  \cdot \Vert_{k,\alpha;\Sigma \cap B_{2R_0}(\bOh)} + \Vert \cdot \Vert_{k,\alpha;\Sigma_\textnormal{cyl} \cap B_{\rho}(\bOh)}\, .\end{equation}

We now assume there is $(-\infty,\tau_{0}] \ni \tau \mapsto \tilde v_{\tau} \in C^{\infty}_\textrm{loc}(\Sigma)$ so that
\[
\Vert \tilde v_\tau\Vert_{C^3;\hat\Sigma\times [Z_0,\infty)} = o(1) \textrm{ as $\tau\to-\infty$},
\]
and so that 
\[
\tilde \Sigma_\tau = \Graph_{\Sigma} \tilde v_{\tau}(\cdot)\textrm{ is a rescaled mean curvature flow}.
\]
(An important special case is $\tilde v_{\tau}\equiv 0$, so that $\tilde \Sigma_\tau \equiv \Sigma$, but we will also consider the case that $\tilde v_{\tau}$ is the ancient flow constructed in Proposition \ref{prop:short-time-exist-Sigma-eps}). 

Let $\tilde \Omega_\tau \subset \Omega$ be the open set such that $\partial \tilde \Omega_\tau = \Sigma_\tau$. We assume that $\check \cM$ lies on one side of $\tilde \Sigma_\tau$, i.e. $\check \cM(\tau) \subset \tilde \Omega_\tau$ for $\tau \leq \tau_0$. We have the following graphicality and decay estimate. Fix $\alpha>0$ as in Theorem \ref{theo:barriers-inner-outer}. 

\begin{proposition}[Applying barriers to obtain decay along cylindrical ends] \label{prop:global-decay} 
There exists $\hat\delta>0, \vartheta>0, R_1\geq 2R_0$ and depending only $\tilde \Sigma_\tau, \Sigma$ with the following property: for every $R\geq R_1$ there exists $\tau_1\leq \tau_0$ such that for $\tau \leq \tau_1$, $\check\cM(\tau) \lfloor (B_{8R}(\bOh))$ is the graph of $\check u(\cdot,\tau)$ defined on some subset of $\Sigma$ (including $\Sigma\cap B_{4R}(\bOh)$) with
$$\delta(\tau) := \sup_{\tau' \leq \tau} \|\check u(\cdot, \tau')- \tilde v(\cdot, \tau') \|_{4,\alpha; \Sigma\cap B_{4R}(\bOh)} \leq \hat\delta\, .$$
Then either $\delta(\tau') = 0$ for some $\tau'<\tau_1$ and $\check{M}(\tau) = \tilde\Sigma_\tau$ with multiplicity one for all $\tau \leq \tau_1$, or
$\delta(\tau)>0$ for all $\tau<\tau_1$ and there exists $\bar\tau\leq \tau_1$ such that for $\tau \leq \bar\tau$
\[
\check\cM(\tau) \lfloor \big((U^\vartheta_\text{cyl}\setminus B_{2R_0}(\bOh))\cap B_{\tfrac{1}{2}\delta(\tau)^{-\frac{1}{1+\alpha}}}(\bOh)\big) 
\]
is the graph of $\check u(\cdot,\tau)$ defined on some subset of $\Sigma_{\text{cyl}}$ such that
\begin{enumerate}
\item $\Vert \check u(\cdot, \tau) - \tilde v(\cdot, \tau) \Vert_{4,\alpha; \textnormal{cyl}}( \tfrac{1}{2}\delta(\tau)^{-\frac{1}{1+\alpha}}) \leq C\delta(\tau)^{\theta}$ for $\tau\leq \bar\tau$ and any $\theta \in (0,\frac{1}{1+\alpha})$ where $C$ depends on $\theta$, and 
\item rotating $\Sigma$ so that some cylindrical end $\Sigma'$ points in the $\be$ direction, then 
$$\check u(\cdot,\tau) - \tilde v(\cdot, \tau) \leq C  \delta(\tau) z^\alpha$$ 
for $\tau \leq \bar \tau, R\leq z \leq \tfrac{1}{2}\delta(\tau)^{-\frac{1}{1+\alpha}}$. 
\end{enumerate}
\end{proposition}
\begin{proof}
By taking $R$ sufficiently large, we can assume that the following property holds. For any cylindrical end $\Sigma'$, after rotating to point in the $\be$-direction, applying Theorem \ref{theo:barriers-inner-outer} (with $w_{\tau}\equiv w$), then for $Z_{1}$ as defined in Theorem \ref{theo:barriers-inner-outer}, it holds that $\Sigma' \cap \{z=Z_1\} \subset B_{R/2}(\bOh)$. 

We assume first that $\delta(\tau)>0$ for all $\tau<\tau_1$. We claim that by increasing $R$ if necessary, for $\tau_1$ sufficiently negative and $\delta_0$ sufficiently small, Theorem \ref{theo:barriers-inner-outer} (with $\overline v_{\tau}\equiv \overline v$ the graphical function of the end $\Sigma'$ over $\hat\Sigma\times \RR$, cf.\ \eqref{eq:cyl.end.graph.cyl}, or $\overline v_{\tau}$ the graphical function of the end of $\Sigma_\tau$ over $\hat\Sigma\times \RR$), with $\eta' = 2\delta(\tau')$ and $T_{2} = \tau'$ applies to 
\[
\check\cM'(\tau) : = \check\cM(\tau) \lfloor (\Omega \cap \{z > Z_{1}\} \cap \{t\bx : \bx \in \partial B_{1}(\bOh), d(\bx,\be) < \vartheta_{0}, t>0\})
\]
for $\vartheta_{0}$ defined as in \eqref{eq:sep.cyl.ends.from.other.ends}. Note that \eqref{eq:sep.cyl.ends.from.other.ends} and Proposition \ref{prop:sublinear-growth} guarantee that after taking $\tau_1$ even more negative, $\check\cM'(\tau) $ is\footnote{The point here being that there might be other components of $\check\cM(\tau)$ in $\Omega \cap \{z > Z_{1}\}$ corresponding to some other conical/cylindrical end that extends into $\{z>Z_{1}\}$, but for sufficiently negative times these components are disjoint and can be dropped without affecting the Brakke flow property.} a (rescaled) Brakke flow in $\Omega \cap \{z > Z_{1}\}$. 

Property (1) in Theorem \ref{theo:barriers-inner-outer} follows from the choice of $\eta'$ by taking $Z_1$ sufficiently large (i.e.~increasing $R$) and choosing $\delta_0>0$ sufficiently small. Property (2) in Theorem \ref{theo:barriers-inner-outer} follows from \eqref{eq:sep.cyl.ends.from.other.ends} and Proposition \ref{prop:sublinear-growth}. Finally, (3) in Theorem \ref{theo:barriers-inner-outer}  holds by assumption. Thus, we find that
\begin{multline}\label{eq:barrier.app.semi.globa.est}
\supp \check\cM'(\tau') \cap \{Z_{1} < z < (\eta')^{-\kappa}\}
 \subset\\
\{ (\bomega,z) + (w_{\tau'}(\bomega,z) + s) \nu_{\hat\Sigma}(\bomega) : z \in [Z_{1},(\eta')^{-\kappa}], s \in [0,(\eta') z^{\alpha}) \} ,
\end{multline}
where $\kappa = (1+\alpha)^{-1}$ is as in Theorem \ref{theo:barriers-inner-outer} and this holds for all $\tau' \leq \bar \tau$. 

For 
\[
b(-e^{-\tau}) = (2\delta(\tau))^{-\frac{1}{1+\alpha}}, \qquad \eps(-e^{-\tau}) = (2\delta(\tau))^{\frac{1}{1+\alpha}},
\]
we claim that $(b(t),\eps(t))$ forms a barrier estimate for $\check \cM$ in the sense of Definition \ref{defi:barrier.est}. Indeed, we observe that as $t\to -\infty$, $b(t) \to \infty$ and $\eps(t)\to 0$ since $\delta(\tau)\to 0$ as $\tau\to -\infty$. Finally, the assumption
\[
\supp\check\cM(\tau) \cap \big(U^\vartheta_\text{cyl}\setminus B_{2R_0}(\bOh)\big) \cap B_{b(t)}(\bOh) \subset U_{\eps(t)}(\Sigma_\text{cyl})
\]
follows from \eqref{eq:barrier.app.semi.globa.est}.  As such, Proposition \ref{prop:barrier.to.graph} yields $\bar\tau< 0$ so that along the cylindrical ends $\check\cM(\tau) \cap B_{\tfrac{1}{2} \delta(\tau)^{-\frac{1}{1+\alpha}}}(\bOh)$ is graphical with $C^{5}$ norm bounded by $\eps_{0}$ (fixed in Proposition \ref{prop:barrier.to.graph}) for all $\tau \leq \bar\tau$. 

The estimate $\check u (\cdot,\tau) - \tilde v(\cdot, \tau) \leq C \delta(\tau) z^{\alpha}$ then follows from the height bounds from the barrier \eqref{eq:barrier.app.semi.globa.est}. This proves (2). Interpolation yields (1). 

Now consider the case $\delta(\tau') = 0$ for some $\tau'<\tau_1$. Note that we  can repeat the first steps of the previous barrier arguments with a sequence $\eta_i>0$, $\eta_i \to 0$ to see that the support of $\check \cM(\tau)$ is contained in $\tilde\Sigma_\tau$ for all $\tau \leq \tau'$. Since $\check \cM(\tau)$ agrees with $\tilde\Sigma_\tau$ on $B_{2R}(\bOh)$ for all $\tau \leq \tau'$, the constancy Lemma implies that $\check \cM(\tau) = \tilde\Sigma_\tau$ with multiplicity one for a.e.~$\tau\leq \tau'$. Unit regularity then implies that $\check{M}(\tau) = \tilde\Sigma_\tau$ with multiplicity one for all $\tau \leq \tau_1$. 
\end{proof}

We extend \eqref{eq:defi.partial.seminorm.1} to the following semi-global norm
\begin{equation}\label{eq:defi.partial.seminorm.2} \Vert \cdot \Vert_{k,\alpha}(\rho)=   \Vert  \cdot \Vert_{k,\alpha; \text{cyl}}(\rho) + \Vert  \cdot \Vert_{k,\alpha;\Sigma_{\textnormal{con}}}^{(1)}\, .\end{equation}
The previous proposition and the decay along the core yields the following graphicality and decay estimate. Recall that $\mu$ is the first eigenvalue of $L_{\Sigma}$. 

\begin{corollary}[Semi-global decay estimate] \label{cor:global-decay.2}
For $\eps\in (0,|\mu|)$, there is $\delta_0>0$, $\vartheta>0$, $\tau_1\ll0$ depending on $\Sigma,\check\cM,\eps$ so that for $\tau \leq \tau_1$, 
\[
\check\cM(\tau) \lfloor (B_{3e^{-2\delta_0 \tau}}(\bOh) \cup U^\vartheta_\textnormal{con}) 
\]
is the graph of $\check u(\cdot,\tau)$ defined on some subset of $\Sigma$ with:
\begin{enumerate}
\item $\Vert \check u(\cdot, \tau) \Vert_{4,\alpha}(2 e^{-\delta_0 \tau}) \leq e^{2\delta_0\tau}$, and 
\item rotating $\Sigma$ so that some cylindrical end $\Sigma'$ points in the $\be$ direction, then $\check u(\cdot,\tau) \leq C e^{-(\mu-\eps)\tau} z^\alpha$ for $\tau \leq \tau_1, R\leq z \leq 2e^{-\delta_0\tau}$. 
\end{enumerate}
\end{corollary}
\begin{proof}
By Lemma \ref{lem:decay-conical-ends}, for any $R>0$ we can find $\tau_1\leq \tau_0$ sufficiently negative and a function $\check u$ defined on $\Sigma_\textrm{con} \cup B_{4R}(\bOh)$ for $\tau \leq \tau_1 < 0$ so that
\begin{equation}\label{eq:con-weak-decay-graphicality-first-lemma}
\lim_{\tau\to-\infty} \big(\Vert \check u(\cdot,\tau)\Vert_{k;\Sigma_{\textnormal{con}}}^{(1)} + \Vert \check u (\cdot,\tau)\Vert_{C^{k}(B_{5R}(\bOh))}\big) = 0 
\end{equation}
for all $k \in \NN$. 

On the other hand, taking $R$ even larger and $\tau_1$ even more negative, Theorem \ref{thm:interior decay}, together with interpolation, yields
 \begin{equation}\label{eq:inner-decay.2}
 \delta(\tau) \leq e^{(-\mu - \tfrac \varepsilon 8) \tau}  
 \end{equation}
 on $\Sigma \cap B_{4R}(\bOh)$ for $\tau \leq \tau_1$.  We combine \eqref{eq:inner-decay.2} with Proposition \ref{prop:estimates-conical-ends} to see that (taking $ \tau_1$ even more negative) for $\tau \leq \tau_1$
 \[
  \Vert \check u(\cdot,\tau)\Vert_{2,\alpha;\Sigma_{\textnormal{con}}}^{(1)}   \leq  e^{(-\mu - \tfrac \varepsilon 4) \tau} \, .\]
Interpolating with \eqref{eq:con-weak-decay-graphicality-first-lemma} we thus find (taking $\bar \tau$ even more negative) for $\tau \leq \bar \tau$
 \begin{equation}\label{eq:semi.global.decay.interior}
  \Vert \check u(\cdot,\tau)\Vert_{4,\alpha;\Sigma \cap B_{2R}(\bOh)} + \Vert \check u(\cdot,\tau)\Vert_{4,\alpha;\Sigma_{\textnormal{con}}}^{(1)}   \leq e^{(-\mu - \tfrac \varepsilon 2) \tau} \, .
 \end{equation}
After increasing $R$ such that $R\geq R_1$ (where $R_1$ is from Proposition \ref{prop:global-decay}) we can decrease $\tau_1$ further such that $\tau_1 \leq \bar{\tau}$ where $\bar\tau$ is as in Proposition \ref{prop:global-decay}. The estimate \eqref{eq:semi.global.decay.interior} combined with Proposition \ref{prop:global-decay} (applied with $\tilde v_\tau \equiv 0$) yield the asserted bounds. This completes the proof. 
\end{proof}

We furthermore record the following decay estimate for the weighted $L^2$-norm. We denote 
$$\Sigma_\rho = \Sigma_\textnormal{con} \cup (\Sigma \cap B_{2R}(\bOh)) \cup (\Sigma_\textnormal{cyl} \cap B_{\rho}(\bOh))\, .$$

\begin{proposition}[Integral decay estimate] \label{prop:nearly-optimal-decay} For every $\varepsilon >0$ there is $C>0$ and $\bar \tau \ll 0$ such that
$$ \int _{\Sigma_{2e^{-\delta_0\tau}}} \check u^2 e^{-\tfrac{1}{4} |\bx|^2} \leq C e^{2(-\mu - \varepsilon)\tau}\, .$$
 \end{proposition}

\begin{proof}
  Using Theorem \ref{thm:interior decay} and Proposition \ref{prop:estimates-conical-ends} (decreasing $\bar\tau$ if necessary) we see that for $\tau \leq \bar \tau$
 \[ \Vert \check u(\cdot,\tau)\Vert_{0;\Sigma \cap B_{2R}(\bOh)} + \Vert \check u(\cdot,\tau)\Vert_{0;\Sigma_{\textnormal{con}}}^{(1)}   \leq C e^{(-\mu -\varepsilon) \tau}\, .\]
 By Corollary \ref{cor:global-decay.2} we have along each cylindrical end $\Sigma'$ that (after rotating to point in the $\be$-direction),
 $$ \check u(\cdot,\tau) \leq C e^{(-\mu -\varepsilon) \tau} z^\alpha $$
 for $\tau \leq \bar \tau$, $R\leq z \leq 2 e^{-\delta_0\tau}$. Combining these estimates yields the stated integral estimate.
\end{proof}

\subsection{Spectral dynamics} 
In this section we continue to consider $(\Sigma,\Omega)\in \cS_n''$ and  $(\check \cM(\tau))_{\tau}$ an ancient integral unit regular rescaled Brakke flow with $(\supp \check \cM)_{\tau} \subset \Omega$ for all $\tau < T<0$. Assume that $\check \cM(\tau)$ limits in $C^{\infty}_{\textrm{loc}}$ to $\Sigma$ with multiplicity one as $\tau\to -\infty$.

Recall that we denote with $\varphi>0$ the first eigenfunction (corresponding to the eigenvalue $\mu$) of the operator $L$ along $\Sigma$, normalized such that $\|\varphi \|_{W} =1$. The following is the main result of this section.

\begin{theorem}[Pointwise dynamics on the compact part]\label{thm:decaying-mode}
There exists a unique constant $\check a=\check a(\check\cM) \geq 0$ with the following significance. Given any $R\gg 1$, there exist $C=C(\check\cM,R) >0$ and $\bar\tau=\bar \tau(\check\cM,R)  \ll-1$ such that
\[
\sup_{\Sigma \cap B_R(\bOh)} \left| \check u(\cdot, \tau) - \check a e^{-\mu \tau}\varphi(\cdot) \right| \leq C e^{\big(-\mu + \tfrac{\delta_0}{4}\big)\tau}
\]
holds for all $\tau \leq \bar\tau$, where $\check u$ and $\delta_0>0$ are as in Corollary \ref{cor:global-decay.2}.
\end{theorem}
Recall that from Corollary \ref{coro:expand-rescaled-mcf-app}  we have that $\check u$ satisfies 
\begin{equation}\label{eq:evolution-check-u}
 \partial_\tau \check u = L\check u + E
 \end{equation}
where $E = \check u E_1 + E_2(\nabla \check u, \nabla \check u)$ satisfy the estimates given in Corollary \ref{coro:expand-rescaled-mcf-app} and in Lemma \ref{lemm:conical-end-decomp-error-lin} along the conical ends. 

We consider a smooth cut-off function $\eta\geq0$ satisfying $\eta(r) =1 $ for $|r|\leq 1, \eta(r) =0$ for $|r|\geq 2, |\eta'|\leq 2$, and $|\eta''| \leq 2$. We define on $\Sigma$ 
\begin{equation}\label{eq:defn.hat.u}
\hat u(\bx,\tau) = \check u (\bx, \tau) \,\hat\eta(\bx,\tau)\, ,
\end{equation}
where $\hat \eta =1$ on $\Sigma \setminus \Sigma_\textnormal{cyl}$, and $\hat \eta(\bx,\tau) = \eta(e^{\delta_0 \tau} |\bx|)$ on $\Sigma_\textnormal{cyl}$. Then we have
\begin{equation}\label{eq:evolution-hat-u}
\partial_\tau \hat u = L\hat u + \hat E =  L\hat u + \hat u E_1 + E_2(\nabla \hat u, \nabla \check u) + F
\end{equation}
where
\begin{equation*}\label{eq:error-hat-u}
 F = - 2 \nabla \hat \eta \cdot \nabla \check u - \check u\, \Delta \hat \eta + \tfrac{1}{2} \check u\, (\bx^T\cdot \nabla \hat \eta) - \check u\, E_2(\nabla \hat \eta, \nabla \check u) + \check u\, \partial_\tau \hat \eta\, .
 \end{equation*}
 
\begin{lemma}[Cutoff error estimate I] \label{prop:estimate-error-hat-u} There exists $\bar\tau = \bar \tau(\check\cM) \ll -1$ so that for  $\tau \leq \bar\tau$ 
$$ \|F(\cdot,\tau)\|_W \leq C e^{5\delta_0 \tau} \exp \left( -\tfrac{1}{4} e^{-2\delta_0 \tau} \right) \, .$$
 \end{lemma}
 
 \begin{proof} Combining Corollary \ref{coro:expand-rescaled-mcf-app} and Proposition \ref{prop:global-decay} we have
 $ E_2(\nabla \hat \eta, \nabla \check u) \leq C |\nabla \hat \eta| |\nabla \check u|$ and thus
 $$ |F(\bx,\tau)| \leq C \|\check u(\cdot, \tau) \|_{C^1(\Sigma_\textnormal{cyl} \cap \{ e^{-\delta_0 \tau} \leq |\bx| \leq 2 e^{-\delta_0 \tau}\})} \chi(\bx, \tau)\, ,$$
 where $\chi =1$ for $\Sigma_\textnormal{cyl} \cap \{ e^{-\delta_0 \tau} \leq |\bx| \leq 2 e^{-\delta_0 \tau}\}$ and $\chi = 0$ otherwise. Thus Proposition \ref{prop:global-decay} yields
 $$ \|F(\cdot, \tau)\|_W \leq C e^{4\delta_0 \tau} \int_\Sigma  \chi^2 e^{-\tfrac{1}{4} |\bx|^2} \leq C   e^{4\delta_0 \tau} \int_{e^{-\delta_0\tau}}^\infty e^{-\tfrac{1}{4} r^2}\, dr \leq C e^{5\delta_0 \tau} \exp \left( -\tfrac{1}{4} e^{-2\delta_0 \tau} \right)\, .$$
This completes the proof. \end{proof}

\begin{lemma}[Cutoff error estimate II]\label{lem:spectral-control-error.1}
There exists $\bar\tau = \bar \tau(\check\cM) \ll -1$ so that for  $\tau \leq \bar\tau$ 
$$ \big| \langle \hat u(\cdot,\tau), \hat E(\cdot,\tau) \rangle_W\big| \leq C e^{\delta_0 \tau} \| \hat u(\cdot,\tau)\|_W^2 +  C e^{5\delta_0 \tau} \exp \left( -\tfrac{1}{4} e^{-2\delta_0 \tau} \right) \|\hat u(\cdot, \tau)\|_W\, .$$
\end{lemma}

\begin{proof}
  Combining Corollary \ref{coro:expand-rescaled-mcf-app} and Proposition \ref{prop:global-decay} we find
  $$ \big| \langle \hat u(\cdot,\tau), \hat u(\cdot,\tau)\, E_1(\cdot,\tau) \rangle_W\big| \leq C \| \check u(\cdot, \tau) \|_{4}(2 e^{-\delta_0 \tau}) \|\hat u(\cdot, \tau)\|^2_W \leq C e^{\delta_0\tau}  \|\hat u(\cdot, \tau)\|^2_W\, .$$
 To control the second error term, we note that using again Corollary \ref{coro:expand-rescaled-mcf-app}, and Proposition \ref{prop:global-decay}
 \begin{equation*}
 \begin{split}
 \big| \langle \hat u(\cdot,\tau),\,   &E_2(\nabla \hat u(\cdot,\tau), \nabla \check u(\cdot,\tau)) \rangle_W\big| = \frac{1}{2} \left| \int_\Sigma E_2(\nabla \hat u^2(\cdot,\tau), \nabla \check u(\cdot,\tau))\ e^{-\frac{1}{4} |\bx|^2}\right|\\
 & = \frac{1}{2} \left| \int_\Sigma \hat u^2\Div \!\Big({E}^\#_2(\nabla \check u(\cdot,\tau)) \, e^{-\frac{1}{4} |\bx|^2}\Big)\right|\\
 & \leq \frac{1}{2}  \left| \int_\Sigma \hat u^2\Div \!\big({E}^\#_2(\nabla \check u(\cdot,\tau))\big) \, e^{-\frac{1}{4} |\bx|^2}\right|
 + \frac{1}{4}  \left| \int_\Sigma \hat u^2 {E}_2(\bx^T,\nabla \check u(\cdot,\tau))\big) \, e^{-\frac{1}{4} |\bx|^2}\right|\\
 & \leq C e^{2\delta_0 \tau} \|\hat u\|_W^2 + \frac{1}{4}  \left| \int_\Sigma \hat u^2 {E}_2(\bx^T,\nabla \check u(\cdot,\tau))\big) \, e^{-\frac{1}{4} |\bx|^2}\right| 
 \end{split}
\end{equation*} 
where ${E}^\#_2$ is the $(1,1)$-tensor obtained via the natural isomorphism $\#:T^*\Sigma \to T\Sigma$ induced via $g$, acting on the first entry of $E_2(\cdot,\cdot)$. Note that along the conical ends we have from Lemma \ref{lemm:conical-end-decomp-error-lin} and Proposition \ref{prop:global-decay} that
$$ |E_2(\bx^T, \nabla \check u(\cdot,\tau))| \leq C |\nabla \check u(\cdot,\tau)| \leq C e^{2\delta_0 \tau}\, , $$
whereas along the cylindrical ends, since $|\bx^T|\leq 2 e^{-\delta_0\tau}$ on the support of $\hat u$, we have with Corollary \ref{coro:expand-rescaled-mcf-app} and (2) in Proposition \ref{prop:global-decay} (recalling that $\delta_{0} < \tfrac{|\mu|-\eps}{\alpha}$).
$$ |E_2(\bx^T, \nabla \check u(\cdot,\tau))| \leq C e^{-\delta_0 \tau} |\nabla \check u(\cdot,\tau)| \leq C e^{\delta_0 \tau}\, .$$
This implies
$$\big| \langle \hat u(\cdot,\tau),\,   E_2(\nabla \hat u(\cdot,\tau), \nabla \check u(\cdot,\tau)) \rangle_W\big|  \leq C e^{\delta_0 \tau} \|\hat u\|_W^2\, .$$
Using Lemma \ref{prop:estimate-error-hat-u} together with 
$$\big| \langle \hat u(\cdot,\tau), F(\cdot,\tau) \rangle_W\big| \leq \|\hat u(\cdot,\tau)\|_W \|F(\cdot,\tau)\|_W$$
yields the desired result.
\end{proof}

We define orthogonal projections $P_i:L^2_W(\Sigma) \to L^2_W(\Sigma)$ via
$$ P_1(v) := \langle v, \varphi \rangle_W \varphi \qquad \text{and} \qquad P_2(v) := v - P_1(v)\, .$$

\begin{lemma}[Cutoff error estimate III]\label{lem:spectral-control-error.2}
There exists $\bar\tau=\bar\tau(\check\cM) \ll -1$ such that for  $\tau \leq \bar\tau$ and $i=1,2$
\begin{equation}\label{eq:spectral-control-error.2}
\big| \langle P_i\hat u(\cdot,\tau), \hat E(\cdot,\tau) \rangle_W\big|  \leq C e^{\delta_0 \tau} \| \hat u(\cdot,\tau)\|_W^2 +  C e^{5\delta_0 \tau} \exp \left( -\tfrac{1}{4} e^{-2\delta_0 \tau} \right) \|\hat u(\cdot, \tau)\|_W\, .
 \end{equation}
 \end{lemma}
\begin{proof} Due to Lemma \ref{lem:spectral-control-error.1} it is sufficient to prove \eqref{eq:spectral-control-error.2} for $i=1$. Note that as in the proof of Lemma \ref{lem:spectral-control-error.1} we have 
$$ \big| \langle P_1\hat u(\cdot,\tau), \hat u(\cdot,\tau)\, E_1(\cdot,\tau) \rangle_W\big| \leq C e^{\delta_0\tau}  \|P_1 \hat u(\cdot, \tau)\|^2_W $$
and 
 $$\big| \langle P_1\hat u(\cdot,\tau), F(\cdot,\tau) \rangle_W\big| \leq C e^{5\delta_0 \tau} \exp \left( -\tfrac{1}{4} e^{-2\delta_0 \tau} \right) \|\hat P_1u(\cdot, \tau)\|_W\, .$$
 Furthermore, we have 
 \begin{equation*}
 \begin{split}
  \langle P_1\hat u(\cdot,\tau), E_2(\nabla \hat u(\cdot,\tau), \nabla \check u(\cdot,\tau)) \rangle_W &= - \int_\Sigma \hat u( \cdot,\tau) \Div\!\Big(P_1\hat u( \cdot,\tau) \, E_2^\#(\nabla \check u(\cdot,\tau))\ e^{-\frac{1}{4} |\bx|^2}\Big)\\
  &=  - \int_\Sigma \hat u( \cdot,\tau) P_1\hat u(\cdot,\tau) \, \Div\!\Big(E_2^\#(\nabla \check u(\cdot,\tau))\ e^{-\frac{1}{4} |\bx|^2}\Big)\\
  &\quad - \int_\Sigma \hat u( \cdot,\tau) E_2(\nabla P_1\hat u(\cdot,\tau), \nabla \check u( \cdot,\tau))e^{-\frac{1}{4} |\bx|^2}
\end{split}
\end{equation*}
The first term on the right hand side can be estimates as in the proof of Lemma \ref{lem:spectral-control-error.1}
$$\left|\int_\Sigma \hat u( \cdot,\tau) P_1\hat u \, \Div\!\Big(E_2^\#(\nabla \check u(\cdot,\tau))\ e^{-\frac{1}{4} |\bx|^2}\Big)\right| \leq Ce^{\delta_0\tau} \|P_1\hat u(\cdot, \tau)\|_W  \|\hat u(\cdot, \tau)\|_W \, .$$
The second term satisfies
$$ \left|\int_\Sigma \hat u( \cdot,\tau) E_2(\nabla P_1\hat u(\cdot,\tau), \nabla \check u( \cdot,\tau))e^{-\frac{1}{4} |\bx|^2}\right| \leq C e^{\delta_0\tau} \|\nabla P_1 \hat u(\cdot,\tau)\|_W \|\hat u(\cdot,\tau)\|_W\, .$$
Since
$$ \|\nabla P_1 \hat u(\cdot,\tau)\|_W \leq \| \nabla \varphi\|_W   \|P_1 \hat u(\cdot,\tau)\|_W\leq C \|\hat u(\cdot,\tau)\|_W$$
for some constant $C$ depending on $\varphi$, combining these inequalities yields the stated estimate.
\end{proof}
We recall that the operator $L$ has a discrete spectrum on $H^1_W(\Sigma)$ and that we denote by $\mu = \mu_1< -1$ the bottom of the spectrum. We write $-1\geq \mu_2>\mu$ for the next eigenvalue.  

\begin{proposition}[Projected spectral dynamics]\label{prop:evolution-projections}
There exists $\bar\tau = \bar\tau(\check \cM) \ll -1$ so that for  $\tau \leq \bar\tau$ the quantities $\hat U_1:= \|P_1\hat u\|_W^2 + \exp(-\tfrac{1}{4}e^{-2\delta_0\tau})$ and $\hat U_2:= \|P_2\hat u\|_W^2$ satisfy
\begin{align*}
\tfrac{d}{d\tau} \hat U_1 &\geq -2 \mu \hat U_1 - C e^{\delta_0 \tau} (\hat U_1 +\hat U_2)\\
\tfrac{d}{d\tau} \hat U_2 &\leq -2 \mu_2 \hat U_2 + C e^{\delta_0 \tau} (\hat U_1 +\hat U_2)\, .
\end{align*}
\end{proposition}
\begin{proof}
 We compute
  \begin{equation}\label{eq:evol_U1-1}
\tfrac{1}{2} \tfrac{d}{d\tau} \hat U_1 = \langle P_1 \hat u,\partial_\tau \hat u\rangle_W + \tfrac{\delta_0}{2}e^{-2\delta_0\tau} \exp(-\tfrac{1}{4}e^{-2\delta_0\tau})\, , 
 \end{equation}
 and 
 \begin{equation}\label{eq:evol_U1-2}
 \langle P_1 \hat u,\partial_\tau \hat u\rangle_W = \langle P_1 \hat u, L \hat u + \hat E\rangle_W = -\mu \|P_1 \hat u\|_W^2 + \langle P_1 \hat u,\hat E\rangle_W\, .
\end{equation}
  Applying Lemma \ref{lem:spectral-control-error.2} yields the first inequality. Similarly we can obtain
 $$ \tfrac{1}{2} \tfrac{d}{d\tau} \hat U_2 = \langle P_2 \hat u,\partial_\tau \hat u\rangle_W =  \langle P_2 \hat u, L \hat u + \hat E\rangle_W  \leq -\mu_2 \|P_2 \hat u\|_W^2 + \langle P_2 \hat u,\hat E\rangle_W\, . $$
 Hence Lemma \ref{lem:spectral-control-error.2} also yields the second inequality.
\end{proof}

\begin{lemma}[Integrated spectral dynamics I] \label{lem:L^2-optimal-decay}
There exists $\bar\tau =\bar\tau(\check\cM) \ll -1$ so that for  $\tau \leq \bar\tau$ it holds that
$$\| P_2 \hat u\|^2_W \leq Ce^{\delta_0 \tau} \left( \|P_1\hat u\|^2_W + \exp( -\tfrac{1}{4} e^{-2\delta_0 \tau})\right)\, $$
and $\|\hat u\|_W \leq Ce^{-\mu \tau}$.
\end{lemma}
\begin{proof}
Consider again $\hat U_1:= \|P_1\hat u\|_W^2 + \exp(-\tfrac{1}{4}e^{-2\delta_0\tau})$ and $\hat U_2:= \|P_2\hat u\|_W^2$ and define
$$ f(\tau) = e^{(\mu + \mu_2)\tau} \hat U_2 - M e^{(\mu + \mu_2 + \delta_0)\tau} \hat{U}_1\, ,$$
where $M$ is a positive constant to be determined. Proposition \ref{prop:evolution-projections} implies that
\begin{equation*}
 \begin{split}
 f'(\tau) &\leq -(\mu_2-\mu) e^{(\mu+\mu_2)\tau} \hat U_2 - M (\delta_0+\mu_2-\mu) e^{(\mu+\mu_2 +\delta_0)\tau} \hat U_1\\
 &\quad + C(1+M e^{\delta_0 \tau}) e^{(\mu+\mu_2+\delta_0)\tau} (\hat U_1 +\hat U_2)\, .
\end{split}
\end{equation*}
We now choose $\bar \tau$ such that
$$ 2Ce^{\delta_0\tau} \leq \mu_2-\mu $$
for $\tau\leq \bar\tau$ and let $M = 2C(\mu_2-\mu)^{-1}$. This implies
$$ f'(\tau) \leq -(\mu_2-\mu) e^{(\mu+\mu_2)\tau} \hat U_2 - 2C e^{(\mu+\mu_2 +\delta_0)\tau} \hat U_1 + 2C e^{(\mu+\mu_2+\delta_0)\tau} (\hat U_1 +\hat U_2) \leq 0\, .$$
By Proposition \ref{prop:nearly-optimal-decay} we have $\lim_{\tau \to -\infty} f(\tau) =0$ and thus $f(\tau) \leq 0$ for $\tau \leq \bar\tau$, which implies
$$ \hat U_2 \leq 2C (\mu_2-\mu)^{-1} e^{\delta_0 \tau} \hat U_1 $$
for $\tau \leq \bar\tau$. This yields the first estimate. 

Returning to Proposition \ref{prop:evolution-projections} we thus find
\[
\tfrac{d}{d\tau} \log e^{2\mu \tau}\hat U_1 \geq  - (C+1) e^{\delta_0 \tau} 
\]
Because the right hand side is integrable on $(-\infty,\bar\tau]$, this implies that $\hat U_{1}\leq C' e^{-2\mu \tau}$. Combining this with the estimate for $\hat U_{2}$ yields the second estimate. 
\end{proof}

\begin{theorem}[Integrated spectral dynamics II] \label{thm:L^2-decay}
There exists $\check a \geq 0$, $C>0$, and $\bar{\tau}=\bar \tau(\check\cM) \ll -1$ such that for $\tau \leq \bar{\tau}$
$$\|\hat u - \check a e^{-\mu \tau} \varphi\|_W^2 \leq C e^{(2|\mu| + \delta_0)\tau}\, .$$
 \end{theorem}
\begin{proof}
As in \eqref{eq:evol_U1-1}, \eqref{eq:evol_U1-2}, we find
\[
 \left | \tfrac{d}{d\tau} \left(e^{2\mu \tau} \|P_1 \hat u\|^2_W \right)\right| \leq 2 e^{2\mu\tau} \langle P_1 \hat u,\hat E\rangle_W. 
\] 
Combining Lemma \ref{lem:spectral-control-error.2} with Lemma \ref{lem:L^2-optimal-decay}, we thus find that
$$ \left | \tfrac{d}{d\tau} \left(e^{2\mu \tau} \|P_1 \hat u\|^2_W \right)\right| \leq C e^{\delta_0\tau}\, $$
for $\tau\leq \bar\tau\ll0$. Thus, the limit
$$ \lim_{\tau \to -\infty} e^{2\mu \tau} \|P_1 \hat u\|^2_W = \check a^2$$
exists, and
$$ \left| e^{2\mu \tau} \|P_1 \hat u\|^2_W - \check a^2 \right| \leq C e^{\delta_0 \tau}\, $$
for $\tau\leq \bar\tau\ll0$. Combined with Lemma \ref{lem:L^2-optimal-decay}, this proves the assertion. 
\end{proof}

\begin{proof}[Proof of Theorem \ref{thm:decaying-mode}] This follows from Theorem \ref{thm:L^2-decay} and interpolation. 
 \end{proof}
 
\begin{remark}We emphasize that for fixed $\check \cM$, the constant $\check a$ is the \emph{same} in both Theorem \ref{thm:L^2-decay} and Theorem \ref{thm:decaying-mode}. 
 \end{remark}
 
 \subsection{Global lower barrier} In this section we continue to consider $(\Sigma,\Omega)\in \cS_n''$ and  $(\check \cM(\tau))_{\tau}$ an ancient (integral unit regular) rescaled Brakke flow with $(\supp \check \cM)_{\tau} \subset \Omega$ for all $\tau < T<0$. Assume that $\check \cM(\tau)$ limits in $C^{\infty}_{\textrm{loc}}$ to $\Sigma$ with multiplicity one as $\tau\to -\infty$.  
 
 In this section, we additionally impose the requirement that Theorem \ref{thm:L^2-decay} holds with constant $\check a >0$. Because we are concerned with an optimal \emph{lower barrier}, the case that $\check a = 0$ will be trivially handled by choosing $\bar M(\tau) \equiv \Sigma$ for all $\tau$. 
 
 We denote by $\bar v:\Sigma \times (-\infty, \tau_0)\to \RR_+$ the function such that 
 $$\tau \mapsto \bar M(\tau)=\Graph_\Sigma(\bar v(\cdot,\tau))$$
  is the smooth ancient rescaled mean curvature flow on one side of  $\Sigma$ constructed in Proposition \ref{prop:exist-ancient-rescaled-MCF-C2-rescaling-argument}. By shifting time, we can assume that
 \begin{equation}\label{eq:entire-sol-cal}
  \lim_{\tau \to -\infty} e^{\mu \tau} \langle \bar v(\cdot,\tau),\varphi \rangle_W =1\, .
\end{equation}
 Note that this condition together with Proposition \ref{prop:cond.unique.2} implies uniqueness under the decay assumption \eqref{eq:estimate-decay}. We also consider the general solution $\cM$ on one side of $\cM_\Sigma$, and denote the corresponding rescaled flow by $\tau \mapsto M(\tau)$. 
 
Since we assumed that $\check a> 0$, we can shift (rescaled) time in the given flow $\check\cM$ to ensure that $\check a = 1$. In other words, we arrange that
 \begin{equation}\label{eq:sol-cal}
  \lim_{\tau \to -\infty} e^{\mu \tau} \langle \hat u,\varphi \rangle_W =1\, ,
\end{equation} 
where $\hat u$ is the cutoff graphical function defined in \eqref{eq:defn.hat.u}. 

Applying Theorem \ref{thm:decaying-mode} to both\footnote{We use that the constant $\check a$ is the same in the point-wise estimate and integrated estimate.} the unknown flow $\check\cM$ and the graphical flow $\bar M$, we can choose a sequence $\tau_m \to -\infty$ such that 
 \begin{equation}\label{eq:lower-barrier}
 (1-\tfrac{1}{m}) e^{-\mu\tau_m}\varphi \leq \min\{\bar v(\cdot, \tau_m), \hat u(\cdot, \tau_m)\} 
 \end{equation} 
on $B_{2m}(\bOh) \cap \Sigma$. We consider a smooth cut-off function $\eta$ with values in $[0,1]$ such that
$$ \eta(r) = 1 \quad\text{for}\ r\leq 1, \qquad \eta(r) = 0\quad\text{for}\ r\geq 2\, ,$$
and define $\bar w^m \in C_c^\infty(\Sigma)$ by
$$ \bar w^m(\bx) = (1-\tfrac{1}{m})  e^{-\mu\tau_m} \eta(|\bx|/m)  \varphi(\bx)\, .$$
Recall (cf.\ Proposition \ref{prop:exist-ancient-rescaled-MCF-C2-rescaling-argument}) that we may assume that $\bar v(\cdot,\tau) \leq C e^{-\mu\tau}$ for some $C>0$ and $\tau \in  (-\infty, \tau_0)$.

\begin{proposition}[Partial lower barrier]\label{prop:est-vm} There exists $\bar\tau = \bar \tau(\check\cM) \in (-\infty, \tau_0)$ and $C>0$ such that for all $m$ sufficiently large there exists a smooth rescaled ancient mean curvature flow  
 $$[\tau_m, \bar \tau]\ni \tau \mapsto \bar M^m(\tau)=\Graph_\Sigma(\bar v^m(\cdot,\tau))$$ such that\\[-2ex]
\begin{itemize}
\item[(1)]  $\bar v^m(\cdot,\tau) \in C^\infty(\Sigma)$ and $\bar v^m(\cdot,\tau)\geq 0$ for all $\tau \in [\tau_m, \bar \tau]$,\\[-2ex]
\item[(2)] $\bar v^m(\cdot,\tau_m) = \bar w^m$,\\[-2ex]
\item[(3)] $\bar v^m \leq v$ holds for $\tau \in [\tau_m, \bar \tau]$\, ,\\[-2ex]
\item[(4)] $(\supp\check\cM)_{\tau}$ avoids $\bar M^m(\tau)$ in the sense that
\[
(\supp\check\cM)_{\tau} \subset \Omega\setminus\{\bx + s \bar v^{m}(\bx,\tau): \bx \in \Sigma,s\in [0,1]\}
\]
for $\tau \in [\tau_m, \bar \tau]$,\\[-2ex]
\item[(5)] $\|\bar v^m(\cdot, \tau)\|_{C^4(\Sigma)} \leq C e^{-\mu\tau}$ for $\tau \in [\tau_m, \bar \tau]$.
\end{itemize}
\end{proposition}

\begin{proof} Note that we can apply Corollary \ref{coro:exist.graph.bootstrap} with initial datum $\bar w^m$ to obtain a solution $\tau \mapsto \bar v^m(\cdot, \tau)$ to graphical rescaled mean curvature flow for $\tau \in [\tau_m,\bar \tau_m]$. We note that the Ecker-Huisken maximum principle implies that 
\begin{equation}\label{eq:upper-bound-vm}
v^m(\cdot, \tau) \leq v(\cdot, \tau)\leq   C e^{-\mu\tau}
\end{equation}
for some $C>0$ for $\tau \in [\tau_m,\bar \tau_m]$. Corollary \ref{coro:exist.graph.bootstrap} then ensures the existence of a uniform lower bound $\bar \tau\leq \bar\tau_m$, when $m$ is sufficiently large. This establishes points (1) through (3). Point (4) follows from Proposition \ref{prop:avoidance-noncompact}. Note that Proposition \ref{coro:exist.graph.bootstrap} implies that $\|\bar v^m(\cdot,\tau)\|_{C^3(\Sigma)}$ is uniformly controlled for all $\tau \in [\tau_n,\bar\tau]$. Point (5) then follows by considering the unrescaled mean curvature flow and applying standard interior parabolic Schauder estimates to the mean curvature flow equation for graphs with controlled, small $C^{1,\alpha}$-norm over $\Sigma$, combined with \eqref{eq:upper-bound-vm}.
 \end{proof}

\begin{proposition}[Global lower barrier] \label{prop:entire-sol-below} Under the assumptions \eqref{eq:entire-sol-cal} and \eqref{eq:sol-cal} it holds that $\check\cM(\tau)$ lies above $\tau \mapsto \bar M(\tau)$ for $\tau \in (-\infty, \bar\tau)$ in the sense that
\[
(\supp\check\cM)_{\tau} \subset \Omega \setminus \{\bx + s \bar v(\bx,\tau)\nu_{\Sigma} : \bx \in \Sigma, s \in [0,1)\}. 
\]
\end{proposition}
\begin{proof}
 Recall that from Corollary \ref{coro:expand-rescaled-mcf-app}  we have that $\bar v^m$ satisfies 
\begin{equation}\label{eq:evolution-vm}
 \partial_\tau \bar v^m = L\bar v^m + E
 \end{equation}
where $E = \bar v^m E_1 + E_2(\nabla \bar v^m, \nabla \bar v^m)$ satisfy the estimates given in Corollary \ref{coro:expand-rescaled-mcf-app}. Proposition \ref{prop:est-vm}, (5) then implies
$$ | \partial_\tau \bar v^m - L\bar v^m| \leq C e^{-2\mu\tau}\, .$$
This yields
$$\big| \tfrac{d}{d\tau} \langle \bar v^m, \varphi\rangle_W + \mu  \langle \bar v^m, \varphi\rangle_W \big| \leq Ce^{-2\mu\tau}\, ,$$
and thus for $\tau \in [\tau_m,\bar \tau]$
$$ \left| e^{\mu \tau} \langle \bar v^m(\cdot,\tau), \varphi\rangle_W - e^{\mu \tau_m}\langle \bar w^m, \varphi\rangle_W\right| \leq \int_{\tau_m}^\tau e^{-\mu s}\, ds \leq \int_{-\infty}^\tau e^{-\mu s}\, ds \leq C e^{-\mu \tau}\, .$$
By \eqref{eq:entire-sol-cal}, \eqref{eq:sol-cal} and \eqref{eq:lower-barrier} we have
$$ \lim_{m\to\infty} e^{\mu \tau_m}\langle \bar w^m, \varphi\rangle_W = 1.$$
Let $\tilde v$ be any subsequential limit of the sequence $\bar v^m$. Note that $(-\infty, \bar\tau) \ni \tau \mapsto \tilde v(\cdot, \tau)$ is a entire graphical solution to rescaled mean curvature flow over $\Sigma$, lying below $\bar v$, which satisfies
$$ \left| e^{\mu\tau}  \langle \tilde v, \varphi\rangle_W  - 1 \right| \leq Ce^{-\mu\tau}$$ 
as well as
$$ \| \tilde v(\cdot\, \tau)\|_{C^3(\Sigma)} \leq C e^{-\mu\tau}\, .$$
Proposition \ref{prop:cond.unique.2} thus implies that $\tilde v \equiv \bar v$. The assertion then follows by combining this with (4) in Proposition \ref{prop:est-vm}. 
\end{proof}

\subsection{Global uniqueness}  In this section we consider\footnote{If one is content with uniqueness in the ancient past, the condition $\Sigma\in\cS_n''$ suffices. The strictly stable end (as opposed to stable end) condition is only used to apply Proposition \ref{prop:cond.unique.1} to upgrade uniqueness in the ancient past to uniqueness for all $t<0$. }   $(\Sigma,\Omega)\in \cS_n'''$ and  $(\check \cM(\tau))_{\tau}$ an ancient (integral unit regular) rescaled Brakke flow with $(\supp \check \cM)_{\tau} \subset \overline \Omega$ for all $\tau < T$. Assume that $\check \cM(\tau)$ limits in $C^{\infty}_{\textrm{loc}}$ to $\Sigma$ with multiplicity one as $\tau\to -\infty$.

Let $\check a$ be given by Theorem \ref{thm:L^2-decay}. 
 
We establish the following global uniqueness theorem.

\begin{theorem}[Global uniqueness result]\label{thm:uniqueness-global}\item
 \begin{itemize}
 \item[(1)]  If $\check a >0 $ then up to a shift in time, $\check \cM$ coincides with the rescaling of the flow constructed in Proposition  \ref{prop:exist-ancient-rescaled-MCF-C2-rescaling-argument} for all $t<0$.
 \item[(2)]  If $\check a =0 $ then $\check \cM$ coincides with $\cM_\Sigma$ for all $t<0$.
\end{itemize}
\end{theorem}
\begin{proof}
We first consider case (1). By shifting time we can assume that \eqref{eq:sol-cal} holds. Thus by Proposition \ref{prop:entire-sol-below} there exists $\bar \tau$, such that for $\tau \in (-\infty,\bar{\tau})$, $\check\cM(\tau)$ lies above\footnote{Here, we mean in the sense of Proposition \ref{prop:entire-sol-below}.} $\bar M(\tau)$, where $\tau \mapsto \bar M(\tau)$ is the  smooth ancient rescaled mean curvature flow on one side of  $\Sigma$ constructed in Proposition \ref{prop:exist-graphical-one-sided}, satisfying \eqref{eq:entire-sol-cal}. We denote with $\bar v:\Sigma \times (-\infty, \tau_0')\to \RR_+$ the function such that  $$\tau \mapsto \bar M(\tau)=\Graph_\Sigma(\bar v(\cdot,\tau))\, .$$
We can assume that $\tau_{0}'\leq \bar \tau$.

Our goal is to prove that there is some $\tau_{0}\in (-\infty,\tau_{0}']$ with $\check \cM(\tau) = \bar M(\tau)$ for all $\tau \in (-\infty,\tau_{0})$. Given this, the uniqueness for all (rescaled) time\footnote{and thus all negative non-rescaled time} follows from Proposition \ref{prop:cond.unique.1}. Below, we will repeatedly (but only finitely many times) take $\tau_{0}$ more negative until we achieve this goal.

Fix $R=R_{1}$ as in Proposition \ref{prop:global-decay} . Note that Theorem \ref{thm:decaying-mode} together with \eqref{eq:entire-sol-cal} and \eqref{eq:sol-cal} implies that for $\tau_0\ll 0$,
\begin{equation}\label{eq:first-decay}
 \sup_{\Sigma \cap B_{5R}(\bOh)}|\check u(\cdot,\tau) - \bar v(\cdot, \tau)|\leq e^{\big(-\mu+\frac{\delta_0}{4}\big)\tau}
\end{equation}
for $\tau \leq \tau_0$. For $\tau\leq \tau_{0}\ll0$, we find
\begin{equation}\label{eq:ancient-past-defn-delta-tau}
\delta(\tau) := \sup_{\tau' \leq \tau} \|\check u(\cdot, \tau')- \bar v(\cdot, \tau') \|_{4,\alpha; \Sigma\cap B_{4R}(\bOh)} \leq \hat\delta\, ,
\end{equation}
where $\hat \delta$ is as in Proposition \ref{prop:global-decay}.
By Proposition \ref{prop:global-decay} we can assume that $\delta(\tau)>0$ for all $\tau \leq \tau_0$.

We now define $w = \check u - \bar v$ (on the domain of definition of $\check u$). Observe that the fact that $\check\cM(\tau)$ lies above $\bar M(\tau)$ implies that $w\geq 0$. We set
\[
W(\tau)^2 : = \sup_{s\leq \tau} \int_{\Sigma\cap B_{5R}(\bOh)} w(\cdot,s)^{2} e^{-\tfrac 14 |\bx|^{2}}.
\]
 We define
\begin{equation}\label{eq:definiton-gamma}
 \gamma:=\sup\{ \gamma'>0\, |\, W(\tau) \leq e^{(1+2\alpha)\gamma' \tau}\ \text{for}\ \tau\leq \tau_0\}\, .
\end{equation}
We assume (for contradiction) that $\gamma < \infty$ for all $\tau_{0}\ll0$. Note that if for some $\tau_{0}$ fixed, $\gamma=\infty$, then $W(\tau)=0$ for all $\tau\leq \tau_0$. Thus, \eqref{eq:control-delta} together with Proposition \ref{prop:global-decay} implies that $\check \cM(\tau) = \bar{M}(\tau)$ for all $\tau \leq \tau_0$, proving the assertion.

To begin, we note that \eqref{eq:first-decay} implies that
\begin{equation}\label{eq:first-decay-W}
 W(\tau) \leq C e^{\big(-\mu+\frac{\delta_0}{4}\big)\tau}
\end{equation}
 for $\tau \leq \tau_0$. This implies that
\begin{equation}\label{eq:first-estimate-gamma}
\gamma\geq \gamma_0:= (1+2\alpha)^{-1}\left( -\mu+\tfrac{\delta_0}{8}\right) > 0\, 
\end{equation}
for $\tau_{0}\ll0$. 

We now show derive pointwise estimates on $w$ from the given estimates on $W$. Interior parabolic Schauder estimates imply that there is a constant $C>0$,  such that 
\begin{equation}\label{eq:control-delta}
 \delta(\tau) \leq C W(\tau)  
\end{equation}
for $\tau\leq \tau_0$. Recall the definition of the semi-global norm as in \eqref{eq:defi.partial.seminorm.2} 
\[ \Vert \cdot \Vert_{k,\alpha}(\rho)=   \Vert  \cdot \Vert_{k,\alpha;\, \text{cyl}}(\rho) + \Vert  \cdot \Vert_{k,\alpha;\Sigma_{\textnormal{con}}}^{(1)}\, .\]
Proposition \ref{prop:global-decay} implies that there is a constant $C_0\geq 2$ such that we can choose
\begin{equation}\label{eq:definition-graphical-radius}
 \rho(\tau):= C_0^{-1} W(\tau)^{-\frac{1}{1+\alpha}}
\end{equation}
with
\[ \|w(\cdot, \tau)\|_{4,\alpha;\,\text{cyl}}(\rho(\tau)) \leq C W(\tau)^\theta\]
for any $\theta \in (0,\frac{1}{1+\alpha})$. In particular, by interpolating Lemma \ref{lem:decay-conical-ends} with the estimate that follows by using \eqref{eq:control-delta} in Proposition \ref{prop:estimates-conical-ends}, 
\begin{equation}\label{eq:control-semi-global-decay}
 \|w(\cdot, \tau)\|_{4,\alpha}(\rho(\tau)) \leq W(\tau)^{\frac{1}{1+2\alpha}}
\end{equation}
for $\tau\leq \tau_0 \ll0$. We note that
\begin{equation}\label{eq:lower-estimate-rho}
 \rho(\tau)\geq 3 e^{-\gamma\tau}\, 
\end{equation}
for $\tau_0$ sufficiently negative. This will be useful below.

In the following part of the argument, we will keep decreasing $\tau_0$, but will emphasize that all constants $C$ will be uniformly bounded independently of $\tau_0\to -\infty$ (and in particular will not depend on the value of $\gamma$). This will allow us to derive a contradiction to the assumption that $\gamma < \infty$ (as long as we take $\tau_{0}$ sufficiently negative).

We now show how to convert estimates on $W(\tau)$ into (potentially) optimal estimates on $\check u$. Propostion \ref{prop:exist-ancient-rescaled-MCF-C2-rescaling-argument} ensures that for $\tau\leq \tau_0$, we have
$$ \|\bar v\|_{C^5(\Sigma)}\leq C e^{-\mu\tau}\, ,$$
and combining this with \eqref{eq:control-semi-global-decay} we find
$$ \| \check u\|_{4,\alpha}(\rho(\tau))\leq C e^{\gamma_\mu\tau}\, ,$$
where we set
$$\gamma_\mu:= \min\{-\mu, \gamma\}\, .$$
Using these estimates in Corollary \ref{lemm:relative-shrinker-mean-curvature},  we find that $w$ satisfies 
\begin{equation}\label{eq:evolution-w}
 \partial_\tau w = L_\Sigma w + E^w
 \end{equation}
where 
$$ E^w(\bx) = w(\bx) F(\bx) + \nabla w(\bx) \cdot \bF(\bx) + \nabla^2 w(\bx) \cdot \mathcal{F}(\bx) \, ,$$
  with the estimate  
 $$|F| + |\bF| + |\cF| + |\nabla \cF| \leq C e^{\gamma_\mu\tau}$$
on $\Sigma_\text{con}\cup(\Sigma_\text{cyl}\cap B_{\rho(\tau)}(\bOh))$ for $\tau\leq \tau_0$.

We consider a smooth cut-off function $\eta\geq0$ satisfying $\eta(r) =1 $ for $|r|\leq 1, \eta(r) =0$ for $|r|\geq 2, |\eta'|\leq 2$, and $|\eta''| \leq 2$. Using \eqref{eq:lower-estimate-rho}, we can define on $\Sigma$ 
\begin{equation}\label{eq:defn.hat.w}
\hat w(\bx,\tau) = w(\bx, \tau) \,\hat\eta(\bx,\tau)\, ,
\end{equation}
where $\hat \eta =1$ on $\Sigma \setminus \Sigma_\textnormal{cyl}$, and $\hat \eta(\bx,\tau) = \eta(10^{-1} e^{\gamma \tau} |\bx|)$ on $\Sigma_\textnormal{cyl}$. We introduce
\begin{equation}\label{eq:def-hat-w}
 \hat W(\tau) := \sup_{s\leq \tau} \|\hat w(\cdot, s)\|_W
\end{equation}
and note that for $\tau_0$ sufficiently negative Proposition \ref{prop:global-decay} (as in the proof of Proposition \ref{prop:nearly-optimal-decay}) combined with \eqref{eq:control-delta} implies that
\begin{equation}\label{eq:hat-W-first-decay}
 \hat W(\tau) \leq C e^{(1+2\alpha)\gamma\tau}\leq C e^{\big(-\mu +\frac{\delta_0}{8}\big)\tau}
\end{equation}
 and we have 
$$W(\tau) \leq \hat W(\tau)\ .$$
The evolution of $\hat w$ is given by
\begin{equation}\label{eq:evolution-hat-w}
\partial_\tau \hat w = L_\Sigma \hat w +  \hat E^w =  L_\Sigma \hat w + \hat w F +  \nabla \hat w \cdot \bF + \nabla^2 \hat w \cdot \mathcal{F} + G
\end{equation}
where
\begin{equation*}\label{eq:error-hat-w}
\begin{split}
 G &= - 2 \nabla \hat \eta \cdot \nabla w - w\, \Delta \hat \eta + \tfrac{1}{2} w \, (\bx^T\cdot \nabla \hat \eta) - w \nabla \hat \eta \cdot \bF\,  - (\nabla w \otimes \nabla \hat \eta + \nabla \hat \eta \otimes \nabla w )\cdot \mathcal{F}\\
 &\quad - w  \nabla^2 \hat \eta \cdot \mathcal{F} + w\, \partial_\tau \hat \eta\, .
\end{split}
 \end{equation*}
Note that $|\nabla \hat \eta|, |\nabla^2\hat\eta| \leq C \chi_A e^{\gamma \tau}$ where $\chi_A$ is the characteristic function of the annulus $A=\{ \bx\, |\,  e^{-\gamma \tau} \leq |\bx| \leq 2 e^{-\gamma \tau}\}$. Note that by \eqref{eq:lower-estimate-rho}, $2e^{-2\gamma\tau} < \rho(\tau)$ for $\tau_0$ sufficiently negative. Thus we can estimate using \eqref{eq:control-semi-global-decay} and \eqref{eq:definiton-gamma}
$$|G| \leq C \chi_A e^{\gamma \tau} \big((|w| + |\nabla w|) + (1 + \gamma) |\bx| |w|)\big)\leq C  \chi_A .$$
We emphasize that the final constant $C$ is independent of $\gamma$ (which depends on $\tau_0$). Indeed,
\[
\chi_A  e^{\gamma \tau}\gamma  |\bx| |w| = \chi_A ((|\bx| e^{\gamma \tau})(\gamma |w|))\leq \chi_A (2\gamma e^{\gamma \tau})
\]
Since $\tau_0 < 0$, $\gamma e^{\gamma \tau}$ is bounded independently of $\gamma>0$, so $C$ can be chosen as claimed. 

Because $e^{-\frac{|\bx|^{2}}{4}}\leq \exp\big(-\tfrac{1}{4}e^{-2\gamma \tau}\big)$ on $A$, we can derive the following coarse estimates
\begin{align}
|\langle \hat w, G \rangle_W| & \leq e^{(2(1+2\alpha)\gamma + \gamma_{\mu})\tau} \label{eq:estimate-error.1} \\
\|G\|_W^2 & \leq e^{(2(1+2\alpha)\gamma + \gamma_{\mu})\tau}\,,\label{eq:estimate-error.2}
\end{align}
valid for $\tau\leq \tau_{0}\ll0$.

Combining \eqref{eq:evolution-hat-w}, \eqref{eq:estimate-error.1}, \eqref{eq:estimate-error.2} with Corollary \ref{lemm:relative-shrinker-mean-curvature} and with integration by parts (off of the $\nabla^{2}\hat w$ term),  we see that\footnote{To handle the error terms arising after integrating by parts, we can use the estimate for $\nabla \cF$ as well as \eqref{eq:relative-estimate-ibp-hits-gaussian} in Corollary \ref{lemm:relative-shrinker-mean-curvature} (to control the terms when the derivative hits the Gaussian weight).  }
\begin{equation}\label{eq:estimate-error.3}
 |\langle \hat w, \hat E^w \rangle_W| \leq C e^{\gamma_\mu\tau} \|\hat w\|^2_{W,1} + e^{(2(1+2\alpha)\gamma + \gamma_{\mu})\tau}\, .
\end{equation}

Thus we can compute
\begin{equation}\label{eq:estimate-evol-hat-w.0}
\tfrac{1}{2}\tfrac{d}{d\tau}\|\hat w\|^2_W = \langle \hat w, L_\Sigma \hat w  + \hat E^w\rangle_W \leq -\mu \|\hat w\|^2_W +  C e^{\gamma_\mu\tau} \|\hat w\|^2_{W,1} + e^{(2(1+2\alpha)\gamma + \gamma_{\mu})\tau}\, ,
\end{equation}
  and therefore
  \begin{equation}\label{eq:estimate-evol-hat-w.1}
\tfrac{1}{2}\tfrac{d}{d\tau}\big(e^{2\mu\tau}\|\hat w \|^2_W\big)  \leq   C e^{(2\mu + \gamma_\mu)\tau} \|\hat w\|^2_{W,1} + e^{(2\mu + 2(1+2\alpha)\gamma + \gamma_{\mu})\tau}\, .
\end{equation}

We recall from \cite[(5.17)]{CCMS:generic1} the identity
\begin{multline}\label{eq:integration-by-parts-hessian-laplacian}
\int_\Sigma (\Delta_\Sigma \hat w -\tfrac{1}{2} \bx \cdot \nabla_\Sigma \hat w)^2 e^{-\tfrac 14 |\bx|^{2}} d\cH^n \\ = \int_\Sigma (|\nabla^2_\Sigma \hat w|^2 - A^2(\nabla_\Sigma \hat w, \nabla_\Sigma \hat w) + \tfrac{1}{2}|\nabla_\Sigma\hat w|^2)\, e^{-\tfrac 14 |\bx|^{2}} d\cH^n
\end{multline}
and the identity following on that, which in the present situation can be written as
\begin{equation*}
 \begin{split}
 \tfrac{1}{2}\tfrac{d}{d\tau}\|\nabla_\Sigma \hat w\|_W^2 &=  - \Vert \Delta_\Sigma \hat w - \tfrac12 \mathbf{x} \cdot \nabla_\Sigma \hat w \Vert_W^2 + \Vert (\tfrac12 + |A_\Sigma|^2)^{\frac12} \nabla_\Sigma \hat w \Vert_W^2\\ & \qquad  + \langle \nabla_\Sigma \hat w, \hat w \, \nabla_\Sigma |A_\Sigma|^2 \rangle_W 
			- \langle \Delta_\Sigma \hat w - \tfrac12 \mathbf{x} \cdot \nabla_\Sigma \hat w, \hat E^w \rangle_W.
\end{split}
\end{equation*}
We can thus estimate, using \eqref{eq:evolution-hat-w}, \eqref{eq:estimate-error.2}, \eqref{eq:integration-by-parts-hessian-laplacian}, as follows
\begin{equation}\label{eq:estimate-evol-hat-w.2}
 \begin{split}
 \tfrac{1}{2}\tfrac{d}{d\tau}\|\nabla_\Sigma \hat w\|_W^2 &\leq  - \tfrac{1}{2} \Vert \Delta_\Sigma \hat w - \tfrac12 \mathbf{x} \cdot \nabla_\Sigma \hat w \Vert_W^2 + \tfrac{1}{2}\|\hat E^w\|_W^2 + C \|w\|_{W,1}^2\\
 &\leq  -\tfrac{1}{2} \Vert \nabla^2_\Sigma \hat w\Vert_W^2 + 2\| \nabla^2 \hat w \cdot \mathcal{F}\|_W^2 + C \|w\|_{W,1}^2 + e^{(2(1+2\alpha)\gamma + \gamma_{\mu})\tau}\\
 &\leq C \|w\|_{W,1}^2 + e^{(2(1+2\alpha)\gamma + \gamma_{\mu})\tau}
 \end{split}
\end{equation}
Alternatively, we can integrate by parts rather than using the spectral estimate in \eqref{eq:estimate-evol-hat-w.0} to find (decreasing $\tau_0$ if necessary and using \eqref{eq:estimate-error.3})
\begin{equation}\label{eq:estimate-evol-hat-w.3}
\begin{split}
 \tfrac{1}{2}\tfrac{d}{d\tau}\|\hat w\|^2_W &= - \Vert \nabla_\Sigma \hat w \Vert_W^2 + \langle \hat w, (\tfrac12 + |A_\Sigma|^2)\hat w + \hat E^w \rangle_W \\
& \leq -  \Vert \nabla_\Sigma \hat w \Vert_W^2 + C \Vert \hat w \Vert_W^2 + C e^{\gamma_\mu\tau} \|\hat w\|^2_{W,1} + e^{(2(1+2\alpha)\gamma + \gamma_{\mu})\tau}\\
&\leq - \tfrac12 \Vert \nabla_\Sigma \hat w \Vert_W^2 + C \Vert \hat w \Vert_W^2 + e^{(2(1+2\alpha)\gamma + \gamma_{\mu})\tau}\, .
 \end{split}
\end{equation}

Combining \eqref{eq:estimate-evol-hat-w.2} and \eqref{eq:estimate-evol-hat-w.3} yields
\begin{equation}\label{eq:final-unique-W1-ode}
 \tfrac{d}{d\tau} (\|\nabla_\Sigma \hat w\|_W^2 + C \|\hat w\|^2_W)\leq C \Vert \hat w \Vert_W^2 + C e^{(2(1+2\alpha)\gamma + \gamma_{\mu})\tau}\leq C \hat W(\tau)^2 + C e^{(2(1+2\alpha)\gamma + \gamma_{\mu})\tau}
 \end{equation}
By \eqref{eq:hat-W-first-decay} find that
\[ \hat W(\tau)^2 \leq C e^{2(1+2\alpha) \gamma \tau} \, ,\]
for $\tau\leq \tau_0 \ll0$, so by integrating \eqref{eq:final-unique-W1-ode}, we find
\[ \|\hat w(\cdot, \tau)\|^2_{W,1} \leq C \gamma^{-1} e^{2(1+2\alpha)\gamma\tau} + C \gamma^{-1} e^{(2(1+2\alpha)\gamma + \gamma_{\mu})\tau} \leq C e^{2(1+2\alpha)\gamma\tau} \, \]
for $\tau\leq \tau_{0}$. In the final inequality we used $\gamma_{\mu}>0$ and $\gamma^{-1} \leq \gamma_{0}^{-1}<\infty$ (where we emphasize that $\gamma_{0}$ is fixed in \eqref{eq:first-estimate-gamma} and thus is independent of $\tau_{0}$). 

We can now return to \eqref{eq:estimate-evol-hat-w.1} to find
\[
\tfrac{1}{2}\tfrac{d}{d\tau}\big(e^{2\mu\tau}\|\hat w \|^2_W\big)  \leq C e^{(2\mu + 2(1+2\alpha)\gamma + \gamma_{\mu})\tau}\, .
\]
Note that $2\mu + 2(1+2\alpha)\gamma \geq \tfrac{\delta_0}{4} > 0$, so we can integrate this to find 
\[ \|\hat w(\cdot, \tau)\|^2_{W} \leq  e^{\big(2(1+2\alpha)\gamma +\frac{\gamma_\mu}{2}\big)\tau} \]
for $\tau \leq \tau_{0}\ll0$. Thus,
\[W(\tau) \leq   e^{\big((1+2\alpha)\gamma +\frac{\gamma_\mu}{4}\big)\tau}\]
for $\tau \leq \tau_0$. But this contradicts the extremality of $\gamma$ in \eqref{eq:definiton-gamma} (for $\tau_{0}$ fixed sufficiently negative so that each step above is justified) and thus $\gamma = \infty$ and $W(\tau) \equiv 0$ for $\tau \leq \tau_0$.

In case (2) we can argue as above, replacing $\tau \mapsto \bar{M}(\tau)$ by the static flow $\tau \mapsto \Sigma$.
 \end{proof}


\part{Application to mean curvature flow in $\RR^3$}

\section{Genus drop for ancient one-sided flows in $\RR^{3}$}\label{sec:genus-drop-ancient}

In this section we consider $(\Sigma,\Omega) \in \cS_{2}^*\setminus \cS^\textrm{gen}_2$. In other words, we consider $\Sigma$ a non-flat self-shrinker in $\RR^3$ that is not a sphere or cylinder, and $\Omega$ a chosen side of $\Sigma$. Recall that by Proposition \ref{prop:ends.shrinkers.R3} it holds that $\Sigma \in \cS_{2}'''$. 

Consider the (one-sided) ancient weak set flow $\cK$ and associated Brakke flow $\cM$ as considered in Proposition \ref{prop:exist-ancient-rescaled-MCF-C2-rescaling-argument} (cf. Corollary \ref{coro:summary-tleq0-non-rescaled}). The next lemma follows by combining \cite[Corollary 1.4]{ColdingMinicozzi:sing-generic} with Proposition \ref{prop:exist-ancient-rescaled-MCF-C2-rescaling-argument}.
\begin{lemma}
$\partial\cK(t)$ is smooth for a.e., $t<0$. 
\end{lemma}

In this section we will show that $\partial\cK(t)$ loses genus as $t\to 0$ (recall that Brendle has shown \cite{Brendle:genus0} that $\Sigma$ has positive genus). We will later use this in the global perturbation argument to ensure that we only need to perturb the initial conditions finitely many times. 

We recall the following definitions:
\begin{definition}
If $M$ is an oriented compact connected surface without boundary, decomposing $M = \cup_{\alpha} M _{\alpha}$ into connected components, then we set
\[
\genus(M) : = \sum_{\alpha} \genus(M_{\alpha}) = \tfrac 12 b_{1}(M). 
\]
\end{definition}
\begin{definition}
If $M^2$ is an oriented compact surface with boundary, the genus of $M$ is defined to be the genus of the oriented surface $\tilde M$ formed by capping off each  circle in $\partial M$ by a disk. 
\end{definition}

\begin{definition}[Simple flow {\cite[Appendix G]{CCMS:generic1}}] For $U\subset \RR^3$ an open set with smooth boundary, a \emph{simple flow} in $U\times I$ is a closed subset of space-time $\cM \subset \RR^{3}\times \RR$ so that there is a compact $2$-manifold with boundary and a continuous map $f: M \times I \to \RR^3$ so that 
\begin{enumerate}
\item $\cM(t) \cap \overline U= f(M\times \{t\})$ (where $\cM(t) = \{\bx \in \RR^3 : (\bx,t) \in \cM\}$) 
\item $f$ is smooth on $M^\circ \times I$ where $M^\circ = M\setminus\partial M$ 
\item $f(\cdot,t)$, $t\in I$ is an embedding of $M^\circ$ into $U$
\item $t\mapsto f(M^\circ\times \{t\})$ is a smooth mean curvature flow
\item $f|_{\partial M\times I}$ is a smooth family of embeddings of $\partial M$ into $\partial U$. 
\end{enumerate}
\end{definition}

We can now state the main result. 
\begin{proposition}[Genus drop for one-sided flows in $\RR^3$]\label{prop:genus-drop-one-sided-flow}
Let $(\Sigma,\Omega)$, $\cK$ be as above. There is $t_{0}<t_{1}<0$, $\kappa_{0} \in (0,t_{1}-t_{0})$, $\kappa_{1} \in (0,-t_{1})$ and $R>0$ with the following properties. 
\begin{enumerate}
\item for any $t \in [t_{0},t_{0}+\kappa_{0}]$ so that $\partial\cK(t)$ is smooth, $\genus(\partial\cK(t) \cap B_{R}(\bOh)) > 0$,
\item  for any $t \in [t_{1},t_{1}+\kappa_{1}]$ so that $\partial\cK(t)$ is smooth, $\genus(\partial\cK(t) \cap B_{R}(\bOh)) = 0$,
\item  $t\mapsto \partial\cK(t)$ is a simple flow in $\{R/2 <  |\bx| < 3R\} \times [t_0,t_1+\kappa_{1}]$,
\item  and
for $(\bx,t) \in \partial\cK(t) \cap \{R/2 < |\bx| < 3R\}\times  [t_0,t_1+\kappa_{1}]$, we have $|\bx^T| \geq \tfrac 12 |\bx|$, for $\bx^T$ the projection of $\bx$ to $T_\bx\partial\cK(t)$. 
\end{enumerate}
\end{proposition}

\begin{remark}
We first sketch the idea of the proof. For simplicity let us assume that $\Sigma$ is compact (in reality one must be somewhat careful with the cylindrical ends, but the main idea  is the same). Write $\Omega(\tau)$ for the inside of the rescaled one-sided flow (so $\Omega(\tau) \subset \Omega(\tau') \subset \Omega$ for $\tau > \tau'$). Suppose that $\partial\Omega(\tau)$ has non-zero genus for all $\tau$ large. Note that $d(\Sigma, \Omega(\tau))$ is exponentially increasing as $\tau\to\infty$ (since the distance between the non-rescaled flows is non-decreasing). 

Thus, the non-trivial situation is when $\Omega$ is the unbounded component of $\RR^3\setminus\Sigma$. If $\partial\Omega(\tau)$ has positive genus, then by a result in algebraic topology (cf.\ Lemma \ref{lemm:H1.subset.R3.bdry} below), there is a non-zero element $[\sigma] \in H_1(\Omega(\tau))$. Because the flows are shrinker mean convex, we find $\sigma \subset \Omega$ and $d(\sigma,\Sigma)$ is arbitrarily large. However, it is easy to see that when $d(\sigma,\Sigma)$ is sufficiently large, then $[\sigma] = 0 \in H_1(\Omega)$. This contradicts White's topological monotonicity result \cite[Theorem 1(iii)]{White:topology-weak}, which says that a $1$-cycle in the complement of a ($2$-dimensional) weak set flow cannot suddenly cease to bound a $2$-chain. 
\end{remark}

\begin{remark}
The argument used here avoids the need for a potentially delicate analysis of the star-shaped properties of the flow (along cylindrical ends) near $t=0$ (cf.\ \cite[Theorem 9.1]{CCMS:generic1}).  
\end{remark}


The following is a direct consequence of Lemma \ref{lem:ends-decomposition}.

\begin{lemma}\label{lemm:ends-decomp-specific}
We can take $R_0$ sufficiently large so that for $R\geq R_0$, $\Sigma$ intersects $\partial B_R(\bOh)$ transversely and $\Sigma \setminus B_R(\bOh)$ (resp.\ $\Omega \setminus B_R(\bOh)$) is diffeomorphic to $\Sigma \cap \partial B_R(\bOh) \times [R,\infty)$ (resp.\ $\Omega \cap \partial B_R(\bOh)\times [R,\infty)$). 
\end{lemma}
Note that $\Sigma \cap \partial B_R(\bOh)$ is a finite union of pairwise disjoint embedded loops in $\partial B_R(\bOh) \approx \SS^2$ and $\Omega\cap\partial B_R(\bOh)$ is a finite union of regions in $\partial B_R(\bOh)$ bounded by said loops. 
\begin{lemma}
There is $r >0$ and $s_{1}<0$ so that $\cK(t) \cap B_{r}(\bOh) = \emptyset$ for $t \in [s_{1},0)$.
\end{lemma} 
\begin{proof}
This follows from (1) in Proposition \ref{theo:basic-prop-cK-cM}.
\end{proof}
As such, taking $t_1 \in [s_{1},0)$ sufficiently close to $0$, we find that for all $t \in [t_{1},0)$,
\begin{equation}\label{eq:genus.drop.defn.t1}
\tfrac{1}{\sqrt{-t}}\cK(t) \cap B_{R_0}(\bOh) = \emptyset,
\end{equation}
wheres $R_0$ is fixed in Lemma \ref{lemm:ends-decomp-specific} above. This fixes $t_{1}$. Now we fix an arbitrary $\kappa_{1} \in (t_{1},0)$. 
\begin{lemma}\label{lemm:genus.drop.defn.t0}
There is $t_{0}<t_{1}$ sufficiently negative, $\kappa_{0} \in (0,t_{1}-t_{0})$ and $R_{1}\geq R_{0}$ so that for any $R\geq R_{1}$ and $t \in [t_{0},t_{0}+\kappa_{0})$, $\tfrac{1}{\sqrt{-t}}\partial\cK(t)$ is a smooth surface meeting $\partial B_{R}(\bOh)$ transversally with $\genus(\tfrac{1}{\sqrt{-t}} \partial\cK(t) \cap B_{R}(\bOh)) > 0$ and $\tfrac{1}{\sqrt{-t}}\partial\cK(t) \setminus B_{R}(\bOh)$ is diffeomorphic to $(\tfrac{1}{\sqrt{-t}}\partial\cK(t) \cap \partial B_{R}(\bOh))\times [R,\infty)$
\end{lemma}
\begin{proof}
Combine Lemmas \ref{lemm:alg.int.pos.genus} and \ref{lemm:ends-decomp-specific}, property (4) in Proposition \ref{prop:exist-ancient-rescaled-MCF-C2-rescaling-argument}, and \cite[Theorem 2]{Brendle:genus0}. 
\end{proof}

\begin{lemma}\label{lemm:genus.drop.arg.near.bdr.sphere.top}
There is $R_{2}\geq R_{1}$ sufficiently large so that:
\begin{enumerate}
\item For $R \geq R_{2}$ and $t \in [t_{0},t_{1}+\kappa_{1}]$ we have that $\tfrac{1}{\sqrt{-t}} \partial\cK(t)$ is smooth in a neighborhood of $\partial B_{R}(\bOh)$ and these surfaces have transverse intersection. 
\item For any $t > t_{0}$, each component of $\overline{(\tfrac{1}{\sqrt{-t}}\cK(t) \setminus \tfrac{1}{\sqrt{-t_0}}\cK(t_{0}))} \cap \partial B_{R}(\bOh)$
is an annulus. 
\end{enumerate}
\end{lemma}
\begin{proof}
By (3) in Proposition \ref{prop:exist-ancient-rescaled-MCF-C2-rescaling-argument} there is $R(t)$ monotonically increasing on $(-\infty,0)$ so that 
\[
\partial \cK(t) \setminus B_{R(t)}(\bOh) 
\]
is a smooth strictly shrinker mean convex flow that can be written as a small $C^{3}$-graph over a subset of $\Sigma$. Take $R_{2} > \max_{t\in[t_{0},t_{1}+\kappa_{1}]}\tfrac{1}{\sqrt{-t}}R(t) = \tfrac{1}{\sqrt{-(t_{1}+\kappa_{1})}} R(t_{1}+\kappa_{1})$. This yields (1). To prove (2), we combine the above reasoning with the geometric reformulation of shrinker mean convexity: $\tfrac{1}{\sqrt{-t}} \cK(t) \subset \tfrac{1}{\sqrt{-s}} \cK(s)$ for $t>s$. 
\end{proof}

 Loosely speaking, the next lemma shows that if for a time $t$ close to zero $(-t)^{-1/2} \partial\cK(t)$ still has genus in a large ball, then this can be detected by a homologically nontrivial  loop in $(-t)^{-1/2} \cK(t)$ (relative to the homology created by the ends of the shrinker) that was trivial in the ancient past. Note that we must take care to avoid loops that are homologically nontrivial only because they wrap around the ends of the shrinker/flow.

\begin{lemma}\label{lemm:one.sided.genus.drop.inding.loop}
Suppose  that for some $t\in [t_{1},t_{1}+\kappa_{1}]$ and $R\geq R_{2}$ it holds that $\partial \cK(t)$ is smooth and $\genus(\tfrac{1}{\sqrt{-t}} \partial\cK(t)\cap B_{R}(\bOh)) > 0$. Then, there is a loop $\sigma \subset \tfrac{1}{\sqrt{-t}} \cK(t)\cap B_{R}(\bOh)$ so that
\begin{itemize}
\item  for any $1$-cycle  $\beta \subset \partial B_{R}(\bOh)\cap \tfrac{1}{\sqrt{-t}}\cK(t)$, it holds that  $[\sigma] \neq [\beta]$ in $H_{1}(\tfrac{1}{\sqrt{-t}} \cK(t)\cap \overline{B_{R}(\bOh)};\ZZ)$, but
\item there is a $1$-cycle $\beta \subset \partial B_{R}(\bOh)\cap \tfrac{1}{\sqrt{-t_{0}}}\cK(t_{0})$ so that $[\sigma] = [\beta]$ in $H_{1}(\tfrac{1}{\sqrt{-t_{0}}}\cK(t_{0})\cap \overline{B_{R}(\bOh)};\ZZ)$. 
\end{itemize}
\end{lemma}
\begin{proof}
Applying Lemma \ref{lemma:genus.loops.inside.first} to the set $\tilde \Omega$ formed by smoothing out the corners of $\tfrac{1}{\sqrt{-t}} \cK(t) \cap B_{R+\eps}(\bOh)$ and ball $B=B_{R}(\bOh)$ yields $\sigma \subset  \tfrac{1}{\sqrt{-t}} \cK(t)\cap B_{R}(\bOh)$ satisfying the first bullet point (since 
$
\genus(\partial\tilde\Omega \cap B_{R}(\bOh)) > 0 $
by assumption). 

Now, we observe that by shrinker mean convexity and \eqref{eq:genus.drop.defn.t1}
\[
\sigma \subset \tfrac{1}{\sqrt{-t_{0}}} \cK(t_{0}) \cap (B_{R}(\bOh)\setminus B_{R_{0}}(\bOh)). 
\]
By Lemma \ref{lemm:genus.drop.defn.t0} we have
\[
\tfrac{1}{\sqrt{-t_{0}}} \cK(t_{0}) \cap (\overline{B_{R}(\bOh)}\setminus B_{R_{0}}(\bOh))
\]
deformation retracts to
\[ 
\tfrac{1}{\sqrt{-t_{0}}} \cK(t_{0})\cap \partial B_{R}(\bOh)
\]
so the second bullet point follows. 
\end{proof}

\begin{lemma}\label{lemm:genus.drop.minimize}
Recall that $\partial\Omega = \Sigma$. For some $R\geq R_{2}$ fixed, suppose that there is a loop $\sigma \subset \Omega \cap B_{R}(\bOh)$ and a $1$-cycle $\beta \subset  \Omega \cap \partial B_{R}(\bOh)$ so that $[\sigma] = [\beta]$ in $H_{1}(\Omega \cap \overline{B_{R}(\bOh)};\ZZ)$. Then, there is a connected smooth embedded self-shrinker $\Gamma\subset \Omega$ with $\partial \Gamma = \sigma$. Moreover,
\[
\sqrt{0} \Gamma \cap \partial B_{1}(\bOh) \qquad \textrm{and} \qquad \sqrt{0} \Sigma \cap \partial B_{1}(\bOh)
\]
are disjoint. 
\end{lemma}

\begin{proof}
As in \cite[Proposition 12]{Brendle:genus0},  we can choose a sequence of Riemannian metrics $g_{k}$ on $\RR^{n+1}$ converging in $C^{\infty}_{\textrm{loc}}(\RR^{n+1})$ to $e^{-\frac{|\bx|^{2}}{2n}}g_{\RR^{n+1}}$ so that $\Omega \cap B_{2k}(\bOh)$ is (weakly) shrinker mean convex in the sense of Meeks--Yau \cite[\S 1]{MeeksYau}. By Lemma \ref{lemm:ends-decomp-specific}, we can choose $\beta_{k} \subset \Omega \cap \partial B_{2k}(\bOh)$ homologous to $\beta$ in $\Omega$. 

Let $\Gamma_{k}$ minimize $g_{k}$-area among oriented embedded surfaces $\Gamma' \subset \Omega$ with $\partial\Gamma' = \sigma - \beta_{k}$. Using standard interior regularity for minimizers (and \cite{HardtSimon:bdry.reg}) we can pass a subsequence and obtain a connected smooth embedded self-shrinker with $\Gamma\subset \Omega$ (thanks to the strong maximum principle) and $\partial\Gamma = \sigma$. The statement about the asymptotic cones follows from Ilmanen's avoidance principle (Theorem \ref{theo:ilmanen-avoidance}) as in Lemma \ref{lemm:lin-grow-trans}. Indeed, for $t<0$, $t\mapsto \sqrt{-t}\Gamma,t\mapsto \sqrt{-t}\Sigma$ are disjoint mean curvature flows in any small ball around any point in $\partial B_{1}(\bOh)$. Applying Ilmanen's avoidance principle, we find that they remain a definite distance apart at $t=0$ (in a smaller ball). This completes the proof. 
\end{proof}

\begin{corollary}\label{coro:genus.drop.avoidance}
Fix $\sigma,\Gamma$ as in Lemma \ref{lemm:genus.drop.minimize}. Assume that $\sigma \subset\frac{1}{\sqrt{-t'}} \cK(t')$ for some $t'<0$. Then $\Gamma \subset\frac{1}{\sqrt{-t'}} \cK(t')$
\end{corollary}
\begin{proof}
Combining Lemma \ref{lemm:genus.drop.minimize} with the sublinear growth estimates in Proposition \ref{prop:sublinear-growth} applied to $\cM(t)$ (recall that $\cM(t)$ is the one-sided Brakke flow with $\supp\cM(t) = \partial\cK(t)$), we have that
\[
d((\tfrac{1}{\sqrt{-t}} \partial\cK(t)) \setminus B_{R}(\bOh), \Gamma \setminus B_{R}(\bOh)) \to \infty
\]
as $R\to\infty$, uniformly for $t\leq t'$. Using this,   together with the fact that $\frac{1}{\sqrt{-t}} \cK(t)$ exhausts $\Omega$ as $t\to -\infty$ and the avoidance principle, we find that $(\tfrac{1}{\sqrt{-t}} \partial\cK(t)) \cap \Gamma = \emptyset$ for all $t\leq t'$. This proves the assertion (because $\Gamma$ is connected, if there was a point in $\Gamma \setminus \tfrac{1}{\sqrt{-t}}\cK(t)$ then there would be a point in $\Gamma \cap \tfrac{1}{\sqrt{-t}}\partial \cK(t)$). 
\end{proof}

With these preparations we can now prove the genus drop result.
\begin{proof}[Proof of Proposition \ref{prop:genus-drop-one-sided-flow}]
Fix $t_{1},\kappa_{1}$ as in \eqref{eq:genus.drop.defn.t1} and $t_{0},\kappa_{0}$ as in Lemma \ref{lemm:genus.drop.defn.t0}. Then, as long as we fix $R\geq \tfrac{1}{\sqrt{-(t_{1}+\kappa_{1})}}R_{2}$ sufficiently large, we can ensure that (3) and (4) hold by using stability of cylinders (Lemma \ref{lem:smooth-stability-cylinder}) and pseudolocality (along the conical ends). 

Property (1) follows from Lemma \ref{lemm:genus.drop.defn.t0} (with $\sqrt{-t}R \geq R_{1}$ in place of $R$). To prove property (2), we assume that for some $t \in [t_{1},t_{1}+\kappa_{1}]$, $\partial\cK(t)$ is smooth and $\genus(\partial\cK(t)\cap B_{R}(\bOh))>0$. Hence,
\[
\genus(\tfrac{1}{\sqrt{-t}} \partial\cK(t)\cap B_{R'}(\bOh)) > 0 
\]
for $R' = \tfrac{1}{\sqrt{-t}} R \geq R_{2}$. Fix $\sigma$ as in Lemma \ref{lemm:one.sided.genus.drop.inding.loop} (with $R'$ in place of $R$) and then $\Gamma$ as in Lemma \ref{lemm:genus.drop.minimize}. By Corollary \ref{coro:genus.drop.avoidance} (with $t=t'$) we find that
\[
\Gamma \subset \tfrac{1}{\sqrt{-t}} \cK(t).
\]
We can perturb $\Gamma$ slightly to find a surface with $\Gamma \subset \tfrac{1}{\sqrt{-t}} \cK(t)$ and $\Gamma$  transversal to $\partial B_{R'}(\bOh)$. This yields 
\[
\sigma = \partial (\Gamma \cap B_{R}(\bOh)) - \beta
\]
for some $\beta \subset (\tfrac{1}{\sqrt{-t}} \cK(t)) \cap \partial B_{R'}(\bOh)$. This contradicts the first bullet point in Lemma \ref{lemm:one.sided.genus.drop.inding.loop}, completing the proof. 
\end{proof}


\section{The first non-generic time for flows in $\RR^3$}\label{sec:main-generic}

\begin{definition}[Generic singularities and first non-generic time]\label{defi:generic-time}
Consider $\cM$ a cyclic integral unit-regular Brakke flow of surfaces in $\RR^3$. 
\begin{itemize}
\item We define $\sing_\textnormal{gen}\cM\subset \sing\cM$ to be the set of singular points for $\cM$ so that one tangent flow (and thus all \cite{ColdingMinicozzi:uniqueness-tangent-flow})  is a multiplicity-one shrinking sphere or cylinder. 
\item Then, we define $T_\textrm{gen}(\cM) : = \inf \ft(\sing\cM \setminus \sing_\textrm{gen}\cM)$, i.e., the first time that there is a ``non-generic'' singularity. 
\item Finally, we set $\fS_\textrm{first}(\cM) : = (\sing \cM \setminus \sing_\textrm{gen}\cM) \cap \{t=T_\textrm{gen}(\cM)\}$, i.e., this is the set of ``non-generic'' singularities at the first time that they exist. 
\end{itemize}
\end{definition}

The following is our main result. 

\begin{theorem}\label{theo:generic.flow.R3}
Suppose that $M\subset \RR^3$ is a closed embedded surface. Then, there exist arbitrarily small $C^\infty$ graphs $M'$ over $M$ and corresponding cyclic integral unit-regular Brakke flows $\cM'$ with $\cM'(0) = \cH^2\lfloor M'$ so that either:
\begin{enumerate}
\item $\cM'$ has only generic singularities: $\sing\cM' = \sing_\textnormal{gen}\cM'$, or 
\item multiplicity occurs: there is $\bx \in \RR^3$ so that any tangent flow to $\cM'$ at the space-time point $(\bx,T_\textnormal{gen}(\cM'))$ is $k\cH^2\lfloor \sqrt{-t}\Sigma$ for $k\geq 2$ and $\Sigma \in \cS_2$ a smooth shrinker.
\end{enumerate}
\end{theorem}

Compare with \cite[Theorem 11.1]{CCMS:generic1} which also left open the possibility that a non-cylindrical shrinker with a cylindrical end could occur. The proof here follows a similar strategy (granted the stronger one-sided uniqueness result Theorem \ref{thm:uniqueness-global})  but differs in some minor, but technically important ways.

In this section we will establish various preliminary results that lead up to the proof of Theorem \ref{theo:generic.flow.R3}. 

\begin{lemma}\label{lemm:first-generic-time.unique.BF}
If $\cM,\cM'$ are cyclic integral unit-regular Brakke flows in $\RR^3$ with $\cM(0) = \cM'(0) = \cH^2\lfloor M$ for a closed embedded surface $M\subset \RR^3$ then $T_\textnormal{gen}(\cM) =T_\textnormal{gen}(\cM')$, $\cM(t) = \cM'(t)$ for $t< T_\textnormal{gen}(\cM)$, and $\fS_\textnormal{first}(\cM) = \fS_\textnormal{first}(\cM')$.
 \end{lemma}
\begin{proof}
By definition of $T_\textnormal{gen}$ and \cite[Corollary F.5]{CCMS:generic1}, we find that $\reg \cM \cap \{t < T_\textrm{gen}(\cM)\}$ is connected. Thus, the assertion follows from unit-regularity and the resolution of the mean convex neighborhood conjecture \cite[Theorem 1.9]{ChoiHaslhoferHershkovits} (cf. \cite{HershkovtisWhite}). 
\end{proof}

\begin{definition}
Consider  $M \subset \RR^3$ a closed embedded surface. Let $U$ denote the unique compact region bounded by $M$. Write $M(t)$ for the level set flow of $M$ and $U(t)$ for the level set flow of $U$. 

Consider a cyclic integral unit-regular Brakke flow $\cM$ in $\RR^3$ with $\cM(0) = \cH^2\lfloor M$. We say
the set of times $\cT_\textrm{reg}(\cM)$ is the complement of the set of singular times $\ft(\sing\cM)$. Note that by upper semi-continuity of density $\cT_\textrm{reg}(\cM)$ is open. Furthermore, for $t\in \cT_\textrm{reg}(\cM)$ the level set flow $M(t)$ of $M$ at time $t$ is a smooth embedded surface $M(t)$ with $M(t) = \partial U(t)$ and $\cM(t) = \cH^2\lfloor M(t)$.
\end{definition}

The next result follows from combining several important results in the literature (cf.\ \cite[Proposition 11.3]{CCMS:generic1}).
\begin{proposition}\label{prop:properties.until.nongeneric.time}
Consider a cyclic integral unit-regular Brakke flow $\cM$ in $\RR^3$ with $\cM(0) = \cH^2\lfloor M$ for a closed embedded surface $M \subset \RR^3$. Write $U$ for the unique compact set bounded by $M$ and $U(t)$ for the level set flow of $U$. Then:
\begin{itemize}
\item the level set flow $M(t)$ of $M$ does not fatten before $T_\textnormal{gen}(\cM)$,
\item almost every $t\in [0,T_\textnormal{gen}(\cM)]$ is a regular time for the flow, and
\item $t\mapsto \genus(M(t))$ is non-increasing for $t \in \cT_\textnormal{reg}$. 
\end{itemize}
Now, assume that $T_\textnormal{gen}(\cM) <\infty$. Then:
\begin{itemize}
\item Every tangent flow at $(\bx,T_\textnormal{gen}(\cM)) \in \sing\cM$ is of the following form for $t<0$
\[
t\mapsto k \cH^2\lfloor\sqrt{-t}\Sigma
\]
where $k \in \NN$ is the multiplicity and $\Sigma$ a smooth embedded self-shrinker and
\item the self-shrinker $\Sigma$ in the previous bullet point satisfies 
\[
\genus(\Sigma) \leq \lim_{\cT_\textrm{reg} \ni t \nearrow T_\textnormal{gen}(\cM)}\genus(M(t)) . 
\]
and $\Sigma \in \cS_2''$ has nice stable ends. 
\end{itemize}
Finally, assume that for some point $(\bx,T_\textnormal{gen}(\cM)) \in \fS_\textnormal{first}(\cM)$ there is some tangent flow of the form $t\mapsto \cH^2\lfloor \sqrt{-t}\Sigma$, i.e., the tangent flow is multiplicity one. Then:
\begin{itemize}
\item all tangent flows at $(\bx,T_\textnormal{gen}(\cM))$ occur with multiplicity-one---possibly with a different $\Sigma$, 
\item $\genus(\Sigma) \geq 1$, and
\item $F(\Sigma) \geq \lambda_1 + \delta_0$ for some numerical constant $\delta_0$. 
\end{itemize}
\end{proposition}
\begin{proof}
By \cite[Theorem 1.9]{ChoiHaslhoferHershkovits} the level set flow of $M$ does not fatten before $T_\textrm{gen}(\cM)$. Hence, by \cite[Corollary 1.4]{ColdingMinicozzi:sing-generic} (and \cite[Corollary F.5]{CCMS:generic1}), we have that $M(t)$ is smooth and $\cM(t) = \cH^2\lfloor M(t)$ for almost every $t \in [0,T_\textrm{gen}(\cM)]$. Combined with \cite[Theorems 3.1, B.2, and C.1]{HershkovtisWhite}, this implies that almost every $t \in [0,T_\textrm{gen}(\cM)]$ is a regular time for the flow. Now, by \cite[Theorem 1]{White:topology-weak}, we have that $t\mapsto \genus(M(t))$ is non-increasing for $t \in \cT_\textrm{reg}$. This verifies the first set of bullet points. 

The second set of bullet points follows from a straightforward adaptation of \cite{Ilmanen:singularities} to avoid the singular time-slices (for $\Sigma \in\cS_2''$, see Proposition \ref{prop:ends.shrinkers.R3}). Finally, for the third set of bullet points, the first one is proven in Proposition \ref{prop:smooth-tf-all-are}. The second follows from \cite[Theorem 2]{Brendle:genus0} and the third from \cite[Corollary 1.2]{BernsteinWang:TopologicalProperty}. This completes the proof. 
\end{proof}

It is convenient to define the quantity 
\[
\genus_{T_\textnormal{gen}}(\cM) : = \lim_{\cT_\textrm{reg} \ni t \nearrow T_\textnormal{gen}(\cM)}\genus(M(t)) . 
\]
We will refer to this below. Note that by Proposition \ref{prop:properties.until.nongeneric.time} we have $\genus_{T_\textnormal{gen}}(\cM) \leq \genus M$. Moreover, by Lemma \ref{lemm:first-generic-time.unique.BF} and Proposition \ref{prop:properties.until.nongeneric.time} we find that the quantities $T_\textrm{gen}(\cM)$ and $\genus_{T_\textnormal{gen}}(\cM)$ depend only on the initial surface $M$, not the chosen Brakke flow $\cM$ (by \cite[Theorem B.6]{HershkovitsWhite:set-theoretic} a flow $\cM$ as above always exists). As such, we will sometimes write $T_\textrm{gen}(M)$ and $\genus_{T_\textnormal{gen}}(M)$.

The following lemma will be important later. 

\begin{lemma}\label{lemm:first-generic-time.points}
Consider $\cM$ a cyclic integral unit-regular Brakke flow in $\RR^3$ with $\cM(0)  = \cH^2\lfloor M$ a closed embedded surface $M\subset \RR^3$. Assume that $T_\textnormal{gen}(\cM) < \infty$ and at any $(\bx,T_{\textnormal{gen}}(\cM)) \in \fS_\textnormal{first}(\cM)$, all tangent flows occur with multiplicity one. Then $\fS_\textnormal{first}(\cM)$ is a non-zero finite set. 
\end{lemma}
\begin{proof}
It is clear that $\fS_\textnormal{first}(\cM)$ is non-zero by definition of the first non-generic time. We claim that $\fS_\textnormal{first}(\cM)$ is closed. Indeed, if a sequence $(\bx_j,T_\textrm{gen}(\cM)) \in \fS_\textnormal{first}(\cM)$ has $(\bx_j,T_\textrm{gen}(\cM)) \to (\bx_\infty,T_\textrm{gen}(\cM))$ we note that
\[
\Theta_{\cM}(\bx_\infty,T_\textrm{gen}(\cM)) \geq \limsup_{j\to\infty} \Theta_{\cM}(\bx_\infty,T_\textrm{gen}(\cM)) \geq \lambda_1 + \delta_0
\]
by upper-semicontinuity of density and the final bullet point in Proposition \ref{prop:properties.until.nongeneric.time}. This implies that $(\bx_\infty,T_\textrm{gen}(\cM)) \in \fS_\textnormal{first}(\cM)$.  

We now claim that $\fS_\textrm{first}(\cM)$ is discrete. (This will finish the proof.) Indeed, if $(\bx_j,T_\textrm{gen}(\cM)) \to (\bx_\infty,T_\textrm{gen}(\cM))$ are distinct elements of $\fS_\textrm{gen}(\cM)$ we can rescale $\cM$ around $(\bx_\infty,T_\textrm{gen}(\cM))$ by $|\bx_j - \bx_\infty|^{-1}$ to yield $\tilde \cM_j$. Note that up to a subsequence, $\tilde\cM_j$ converges to a tangent flow $\tilde\cM$ to $\cM$ at $(\bx_\infty,T_\textrm{gen}(\cM))$. By Proposition \ref{prop:properties.until.nongeneric.time} and the multiplicity one assumption, $\tilde\cM = \cH^2\lfloor \Sigma$ for some $\Sigma\in \cS_2''$. Rescaling the points $\bx_j$ and passing to the limit we find $\tilde \bx \in \RR^3$ with $|\tilde \bx| = 1$ and $\Theta_{\tilde\cM}(\bx,0) \geq \lambda_1 + \delta_0$ (by upper-semicontinuity of density and the final bullet point in Proposition \ref{prop:properties.until.nongeneric.time}). This contradicts the fact that all tangent flows to $\tilde\cM$ at $(\bx,0) \neq (\bOh,0)$ are multiplicity one cylinders or planes. This completes the proof. 
\end{proof}

\begin{proposition}\label{prop:generic.flow.R3.iterate}
Consider a cyclic integral unit-regular Brakke flow $\cM$ in $\RR^3$ with $\cM(0) = \cH^2\lfloor M$ for a closed embedded surface $M \subset \RR^3$. Write $U$ for the unique compact region bounded by $M$ and choose the inwards pointing unit normal. Assume that $T_\textnormal{gen}(\cM) < \infty$ and at any $(\bx,T_{\textnormal{gen}}(\cM)) \in \fS_\textnormal{first}(\cM)$, all tangent flows occur with multiplicity one. If $M_j$ is a sequence of normal graphs of functions $0<u_j$ over $M$ with $u_j\to 0$ in $C^\infty(M)$, then 
\[
\limsup_{j\to\infty} \genus_{T_\textnormal{gen}}(M_j) \leq  \genus_{T_\textnormal{gen}}(M) - 1. 
\]
\end{proposition}

\begin{remark}
This result would also hold (with the same proof) if we assumed that $0>u_j$. 
\end{remark}

Before proving Proposition \ref{prop:generic.flow.R3.iterate}, we observe that it immediately implies Theorem \ref{theo:generic.flow.R3} after finitely many iterations.

\begin{proof}[Proof of Theorem \ref{theo:generic.flow.R3} using Proposition \ref{prop:generic.flow.R3.iterate}]
For $M\subset \RR^3$ a closed embedded surface, fix (using \cite[Theorem B.6]{HershkovitsWhite:set-theoretic}) a cyclic integral unit-regular Brakke flow $\cM$ with $\cM(0) = \cH^2\lfloor M$.  If $T_{\textrm{gen}}(\cM) = \infty$ then $\sing \cM = \sing_\textrm{gen}\cM$ so we're in case (1) of Theorem \ref{theo:generic.flow.R3} with $M'=M$ (i.e.\ no perturbation is necessary).  If $T_{\textrm{gen}}(\cM) < \infty$, then either it holds that all $(\bx,T_{\textnormal{gen}}(\cM)) \in \cS_\textnormal{first}(\cM)$ have some tangent flow with multiplicity $k=1$ or some $(\bx,T_{\textnormal{gen}}(\cM)) \in \cS_\textnormal{first}(\cM)$ has only tangent flows with multiplicity $k\geq 2$. In the latter case, we're in case (2) of Theorem \ref{theo:generic.flow.R3} (again with $M'=M$). 

It thus remains to consider the case where $T_{\textrm{gen}}(\cM) < \infty$ and at all $(\bx,T_{\textnormal{gen}}(\cM)) \in\fS_\textnormal{first}(\cM)$, some tangent flows occur with multiplicity-one. By Proposition \ref{prop:properties.until.nongeneric.time} (the ``...Finally,'' bullet points) it holds that at all $(\bx,T_{\textnormal{gen}}(\cM)) \in\fS_\textnormal{first}(\cM)$, all tangent flows occur with multiplicity-one. 

Note also that by Proposition \ref{prop:properties.until.nongeneric.time} we see that $\genus_{T_\textrm{gen}}(M) \geq 1$. Using Proposition \ref{prop:generic.flow.R3.iterate} we can thus find $M'$ a (small) normal graph over $M$ with
\[
\genus_{T_\textrm{gen}}(M') \leq \genus_{T_\textrm{gen}}(M) - 1. 
\]
Now, by the argument in the previous paragraph if $M'$ fails to satisfy the desired properties in Theorem \ref{theo:generic.flow.R3} for some (and thus any) cyclic unit-regular integral Brakke flow $\cM'$ with $\cM'(0) = \cH^2\lfloor M'$ it holds that $T_{\textrm{gen}}(\cM') < \infty$ and at all $(\bx,T_{\textnormal{gen}}(\cM')) \in \fS_\textnormal{first}(\cM)$, all tangent flows occur with multiplicity one. We can thus apply  Proposition \ref{prop:generic.flow.R3.iterate} with $M'$ in place of $M$. Note that at each step the genus at the non-generic time must strictly decrease. As such, the perturbations must eventually satisfy the asserted properties in Theorem \ref{theo:generic.flow.R3}. (By choosing all perturbations sufficiently small, we can arrange that the final surface is a small graph over $M$.) This completes the proof.
\end{proof}

The remainder of this section is thus devoted to the proof of Proposition \ref{prop:generic.flow.R3.iterate}. Below we fix a cyclic integral unit-regular Brakke flow $\cM$ in $\RR^3$ with $\cM(0) = \cH^2\lfloor M$ for a closed embedded surface $M \subset \RR^3$.  Assume that $T_\textnormal{gen}(\cM) < \infty$ and at any $(\bx,T_{\textnormal{gen}}(\cM)) \in \fS_\textnormal{first}(\cM)$, all tangent flows occur with multiplicity one. 

Consider $M_j$ is a sequence of normal graphs of functions $0<u_j$ over $M$ with $u_j\to 0$ in $C^\infty(M)$ and fix cyclic integral unit-regular Brakke flows $\cM_j$ with $\cM_j(0) = \cH^2\lfloor M_j$. Using  Lemma \ref{lemm:first-generic-time.unique.BF} it suffices to assume (after passing to a subsequence and replacing $\cM$ by the limiting Brakke flow) that $\cM_j \rightharpoonup \cM$. 

Define
\begin{equation}\label{eq:R3-defn-dj}
d_j : = \max_{(\bx,T_{\textrm{gen}}(\cM)) \in \fS_\textrm{first}(\cM)} d((\bx,T_{\textrm{gen}}(\cM)), \supp\cM_j)
\end{equation}
(using the parabolic distance). Note that $0<d_j =o(1)$. (Lemma \ref{lemm:first-generic-time.points} implies that $\fS_\textrm{first}(\cM)$ is a finite set.)

Before proving Proposition \ref{prop:generic.flow.R3.iterate} we first show that the first non-generic time does not significantly decrease under the perturbation. 

\begin{lemma}\label{lemm:Tgen-est-one-sidedR3}
For $M,M_j,\cM,\cM_j$ as above we have $T_\textnormal{gen}(\cM_j) \geq T_\textnormal{gen}(M) - o(d_j^2)$ as $j\to\infty$. 
\end{lemma}
\begin{proof}
Passing to a subsequence, we can assume for contradiction that there is some $\eta>0$ so that 
\[
T_\textnormal{gen}(\cM_j) \leq T_\textnormal{gen}(M) - \eta d_j^2
\]
By stability of cylindrical (generic) singularities (this follows from upper-semicontinuity of density and the final bullet point in Proposition \ref{prop:properties.until.nongeneric.time}) we see that in fact
\begin{equation}\label{eq:gen-time-not-sig-decrease-contr}
T_\textnormal{gen}(M) - o(1) \leq T_\textnormal{gen}(\cM_j) \leq T_\textnormal{gen}(M) - \eta d_j^2
\end{equation}
as $j\to\infty$. By translating in space-time and using \eqref{eq:gen-time-not-sig-decrease-contr} we can assume that $T_\textrm{gen}(\cM) = 0$ and that there is $(\bx_j,T_j) \in \fS_\textrm{first}(\cM_j)$ with $T_j = T_\textrm{gen}(\cM_j)$ so that $(\bx_j,T_j)\to (\bOh,0) \in \fS_\textrm{first}(\cM)$ but
\begin{equation}\label{eq:gen-time-not-sig-decrease-contr-2}
T_j \leq - \eta d_j^2. 
\end{equation}
Now, rescale $\cM,\cM_j,(\bx_j,T_j)$ around the space-time origin to $\tilde \cM_j,\check \cM_j,(\check\bx_j,\check T_j)$ so that $|(\check \bx_j,\check T_j)| = 1$. Passing to a further subsequence, we assume that $\tilde \cM_j \rightharpoonup \tilde\cM$, $\check \cM_j \rightharpoonup \check \cM$, $(\check\bx_j,\check T_j) \to (\check\bx_0,\check T_0)$. Note that $\check T_0 \leq 0$ and $\tilde\cM(t) = \cH^2\lfloor\sqrt{-t}\Sigma$ for some $\Sigma\in \cS_2''$.

We claim that is $\check \cM$ either $\check\cM(t) = \cH^2\lfloor \sqrt{-t}\Sigma$ for all $t<0$ or $\check \cM$ is (up to a parabolic dilation around the origin) the ancient flow constructed in Proposition  \ref{prop:exist-ancient-rescaled-MCF-C2-rescaling-argument}. We will verify this claim below but first will show how it completes the proof. By upper-semicontinuity of density and the final bullet point in Proposition \ref{prop:properties.until.nongeneric.time} we see that $\Theta_{\check\cM}(\check \bx_0,\check T_0) \geq \lambda_1 + \delta_0$.  Thus, the first case ($\check\cM(t) = \cH^2\lfloor \sqrt{-t}\Sigma$) cannot occur, since $\check T_0 \leq 0$, so we find that $\check\cM$ is a parabolic dilation of the ancient flow constructed  in Proposition  \ref{prop:exist-ancient-rescaled-MCF-C2-rescaling-argument}. By (4) in Proposition \ref{coro:summary-tleq0-non-rescaled} we see that $\check T_0 = 0$. By the scaling invariance of \eqref{eq:gen-time-not-sig-decrease-contr-2}, $\check T_j \to 0$ and the definition of $d_j$ we see that 
$$ d((\bOh,0),\supp \check \cM_j) \to 0$$
and thus $d((\bOh,0),\supp\check \cM) = 0$. This contradicts (1) in Proposition \ref{theo:basic-prop-cK-cM}.

We now prove the claim about $\check \cM$. We can choose an open set $\Omega$ so that $\partial\Omega = \Sigma$ and $\supp\check\cM(t) \subset \sqrt{-t} \overline\Omega$ (to see that $\check \cM$ is contained to one side of $\sqrt{-t}\Sigma$ we can use the fact that the level set flow of the inside $U_j$ of $M_j$ contains the level set flow $U$, the inside of $M$). We claim that \begin{equation}\label{eq:R3-sing-time-lemm-ent-bd}
\lambda(\check \cM) \leq F(\Sigma).
\end{equation}
 Indeed, we can choose $(\by_j,s_j) \in \RR^3\times \RR$ with $(\by_j,s_j) \to (\bOh,0)$ and $r_j\to 0$ and 
 $$\Theta_{\cM_j}((\by_j,s_j),r_j) = \lambda(\check\cM) + o(1)\, .$$ 
 For $r>0$ fixed, we can take $j$ large and apply Huisken's monotonicity to yield 
\[
\Theta_{\cM_j}((\by_j,s_j),r) \geq \Theta_{\cM_j}((\by_j,s_j),r_j) = \lambda(\check\cM) + o(1).
\]
Sending $j\to\infty$, we find
\[
\Theta_\cM((\bOh,0),r) \geq \lambda(\check\cM). 
\]
Then, sending $r\to 0$, we obtain \eqref{eq:R3-sing-time-lemm-ent-bd}. 

Since $\supp\check\cM(t) \subset \sqrt{-t} \overline\Omega$ we find that any tangent flow to $\check\cM$ at $-\infty$ is associated to an integral unit-regular varifold shrinker ($F$-stationary varifold) $V$ with $\supp V\subset \overline \Omega$. The Frankel property for shrinkers (Corollary \ref{cor:frank.shrink}) implies that $\supp V \cap \Sigma \not = \emptyset$. Then the strong maximum principle for stationary varifolds \cite{SolomonWhite,Ilmanen:maximum} (either applies since $\Sigma$ is smooth) gives $V = m \cH^2\lfloor \Sigma$ for some $m \in \NN$. The entropy bound \eqref{eq:R3-sing-time-lemm-ent-bd} then implies that $m=1$. 

We can thus apply the global uniqueness result Theorem \ref{thm:uniqueness-global} to see that either $\check\cM(t) = \cH^2\lfloor \sqrt{-t}\Sigma$ for all $t<0$ or $\check \cM$ is (up to a parabolic dilation around the origin) the ancient flow constructed in Proposition  \ref{prop:exist-ancient-rescaled-MCF-C2-rescaling-argument}. This proves the claim and thus completes the proof. 
\end{proof}

We now complete the proof of Theorem \ref{theo:generic.flow.R3} by establishing Proposition \ref{prop:generic.flow.R3.iterate}. 
\begin{proof}[Proof of Proposition \ref{prop:generic.flow.R3.iterate}]
Recall that $M\subset \RR^3$ is a closed embedded surface and $M_j$ are the normal graphs of functions $0<u_j$ with $u_j\to0$ in $C^\infty(M)$. Write $U_j$ for the unique compact region bounded by $M$ and let $U_j(t)$ denote the level set flow of $U_j$. Write $\cU_j$ for the space-time track of $U_j(t)$.

We have fixed cyclic integral unit-regular Brakke flows $\cM_j$ with $\cM_j(0) = \cH^2\lfloor M_j$ and so that $\cM_j\rightharpoonup \cM$ a cyclic integral unit-regular Brakke flow with $\cM(0) = \cH^2\lfloor M$ and so that all tangent flows to $\cM$ at time $T_\textrm{gen}(\cM)$ occur with multiplicity one. 

By Lemma \ref{lemm:first-generic-time.points}, the set $\fS_\textrm{first}(\cM)$ is finite. Thus, by passing to a subsequence and recalling the definition of $d_j$ in \eqref{eq:R3-defn-dj} we can translate $\cM,\cM_j$ by the same fixed space-time vector to assume that $T_\textrm{gen}(\cM) = 0$ and
\[
d_j = d((\bOh,0),\supp\cM_j)
\]
(i.e., that the maximum in $d_j$ is achieved at the space-time origin). Define $\tilde \cM_j,\check\cM_j$ by parabolically dilating around the space-time origin by $d_j^{-1}$. Exactly as in the proof of Lemma \ref{lemm:Tgen-est-one-sidedR3} we find that (after passing to a subsequence) $\tilde\cM_j$ limits to $\cH^2\lfloor\sqrt{-t}\Sigma$ and $\check \cM_j \rightharpoonup \check\cM$ a parabolic dilation of the ancient flow constructed in Proposition  \ref{prop:exist-ancient-rescaled-MCF-C2-rescaling-argument}. (Note that the other possibility considered in the proof of Lemma \ref{lemm:Tgen-est-one-sidedR3}, namely $\check\cM = \cH^2 \lfloor \sqrt{-t}\Sigma$ cannot occur since $d((\bOh,0),\supp\check\cM_j) = 1$ by the choice of rescaling.) 

Because the ancient flow $\check \cM = \cH^n \lfloor \partial\cK(t)$ is a parabolic dilation of the flow considered in Section \ref{sec:genus-drop-ancient}, we can apply Proposition \ref{prop:genus-drop-one-sided-flow} to find $t_0 < t_1 < 0$ and $\kappa_0 \in (0,t_1-t_0),\kappa_1 \in (0,-t_1)$ and $R>0$ so that 
\begin{enumerate}
\item for any $t \in [t_{0},t_{0}+\kappa_{0}]$ so that $\partial\cK(t)$ is smooth, $\genus(\partial\cK(t) \cap B_{R}(\bOh)) > 0$,
\item  for any $t \in [t_{1},t_{1}+\kappa_{1}]$ so that $\partial\cK(t)$ is smooth, $\genus(\partial\cK(t) \cap B_{R}(\bOh)) = 0$,
\item  $t\mapsto \partial\cK(t)$ is a simple flow in $\{R/2 <  |\bx| < 3R\} \times [t_0,t_1+\kappa_{1}]$,
\item  and
for $(\bx,t) \in \partial\cK(t) \cap \{R/2 < |\bx| < 3R\}\times  [t_0,t_1+\kappa_{1}]$, we have $|\bx^T| \geq \tfrac 12 |\bx|$, for $\bx^T$ the projection of $\bx$ to $T_\bx\partial\cK(t)$.  
\end{enumerate}
(Note that the various constants will be changed by the parabolic dilation, but once this is taken into account they will be fixed with the above properties and will not depend on $j$. This is all we will need later.) Note that by Lemma \ref{lemm:Tgen-est-one-sidedR3} and choice of translation/rescaling we have $T_\textrm{gen}(\check\cM_j) \geq -o(1)$. In particular, for $j$ sufficiently large, $T_\textrm{gen}(\check\cM_j) \geq t_1 + \kappa_1$. 

By unit-regularity and (3) and (4) above, we conclude that for $j$ sufficiently large,
\begin{equation}\label{eq:checkM_j-simple-flow}
\textrm{$\check\cM_j$ is a simple flow in $\{3R/4<|\bx| < 2R\}\times [t_0 + \kappa_0/2,t_1+\kappa_1/2]$.}
\end{equation}
Furthermore, by (1) (and (6) in Proposition \ref{coro:summary-tleq0-non-rescaled} combined with unit-regularity)  there is $t_{0,j} \in [t_0 + \kappa_0/2,t_0+3\kappa_0/4] \cap \cT_\textrm{reg}(\check\cM_j)$ with 
\begin{equation}\label{eq:R3genusdrop.genus.inside}
\genus(\check M_j(t_{0,j}) \cap B_R(\bOh)) > 0.
\end{equation}
Similarly, by (2) there is $t_{1,j} \in [t_1 + \kappa_1/2,t_1+3\kappa_1/4] \cap \cT_\textrm{reg}(\check\cM_j)$ with 
\begin{equation}\label{eq:R3genusdrop.genus.inside.dies}
\genus(\check M_j(t_{1,j}) \cap B_R(\bOh)) = 0.
\end{equation}

Discarding finitely many terms we can ensure (thanks to $\cM_j\rightharpoonup\cM$ and the monotonicity of genus) that for all $j$, it holds that 
\[
g_{0,j} : = \genus({\check M_j(t_{0,j})}) \leq \genus_{T_\textrm{gen}}(M). 
\]
Defining $g_{1,j} : = \genus(\check M_j(t_{1,j}))$, we claim that
\begin{equation}\label{eq:genus-ineq-R3-in-proof}
 g_{1,j} \leq g_{0,j} - 1
\end{equation}
for $j$ sufficiently large. Thanks to (1)-(4) above, the proof of \eqref{eq:genus-ineq-R3-in-proof} is now completed exactly as in \cite[Proposition 11.4]{CCMS:generic1} by using the localized topological monotonicity results proven in \cite[Appendix G]{CCMS:generic1}. 

For the reader's convenience we recall the proof here. If \eqref{eq:genus-ineq-R3-in-proof} fails, we can take $j$ sufficiently large so that \eqref{eq:R3genusdrop.genus.inside} and \eqref{eq:R3genusdrop.genus.inside.dies} hold but 
\[
g_{0,j} = g_{1,j} : = g
\]
For this $j$ fixed, we write $W_j[t]:=(\check U_j(t))^c$ and $W_j[a,b] : = \check \cU_j^c \cap \ft^{-1}([a,b])$ (where $\check \cU_j$ is the space-time track of the level set flow of $U_j$ translated in space-time and then parabolically rescaled by $d_j^{-1}$ as above and $\check U_j(t)$ is the time $t$ slice.)

Fix $B=B_R(\bOh)$ and recall \eqref{eq:checkM_j-simple-flow}. By Lemmas \ref{lemm:H1.subset.R3.bdry} and \ref{lemma:genus.loops.inside.first} and \eqref{eq:R3genusdrop.genus.inside.dies} we can choose loops 
\[
\gamma_1^{(1)},\dots, \gamma_g^{(1)} \subset W_j[t_{1,j}]
\]
so that 
\[
\{[\gamma_1^{(1)}],\dots, [\gamma_g^{(1)}]\} \subset H_1(W_j[t_{1,j}],\ZZ)
\]
is linearly independent and each $\gamma_j^{(1)}$ satisfies one of the following:
\begin{itemize}
\item $\gamma_j^{(1)}$ is contained in $\overline B^c$ , or
\item there is some component $\cB_j^{(1)}$ of $\partial B\cap W_j[t_{1,j}]$ that has non-zero mod-2 intersection with $\gamma_j^{(1)}$ and zero mod-2 intersection with each previous $\gamma_i^{(1)}$, $i<j$. 
\end{itemize}
(No $\gamma_j^{(1)}$ is contained in $B$, thanks to \eqref{eq:R3genusdrop.genus.inside.dies}.) By \cite[Theorem 6.2]{White:topology-weak} we have that 
\[
\{[\gamma_1^{(1)}],\dots, [\gamma_g^{(1)}]\} \subset H_1(W_j[t_{0,j},t_{1,j}],\ZZ)
\]
is linearly independent. We now choose loops $\gamma_1^{(0)},\dots,\gamma_g^{(0)} \subset W_j[t_{0,j}]$ so that
\begin{itemize}
\item Each $\gamma_j^{(1)}$ is homotopic to $\gamma_j^{(0)}$ in $W_j[t_{0,j},t_{1,j}]$; see \cite[Theorem 5.4]{White:topology-weak}.
\item If $\gamma^{(1)}_j$ is contained in $\overline B^c$ then so is $\gamma^{(0)}_j$ (in fact, so is the entire homotopy between them); see Theorem \ref{theo:homotope-time-zero-MCF-complement-simple-flow}. 
\item If $\gamma^{(1)}_j$ is not entirely contained in $\overline B^c$ then there is some component $\cB^{(0)}_j$ of $\partial B\cap W_j[t_{0,j}]$ that has non-zero mod-2 intersection with $\gamma_j^{(0)}$ and zero mod-2 intersection with each previous $\gamma_i^{(0)}$, $i<j$; this follows from simplicity of the flow near $\partial B$ and the fact that mod-2 intersection is preserved under homotopy. 
\end{itemize}
By Lemma \ref{lemma:genus.loops.inside.second} we have that 
\[
\genus(\check M_j(t_{0,j})\cap B_R(\bOh)) = 0. 
\]
This contradicts \eqref{eq:R3genusdrop.genus.inside}, completing the proof of \eqref{eq:genus-ineq-R3-in-proof} and thus the proposition. 
\end{proof}

\appendix

\part*{Appendices}

\section{Graphs over shrinkers} \label{app:graph-shrinker}

Recall the conventions for the second fundamental form and mean curvature as described in Section \ref{subsec:convent}. Suppose that $\Sigma^n\subset \RR^{n+1}$ is a smooth shrinker (the estimates here are local so $\Sigma$ could just be a small piece of a shrinker) and consider for $I \subset \RR$ an interval of time $u : \Sigma \times I \to \RR$ so that
\[
|u||A| < \eta < 1
\]
along $\Sigma\times I$, where $A$ is the second fundamental form of $\Sigma$. This allows us to define the graph
\[
\Gamma_\tau  : =\{ \bx + u(\bx,\tau) \nu_\Sigma(\bx): \bx \in \Sigma\}. 
\]
We compute here various geometric quantities associated to $\Gamma_\tau$. Define
\begin{equation}\label{app:graphical-eqns-defn-v}
v(\bx,\tau) = (1+|(\Id - u S_\Sigma)^{-1}(\nabla_\Sigma u)|^2)^{\frac 12}. 
\end{equation}
\begin{lemma}
The upwards pointing normal along $\Gamma_\tau$ is
\begin{equation}\label{eq:app-unit-normal-graph}
\nu_\Gamma = v^{-1}(-(\Id - u S_\Sigma)^{-1}\nabla_\Sigma u + \nu_\Sigma).
\end{equation}
In particular
\begin{equation}\label{eq:app-defi-vdotv-v}
v = (\nu_\Sigma \cdot \nu_\Gamma)^{-1}. 
\end{equation}
\end{lemma}
\begin{proof}
Note that $\bZ \in T_{\bx} \Sigma$ if and only if 
\[
(\Id - uS_\Sigma) \bZ  + du(\bZ) \nu_\Sigma( \bx)  \in T_{\bx + u(\bx,\tau) \nu_\Sigma(\bx)}\Gamma_\tau. 
\]
Since the shape operator is self-adjoint, we thus find that the expression for $\nu_\Gamma$ is orthogonal to $\Gamma_\tau$. The assertion thus follows from the definition of $v$. 
\end{proof}
For $\ell \in \{0,1\}$ set
\begin{equation}\label{eq:defn-sigma-x-app}
\sigma_{\ell}(\bx,\tau) : = \sum_{j=0}^\ell \sum_{k=0}^{4-2j} |\partial_\tau^j\nabla^k u(\bx,\tau)|. 
\end{equation}
We now study the form of the error terms in the linearization of the rescaled mean curvature flow equation. 
\begin{lemma}\label{lemm:expand-H-app}
The mean curvature of $\Gamma_\tau$ at $\bx + u(\bx,\tau) \nu_\Sigma(\bx)$ satisfies
\begin{equation}
v(\bx,\tau) H_\Gamma(\bx + u(\bx,\tau) \nu_\Sigma(\bx),\tau) = H_\Sigma(\bx) +  (\Delta_\Sigma u + |A_\Sigma |^2 u)(\bx) + E^H
\end{equation}
where the error $E^H$ can be decomposed into terms of the form 
\[
E^H = u E^H_1  + E^H_2(\nabla_\Sigma u,\nabla_\Sigma u) 
\]
where $E^H_1 \in C^\infty(\Sigma)$ and $E^H_2 \in C^\infty(\Sigma ; T^* \Sigma \otimes T^*\Sigma)$ satisfy the following estimates: 
\[
|\partial_\tau  E_1^H(\bx,\tau)| \leq C_1^H\sigma_1(\bx,\tau),\qquad  \sum_{k=0}^2 |\nabla^k_\Sigma E_1^H(\bx,\tau)|  \leq C_1^H \sigma_0(\bx,\tau)
\]
and\footnote{recall that $E_2$ is a section of $T^*\Sigma\otimes T^*\Sigma$ so e.g., $\nabla_\Sigma E_2$ is a section of $T^*\Sigma\otimes T^*\Sigma \otimes T^*\Sigma$}
\[
|\partial_\tau  E_2^H(\bx,t)| \leq C_2^H (1+\sigma_1(\bx,\tau)),\qquad \sum_{k=0}^2 |\nabla^k E_2^H(\bx,t)|  \leq C_2^H (1+ \sigma_0(\bx,\tau))
\]
where $C_1^H,C_2^H$ depend only on $\eta$ and an upper bound for $\sum_{k=0}^3 |\nabla^k A|(\bx)$. 
\end{lemma}
\begin{proof}
The area formula \cite[\S 2.5]{Simon:GMT} implies that for $\tau$ fixed, $\Psi(\bx) = \bx + u(\bx,\tau)\nu_\Sigma(\bx)$ satisfies
\[
\Psi^* d\mu_\Gamma = \sqrt{\det \cG } \, d\mu_\Sigma
\]
where $\cG: T_{\bx}\Sigma \to T_{\bx}\Sigma$ is the linear map
\[
\cG = (\Id - u S)^2 + \nabla_\Sigma u \otimes du. 
\]
It is straightforward to check that $\cG$ is invertible as long as $|u| |A|< 1$. 

Set $u_s = u+s\varphi$ for $\varphi\in C^\infty_c(\Sigma)$ arbitrary. Then (writing $\cG_s$ for $\cG$ with $u$ replaced by $u+s\varphi$) the first variation formula yields
\[
\frac{d}{ds}\Big|_{s=0} \int_\Sigma \sqrt{\det \cG_s } \, d\mu_\Sigma = - \int_\Sigma H_\Gamma (\nu_\Sigma \cdot \nu_\Gamma) \varphi \sqrt{\det \cG} \, d\mu_\Sigma.
\]
On the other hand, we compute 
\[
\frac{d}{ds}\Big|_{s=0} \cG_s = -2  (\Id - u S)  S \varphi+ \nabla \varphi \otimes du + \nabla u \otimes d\varphi,
\]
so
\begin{align*}
\frac{d}{ds}\Big|_{s=0} \sqrt{\det \cG_s} & = \tfrac 12 \tr \left( \cG^{-1} \tfrac{d}{ds}\Big|_{s=0} \cG_s \right)  \sqrt{\det \cG}\\
& = \tr \left( \cG^{-1} \left( \nabla u \otimes d\varphi - (S - u S^2)  \varphi \right)\right)  \sqrt{\det \cG}.
\end{align*} 
Integrating by parts, we thus find 
\begin{equation}\label{eq:mean-curv-graph-div-form}
v H_\Gamma   = \frac{v^2}{\sqrt{\det\cG}} \Div (\sqrt{\det\cG} \, \cG^{-1} \nabla u) + v^2 ( \tr \cG^{-1} S -   \tr \cG^{-1} S^2 u ).
\end{equation}
where we emphasize that all quantities (on the right hand side) are taken along $\Sigma$. 

We now simplify this expression. For notational simplicity, below we will write $f \approx g$ if $f=g + u e_1 + e_2(\nabla u,\nabla u)$ for $e_1,e_2$ satisfying the estimates for $E_1^H,E_2^H$ in the statement of the lemma (i.e., terms we will absorb in to the error terms). Note that $v\approx 1$ and if we abuse the notation slightly to allow for tensor valued expressions, $\cG \approx \Id -2uS$ so $\cG^{-1}\approx \Id + 2uS$. We compute
\begin{align*}
 \tr \cG^{-1} S -   \tr \cG^{-1} S^2 u & =  \tr  \left((1+2u S) S \right) - \tr \left( S^2 u\right) \\
& +  \tr  \left( (\cG^{-1} - \Id - 2u S) S  -  (\cG^{-1} - \Id) S^2 u \right)\\
& =  H + |A|^2 u\\
& +  \tr  \left( (\cG^{-1} - \Id - 2u S_\Sigma) S  -  (\cG^{-1} - \Id) S^2 u \right),
\end{align*}
which implies that
\begin{align*}
v^2 \left( \tr \cG^{-1} S -   \tr \cG^{-1} S^2 u\right) & \approx  H + |A|^2 u. 
\end{align*}
We now consider the divergence term in \eqref{eq:mean-curv-graph-div-form}. Write 
\[
\frac{v^2}{\sqrt{\det\cG}} \Div (\sqrt{\det\cG} \, \cG^{-1} \nabla u) \approx \Div(\cG^{-1}\nabla u) + \frac 12 \nabla \log \det \cG \cdot \nabla u
\]
Since $\nabla u$ enters $\cG$ quadratically, we easily see that $ \nabla \log \det \cG \cdot \nabla u \approx 0$. Finally, we have
\begin{align*}
\Div(\cG^{-1}\nabla u) & \approx \Div \left( (\Id + \nabla u \otimes du)^{-1} \nabla u\right) \\
& = \Div \left( \left( (\Id - \frac{ \nabla u \otimes du}{1+|\nabla u|^2}  \right) \nabla u\right)\\
& = \Div \left( \frac{\nabla u }{1+|\nabla u|^2} \right) \approx \Delta u .
\end{align*}
This completes the proof. 
\end{proof}
\begin{lemma}\label{lemm:expand-xdotnu-app}
The support function along $\Gamma_\tau$ satisfies
\begin{equation}
v(\bx,\tau) \left( \bx_\Gamma \cdot \nu_\Gamma \right) = \bx \cdot \nu_\Sigma + u(\bx,\tau) - \bx^T \cdot \nabla u + u E^{\bx\cdot \nu}
\end{equation}
at $ \bx_\Gamma = \bx + u(\bx,\tau)  \nu_\Sigma(\bx)$, where $E^{\bx\cdot \nu}$ satisfies
\[
|\partial_\tau  E^{\bx\cdot\nu}(\bx,\tau)| \leq C \sigma_1(\bx,\tau) , \qquad  \sum_{k=0}^2 |\nabla^k E^{\bx\cdot\nu}(\bx,\tau)|  \leq C \sigma_0(\bx,\tau)
\]
for $C$ depending only on $\eta$ and an upper bound\footnote{in particular $C$ does \emph{not} depend on $|\bx|$} for $\sum_{k=0}^3 |\nabla^k A|(\bx)$. 
\end{lemma}
\begin{proof}
Differentiating the shrinker equation $H + \tfrac 12 \bx \cdot \nu_\Sigma = 0$ we find
\[
S(\bx^T) =  2 \nabla H . 
\]
Hence, using \eqref{eq:app-unit-normal-graph} we compute
\begin{align*}
v(\bx,\tau) \left( \bx_\Gamma \cdot \nu_\Gamma \right) & = \bx \cdot \nu_\Sigma  + u - \bx \cdot (\Id - u S)^{-1} \nabla u\\
&  = \bx \cdot \nu_\Sigma  + u - \bx\cdot \nabla u - u \sum_{k=1}^\infty u^{k-1} \bx \cdot S^k(\nabla u)\\
&  = \bx \cdot \nu_\Sigma  + u - \bx\cdot \nabla u - u \sum_{k=0}^\infty u^{k} S(\bx^T) \cdot S^k(\nabla u)\\
&  = \bx \cdot \nu_\Sigma  + u - \bx^T \cdot \nabla u - 2 u \sum_{k=0}^\infty u^{k} \nabla H \cdot S^k(\nabla u).
\end{align*}
This completes the proof. 
\end{proof}
\begin{corollary}\label{coro:expand-rescaled-mcf-app}
We have
\begin{align*}
 v(\bx,\tau) (\partial_{\tau}\bx_{\Gamma} - \bH -\tfrac 12 \bx_{\Gamma}) \cdot \nu_{\Gamma} & = \partial_{\tau}u  - {\big (\underbrace{\Delta u - \tfrac 12 \bx^T \cdot \nabla u + ( |A|^2 + \tfrac 12) u}_{=L u} \big) }+ E 
\end{align*}
at $\bx_{\Gamma} = \bx + u(\bx,\tau)\nu_{\Gamma}(\bx)$ for  $E=uE_1 + E_2(\nabla u,\nabla u)$ for $E_1,E_2$ satisfying
\[
|\partial_\tau  E_1(\bx,\tau)| \leq C_1\sigma_1(\bx,\tau), \qquad \sum_{k=0}^2 |\nabla^k E_1(\bx,\tau)|  \leq C_1\sigma_0(\bx,\tau)
\]
and
\[
|\partial_\tau  E_2(\bx,\tau)| \leq C_2(1+\sigma_1(\bx,\tau)),\qquad \sum_{k=0}^2 |\nabla^k E_2(\bx,\tau)|  \leq C_2 (1+ \sigma_0(\bx,\tau))
\]
where $C_1,C_2$ depend only on $\eta$ and an upper bound for $\sum_{k=0}^3 |\nabla^k_\Sigma A|(\bx)$. 

In particular, if $\Gamma_\tau$ is a solution to rescaled mean curvature flow, i.e.,
\[
(\partial_\tau \bx)^\perp = \bH + \tfrac 12\bx^\perp
\]
then 
\begin{equation}\label{eq:error-term-linearize}
\partial_\tau u = Lu + E
\end{equation}
for $E$ as above. 
\end{corollary}
\begin{proof}
Recall that our convention is that $\bH = H \nu_{\Gamma} $. Thus, when combined with Lemmas \ref{lemm:expand-H-app} and \ref{lemm:expand-xdotnu-app}, we compute
\begin{align*}
& (\partial_{\tau}\bx_{\Gamma} - \bH - \tfrac 12 \bx_{\Gamma}) \cdot \nu_{\Gamma} \\
& = (\nu_{\Sigma}\cdot \nu_{\Gamma}) \partial_{\tau}u - H - \tfrac 12 \bx_{\Gamma}\cdot\nu_{\Gamma} \\
& =  - \underbrace{( H_{\Sigma}(\bx) + \tfrac 12 \bx\cdot \nu_{\Sigma})}_{=0}  + \partial_{\tau}u - (\Delta_{\Sigma} u + |A_{\Sigma}|^{2}u - \tfrac 12 \bx^{T}\cdot \nabla u + \tfrac 12 u  ) + E^{\bx\cdot\nu} - E^{H} . 
\end{align*}
Combined with the estimates from Lemmas \ref{lemm:expand-H-app} and \ref{lemm:expand-xdotnu-app}, this completes the proof.  
\end{proof}

Finally, we recall the improved estimates that hold along an asymptotically conical end $\Sigma_\textrm{con}$, proven in \cite[Lemma 3.6]{CCMS:generic1} .
\begin{lemma}\label{lemm:conical-end-decomp-error-lin}
Along an asymptotically conical end $\Sigma = \Sigma_\textnormal{con}$ there is $\eta,C$ depending on $\Sigma$ so that if
\[
|\bx|^{-1} |u(\bx)| + |\nabla u(\bx)| + |\bx| |\nabla^2u| \leq \eta 
\]
then the error term $E$ in \eqref{eq:error-term-linearize} decomposes as
\begin{multline} \label{eq:linearized.equation.error.decomposition.con}
		E(u)(\mathbf{x}) = u(\mathbf{x}) E_1(\mathbf{x}, u(\mathbf{x}), \nabla u(\mathbf{x}), \nabla^2 u(\mathbf{x})) \\
		+  \nabla u(\mathbf{x}) \cdot \mathbf{E}_2(\mathbf{x}, u(\mathbf{x}), \nabla u(\mathbf{x}), \nabla^2 u(\mathbf{x})),
	\end{multline}
	where $E_1$, $\mathbf{E}_2$ are smooth functions on the following domains
\begin{align*}
	&E_1(\mathbf{x}, \cdot, \cdot, \cdot) : \RR \times T_{\mathbf{x}} \Sigma \times \operatorname{Sym}(T_{\mathbf{x}} \Sigma \otimes T_{\mathbf{x}} \Sigma) \to \RR,\\
	&\mathbf{E}_2(\mathbf{x}, \cdot, \cdot, \cdot) : \RR \times T_{\mathbf{x}} \Sigma \times \operatorname{Sym}(T_{\mathbf{x}} \Sigma \otimes T_{\mathbf{x}} \Sigma) \to T_{\mathbf{x}} \Sigma.
\end{align*}
Moreover, we can estimate:
	\begin{align*} 
		& r(\mathbf{x}) ^{2+j-\ell} |\nabla_{\mathbf{x}}^i \nabla_z^j \nabla_{\mathbf{q}}^k \nabla_{\mathbf{A}}^\ell E_1(\mathbf{x}, z, \mathbf{q}, \mathbf{A})| \nonumber \\
		& \qquad \qquad \qquad \leq C( r(\mathbf{x})^{-1} |z| + |\mathbf{q}| + r(\mathbf{x}) |\mathbf{A}|)^{\max\{0, 1-j-k-\ell\}}, \\ 
		& r(\mathbf{x}) ^{1+j-\ell} |\nabla_{\mathbf{x}}^i \nabla_z^j \nabla_{\mathbf{q}}^k \nabla_{\mathbf{A}}^\ell \mathbf{E}_2(\mathbf{x}, z, \mathbf{q}, \mathbf{A})| \nonumber \\
		& \qquad \qquad \qquad \leq C( r(\mathbf{x})^{-1} |z| + |\mathbf{q}| + r(\mathbf{x}) |\mathbf{A}|)^{ \max\{0, 1-j-k-\ell\}} .
	\end{align*}
	In the above, $C = C(\Sigma,i,j,k,\ell)$, and $i,j,k,\ell \geq 0$.
\end{lemma}

Along an asymptotically cylindrical end the estimates can be extended as follows.

\begin{lemma}\label{lemm:cylindrical-end-decomp-error-lin}
Along the core $\Sigma \cap B_R(\bOh)$ and every asymptotically cylindrical end $\Sigma = \Sigma_\textnormal{cyl}$ there is $\eta,C$ depending on $\Sigma$ so that if $ \|u\|_{C^2} \leq \eta$ as well as $|u| |A| \leq \eta$ then the error term $E$ in \eqref{eq:error-term-linearize} decomposes as
\begin{multline} \label{eq:linearized.equation.error.decomposition.cyl}
		E(u)(\mathbf{x}) = u(\mathbf{x}) E_1(\mathbf{x}, u(\mathbf{x}), \nabla u(\mathbf{x}), \nabla^2 u(\mathbf{x})) \\
		+  \nabla u(\mathbf{x}) \cdot \mathbf{E}_2(\mathbf{x}, u(\mathbf{x}), \nabla u(\mathbf{x}), \nabla^2 u(\mathbf{x})),
	\end{multline}
	where $E_1$, $\mathbf{E}_2$ are smooth functions as in Lemma \ref{lemm:conical-end-decomp-error-lin}, where we can estimate:
	\begin{align*} 
		&  |\nabla_{\mathbf{x}}^i \nabla_z^j \nabla_{\mathbf{q}}^k \nabla_{\mathbf{A}}^\ell E_1(\mathbf{x}, z, \mathbf{q}, \mathbf{A})|  \leq C( |z| + |\mathbf{q}| + |\mathbf{A}|)^{\max\{0, 1-j-k-\ell\}}, \\ 
		& |\nabla_{\mathbf{x}}^i \nabla_z^j \nabla_{\mathbf{q}}^k \nabla_{\mathbf{A}}^\ell \mathbf{E}_2(\mathbf{x}, z, \mathbf{q}, \mathbf{A})|  \leq C(  |z| + |\mathbf{q}| + |\mathbf{A}|)^{ \max\{0, 1-j-k-\ell\}} .
	\end{align*}
	In the above, $C = C(\Sigma,i,j,k,\ell)$, and $i,j,k,\ell \geq 0$.
\end{lemma}
\begin{proof}
 As in the proof of \cite[Lemma 3.6]{CCMS:generic1} one can split the error into $E^H(u)$ and $E^{\bx \cdot \nu}(u)$. The estimate for the $E^H(u)$ term is as in \cite[Lemma 3.6]{CCMS:generic1}, whereas for the $E^{\bx \cdot \nu}(u)$ term we can argue as in the proof of Lemma \ref{lemm:expand-xdotnu-app}.
\end{proof}

We also need to understand the linearization of the shrinker mean curvature of two graphs over $\Sigma$, relative to each other and relative to the base $\Sigma$.

\begin{lemma} \label{lemm:relative-shrinker-mean-curvature} For $\delta < \tfrac 12 (\sup_\Omega |A_\Sigma|)^{-1}$ and  $u_i \in C^\infty(\Sigma), i = 0,1$ with 
\begin{equation}\label{eq:C2-bound}
 |u_i(\bx)|+ |\nabla_\Sigma u_i(\bx)| + |\nabla^2_\Sigma u_i(\bx) | + |\nabla^3_\Sigma u_i(\bx) | \leq \delta
\end{equation}
for all $\bx \in \Sigma$. Letting
 $$\Gamma_i:= \{ \bx + u_i(\bx) \nu_\Sigma(\bx): \bx \in \Sigma\}$$
 then denoting $w =u_1-u_0$ and $v_i:= (\nu_{\Gamma_i} \cdot \nu_\Sigma)^{-1}$ there exists $C=C(\sup_{\Sigma} |A_{\Sigma}| + |\nabla A_{\Sigma}|+ |\nabla^2 A_{\Sigma}|)$ such that
 \begin{equation*}
  v_1\Big(H_{\Sigma_1} + \tfrac 12 \bx_{\Sigma_1} \cdot \nu_{\Sigma_1}\Big) - v_0\Big(H_{\Sigma_0}+ \tfrac 12\bx_{\Sigma_0}\cdot \nu_{\Sigma_0}\Big)  
  = L_\Sigma w + E^w
 \end{equation*}
 where $\bx_{\Sigma_i} = \bx + u_i(\bx)\nu_\Sigma(\bx)$ and the error term $E$ satisfies
 $$ E^w(\bx) = w(\bx) F(\bx) + \nabla w(\bx) \cdot \bF(\bx) + \nabla^2 w(\bx) \cdot \mathcal{F}(\bx) \, ,$$
 with the estimate
 $$|F| + |\bF| + |\cF| + |\nabla \cF| \leq C\delta$$
 for all $\bx \in \Sigma$. Furthermore, under the assumption that $\Sigma$ is a properly embedded shrinker, it holds
\begin{equation} \label{eq:relative-estimate-ibp-hits-gaussian}\left|\int_\Sigma w (\bx^T\otimes \nabla_\Sigma w) \cdot \cF\, \rho\, d\cH^n\right| \leq C \delta \|w\|^2_{W,1}\, .\end{equation}
 In the case $\Sigma_\textnormal{con} \subset \Sigma$, all the above statements remain the same if along the conical ends $\Sigma_\textnormal{con}$ one replaces \eqref{eq:C2-bound} with the weaker assumption
 $$ \|u_i\|_{3,\alpha;\Sigma_\textnormal{con}}^{(1)} \leq \delta$$
 for $i=1,2$.
 \end{lemma}

\begin{proof} The proof follows similarly as the proof of \cite[Corollary 5.2]{CCMS:generic1}. Denoting 
$$w: =u_1-u_2,\  E^w:= E(u_1)-E(u_2)$$ 
we have
$$ \big(\tfrac{\partial}{\partial \tau} - L \big) w = E^w\, .$$
Using Lemma \ref{lemm:cylindrical-end-decomp-error-lin} and the fundamental theorem of calculus, we obtain as in \cite[(5.11)]{CCMS:generic1}
 \begin{equation} \label{eq:uniqueness.w.error}
 \begin{split}
		E^w 	& = w E_1(\cdot, u_2, \nabla_\Sigma u_2, \nabla_\Sigma^2 u_2) + \nabla_\Sigma w \cdot \mathbf{E}_2(\cdot, u_2, \nabla_\Sigma u_2, \nabla_\Sigma^2 u_2)  \\
			& \quad + \Big[ u_1 \int_0^1 D_z E_1(\cdots) \, dt \Big] w  + \Big[ u_1 \int_0^1 D_{\mathbf{q}} E_1(\cdots) \, dt \Big] \cdot \nabla_\Sigma w  \\
			& \quad + \Big[ u_1 \int_0^1 D_{\mathbf{A}} E_1(\cdots) \, dt \Big] \cdot \nabla^2_\Sigma w + \Big[ \nabla_\Sigma u_1 \cdot \int_0^1 D_z \mathbf{E}_2(\cdots) \, dt \Big] w  \\
			& \quad + \Big[ \nabla_\Sigma u_1 \cdot \int_0^1 D_{\mathbf{q}} \mathbf{E}_2(\cdots) \, dt \Big] \cdot \nabla_\Sigma w\\
			&\quad  + \Big[ \nabla_\Sigma u_1 \cdot \int_0^1 D_{\mathbf{A}} \mathbf{E}_2(\cdots) \, dt \Big] \cdot \nabla^2_\Sigma w,
			\end{split}
	\end{equation}
where, in all six instances, $\cdots$ stands for $(\cdot, u_2 + tw, \nabla_\Sigma u_2 + t \nabla_\Sigma w, \nabla_\Sigma^2 u_2 + t \nabla^2_\Sigma w)$. We note that we can thus write
$$ E^w= w F + \nabla_\Sigma w \cdot \mathbf{F} + \nabla^2_\Sigma w \cdot \mathcal{F}\, .$$
The estimates for $F,\mathbf{F},\mathcal{F},\nabla\mathcal{F}$ then follow from Lemma \ref{lemm:cylindrical-end-decomp-error-lin}.

In case $\Sigma$ is a properly embedded shrinker, recall that Ecker's Sobolev inequality \cite{Ecker:Sobolev} (cf.\ \cite[Proposition 3.9]{ChodoshSchulze}) implies that for $f \in H^1_W(\Sigma)$
$$ \| |\bx| f  \|^2_W \leq 4n \| f\|_{W,1}^2\, ,$$
and we can thus estimate
$$ \left| \int_\Sigma w\, (\bx^T \otimes \nabla_\Sigma w) \cdot \mathcal{F} \,e^{-\tfrac14|\bx|^{2}} d\mathcal{H}^n \right| \leq C \delta (\| |\bx|w\|_W \|\nabla_\Sigma w\|_W) \leq C \delta \| w\|_{W,1}^2 \, .$$
The claim for the conical ends $ \|u_i\|_{2,\alpha;\Sigma_\text{con}}^{(1)} \leq \delta$ for $i=1,2$, follow similarly, (compare
\cite[(5.11)]{CCMS:generic1}), using Lemma \ref{lemm:conical-end-decomp-error-lin} instead of  Lemma \ref{lemm:cylindrical-end-decomp-error-lin}.
\end{proof}

\section{Positive graphs over shrinkers} \label{app:graph-shrinker-pos}
We continue the analysis in Appendix \ref{app:graph-shrinker} and consider $u:\Sigma\times I \to \RR$ (where we continue to allow for $\Sigma$ to be possibly incomplete) so that the graph $\Gamma_\tau$ is flowing by rescaled mean curvature flow. We continue to require that
\[
|u| |A| \leq \eta < 1
\]
on $\Sigma\times I$. In this appendix, we additionally require that the function $u$ is positive for all $(\bx,t) \in \Sigma \times I$. We set $v = \log u$. Note that
\begin{align*}
\partial_\tau v & = \frac{\partial_\tau u}{u}\\
\nabla v & = \frac{\nabla u}{u}\\
\nabla^2 v & = \frac{\nabla^2 u}{u} - d v \otimes dv. 
\end{align*}
Recall the quantities $\sigma_0$, $\sigma_1$ defined in \eqref{eq:defn-sigma-x-app} and  that Corollary \ref{coro:expand-rescaled-mcf-app} implies that  
\begin{equation}\label{eq:app_posgraphs.1}
\partial_\tau u = L u + E
\end{equation}
where $E=uE_1 + E_2(\nabla u,\nabla u)$ satisfy the estimates stated above. Below we show that these estimates yield estimates for $u^{-1}E$. 
\begin{proposition} \label{prop:app_posgraphs.2}
For $(\bx,\tau) \in \Sigma\times I$, assume that $\sigma_1(\bx,\tau) \leq 1$. There is $C$ depending only on $\eta$ and an upper bound for $\sum_{k=0}^3 |\nabla^k A|(\bx)$ so that
\begin{align*}
|(u^{-1} E)| &\leq C \sigma_0 (1+ |\nabla v|) \\
|\nabla v| |\nabla(u^{-1} E)| & \leq C \sigma_0 (1 +  |\nabla v|^3)\\
|\partial_\tau(u^{-1}E)| & \leq C \sigma_1 (1 +  |\nabla v|^3)
\end{align*}
at $(\bx,\tau)$. 
\end{proposition}
\begin{proof}
It suffices to only consider $E_2$. Writing $C$ for a constant that will change from line to line but only depends on $\eta$ and an upper bound for $\sum_{k=0}^3 |\nabla^k A|(\bx)$, we have:
\begin{align*}
|u^{-1}E_2(\nabla u,\nabla u)| & \leq C (1+ \sigma_0) |\nabla v| |\nabla u| \leq C \sigma_0 |\nabla v|.
\end{align*}
For the second estimate, we have 
\begin{align*}
|\nabla v| |\nabla (u^{-1}E_2(\nabla u,\nabla u))| & \leq C |\nabla v|(u^{-2} |\nabla u|^3 + u^{-1} |\nabla u|^{2} + u^{-1}|\nabla^2 u||\nabla u|)\\
& \leq C \sigma_0 (|\nabla v|^3 + |\nabla v|^2).
\end{align*}
Similarly,
\begin{align*}
|\partial_\tau (u^{-1}E_2(\nabla u,\nabla u))| & \leq C (u^{-2} |\partial_\tau u| |\nabla u|^2 + u^{-1}|\nabla u|^2 + u^{-1}|\nabla u| |\partial_\tau \nabla u|) \\
& \leq C \sigma_1 (|\nabla v|^2 + |\nabla v|).
\end{align*}
Clearly both of these estimates imply the asserted bound. 
\end{proof}
From \eqref{eq:app_posgraphs.1} we see that $v$ satisfies
\begin{equation}\label{eq:app_posgraphs.2}
\partial_\tau v = \Delta v + |\nabla v|^2 - \tfrac{1}{2} \bx^T \cdot \nabla v  +\tfrac{1}{2} + |A|^2 + u^{-1} E\, .
 \end{equation}
 We denote the corresponding linearized operator by $\cN$, defined via
 $$ \cN w: = \Delta w + 2 \nabla v \cdot \nabla w -\tfrac{1}{2} \bx^T \cdot \nabla w\, .$$
 We extend $\cN$ to act on tensors in the usual way, i.e.,
\[
\cN T : = \sum_{i=1}^n \nabla^2_{\bE_i,\bE_i} T + 2\nabla_{\nabla v} T - \frac 12 \nabla_{\bx^T} T. 
\]
for $\bE_1,\dots,\bE_n$ an orthonormal basis of $T_pM$. 
\begin{proposition}\label{prop:app_posgraphs.3}
Denoting $v_\tau$ by the $\tau$-derivative of $v$, we have 
\begin{align}
\partial_\tau v_\tau &=  \cN v_\tau + \partial_\tau(u^{-1}E)\, , \label{eq:app_posgraphs.3}\\
\partial_\tau \nabla_\Sigma v & =  \cN \nabla v  - \tfrac 12 \nabla v + S^2(\nabla v) + \nabla |A|^2 + \nabla(u^{-1}E) 
 \, , \label{eq:app_posgraphs.4}\\
\partial_\tau|\nabla v|^2 & = \cN |\nabla v|^2 - 2 |\nabla^2 v|^2 - |\nabla v|^2 +  2 \nabla v \cdot S^2(\nabla v) + 2 \nabla v \cdot \nabla |A|^2 \label{eq:app_posgraphs.5} \\
&\quad + 2
\nabla v \cdot \nabla(u^{-1}E) \nonumber
\end{align}
\end{proposition} 
\begin{proof}
Equation \eqref{eq:app_posgraphs.3} follows directly from differentiating \eqref{eq:app_posgraphs.2} with respect to $\tau$. To prove \eqref{eq:app_posgraphs.4}, we consider the gradient of \eqref{eq:app_posgraphs.2}. We first note that (recalling the conventions \eqref{eq:convent.shape.oper} and \eqref{eq:convent.sff})
\[
\nabla (\bx^T \cdot \nabla v) = \nabla v + \nabla_{\bx^T} \nabla v -2 H S(\nabla v),
\]
so
\begin{align*}
\partial_\tau \nabla v = \nabla \Delta v + 2\nabla_{\nabla v}\nabla v - \tfrac 12 \nabla_{\bx^T} \nabla v - \tfrac 12 \nabla v +  H S(\nabla v)  + \nabla |A|^2 + \nabla(u^{-1}E).
\end{align*}
Using the commutator identities \eqref{eq:app.3rd.der.12.slot.perm} and \eqref{eq:app.3rd.der.23.slot.perm} proven below, we find 
\[
\nabla \Delta v = \sum_{i=1}^n \nabla^2_{\bE_i,\bE_i} \nabla v - \Ric(\nabla v,\cdot)^\sharp . 
\]
On the other hand, the (traced) Gauss equations \eqref{eq:Gauss.convention.traced} yield 
\[
\Ric(\nabla v,\cdot)^\sharp = H S(\nabla v) - S^2(\nabla v),
\]
so putting these expressions together, \eqref{eq:app_posgraphs.4} follows. Combining \eqref{eq:app_posgraphs.4} with 
\[
\cN(\nabla v\otimes \nabla v) = (\cN\nabla v) \otimes \nabla v + \nabla v \otimes \cN \nabla v + 2 \sum_{i=1}^n (\nabla_{\bE_i} \nabla v) \otimes (\nabla_{\bE_i} \nabla v)
\]
and taking the trace yields \eqref{eq:app_posgraphs.5}. \end{proof}

\begin{lemma}
For a function $f$ on a Riemannian manifold $(M^n,g)$ 
\begin{align}
\nabla^3_{\bX,\bY,\bZ} f & = \nabla^3_{\bY,\bX,\bZ} f + R(\bX,\bY,\bZ,\nabla f) \label{eq:app.3rd.der.12.slot.perm}\\
\nabla^3_{\bX,\bY,\bZ} f & = \nabla^3_{\bX,\bZ,\bY} f \label{eq:app.3rd.der.23.slot.perm}
\end{align}
\end{lemma}
\begin{proof}
We can assume that $\bX,\bY,\bZ$ are parallel at the point under consideration and compute
\begin{align*}
\nabla^3_{\bX,\bY,\bZ} f & = \nabla_\bX \nabla^2_{\bY,\bZ} f \\
& = \nabla_\bX \nabla_\bY \nabla_\bZ f - \nabla_\bX \nabla_{\nabla_\bY\bZ} f\\
& = \nabla_\bY \nabla_\bX \nabla_\bZ f -  \nabla_{\nabla_\bX\nabla_\bY\bZ} f\\
& = \nabla_\bY \nabla^2_{\bX,\bZ} f + \nabla_\bY \nabla_{\nabla_\bX\bZ} f -  \nabla_{\nabla_\bX\nabla_\bY\bZ} f\\
& =   \nabla^3_{\bY,\bX,\bZ} f -  \nabla_{\nabla_\bX\nabla_\bY\bZ- \nabla_\bY\nabla_\bX\bZ} f\\
& = \nabla^3_{\bY,\bX,\bZ} f + R(\bX,\bY,\bZ,\nabla f). 
\end{align*}
This proves \eqref{eq:app.3rd.der.12.slot.perm}. The first line of the calculation also clearly implies \eqref{eq:app.3rd.der.23.slot.perm}. 
\end{proof}

 Combining Proposition \ref{prop:app_posgraphs.2} with Proposition \ref{prop:app_posgraphs.3} yields
 
\begin{proposition}\label{prop:app_evolution_eq}
For $(\bx,\tau) \in \Sigma\times I$, assume that $\sigma_1(\bx,\tau) \leq 1$. There is $C$ depending only on $\eta$ and an upper bound for $\sum_{k=0}^3 |\nabla^k A|(\bx)$ so that $w:= v_\tau - \tfrac12 |\nabla v|^2$ satisfies
\[ w_\tau \geq \cN w+ |\nabla^2 v|^2 + \tfrac 12  |\nabla v|^2 -  |\nabla v|^2 |A|^2 - |\nabla v| |\nabla|A|^2| - C \sigma_1 (1+ |\nabla v|^3)\, . \]
\end{proposition}

\medskip 

\subsection{Li--Yau estimate}
 In the following we assume that $\Sigma \in \cS_n'$ (i.e., $\Sigma$ has nice ends). We define
 \begin{equation}\label{eq:app_Li-Yau.1}
\Sigma' = \Sigma \cap B_{2 R}(\bOh) 
 \end{equation}
 where $R$ is assumed to satisfy the following three conditions
 \begin{equation}\label{eq:app_Li-Yau.2}
R \geq \max \{10^{10}n, \sup_{\Sigma} |A|, \sup_\Sigma  |\nabla |A|^2| \},
 \end{equation}
 as well as 
 \begin{equation}\label{eq:app_Li-Yau.2.1}
0 \leq \inf_{r \in [9nR^2,\infty)} ( \tfrac{5}{24n}r^2-10^{-8}n^{-1} r^{\frac{3}{2}}-2),
 \end{equation}
 and \eqref{eq:app_Li-Yau.3} below.
 
 We introduce a cut-off function $\gamma$ such that $\gamma(r) = 0$ for $r\geq 4$, $\gamma(r) = 1$ for $r\leq 1, -2 \leq \gamma'(r) \leq 0$ and $|\gamma''| \leq 2$ for all $r$. We let $\eta(\bx):= \gamma(|\bx|^2R^{-2})$. Since $\Sigma$ has nice ends we can assume that $R$ is sufficiently large such that
 \begin{equation}\label{eq:app_Li-Yau.3}
|\Delta_\Sigma \eta | \leq R^{-1}\, ,
 \end{equation}
 where $C=C(\sup_\Sigma |A|^2)$.
 \begin{proposition}[Li--Yau type estimate]\label{prop:Li-Yau}
Assume $u:\Sigma'\times (-\infty, 0) \to (0,\infty)$ is such that the graph $\Gamma_\tau$ is flowing by rescaled mean curvature flow and that $\sup_{\Sigma'} \sigma_1(\cdot, \tau) \to 0$ as $\tau \to -\infty$.  
Then there is $\tau_0 \ll -1$ such that
\[
|\nabla \log u|^{2} \leq 20n R^2 + 2 \partial_{\tau} \log u
\]
holds for $(\bx,\tau) \in B_R(0)\cap \Sigma \times (-\infty,\tau_0]$. 
\end{proposition}
\begin{proof} We fix $\tau_0$ sufficiently negative below.  Given $T\leq \tau_0$, we define
\[ W(\bx,\tau) = w(\bx, \tau) \eta^2(\bx) + h(\tau)\, , \]
where $w=v_\tau - \tfrac12 |\nabla v|^2$ and
\[ h(\tau) = 10 n (\tau-T + R^{-2})^{-1}\, .\]
From Proposition \ref{prop:app_evolution_eq} we see that
\begin{equation*}
\begin{split}
 \partial_\tau W &\geq \cN W - \tfrac{1}{10n} h^2 + \eta^2 |\nabla^2 v|^2 - \eta^2  |\nabla v|^2|A|^2 - \eta^2 |\nabla v| |\nabla |A|^2|\\
 &\quad - C \sigma_1 \eta^2(1+|\nabla v|^3) - w \Delta \eta^2 - 4 \eta \nabla w \cdot \nabla \eta - 4\eta w \nabla v \cdot \nabla \eta\\
 &\quad  + w\eta\,  \bx^T \cdot \nabla \eta\, .
\end{split} 
\end{equation*}
 Towards a contradiction, we assume that there exists a point $(\bp_1,\tau_1)$ such that $\tau_1\leq T$, $W(\bp_1,\tau_1)=0$, and $W(\cdot,\tau)>0$ holds for  $T-R^2 < \tau<\tau_1$
 . Since $\nabla W(\bp_1,\tau_1)=0$, we have $\eta \nabla w=-2w \nabla \eta$ at $(\bp_1,\tau_1)$. In addition, we have $W_\tau-\mathcal{N} W\leq 0$ at $(p_1,\tau_1)$. Moreover, $W=0$ yields $w\leq 0$ at $(\bp_1,\tau_1)$ and we have $\bx^T\cdot \nabla \eta \leq 0$ at $(\bp_1,\tau_1)$ by the definition of $\eta$. Hence at $(\bp_1,\tau_1)$,
\begin{equation*}
\begin{split}
0&\geq   -\tfrac{1}{10n} h^2+\eta^2|\nabla^2 v|^2- \eta^2|A|^2|\nabla v|^2 - \eta^2 |\nabla v| |\nabla |A|^2|\\
&\quad  -C\sigma_1 \eta^2(1+ |\nabla  v|^3)-2 w\eta \Delta \eta +6w |\nabla \eta|^2 -4\eta w\nabla v \cdot \nabla \eta.
\end{split}
\end{equation*}
By using the definition of $R,\eta$ and using $|\nabla^2 v|^2 \geq n^{-1} |\Delta v|^2$ this implies 
\begin{equation}\label{eq:app_Li-Yau.4}
\begin{split}
0&\geq   -\tfrac{1}{10n} h^2+\tfrac{1}{n}\eta^2|\Delta v|^2- \eta^2R^2|\nabla v|^2 - \eta^2 R |\nabla v| -C\sigma_1 \eta^2(1+ |\nabla  v|^3)\\
& \quad- 20 R^{-1}|w|(1+\eta|\nabla v|)
\end{split}
\end{equation}
at $(\bp_1,\tau_1)$.  

On the other hand, by using \eqref{eq:app_posgraphs.2} and $|\bx|\leq 2R$ on $B_{2R}(\bOh)$ together with Proposition \ref{prop:app_posgraphs.2}, we have
\[
w \geq \Delta v+\tfrac12 |\nabla v|^2-(R+C\sigma_0) |\nabla v|+\tfrac12+|A|^2 - C\sigma_0,
\]
for some $C$ depending on $\sup_{B_{2R}(\bOh)} \sum_{k=0}^3 |\nabla^k A|$. By choosing $\tau_0$ sufficiently negative, we can arrange that $C\sigma_0 \leq \min\{\tfrac 12, \tfrac 13 R\}$ (note that the constant $C$ here depends only on $\Sigma$, since it comes in through the applications of Propositions \ref{prop:app_posgraphs.2} and Proposition \ref{prop:app_evolution_eq}), we find that
\begin{equation*}
w \geq \Delta v+\tfrac12 |\nabla v|^2-\tfrac43 R |\nabla v|\geq \Delta v+  \tfrac14  |\nabla v|^2- 2R^2,
\end{equation*}
for $\tau \leq \tau_0$. Since $w\leq 0$ at $(\bp_1,\tau_1)$,  we see that at this point 
\begin{equation*}
-w \geq -w\eta^2= h \geq 10nR^2
\end{equation*}
and thus  (note that $n\geq 2$)
\begin{equation}\label{eq:Laplace_v-bound}
 \tfrac14 |\nabla v|^2+\tfrac{9}{10}|w|\leq -\Delta v,
\end{equation}
at $(\bp_1,\tau_1)$. This yields
\begin{equation*}
\tfrac18 \eta^2 |\Delta v|^2\geq \tfrac18 \left(\tfrac{9}{10} \right)^2w^2\eta^2\geq \tfrac{1}{10} w^2\eta^2\geq \tfrac{1}{10}h^2.
\end{equation*}
Combining with \eqref{eq:app_Li-Yau.4} implies 
\begin{equation*}
\begin{split}
0&\geq    \tfrac{7}{8n}\eta^2|\Delta v|^2-4 R^2\eta^2|\Delta v| - 2R\eta^2 |\Delta v|^{\tfrac{1}{2}}\\
&\quad -C\sigma_0 \eta^2(|\Delta  v|^{\frac{3}{2}}+1)-30R^{-1}|\Delta v|(1+2\eta|\Delta v|^{\frac{1}{2}}).
\end{split}
\end{equation*}
We can choose $\tau_0$ sufficiently negative so that we have
\begin{align*}
0\geq    \tfrac{7}{8n}\eta^2|\Delta v|^2-10^{-8}n^{-1}\eta |\Delta v|^{\frac{3}{2}}- 6  R^2 |\Delta v|-2.
\end{align*}
On the other hand, by using \eqref{eq:Laplace_v-bound} we derive
\begin{equation*}
 \eta^2 |\Delta v| \geq \tfrac{9}{10} |w|\eta^2 =\tfrac{9}{10}h\geq 9nR^2 .
\end{equation*}
Combining these expressions we derive a contradiction to \eqref{eq:app_Li-Yau.2.1}. In conclusion, we have $W \geq 0$ for $\tau \leq T$. Therefore, given $T \leq \tau_0$
\begin{equation*}
0\leq W(x,T)=w(x,T)+h(T)=w(x,T)+10nR^2,
\end{equation*}
holds in $B_R(0)\cap \Sigma$. 
\end{proof}
 
\section{Integral Brakke $\bX$-flows} \label{app:integral-brakke-X-flows}
We record here the notion of Brakke flows with a transport term as introduced by Hershkovits--White in \cite{HershkovitsWhite:set-theoretic}. An ($n$-dimensional) \emph{integral Brakke $\bX$-flow} in $\RR^{n+1}$ (or more generally a Riemannian manifold) is a $1$-parameter family of Radon measures $(\mu(t))_{t\in I}$ over an interval $I\subset \RR$ so that:
\begin{enumerate}
\item For almost every $t \in I$ there is an integral $n$-dimensional integral varifold $V(t)$ with $\mu(t) = \mu_{V(t)}$ so that $V(t)$ has locally bounded first variation and has mean curvature $\bH$ orthogonal to $\textrm{Tan}(V(t),\cdot)$ almost everywhere. 
\item For a bounded interval $[t_1,t_2] \subset I$ and compact set $K\subset \RR^{n+1}$, it holds
\[
\int_{t_1}^{t_2} \int_K (1+|\bH|^2) d\mu(t) dt < \infty.
\]
\item For $[t_1,t_2]\subset I$ and $f \in C^1_c(\RR^{n+1}\times [t_1,t_2] ; [0,\infty))$ then
\begin{align*}
& \int f(\cdot,t_2) d\mu(t_2) - \int f(\cdot,t_1) d\mu(t_1) \\
& \leq \int_{t_1}^{t_2} \int \left( \tfrac{\partial f}{\partial t} + \nabla^\perp f \cdot (\bH +\bX) - f \bH \cdot (\bH + \bX^\perp)  \right) d\mu(t) dt.
\end{align*}
\end{enumerate}
Convergence/compactness of integral Brakke $\bX$-flows is essentially the same as in the standard $\bX=0$ case. See \cite{HershkovitsWhite:set-theoretic} for further discussion.

\section{Ilmanen's avoidance principle}\label{app:ilmanen.avoidance}
For $R>0$ and $(\bx_0,t_0) \in \RR^{n+1}\times \RR$, we define 
\begin{equation*}
u(\mathbf{x}, t) := (R^{2} - |\mathbf{x}-\mathbf{x}_{0}|^2 - 2n(t-t_{0}))_+ .
\end{equation*}
Let $W \subset \RR^{n+1}\times \RR$ be open and assume that $u(\cdot,t)$ vanishes on $\partial W(t)$ for all $t \in \mathfrak{t}(W)$. For $\mathbf{p}, \mathbf{q} \in W(t)$, define the distance
\begin{equation} \label{eq:ilmanen-distance-function} d_t(\mathbf{p}, \mathbf{q}) := \inf \Big\{ \int_\gamma u(\gamma(s), t)^{-1} \, ds : \gamma \text{ is a curve joining } \mathbf{p}, \mathbf{q} \text{ in } W(t) \Big\}. 
\end{equation}
We use the standard convention that $\inf \emptyset = \infty$. Note that $d_t$ is just the distance in the (complete) conformally Euclidean metric $g_t := u(\cdot,t)^{-2} g_{\RR^{n+1}}$. More generally, we can consider the distance between two closed sets in $W_t$ defined in the usual way.

\begin{theorem}[Ilmanen]\label{theo:ilmanen-avoidance}
Consider two closed weak set flows $\cM$, $\cM'$ in $\mathbb{R}^{n+1}$ and constants satisfying $R>0,\gamma>0$, $a<b<a + \frac{R^{2}-\gamma}{2n}$. Assume that
\[
\cM(t) \cap B_{\sqrt{\gamma + R^{2} - 2n(t-a)}}(\mathbf{x}_{0}) \quad \textrm{and} \quad \cM'(t) \cap B_{\sqrt{\gamma + R^{2} - 2n(t-a)}}(\mathbf{x}_{0})
\]
are disjoint for $t \in [a,b)$. Then, using this choice of $R$ and $\mathbf{x}_{0}$ along with $t_{0}=a$ in the definition of $u$ and $d_t$ above, we have that $t\mapsto d_{t}(\cM(t),\cM'(t))$ is non-decreasing for $t \in [a,b)$ and
\[
\cM(b) \cap \cM'(b) \cap B_{\sqrt{R^{2} - \gamma - 2n(b-a)}}(\mathbf{x}_{0}) = \emptyset.
\]
\end{theorem}

See \cite[Appendix C]{CCMS:generic1} for a proof. 

In particular we recall that a $F$-stationary varifold is the varifold version of self-shrinker. The avoidance principle immediately implies (cf.\ \cite[Corollary C.4]{CCMS:generic1})
\begin{corollary}[Frankel property for self-shrinkers]\label{cor:frank.shrink}
If $V,V'$ are $F$-stationary varifolds then $\supp V \cap \supp V' \not = \emptyset$. 
\end{corollary}

\section{Brakke flows avoid smooth flows with bounded curvature}\label{app:bf-avoid-bd-curv}

Recall that (in $\RR^{n+1}$) a Brakke flow avoids another Brakke flow as long as one of them has compact initial condition (and they are both initially disjoint). See \cite[10.6]{Ilmanen:elliptic}. 

When both flows have non-compact initial data, it is less clear to what extent this holds. We note that the Ecker--Huisken maximum principle can be used when the flows are both smooth and one is graphical over the other one (and in \cite{CCMS:generic1} we observed that one can weaken the smoothness assumption in a compact set by using Ilmanen's avoidance principle). In this appendix we record the following observation of a related but slightly different nature.

\begin{proposition}\label{prop:avoidance-noncompact}
Suppose that $\Omega_t$ is a family of open sets with the property that $[0,T]\ni t \mapsto \partial\Omega_t := M_t$ is a smooth mean curvature flow with uniformly bounded curvature
\begin{equation}\label{eq:bd.curv.flow.unique}
|A_{M_t}| \leq C
\end{equation}
and uniformly controlled area ratios on small scales, i.e.,
\begin{equation}\label{eq:bd.loc.area.ratios}
\limsup_{r\to 0} \sup_{\bx \in \RR^{n+1}, t\in [0,T]} \frac{|M_t \cap B_r(\bx)|}{\omega_n r^n} = 1.
\end{equation}
 If $[0,T]\ni t\mapsto \tilde \mu_t$ is a Brakke flow with $\supp \tilde \mu_0 \cap \overline{\Omega}_0 = \emptyset$, then it holds that $\supp \tilde\mu_t \cap  \overline\Omega_t = \emptyset$ for all $t\in (0,T]$. 
\end{proposition}

The point is that neither of the flows are assumed to have compact support. 

\begin{proof}
For $R>0$ (large) fix a compact set $K_0^R$ as follows. Choose $|\delta_R| < 1/R$ so that $\partial \Omega_0$ is tranverse to $\partial B_{R+\delta_R}(\bOh)$. Then, we can fix $K_0^R \subset \overline{\Omega_0\cap B_{R+\delta_R}(\bOh)}$ so that (i) $K_0 ^R$ does not fatten under the level-set flow and (ii) $\partial K_0^R$ is smooth (iii) $\partial K_0^R \cap B_{R/2}(\bOh)$ is a graph over a portion of $M_0$ with $C^{\lfloor R\rfloor}$-norm $\leq 1/R$ and (iv) the condition \eqref{eq:bd.loc.area.ratios} holds for $\partial K_0^R$ in place of $M_t$, uniformly as $R\to\infty$. 

Let $K^R_t$ denote the level set flow of $K_0^R$. We note that $K_t^R$ is a set-theoretic sub-solution (cf.\ \cite[10.1]{Ilmanen:elliptic}) with compact initial condition and moreover because $\tilde \mu_t$ is a Brakke flow, $\supp \tilde\mu_t$ is a set-theoretic subsolution by \cite[10.5]{Ilmanen:elliptic}. Using the fact that two set-theoretic solutions avoid each-other when one has compact initial conditions by \cite[10.2]{Ilmanen:elliptic}, we find that $\supp \tilde \mu_t  \cap K_0^R = \emptyset$ for all $t \in [0,T]$.

Let $R\to \infty$. We claim that $K^R_t$ converges to $K_t$ for $t \in [0,T]$ in the local Hausdorff sense. In fact, $\partial K_t^R$ converges in the locally smooth graphical sense to $M_t$. To see this, first note that $t\mapsto \mu_t^R:= \cH^n\lfloor \partial^* K_t^R$ is an integral unit-regular Brakke flow by \cite[11.3]{Ilmanen:elliptic} (since $K_0^R$ is non-fattening under the level-set flow by (i) above). Using (iv) and a covering argument we can find a subsequence $R_i\to\infty$ and pass $\mu_t^{R_i}$ to the limit as a Brakke flow $\mu_t$. Note that $\mu_0 = \cH^n \lfloor M_0$. By Lemma \ref{lemm:uniqueness-BF} below, we find that $\mu_t = \cH^n\lfloor M_t$ for all $t \in [0,T]$. (Thus, we did not need to take a subsequence.) Unit regularity of $\mu_t^{R}$ then shows that $\mu_t^{R}$ is actually smooth on compact sets for $R$ large and converges locally smoothly. This yields the asserted convergence of $K_t^R$ to $K_t$.

Suppose that $\tilde\mu_t$ is as in the statement of the proposition but there is $ t\in(0,T]$ and $\bx \in K_t \cap \supp\tilde \mu_t$. Choose  {$\gamma,R_0>0$ so that $R_0^2 - \gamma > 2nT$ and $\bx \in B_{\sqrt{R_0^2 - 2nT -\gamma}}(\bOh)$. By Ilmanen's avoidance principle, Theorem \ref{theo:ilmanen-avoidance}, applied to $K_t^R$ and $\supp\tilde \mu_t$, (with constants $a=0,b=T,R_0,\gamma$)} we have that 
\[
d_t(K_t^R,\supp \tilde\mu_t) \geq d_0(K_t^R,\supp \tilde\mu_0), 
\]
where $d_t$ is as in \eqref{eq:ilmanen-distance-function}.
Note that $d_0(K_t^R,\supp \tilde\mu_0)\geq c>0$ as $R\to\infty$ by (iii) above (and disjointness of $K_0$ and $\supp\tilde\mu_0$). On the other hand, we can take $l_R$ to be a line segment connecting $\bx$ to $K_t^R$. Because $K_t^R$ is converging to $K_t$, we find that
\[
d_t(K_t^R,\supp\tilde\mu_t) \leq \int_{l_R} u(l_R(s),t)^{-1}ds = o(1)
\]
as $R\to\infty$ (with the notation as in Appendix \ref{app:ilmanen.avoidance} above). This is a contradiction, completing the proof. 
\end{proof}

\begin{lemma}\label{lemm:uniqueness-BF}
Suppose that $[0,T]\ni t\mapsto M_t$ is a smooth mean curvature flow with the property that $M_t$ satisfies \eqref{eq:bd.curv.flow.unique} and \eqref{eq:bd.loc.area.ratios}. Then for $[0,T] \ni t\mapsto \mu_t$ an integral unit-regular Brakke flow with $\mu_0 = \cH^n\lfloor M_0$, then it holds that $\mu_t = \cH^n\lfloor M_t$ for all $t\in[0,T]$. 
\end{lemma}
\begin{proof}
By unit-regularity and \eqref{eq:bd.curv.flow.unique}, \eqref{eq:bd.loc.area.ratios}, the set of times $t$ with $\mu_t = \cH^n \lfloor M_t$ is closed. To prove openness, we can note that again by unit-regularity and \eqref{eq:bd.curv.flow.unique}, \eqref{eq:bd.loc.area.ratios}, if $\mu_t = \cH^n \lfloor M_t$ then for $s>0$ small, $\mu_{t+s} = \cH^n \lfloor M_{t+s}'$ for some smooth mean curvature flow $M_{t+s}'$ with bounded curvature. Then, $M_{t+s}'=M_{t+s}$ by uniqueness of smooth mean curvature flows with bounded mean curvature \cite{ChenYin}.  Alternatively, one could argue that $M_{t+s}'$ is a small graph (with zero initial conditions at $s=0$) over $M_{t+s}$ and apply the Ecker--Huisken maximum principle. 
\end{proof}

\section{The topology of regions in $\RR^3$}\label{sec:top-regions-R3}

We collect some useful results concerning the topology of regions in $\RR^3$. Recall the following definitions. 

\begin{definition}
If $M$ is an oriented compact connected surface without boundary, decomposing $M = \cup_{\alpha} M _{\alpha}$ into connected components, then we set
\[
\genus(M) : = \sum_{\alpha} \genus(M_{\alpha}) = \tfrac 12 b_{1}(M). 
\]
\end{definition}
\begin{definition}\label{defi:genus.surface.w.bdry}
If $M^2$ is an oriented compact surface with boundary $\partial M$, the genus of $M$ is defined to be the genus of the oriented surface $\tilde M$ formed by capping off each  circle in $\partial M$ by a disk. 
\end{definition} 

\begin{lemma}\label{lemm:alg.int.pos.genus}
If $M^2$ is an oriented compact surface with boundary so that there are two curves in $M^2$ with non-zero algebraic intersection, then $\genus(M) > 0$. 
\end{lemma}
\begin{proof}
We can assume that $M$ is connected. The capped off surface $\tilde M$ has two curves with non-zero algebraic intersection. These curves cannot be nullhomotopic  in $\tilde M$ (since algebraic intersection is unchanged by homotopy), so $\tilde M$ is not a sphere. 
\end{proof}

The following result is well-known, but we give the proof for completeness. 
\begin{lemma}\label{lemm:H1.subset.R3.bdry}
Suppose that a closed set $\Omega\subset \RR^{3}$ is a $3$-manifold with compact boundary $\partial\Omega$. Set $g = \genus(\partial \Omega)$. Then $H_{1}(\Omega;\ZZ) = \ZZ^{g}$. 
\end{lemma}
\begin{proof}
By compactifying, it suffices to prove that for $\Omega\subset S^{3}$ a compact $3$-submanifold with boundary $\partial\Omega$ of genus $g$, it holds that $H_{1}(\Omega;\ZZ) = \ZZ^{g}$ (adding the point at infinity to $\Omega$ will not change $H_{1}$). Write $\Omega ' = S^{3} \setminus \Omega^{\circ}$. 

 Alexander duality (cf.\ \cite[Theorem 3.44]{Hatcher}) yields
 \[
 H_{1}(\Omega;\ZZ) \oplus H_{1}(\Omega';\ZZ) = H^{1}(\partial\Omega;\ZZ) = \ZZ^{2g}. 
 \]
In particular, $ H_{1}(\Omega;\ZZ), H_{1}(\Omega';\ZZ)$ are free and finitely generated with
\[
\rank  H_{1}(\Omega;\ZZ)+ \rank H_{1}(\Omega';\ZZ) = 2g. 
\]
Thus, using the universal coefficients theorem for homology \cite[Theorem 3A.3]{Hatcher}, it suffices to prove that 
\begin{equation}\label{eq:homol.calc.geq}
\dim H_{1}(\Omega;\QQ),\dim H_{1}(\Omega';\QQ) \geq \genus M,
\end{equation}
We prove \eqref{eq:homol.calc.geq} for $\Omega$ since the argument for $\Omega'$ will be identical.

 We will write $i : \partial \Omega \to \Omega$ for the inclusion map. We now follow \cite[Lemma 3.5]{Hatcher:3mfld} closely. Poincar\'e duality and the long exact sequence for the pair $(\Omega,\partial\Omega)$ yields the commutative diagram
\[
\xymatrix{
H_{2}(\Omega,\partial\Omega;\QQ) \ar[r]^{\partial} \ar@{}[d]|*=0[@]{\approx}  & H_{1}(\partial\Omega;\QQ) \ar[r]^{i_{*}} \ar@{}[d]|*=0[@]{\approx} & H_{1}(\Omega;\QQ) \ar@{}[d]|*=0[@]{\approx}\\
H^{1}(\Omega;\QQ) \ar[r]^{i^{*}}  & H^{1}(\partial\Omega;\QQ) \ar[r]^{\delta}  & H^{2}(\Omega,\partial\Omega;\QQ). 
}
\]
The map $i^{*}$ is the dual map to $i_{*}$ (since we are working over a field). Thus, $\dim \ker i_{*} = \dim \coker i^{*}$. Combined with exactness of the rows in the above diagram, we find
\[
\dim \image \partial = \dim \ker i_{*} = \dim\coker i^{*} = \dim\coker \partial
\]
Thus, 
\[
\dim\ker i_{*} = \dim\image\partial = \tfrac 12 \dim H_{1}(\partial\Omega,\QQ) = g.
\]
Thus,
\[
\dim H_{1}(\Omega;\QQ) \geq \dim \image i_{*} = \dim H_{1}(\partial\Omega;\QQ) - \dim \image i_{*} = g. 
\]
This establishes \eqref{eq:homol.calc.geq} and thus completes the proof.
\end{proof}
The next lemma is a refinement of \cite[Lemma 11.7]{CCMS:generic1} (using Lemma \ref{lemm:H1.subset.R3.bdry} in place of Alexander duality). 
\begin{lemma}\label{lemma:genus.loops.inside.first}
Suppose that a closed set $\Omega\subset \RR^{3}$ is a $3$-manifold with compact boundary $\partial\Omega$.  Set $g = \genus(\partial \Omega)$. Assume that $\partial\Omega$ is transverse to a sphere $\partial B\subset \RR^{3}$. There are loops $\gamma_{1},\dots,\gamma_{g}$ so that $\{[\gamma_{1}],\dots,[\gamma_{g}]\} \subset H_{1}(\Omega;\ZZ)$ is linearly independent and for $i=1,\dots,g$, either
\begin{itemize}
\item $\gamma_{i}$ is contained in $B$ or in $\overline B^{c}$, or
\item there is a component $\cU_{i}$ of $\partial B\cap \Omega$ that has non-zero mod-2 intersection with $\gamma_{i}$ and zero mod-2 intersection with each previous $\gamma_{j}$, $j<i$. 
\end{itemize}
Moreover, we can arrange that exactly $2\genus(S\cap B)$ of the $\gamma_{i}$ are contained entirely in $B$ and if, in $H_{1}(\Omega \cap \overline B;\ZZ)$, 
\[
\sum_{\{i : \gamma_{i}\subset B\}} n_{i}[\gamma_{i}] = [\beta]
\]
for some cycle $\beta \subset \Omega \cap \partial B$, then all of the coefficients $n_{i}\in\ZZ$ vanish. 
\end{lemma}
\begin{proof}
The proof is essentially identical to \cite[Lemma 11.7]{CCMS:generic1} (with Lemma \ref{lemm:H1.subset.R3.bdry} in place of Alexander duality). For the reader's convenience we give the proof below. 

We set $S=\partial\Omega$.  We induct on the number of components $b\geq 0$ of $S \cap \partial B$. 

We begin with the base case $b=0$. Here, the claim follows from Lemma \ref{lemm:H1.subset.R3.bdry}. (To ensure that the loops constructed are contained either in $B$ or $\overline B^c$, we can apply Lemma \ref{lemm:H1.subset.R3.bdry} to $\Omega \cap \overline B$ and $\Omega \setminus B$ separately.) 

For the inductive step, if we consider the $b$ components of $S\cap \partial B$, there is an inner most component (non-unique) $\alpha$ of $S\cap \partial B$ with a disk $D\subset \partial B$ with $\partial D = \alpha$ and $S\cap D^\circ = \emptyset$. We form a new surface $S'$ by removing an annulus $A = U_{\eps/10}(\alpha)\subset S$ and then gluing in two disks that are slight deformations of $D$ in/out of $B$ to cap off the boundary components $S\setminus A$. We can arrange that this occurs in $U_\eps(D)\subset \RR^3$ with $\eps>0$ small enough so that $U_\eps(D)$ is contractible. Now, we can form $\Omega'$ with $\partial\Omega' = S'$ and so that $\Omega'$ agrees with $\Omega$ outside of $U_\eps(D)$. 

Note that $S'$ has $b-1$ components in $S'\cap \partial B$, so we will be able to use the inductive hypothesis. We also note that $\genus(S'\cap B) = \genus(S\cap B)$. 

We consider two cases depending on whether or not $\alpha$ separates the component of $S$ that contains it or not. 

\emph{Separating case.} Suppose $\alpha$ is separating. In this case, it is easy to see that $\genus(S') = \genus(S) = g$. The inductive hypothesis yields a set of loops $\gamma_1,\dots,\gamma_g$ satisfying the claim with $\Omega'$ replacing $\Omega$. As long as $\eps>0$ was chosen small enough above, we can assume that the $\gamma_j$ are contained in $\Omega$. We claim that these curves satisfy the assertion of the lemma (for $\Omega$).  The two bullet points follow immediately from the construction. 

Write $I_B$ for the set of indices with $\gamma_i \subset B$ and similarly for $I_{B^c}$. Suppose that there are integers $n_i$ so that
\[
\sum_{i \in I_B} n_i[\gamma_i] = [\beta]
\]
in $H_1(\Omega \cap \overline B;\ZZ)$ for some $1$-cycle $\beta \subset \Omega \cap \partial B$. Write $\beta = \beta' + \beta''$ where $\beta''$ are the components of $\beta$ intersecting $D$. Note that in fact $\beta''\subset D$ and thus $\beta''$ bounds in $D$. As such, we can reduce to the case where there is a $2$-chain $\sigma \subset  (\Omega \setminus U_\eps(D)) \cap \overline B$ with
\[
\partial\sigma = \beta' - \sum_{i\in I_B} n_i\gamma_i. 
\]
(Indeed, by assumption there is such a $\sigma$ in $\Omega\cap\overline B$. Since $\partial\sigma \cap U_\eps(D) = \emptyset$, we can then push $\sigma$ out of $U_\eps(D)$.) Because $(\Omega \setminus U_\eps(D)) \cap \overline B \subset \Omega' \cap \overline B$, all $n_i$ vanish by the inductive hypothesis. This proves the final assertion in the claim. 

Finally, it remains to check that $\{[\gamma_1],\dots,[\gamma_g]\}\subset H_1(\Omega;\ZZ)$ is linearly independent. Suppose that 
\[
n_1[\gamma_1]+\dots + n_g[\gamma_g] = 0 \textrm{ in $H_1(\Omega;\ZZ)$}.
\]
By construction (specifically the second bullet point) we can consider the intersection with the regions $\cU_i$ (proceeding from large to small indices) to see that $n_i = 0$ for any $\gamma_i$ that intersects $\partial B$. Using Mayer--Vietoris (or transversality) we can find a $1$-cycle $\beta \subset \Omega\cap \partial B$ with 
\[
\sum_{i\in I_B} n_i[\gamma_i] = [\beta] \textrm{ in $H_1(\Omega\cap \overline B;\ZZ)$ and } [\beta] = - \sum_{i\in I_{B^c}} n_i[\gamma_i] \textrm{ in $H_1(\Omega\cap B^c;\ZZ)$}.
\]
As proven above, this yields $n_i=0$ for $i\in I_B$. As above, if we write $\beta=\beta'+\beta''$ for $\beta''$ the components intersecting $D$ (so $\beta''\subset D$) then we find that $[\beta']=0$ in $H_1(\Omega' \cap \overline B;\ZZ)$ and 
\[
\sum_{i\in I_{B^c}} n_i[\gamma_i] = - [\beta']
\]
in $H_1(\Omega'\setminus B;\ZZ)$. Thus,we find that $\sum_{i\in I_{B^c}} n_i[\gamma_i] = 0$ in $H_1(\Omega';\ZZ)$. Thus $n_i=0$ for $i\in I_{B^c}$ by the inductive hypothesis. This completes the proof of linear independence, and thus finishes the consideration of the separating case. 

\emph{Nonseparating case}. If $\alpha$ does not separate the component of $S$ that contains it, we find that $\genus(S') = g-1$. We can use the inductive hypothesis to find $\gamma_1,\dots,\gamma_{g-1} \subset \Omega \cap \Omega'$ satisfying the assertion for $\Omega'$. We now define $\gamma_g$ depending on whether or not $D\subset \Omega$. If $D \subset \Omega$ we can choose a loop in $S$ intersecting $\alpha$ precisely once and then push this loop into $\Omega$ to find $\gamma_g$ a loop in $\Omega$ intersecting $D$ exactly once. If $D \cap \Omega =\emptyset$ then we define $\gamma_g$ to be $\alpha$ pushed into $\Omega$ and out of $B$. The various assertions are checked exactly as above. This completes the proof. 
\end{proof}

The next lemma is an analogue  of \cite[Lemma 11.7]{CCMS:generic1} in this setting. 

\begin{lemma}\label{lemma:genus.loops.inside.second}
Suppose that a closed set $\Omega\subset \RR^{3}$ is a $3$-manifold with compact boundary $\partial\Omega$. Set $g=\genus(\partial\Omega)$. Assume that $\partial\Omega$ is transverse to a sphere $\partial B\subset \RR^3$. Assume that we have $\{[\gamma_1],\dots,[\gamma_g]\}\subset H_1(\Omega;\ZZ) \approx \ZZ^g$ linearly independent and where each $\gamma_i$ satisfies one of the following conditions:
\begin{itemize}
\item $\gamma_i$ is contained in $B$ or $\overline B^c$, or 
\item there is a component $\cU_i$ of $\partial B\setminus S$ that has non-zero mod-2 intersection with $\gamma_i$ and zero mod-2 intersection which each previous $\gamma_j$, $j<i$. 
\end{itemize}
Then, if $\genus(\partial\Omega \cap B) >0$, at least one of the $\gamma_i$ is contained in $B$. 
\end{lemma}
\begin{proof}
Assume this does not hold. Lemma \ref{lemma:genus.loops.inside.first} and the assumption that $\genus(\partial\Omega\cap B)> 0$ yields $\eta \subset \Omega\cap B$ so that $m\eta$ is not homologous in $\Omega$ to a cycle in $\partial B\cap \Omega$ for any $m\in\ZZ\setminus\{0\}$. We claim that 
\[
\{[\gamma_1],\dots,[\gamma_g],[\eta]\} \subset H_1(\Omega;\ZZ) \approx\ZZ^g
\]
is linearly independent, a contradiction. 

If it were linearly dependent, we could write
\[
m [\eta] = \sum_{i=1}^g n_i[\gamma_i]\textrm{ in $H_1(\Omega;\ZZ)$}.
\]
The second bullet point implies that $n_i = 0$ if $\gamma_i$ is not contained entirely in $B^c$. Mayer--Vietoris then implies that $m\eta$ is homologous to a cycle in $\partial B\cap \Omega$, so $m=0$. This yields a linear dependency between the $[\gamma_i]$'s, a contradiction. This completes the proof. 
\end{proof}

\section{Localized topological monotonicity} \label{app:loc-top-mon}

We recall here some results from \cite[Appendix G]{CCMS:generic1}. We begin with the following result. 
\begin{definition}[Simple flow {\cite[Appendix G]{CCMS:generic1}}] For $U\subset \RR^3$ an open set with smooth boundary, a \emph{simple flow} in $U\times I$ is a closed subset of space-time $\cM \subset \RR^{3}\times \RR$ so that there is a compact $2$-manifold with boundary and a continuous map $f: M \times I \to \RR^3$ so that 
\begin{enumerate}
\item $\cM(t) \cap \overline U= f(M\times \{t\})$ (where $\cM(t) = \{\bx \in \RR^3 : (\bx,t) \in \cM\}$) 
\item $f$ is smooth on $M^\circ \times I$ where $M^\circ = M\setminus\partial M$ 
\item $f(\cdot,t)$, $t\in I$ is an embedding of $M^\circ$ into $U$
\item $t\mapsto f(M^\circ\times \{t\})$ is a smooth mean curvature flow
\item $f|_{\partial M\times I}$ is a smooth family of embeddings of $\partial M$ into $\partial U$. 
\end{enumerate}
\end{definition}

We now recall some definitions from \cite{White:topology-weak} and \cite[Appendix G]{CCMS:generic1} (since we will only use these results in $\RR^3$ we do not state them in general dimensions). For $\cM \subset \RR^{3}\times [0,T]$ we set 
\[
W[t] : =\cM^c \cap \ft^{-1}(\{t\}), \qquad W[0,T] : = \cM^c \cap \ft^{-1}([0,T]).
\]
Similarly, for $\Omega\subset \RR^3$ an open set with smooth boundary, we set
\[
W_\Omega[t] : =\cM^c \cap \Omega \cap \ft^{-1}(\{t\}), \qquad W[0,T] : = \cM^c \cap (\Omega\times [0,T]) \cap \ft^{-1}([0,T]).
\]
The following result proven  \cite[Appendix G]{CCMS:generic1} localizes the corresponding result in \cite{White:topology-weak}.

\begin{theorem}\label{theo:homotope-time-zero-MCF-complement-simple-flow}
Let $\cM$ be a level set flow and $\Omega\subset \RR^3$ be an open set with smooth boundary. Assume that $\cM$ is a simple flow in $U\times [0,T]$ for some tubular neighborhood $U$ of $\partial\Omega$. Then any loop in $W_\Omega[0,T]$ is homotopic to one in $W_\Omega[0]$. In particular 
\[
H_1(W_\Omega[0];\ZZ) \to H_1(W_\Omega[0,T];\ZZ)
\]
is surjective. 
\end{theorem}
\section{Some remarks concerning tangent flows in $\RR^3$}\label{app:unique-weak-tf}

Huisken's monotonicity and compactness of Brakke flows implies that we can extract tangent flows to a given Brakke flow around every point in space-time. It is a well-known problem to show that such tangent flows are unique, i.e., independent of the chosen sequence of dilations. (Several uniqueness results are known \cite{Schulze:Loj,ColdingIlmanenMinicozzi,ColdingMinicozzi:uniqueness-tangent-flow,ChodoshSchulze,Zhu:mcvx.loj,Zhu:mcvx.loj.rigid} when the tangent flow is associated to a multiplicity one smooth shrinker.) 

In this appendix we observe that certain soft arguments (along with existing results on the structure and compactness of shrinkers) can be used to prove that if a tangent flow with certain properties exists, then any other tangent flow must have these properties (even if uniqueness is not clear). The arguments used here are inspired by work of Bernstein--Wang \cite{BernsteinWang:high-mult-unique} who proved that if one tangent flow is a shrinking sphere/cylinder with multiplicity, then all tangent flows are (but the cylinder could \emph{a priori} rotate between distinct tangent flows). 

\subsection{Smooth multiplicity-one tangent flows}

The main result here is Proposition \ref{prop:smooth-tf-all-are} which roughly says that in $\RR^3$ a Brakke flow has one tangent flow supported on a smooth shrinker with unit multiplicity then all tangent flows are of this form. Using the idea of \cite{BernsteinWang:high-mult-unique} we will see that the key to proving this is the following lemma.  

\begin{lemma}
Suppose that $V_i$ are $F$-stationary cyclic integral varifolds in $\RR^3$ with $V_i \rightharpoonup \cH^2\lfloor\Sigma$ for $\Sigma \in \cS_2$. Then for $i$ sufficiently large, $V_i = \cH^2\lfloor \Sigma_i$ for some $\Sigma_i \in \cS_2$
\end{lemma}
\begin{proof}
Assume for contradiction that there is $p_i \in \sing V_i$. Allard's theorem guarantees that $|\bp_i| \to \infty$. Because $V_i$ is cyclic and integral, $\Theta_{V_i}(\bp_i) \geq 2$. Thus, there is $0<r_i=o(1)$ so that
\[
\int \rho_{(p_i,-1+r_i^2)}(\bx,-1) d\mu_{V_i} \geq 2-o(1)
\]
as $i\to\infty$ (where $\rho$ is the Gaussian density defined in \eqref{eq:gauss.dens.defn}). 

Let $\cM_i$ be the Brakke flow (for $t<0$) associated to $V_i$ (and similarly $\cM$ the flow associated to $\Sigma$). Note that $\cM_i \rightharpoonup \cM$. Define $\tilde t_i : = -|\bp_i|^{-2} = o(1)$, $\tilde \bp_i : = |\bp_i|^{-1}\bp_i$ and $\tilde r_i := |\bp_i|^{-1}r_i$. Set $\tilde X_i : = (\tilde \bp_i,\tilde t_i + \tilde r_i^2)$. Observe that 
\[
\Theta_{\cM_i}(\tilde X_i,\tilde r_i) = \int \rho_{(\bp_i,-1+r_i^2)}(\bx,-1) d\mu_{V_i} \geq 2-o(1).
\]
Thus, Huisken's monotonicity formula yields $\Theta_{\cM_i}(\tilde X_i,\tilde r_i + r) \geq 2-o(1)$ as $i\to\infty$, for any $r>0$ fixed. We can pass to a subsequence and assume that $\tilde X_i \to \tilde X = (\tilde \bp,0)$, for $|\tilde \bp|=1$. By the convergence $\cM_i\rightharpoonup\cM$ we thus have
$\Theta_\cM(\tilde X,r) \geq 2$ for all $r>0$, so in particular 
\[
\Theta_\cM(\tilde X)\geq 2.
\]
On the other hand, by \cite{Wang:ends-conical} (cf.\ Proposition \ref{prop:ends.shrinkers.R3}), $\Sigma \in \cS_2''$, implying that $\Theta_\cM(\tilde X) \in \{1,\lambda_1\}$. This is a contradiction, completing the proof.  
\end{proof}

We recall the distance on Radon measures $d_V$ defined in e.g.\ \cite[(2.11)]{ColdingIlmanenMinicozzi},
\[
d_V(\mu_1,\mu_2) = \sum_k 2^{-k} \left| \int f_k e^{-\frac 14 |\bx|^2} d\mu_1 - \int f_k e^{-\frac 14 |\bx|^2} d\mu_2  \right|
\]
where $\{f_n\}$ is a fixed countable dense subset of the unit ball in $C^0_c(\RR^3)$. Recall that $d_V$ metrizes the weak* topology on the set of Radon measures $\mu$ with $\int  e^{-\frac 14 |\bx|^2} d\mu < \infty$.  

Write $\mathfrak{F}$ for the set of cyclic integral $F$-stationary varifolds in $\RR^3$. By an abuse of notation we will consider $\cS_2 \subset \mathfrak{F}$ by identifying $\Sigma\in\cS_2$ with $\cH^2\lfloor \Sigma$. We can rewrite the previous lemma as follows:
\begin{corollary}\label{coro:close-smooth-shrink-impl-smooth}
There is $\eps_0>0$ so that if $V \in \mathfrak{F}$ has $d_V(V,\cS_2) \leq \eps_0$ then $V \in \cS_2$. 
\end{corollary}
On the other hand, we recall the following standard result. 
\begin{lemma}\label{lemm:huisk-mon-rescaled-shrinker-close}
Consider $\cN$ a rescaled Brakke flow in $\RR^3$ with entropy $\lambda(\cN) < \infty$. Then $\lim_{\tau \to\infty}F(\cN(\tau)) : = \theta_0$ exists and $\lim_{\tau\to \infty} d_V(\cN(\tau),\mathfrak{F}) = 0 $. Furthermore, for $s_0 > 0$ fixed,
\[
\limsup_{\tau\to \infty} \sup_{|\tau' - \tau|\leq s_0} d_V(\cN(\tau),\cN(\tau')) = 0. 
\]
\end{lemma}
\begin{proof}
This is a direct consequence of Huisken's monotonicity formula for Brakke flows (in the rescaled setting) cf.\ \cite{Ilmanen:singularities}. 
\end{proof}

\begin{proposition}\label{prop:smooth-tf-all-are}
Consider $\cM$ a cyclic integral unit-regular Brakke flow in $\RR^3$. Suppose that some tangent flow to $\cM$ at $X \in \RR^3\times \RR$ satisfies  $t\mapsto \cH^2\lfloor\sqrt{-t}\Sigma$ for $t<0$, where $\Sigma \in \cS_2$ is a smooth self-shrinker. Then, any tangent flow to $\cM$ at $X$ is of the form  $t\mapsto \cH^2\lfloor\sqrt{-t}\Sigma'$ for $t<0$, for some $\Sigma' \in \cS_2$. 
\end{proposition}
\begin{proof}
Let $\cN$ denote the rescaled Brakke flow obtained from $\cM$ with respect to the point $X$. By assumption there is $\tau_i \to\infty$ so that $\cN(\tau_i) \rightharpoonup \cH^2\lfloor \Sigma$ as $i\to\infty$. For $\eps_0>0$ defined in Corollary \ref{coro:close-smooth-shrink-impl-smooth} we consider
\[
\tau_i' : = \sup\{\tau' \geq \tau_i : d_V(\cN(\tau'),\cS_2) \leq \eps_0/2 \textrm{ for all } \tau \in [\tau_i,\tau']\}. 
\]
We consider two cases. First, we suppose that $\tau_i' = \infty$ for some $i$. Consider $s_i\to\infty$ and pass to a subsequence (see Lemma \ref{lemm:huisk-mon-rescaled-shrinker-close}) so that $\cN(s_i)\rightharpoonup V \in \cF$. By assumption, we have $d_V(V,\cS_2) \leq \eps_0$ so $V \in \cS_2$ by Corollary \ref{coro:close-smooth-shrink-impl-smooth}. Thus, the proof is completed in this case. 

It thus suffices to rule out the possibility that $\tau_i' <\infty$ for all $i\to\infty$. First, note that $\tau_i'-\tau_i \to\infty$. Indeed, Lemma  \ref{lemm:huisk-mon-rescaled-shrinker-close} implies that for $s_0>0$ fixed 
\[
\sup_{\tau' \in [\tau_i,\tau_i+s_0]} d_V(\cN(\tau'),\cS_2) \leq d_V(\cN({\tau_i}),\cS_2) + o(1) = o(1)
\]
as $i\to\infty$. Now, we have
\[
\sup_{\tau' \in [\tau_i'-1,{ \tau_i'}+1]} d_V(\cN(\tau'),\cS_2) \leq o(1) + d_V(\cN(\tau_i'-1),\cS_2)
\]
Note that  Lemma \ref{lemm:huisk-mon-rescaled-shrinker-close} implies that after passing to a subsequence, $\cN(\tau_i'-1)$ converges to an element $V$ of $\mathcal{F}$. This element satisfies $d_V(V,\cS_2) \leq \eps_0/2$. Thus by Corollary \ref{coro:close-smooth-shrink-impl-smooth}, we have $V \in \cS_2$. In particular, this yields $d_V(\cN(\tau_i'-1),\cS_2) = o(1)$. Thus 
\[
\sup_{\tau' \in [\tau_i'-1,{\tau_i'}+1]} d_V(\cN(\tau'),\cS_2) = o(1).
\]
This contradicts the definition of $\tau_i'$, completing the proof. 	
\end{proof}

\subsection{Genus and multiplicity at the first non-generic time} The results of this subsection are not explicitly used in this paper, but we include it because it may be of some interest elsewhere. Consider $\cM$ an integral unit-regular Brakke flow with $\cM(0) = \cH^2\lfloor M$ for $M$ a closed embedded surface in $\RR^3$.  We recall here (part of) Definition \ref{defi:generic-time}. 

\begin{definition}[Generic singularities and first non-generic time]\item
\begin{itemize}
\item We define $\sing_\textnormal{gen}\cM\subset \sing\cM$ to be the set of singular points for $\cM$ so that one (and thus all \cite{ColdingMinicozzi:uniqueness-tangent-flow}) tangent flow is a multiplicity-one shrinking sphere or cylinder. 
\item Then, we define $T_\textrm{gen}(\cM) : = \inf \ft(\sing\cM \setminus \sing_\textrm{gen}\cM)$, i.e., the first time that there is a ``non-generic'' singularity. 
\end{itemize}
\end{definition}

Recall that any tangent flow to $\cM$ at $(\bx,T_\textrm{gen}(\cM)) \in \sing \cM$ is of the form 
\[
t\mapsto k \cH^2\lfloor \sqrt{-t}\Sigma
\]
 for some smooth properly embedded self-shrinker $\Sigma \in\cS_2$ with $\genus(\Sigma) \leq \genus(M) : = g_0$. Write $\cS_2(g)\subset \cS_2$ for the set of genus $g$ shrinkers. Continuing with the notation from the previous section, we will write $k\cS_2(g) : = \{k \cH^2\lfloor\Sigma : \Sigma \in \cS_2(g)\} \subset \mathfrak{F}$ for the multiplicity $k$ shrinkers of genus $g$. 
 
\begin{lemma}\label{lemm:gen.mult.fixed.shrink.closed}
The set $k\cS_2(g)$ is closed with respect to $d_V$. 
\end{lemma}
\begin{proof}
Consider $\mu \in \mathfrak{F}$ and $\mu_j = k \cH^2 \lfloor \Sigma_j \in k\cS_2(g)$ with $d_V(\mu_j , \mu)\to 0$. By \cite{ColdingMinicozzi:compactness-shrinkers,SunWang:compact-genus} (cf.\  \cite[Proposition 3.4]{BernsteinWang:high-mult-unique}) we can pass to a subsequence so that $\Sigma_j$ converges in $C^\infty_\textrm{loc}(\RR^3)$ to $\Sigma \in \cS_2(g)$ with multiplicity one. Thus, the weak* limit of $\mu_j = k \cH^2 \lfloor \Sigma_j$ is $k\cH^2 \lfloor \Sigma=\mu$. This completes the proof. 
\end{proof}

Let $\cN$ denote the rescaled Brakke flow obtained by rescaling $\cM$ around $X = (\bx,T_\textrm{gen}(\cM))$. We record the following addendum to Lemma \ref{lemm:huisk-mon-rescaled-shrinker-close}. Set $\mathfrak{F}^{k_0,g_0}:=\cup_{k=1}^{k_0} \cup_{g=0}^{g_0}k\cS_2$. 
\begin{lemma}\label{lemm:huisk-mon-rescaled-shrinker-close-addendum}
There is $k_0\geq 0$ so that $\lim_{\tau\to\infty} d_V(\cN(\tau),\mathfrak{F}^{k_0,g_0}) = 0$. 
\end{lemma}
\begin{proof}
Since $\Sigma \in \cS_2$ has $F(\Sigma)\geq 1$ (by Brakke's theorem, cf.\ \cite{White:Brakke}) we can choose $k_0 > \Theta_\cM(X)$ and then apply \cite{White:topology-weak,Ilmanen:singularities} (cf.\ Proposition \ref{prop:properties.until.nongeneric.time}). 
\end{proof}
\begin{proposition}\label{prop:unique-tf-genus-mult}
Suppose that one tangent flow to $\cM$ at $(\bx,T_\textnormal{gen}(\cM))$ is associated to a genus $g$ shrinker with multiplicity $k$. Then, any other tangent flow to $\cM$ at $(\bx,T_\textnormal{gen}(\cM))$ is associated to a (possibly different) genus $g$ shrinker with multiplicity $k$. 
\end{proposition} 
\begin{proof}
We follow a similar strategy to the proof of Proposition \ref{prop:smooth-tf-all-are}. Write $\cN$ for the rescaled Brakke flow associated to $\cM$ at $(\bx,T_\textnormal{gen}(\cM))$. By assumption there is $\tau_i \to \infty$ with $\cN(\tau_i) \rightharpoonup k\cH^2\lfloor \Sigma \in k \cS_2^g$. Using Lemma \ref{lemm:gen.mult.fixed.shrink.closed} we can fix $\eps_0>0$ so that if $V \in \mathfrak{F}^{k_0,g_0}$ has $d_V(V,k'\cS_2(g'))\leq \eps_0$ for some $k' \in \{1,\dots,k_0\},g'\in\{0,\dots,g_0\}$ then $V \in k'\cS_2(g')$. Then, we can define
\[
\tau_i' : = \sup\{\tau' \geq \tau_i : d_V(\cN(\tau'),k\cS_2^g) \leq \eps_0/2 \textrm{ for all } \tau \in [\tau_i,\tau']\}. 
\]
If $\tau_i' = \infty$ for some $i$, then for $s_i\to\infty$ so that $\cN(s_i) \to V \in \mathfrak{F}^{k_0,g_0}$ we have $V \in k\cS_2^g$ by the choice of $\eps_0$. Thus, it remains to rule out $\tau_i' <\infty$ for all $i$. 

As in Proposition \ref{prop:smooth-tf-all-are}, Lemmas \ref{lemm:huisk-mon-rescaled-shrinker-close} and \ref{lemm:huisk-mon-rescaled-shrinker-close-addendum} yield  $\tau_i'-\tau_i\to\infty$. Moreover, up to a subsequence $\cN(\tau_i'-1)$ converges to $V \in \mathfrak{F}^{k_0,g_0}$ and by choice of $\tau_i'$ and $\eps_0$, we thus see that $V \in k \cS_2^g$. Lemma \ref{lemm:huisk-mon-rescaled-shrinker-close} then yields a contradiction, completing the proof. 
\end{proof}

\bibliographystyle{alpha}
\bibliography{generic4}

\end{document}